\pdfoutput=1
\documentclass[11pt]{amsart}
\usepackage{lmodern}
\usepackage{ifthen}
\usepackage{amsfonts}
\usepackage{amsxtra}
\usepackage{amssymb}
\usepackage{mathdots}
\usepackage{array}
\usepackage[margin=0.8in]{geometry}
\usepackage{xcolor}
\definecolor{cite}{rgb}{0.30,0.60,1.00}
\definecolor{url}{rgb}{0.00,0.00,0.80}
\definecolor{link}{rgb}{0.40,0.10,0.20}
\usepackage[colorlinks,linkcolor=link,urlcolor=url,citecolor=cite,pagebackref,breaklinks]{hyperref}
\usepackage{bbm}
\usepackage{mathtools}
\usepackage{mathrsfs}
\usepackage{appendix}
\usepackage[all]{xy}
\usepackage[lite,abbrev,msc-links,alphabetic]{amsrefs}

\usepackage{graphicx}
\usepackage{multirow}
\usepackage{pstricks}
\usepackage{pst-pdf}
\usepackage{enumitem}
\usepackage{pifont}
\usepackage{marvosym}
\usepackage{txfonts}

\usepackage[OT2,T1]{fontenc}
\DeclareSymbolFont{cyrletters}{OT2}{wncyr}{m}{n}
\DeclareMathSymbol{\Sha}{\mathalpha}{cyrletters}{"58}

\usepackage{graphicx}

\makeatletter
\providecommand*{\Dashv}{%
  \mathrel{%
    \mathpalette\@Dashv\vDash
  }%
}
\newcommand*{\@Dashv}[2]{%
  \reflectbox{$\m@th#1#2$}%
}
\makeatother

\usepackage{nameref}

\makeatletter
\let\orgdescriptionlabel\descriptionlabel
\renewcommand*{\descriptionlabel}[1]{%
  \let\orglabel\label
  \let\label\@gobble
  \phantomsection
  \edef\@currentlabel{#1\unskip}%
  \let\label\orglabel
  \orgdescriptionlabel{#1}%
}
\makeatother

\numberwithin{equation}{section}

\theoremstyle{plain}
\newtheorem{proposition}{Proposition}[subsection]

\newtheorem{corollary}[proposition]{Corollary}
\newtheorem{lem}[proposition]{Lemma}
\newtheorem{theorem}[proposition]{Theorem}

\theoremstyle{definition}
\newtheorem{definition}[proposition]{Definition}
\newtheorem{construction}[proposition]{Construction}
\newtheorem{notation}[proposition]{Notation}
\newtheorem{assumption}[proposition]{Assumption}
\newtheorem{hypothesis}[proposition]{Hypothesis}

\theoremstyle{remark}
\newtheorem{remark}[proposition]{Remark}

\renewcommand{\b}[1]{\mathbf{#1}}
\renewcommand{\c}[1]{\mathcal{#1}}
\renewcommand{\d}[1]{\mathbb{#1}}
\newcommand{\f}[1]{\mathfrak{#1}}
\renewcommand{\r}[1]{\mathrm{#1}}
\newcommand{\s}[1]{\mathscr{#1}}
\renewcommand{\sf}[1]{\mathsf{#1}}
\renewcommand{\(}{\left(}
\renewcommand{\)}{\right)}
\newcommand{\res}{\mathbin{|}}
\newcommand{\ol}[1]{\overline{#1}{}}

\newcommand{\ul}{\underline}
\renewcommand{\leq}{\leqslant}
\renewcommand{\geq}{\geqslant}

\newcommand{\bG}{\b G}

\newcommand{\bM}{\b M}
\newcommand{\bN}{\b N}

\newcommand{\bP}{\b P}
\newcommand{\bQ}{\b Q}
\newcommand{\bR}{\b R}

\newcommand{\bT}{\b T}

\newcommand{\bX}{\b X}

\newcommand{\bff}{\b f}
\newcommand{\bg}{\b g}
\newcommand{\bh}{\b h}
\newcommand{\bi}{\b i}

\newcommand{\cA}{\c A}
\newcommand{\cB}{\c B}
\newcommand{\cC}{\c C}
\newcommand{\cD}{\c D}
\newcommand{\cE}{\c E}
\newcommand{\cF}{\c F}
\newcommand{\cG}{\c G}
\newcommand{\cH}{\c H}
\newcommand{\cI}{\c I}

\newcommand{\cL}{\c L}
\newcommand{\cM}{\c M}
\newcommand{\cN}{\c N}
\newcommand{\cO}{\c O}

\newcommand{\cS}{\c S}
\newcommand{\cT}{\c T}

\newcommand{\cV}{\c V}

\newcommand{\cX}{\c X}

\newcommand{\dA}{\d A}
\newcommand{\dB}{\d B}
\newcommand{\dC}{\d C}

\newcommand{\dF}{\d F}

\newcommand{\dK}{\d K}

\newcommand{\dN}{\d N}

\newcommand{\dP}{\d P}
\newcommand{\dQ}{\d Q}

\newcommand{\dT}{\d T}

\newcommand{\dZ}{\d Z}
\newcommand{\fA}{\f A}

\newcommand{\fC}{\f C}
\newcommand{\fD}{\f D}

\newcommand{\fK}{\f K}
\newcommand{\fL}{\f L}

\newcommand{\fN}{\f N}

\newcommand{\fP}{\f P}
\newcommand{\fQ}{\f Q}

\newcommand{\fS}{\f S}
\newcommand{\fT}{\f T}
\newcommand{\fU}{\f U}

\newcommand{\fX}{\f X}
\newcommand{\fY}{\f Y}

\newcommand{\fa}{\f a}

\newcommand{\fc}{\f c}

\newcommand{\ff}{\f f}

\newcommand{\fl}{\f l}
\newcommand{\fm}{\f m}
\newcommand{\fn}{\f n}

\newcommand{\fp}{\f p}
\newcommand{\fq}{\f q}

\newcommand{\rB}{\r B}

\newcommand{\rD}{\r D}
\newcommand{\rE}{\r E}
\newcommand{\rF}{\r F}
\newcommand{\rG}{\r G}
\newcommand{\rH}{\r H}
\newcommand{\rI}{\r I}
\newcommand{\rJ}{\r J}
\newcommand{\rK}{\r K}
\newcommand{\rL}{\r L}
\newcommand{\rM}{\r M}
\newcommand{\rN}{\r N}

\newcommand{\rP}{\r P}
\newcommand{\rQ}{\r Q}
\newcommand{\rR}{\r R}
\newcommand{\rS}{\r S}
\newcommand{\rT}{\r T}
\newcommand{\rU}{\r U}
\newcommand{\rV}{\r V}
\newcommand{\rW}{\r W}
\newcommand{\rX}{\r X}
\newcommand{\rY}{\r Y}
\newcommand{\rZ}{\r Z}

\newcommand{\rb}{\r b}

\newcommand{\rd}{\r d}

\newcommand{\rf}{\r f}
\newcommand{\rg}{\r g}
\newcommand{\rh}{\r h}
\newcommand{\ri}{\r i}
\newcommand{\rj}{\r j}

\renewcommand{\rm}{\r m}

\newcommand{\rp}{\r p}

\newcommand{\sF}{\s F}

\newcommand{\sK}{\s K}
\newcommand{\sL}{\s L}
\newcommand{\sM}{\s M}

\newcommand{\sP}{\s P}

\newcommand{\sS}{\s S}
\newcommand{\sT}{\s T}
\newcommand{\sU}{\s U}
\newcommand{\sV}{\s V}
\newcommand{\sW}{\s W}
\newcommand{\sX}{\s X}

\newcommand{\sZ}{\s Z}

\newcommand{\sfH}{\sf H}

\newcommand{\sfR}{\sf R}

\newcommand{\sfT}{\sf T}

\newcommand{\sfm}{\sf m}
\newcommand{\sfn}{\sf n}

\newcommand{\sfp}{\sf p}

\newcommand{\sfs}{\sf s}

\newcommand{\sfv}{\sf v}

\newcommand{\sfx}{\sf x}

\newcommand{\tD}{\mathtt{D}}

\newcommand{\tF}{\mathtt{F}}

\newcommand{\tI}{\mathtt{I}}

\newcommand{\tP}{\mathtt{P}}
\newcommand{\tQ}{\mathtt{Q}}
\newcommand{\tR}{\mathtt{R}}

\newcommand{\tT}{\mathtt{T}}

\newcommand{\tV}{\mathtt{V}}

\newcommand{\tb}{\mathtt{b}}
\newcommand{\tc}{\mathtt{c}}

\newcommand{\ti}{\mathtt{i}}
\newcommand{\tj}{\mathtt{j}}

\newcommand{\tl}{\mathtt{l}}

\newcommand{\tw}{\mathtt{w}}

\newcommand{\balpha}{\boldsymbol{\alpha}}

\newcommand{\bkappa}{\boldsymbol{\kappa}}
\newcommand{\blambda}{\boldsymbol{\lambda}}

\newcommand{\bbL}{\boldsymbol{L}}
\newcommand{\bbM}{\boldsymbol{M}}
\newcommand{\bbT}{\boldsymbol{T}}
\newcommand{\bbV}{\boldsymbol{V}}

\newcommand{\obj}{\text{\Flatsteel}}
\newcommand{\qbinom}[2]{\genfrac{[}{]}{0pt}{}{#1}{#2}}

\newcommand{\pres}[1]{{}^{#1}}
\newcommand{\floor}[1]{\lfloor{#1}\rfloor}
\newcommand{\ceil}[1]{\lceil{#1}\rceil}

\newcommand{\ab}{\r{ab}}

\newcommand{\bad}{\r{bad}}

\newcommand{\CF}{\mathbbm{1}}

\newcommand{\cl}{\r{cl}}

\newcommand{\cosp}{\r{cosp}}
\newcommand{\cris}{\r{cris}}

\newcommand{\cyc}{\r{cyc}}

\newcommand{\dr}{\r{dR}}
\newcommand{\et}{{\acute{\r{e}}\r{t}}}

\newcommand{\free}{\r{free}}
\newcommand{\gr}{\r{gr}}
\renewcommand{\graph}{\triangle}

\newcommand{\HOM}{\c{H}om}

\newcommand{\id}{\r{id}}
\newcommand{\inc}{\r{inc}}

\newcommand{\loc}{\r{loc}}
\newcommand{\lr}{\r{lr}}

\newcommand{\mix}{\r{mix}}
\newcommand{\mnm}{\r{min}}

\newcommand{\no}{\r{no}}
\newcommand{\ns}{\r{ns}}
\newcommand{\ordi}{\r{ord}}

\newcommand{\ram}{\r{ram}}

\newcommand{\sh}{\r{sh}}

\newcommand{\sing}{\r{sing}}
\renewcommand{\sp}{\r{sp}}

\newcommand{\unr}{\r{unr}}
\newcommand{\ur}{\r{ur}}

\newcommand{\even}{\mathrm{even}}
\newcommand{\eff}{\mathrm{eff}}
\newcommand{\odd}{\r{odd}}

\newcommand{\Mod}{\sf{Mod}}
\newcommand{\Set}{\sf{Set}}
\newcommand{\Fun}{\sf{Fun}}
\newcommand{\Sch}{\sf{Sch}}
\newcommand{\defin}{\r{def}}
\newcommand{\indef}{\r{indef}}

\DeclareMathOperator{\AJ}{AJ}

\DeclareMathOperator{\Aut}{Aut}

\DeclareMathOperator{\BC}{BC}

\DeclareMathOperator{\Char}{char}
\DeclareMathOperator{\coker}{coker}

\DeclareMathOperator{\diag}{diag}
\DeclareMathOperator{\disc}{disc}

\DeclareMathOperator{\DL}{DL}
\DeclareMathOperator{\End}{End}
\DeclareMathOperator{\Ext}{Ext}
\DeclareMathOperator{\Fil}{Fil}

\DeclareMathOperator{\Fr}{Fr}

\DeclareMathOperator{\Gal}{Gal}

\DeclareMathOperator{\GL}{GL}
\DeclareMathOperator{\Gr}{Gr}
\DeclareMathOperator{\GU}{GU}
\DeclareMathOperator{\Hom}{Hom}
\DeclareMathOperator{\IP}{Im}
\DeclareMathOperator{\IM}{im}
\DeclareMathOperator{\ind}{ind}

\DeclareMathOperator{\IC}{IC}

\DeclareMathOperator{\Ker}{ker}

\DeclareMathOperator{\length}{length}
\DeclareMathOperator{\Lie}{Lie}

\DeclareMathOperator{\modulo}{mod}
\DeclareMathOperator{\Nm}{Nm}

\DeclareMathOperator{\ord}{ord}
\DeclareMathOperator{\Perv}{Perv}

\DeclareMathOperator{\rank}{rank}

\DeclareMathOperator{\Res}{Res}
\DeclareMathOperator{\Sel}{Sel}
\DeclareMathOperator{\Stub}{Stub}
\DeclareMathOperator{\Sh}{Sh}
\DeclareMathOperator{\SL}{SL}

\DeclareMathOperator{\Spec}{Spec}

\DeclareMathOperator{\Sym}{Sym}
\DeclareMathOperator{\Tor}{Tor}
\DeclareMathOperator{\Tr}{Tr}

\DeclareMathOperator{\val}{val}
\DeclareMathOperator{\Ver}{Ver}

\begin{document}

\title{Iwasawa's main conjecture for Rankin--Selberg motives in the anticyclotomic case}

\author{Yifeng Liu}
\address{Institute for Advanced Study in Mathematics, Zhejiang University, Hangzhou 310058, China}
\email{liuyf0719@zju.edu.cn}

\author{Yichao Tian}
\address{Morningside Center of Mathematics, AMSS, Chinese Academy of Sciences, Beijing 100190, China}
\email{yichaot@math.ac.cn}

\author{Liang Xiao}
\address{New Cornerstone Lab, Beijing International Center for Mathematical Research, Peking University, Beijing 100871, China}
\email{lxiao@bicmr.pku.edu.cn}

\date{\today}

\begin{abstract}
  In this article, we study the Iwasawa theory for cuspidal automorphic representations of $\GL(n)\times\GL(n+1)$ over CM fields along anticyclotomic directions, in the framework of the Gan--Gross--Prasad conjecture for unitary groups. We prove one-side divisibility of the corresponding Iwasawa main conjecture: when the global root number is $1$, the $p$-adic $L$-function belongs to the characteristic ideal of the Iwasawa Bloch--Kato Selmer group; when the global root number is $-1$, the square of the characteristic ideal of a certain Iwasawa module is contained in the characteristic ideal of the torsion part of the Iwasawa Bloch--Kato Selmer group (analogous to Perrin-Riou's Heegner point main conjecture).
\end{abstract}

\maketitle

\tableofcontents

\section{Introduction}
\label{ss:1}

The study of anticyclotomic Iwasawa theory dates back to the pioneer work of Bertolini and Darmon \cite{BD05}, in which they proved the one-side divisibility ($p$-adic $L$-function belongs to the characteristic ideal) of the (definite) Iwasawa main conjecture for rational elliptic curves along the anticyclotomic direction of an imaginary quadratic field, by constructing an Euler system that is a variant of Kolyvagin's original system \cite{Kol90},\footnote{During the preparation of this article, we noticed that the preprint \cite{BLV} has claimed a proof of the other direction of the definite anticyclotomic Iwasawa main conjecture for elliptic curves.} which was later streamlined by Howard known as the bipartite Euler system \cite{How06}. Around the same time, Howard \cite{How04} proved the (same) one-side divisibility of an indefinite version of the anticyclotomic Iwasawa main conjecture for rational elliptic curves proposed by Perrin-Riou \cite{PR87}. In years after, the results of Bertolini--Darmon and Howard have been generalized to modular forms of general weights or Hilbert modular forms (see, for example, \cites{PW11,Lon12,Fou13,VO13,CH15,BCK21,Wan21,Wan23}), all for $\GL(2)$ (or rather $\GL(1)\times\GL(2)$).

In this article, we study the anticyclotomic Iwasawa theory for cuspidal automorphic representations of $\GL(n)\times\GL(n+1)$ over CM fields for an arbitrary positive integer $n$, in the framework of the Gan--Gross--Prasad conjecture \cite{GGP12} for unitary groups. This article can also be viewed as an upgrade of \cite{LTXZZ} to the level of Iwasawa algebra.

Let $F/F^+$ be a totally imaginary quadratic extension of a totally real number field, with the maximal abelian extension $F^\ab$ of $F$ in a fixed algebraic closure $\ol{F}$ of $F$. We define an \emph{anticyclotomic extension} of $F$ to be an extension $\cF/F$ contained in $F^\ab$ such that $\cF/F^+$ is Galois for which the unique nontrivial element in $\Gal(F/F^+)$ acts on $\Gal(\cF/F)$ by $-1$. We say that an anticyclotomic extension $\cF/F$ is \emph{free} if $\Gal(\cF/F)$ is torsion free. For an anticyclotomic extension $\cF/F$, we denote by $\Sigma_\cF$ the set of primes of $F$ ramified in $\cF$ and $\Sigma_\cF^+$ the set of primes of $F^+$ underlying those in $\Sigma_\cF$.

\subsection{Main results: the symmetric power setting}
\label{ss:main1}

Consider two \emph{modular} elliptic curves $A$ and $A'$ over $F^+$ such that $\End(A_{\ol{F}})=\End(A'_{\ol{F}})=\dZ$ and an integer $n\geq 1$.

Let $p$ be a rational prime. We say that an anticyclotomic extension $\cF/F$ is \emph{$p$-ordinary} (with respect to $A,A'$) if every prime $v\in\Sigma_\cF^+$ divides $p$, splits in $F$, and satisfies that both $A_v$ and $A'_v$ have good ordinary reduction.\footnote{In particular, the trivial extension $F/F$ is a $p$-ordinary anticyclotomic extension.}

Take a free $p$-ordinary anticyclotomic extension $\cF/F$. Combining with the recent breakthrough on the automorphy of symmetric powers of Hilbert modular forms \cite{NT}, the automorphy of quadratic base change \cite{AC89}, and \cite{Liu5}*{Theorem~5.2}, one has an element
\[
\sL_\cF(\Sym^{n-1}A_F\times\Sym^nA'_F)\in\Lambda_{\cF,\dQ_p}\coloneqq\dZ_p[[\Gal(\cF/F)]]\otimes_{\dZ_p}\dQ_p,
\]
unique up to a scalar in $\dQ_p^\times$, satisfying the following interpolation property: There exists a constant $C\in\dC^\times$ such that for every \emph{finite} order character $\chi\colon\Gal(\cF/F)\to\ol\dQ_p^\times$ that is ramified at \emph{all} places in $\Sigma_\cF$ and every embedding $\iota\colon\ol\dQ_p\to\dC$, we have
\[
\iota\sL_\cF(\Sym^{n-1}A_F\times\Sym^nA'_F)(\chi)=C\cdot
\iota\prod_{v\in\Sigma_\cF^+}\(\frac{q_v^{\frac{n(n+1)(2n+1)}{6}}}{\alpha_v^{2\psi(n-1)}\alpha'^{2\psi(n)}_v}\)^{\fc_v(\chi)}
\cdot L(n,\Sym^{n-1}A_F\times\Sym^nA'_F\otimes\iota\chi),
\]
where $\psi(N)\coloneqq N^2+(N-2)^2+\cdots$ for an integer $N\geq 0$; and for every $v\in\Sigma_\cF^+$,
\begin{itemize}[label={\ding{118}}]
  \item $q_v$ denotes the residue cardinality of $F^+_v$;

  \item $\alpha_v$ and $\alpha'_v$ denote the unique root in $\dZ_p^\times$ of the polynomials $T^2-a(A_v)T+q_v$ and $T^2-a(A'_v)T+q_v$, respectively;

  \item $\fc_v(\chi)>0$ denotes the conductor of $\chi$ at (either place in $\Sigma_\cF$ above) $v$.
\end{itemize}
Note that $n$ is the center of the (complex) $L$-function $L(s,\Sym^{n-1}A_F\times\Sym^nA'_F\otimes\iota\chi)$ in the classical convention.

As usual, we denote by $\rT_p$ the $p$-adic Tate module (functor), and put $\rV_p\coloneqq\rT_p\otimes_{\dZ_p}\dQ_p$ and $\rW_p\coloneqq\rT_p\otimes_{\dZ_p}\dQ_p/\dZ_p$. For every \emph{finite} extension $F'/F$, we have the Bloch--Kato Selmer groups \cite{BK90}
\begin{align*}
\rH^1_f(F',\Sym^{n-1}\rW_p(A)\otimes_{\dZ_p}\Sym^n\rW_p(A')(1-n))
&\subseteq\rH^1(F',\Sym^{n-1}\rW_p(A)\otimes_{\dZ_p}\Sym^n\rW_p(A')(1-n)), \\
\rH^1_f(F',\Sym^{n-1}\rT_p(A)\otimes_{\dZ_p}\Sym^n\rT_p(A')(1-n))
&\subseteq\rH^1(F',\Sym^{n-1}\rT_p(A)\otimes_{\dZ_p}\Sym^n\rT_p(A')(1-n)),
\end{align*}
for the (canonically polarized) Galois representation $\Sym^{n-1}\rV_p(A)\otimes_{\dQ_p}\Sym^n\rV_p(A')(1-n)$ (see \S\ref{ss:main} for the precise definition). Define
\begin{align*}
&\sX(\cF,\Sym^{n-1}\rV_p(A)\otimes\Sym^n\rV_p(A')(1-n)) \\
&\quad\coloneqq
\Hom_{\dZ_p}\(\varinjlim_{F\subseteq F'\subseteq\cF}\rH^1_f(F',\Sym^{n-1}\rW_p(A)\otimes_{\dZ_p}\Sym^n\rW_p(A')(1-n)),\dQ_p/\dZ_p\)\otimes_{\dZ_p}\dQ_p, \\
&\sS(\cF,\Sym^{n-1}\rV_p(A)\otimes\Sym^n\rV_p(A')(1-n)) \\
&\quad\coloneqq\(\varprojlim_{F\subseteq F'\subseteq\cF}\rH^1_f(F',\Sym^{n-1}\rT_p(A)\otimes_{\dZ_p}\Sym^n\rT_p(A')(1-n))\)\otimes_{\dZ_p}\dQ_p,
\end{align*}
where the colimit/limit are taken with respect to restriction/corestriction maps; they are naturally finitely generated modules over $\Lambda_{\cF,\dQ_p}$.

\begin{theorem}\label{th:elliptic}
Let $n\geq 1$ be an integer that is at most $2$ when $F^+=\dQ$. Let $A$ and $A'$ be two modular elliptic curves over $F^+$ such that $\End(A_{\ol{F}})=\End(A'_{\ol{F}})=\dZ$ and that $A_{\ol{F}}$ and $A'_{\ol{F}}$ are not isogenous to each other. The following holds for all but finitely many rational primes $p$ (see Remark \ref{re:elliptic} below for precise conditions on $p$):

For every free $p$-ordinary anticyclotomic extension $\cF/F$ (so that $\Lambda_{\cF,\dQ_p}$ is a UFD), if $\sL_\cF(\Sym^{n-1}A_F\times\Sym^nA'_F)\neq 0$, then
\begin{enumerate}[label=(\alph*)]
  \item $\sX(\cF,\Sym^{n-1}\rV_p(A)\otimes\Sym^n\rV_p(A')(1-n))$ is a torsion $\Lambda_{\cF,\dQ_p}$-module;

  \item $\sS(\cF,\Sym^{n-1}\rV_p(A)\otimes\Sym^n\rV_p(A')(1-n))$ vanishes;

  \item $\sL_\cF(\Sym^{n-1}A_F\times\Sym^nA'_F)$ belongs to the characteristic ideal of the $\Lambda_{\cF,\dQ_p}$-module $\sX(\cF,\Sym^{n-1}\rV_p(A)\otimes\Sym^n\rV_p(A')(1-n))$.\footnote{For (c), one does not need the nonvanishing of $\sL_\cF(\Sym^{n-1}A_F\times\Sym^nA'_F)$ since otherwise the conclusion is automatic.}
\end{enumerate}
\end{theorem}

This theorem follows by combining Theorem \ref{th:main}(1) and Theorem \ref{th:admissible}(1) below. The above theorem recovers \cite{LTXZZ}*{Theorem~1.1.1} by taking $\cF=F$.

\subsection{Main results: the automorphic setting}
\label{ss:main2}

\begin{definition}\label{de:relevant}
We say that a complex representation $\Pi$ of $\GL_N(\dA_F)$ with $N\geq 1$ is \emph{relevant} if
\begin{enumerate}
  \item $\Pi$ is an irreducible cuspidal automorphic representation;

  \item $\Pi\circ\tc\simeq\Pi^\vee$, where $\tc\in\Gal(F/F^+)$ is the complex conjugation;

  \item for every archimedean place $w$ of $F$, $\Pi_w$ is isomorphic to the (irreducible) principal series representation induced by the characters $(\arg^{1-N},\arg^{3-N},\ldots,\arg^{N-3},\arg^{N-1})$, where $\arg\colon\dC^\times\to\dC^\times$ is the \emph{argument character} defined by the formula $\arg(z)\coloneqq z/\sqrt{z\ol{z}}$.
\end{enumerate}
It is clear that being relevant is stable under contragredient.
\end{definition}

Take an integer $n\geq 1$, and denote by $n_0$ and $n_1$ the unique even and odd numbers in $\{n,n+1\}$, respectively. Consider two relevant representations $\Pi_0$ and $\Pi_1$ of $\GL_{n_0}(\dA_F)$ and $\GL_{n_1}(\dA_F)$, respectively, and a strong coefficient field $E\subseteq\dC$ of both $\Pi_0$ and $\Pi_1$ (Remark \ref{re:galois}(3)).

Let $\wp$ be a $p$-adic place of $E$. We say that an anticyclotomic extension $\cF/F$ is \emph{$\wp$-ordinary} (with respect to $\Pi_0,\Pi_1$) if every prime $w\in\Sigma_\cF$ divides $p$, splits over $F^+$, and satisfies that both $\Pi_{0,w}$ and $\Pi_{1,w}$ are unramified and $\wp$-ordinary (Definition \ref{de:ordinary}).

Take a free $\wp$-ordinary anticyclotomic extension $\cF/F$. For $\alpha=0,1$, we have the Galois representation $\rho_{\Pi_\alpha,\wp}$ of $F$ with coefficients in $E_\wp$ associated with $\Pi_0$, normalized so that $\rho_{\Pi_\alpha^\vee,\wp}\simeq\rho_{\Pi_\alpha,\wp}^\vee(1-n_\alpha)$. In particular, $\rho_{\Pi_0,\wp}\otimes\rho_{\Pi_1,\wp}(n)$ is polarizable. Similarly, we have Selmer groups
\[
\sX(\cF,\rho_{\Pi_0,\wp}\otimes\rho_{\Pi_1,\wp}(n)),\quad
\sS(\cF,\rho_{\Pi_0,\wp}\otimes\rho_{\Pi_1,\wp}(n)),
\]
which are finitely generated modules over
\[
\Lambda_{\cF,E_\wp}\coloneqq\dZ_p[[\Gal(\cF/F)]]\otimes_{\dZ_p}E_\wp.
\]
In addition, we denote by $\sX_0(\cF,\rho_{\Pi_0,\wp}\otimes\rho_{\Pi_1,\wp}(n))$ the maximal $\Lambda_{\cF,E_\wp}$-torsion submodule of $\sX(\cF,\rho_{\Pi_0,\wp}\otimes\rho_{\Pi_1,\wp}(n))$.

By Definition \ref{de:relevant}(2), the global Rankin--Selberg epsilon factor $\epsilon(\Pi_0\times\Pi_1)$ is either $1$ or $-1$. In \cite{Liu5}, we construct (see \S\ref{ss:main} for more details)
\begin{itemize}[label={\ding{118}}]
  \item when $\epsilon(\Pi_0\times\Pi_1)=1$, an element
      \[
      \sL_\cF(\Pi_0\times\Pi_1)\in\Lambda_{\cF,E_\wp},
      \]
      well-defined up to a scalar in $E_\wp^\times$, interpolating central $L$-values $L(\tfrac{1}{2},(\Pi_0\times\Pi_1)\otimes\iota\chi)$ for \emph{finite} order characters $\chi\colon\Gal(\cF/F)\to(\ol{E_\wp})^\times$ and $E$-linear embeddings $\iota\colon\ol{E_\wp}\to\dC$;

  \item when $\epsilon(\Pi_0\times\Pi_1)=-1$, a $\Lambda_{\cF,E_\wp}$-submodule
     \[
     \sK(\cF,\rho_{\Pi_0,\wp}\otimes\rho_{\Pi_1,\wp}(n))
     \subseteq\sS(\cF,\rho_{\Pi_0,\wp}\otimes\rho_{\Pi_1,\wp}(n)),
     \]
     which is generated by, roughly speaking, the Abel--Jacobi image of a family of Gan--Gross--Prasad diagonal cycles in the anticyclotomic tower. Put
     \[
     \sT(\cF,\rho_{\Pi_0,\wp}\otimes\rho_{\Pi_1,\wp}(n))\coloneqq
     \frac{\sS(\cF,\rho_{\Pi_0,\wp}\otimes\rho_{\Pi_1,\wp}(n))}{\sK(\cF,\rho_{\Pi_0,\wp}\otimes\rho_{\Pi_1,\wp}(n))}.
     \]
\end{itemize}

\begin{notation}\label{no:iwasawa}
For every free $\wp$-ordinary anticyclotomic extension $\cF/F$ (so that $\Lambda_{\cF,E_\wp}$ is a UFD) and every finitely generated $\Lambda_{\cF,E_\wp}$-module $\sM$, we denote by $\Char_\cF(\sM)$ the characteristic ideal of $\sM$.
\end{notation}

\begin{theorem}[Theorem \ref{th:iwasawa}]\label{th:main}
Let $n\geq 1$ be an integer that is at most $2$ when $F^+=\dQ$. Let $\Pi_0$ and $\Pi_1$ be two relevant representations of $\GL_{n_0}(\dA_F)$ and $\GL_{n_1}(\dA_F)$, respectively, and $E\subseteq\dC$ a strong coefficient field of both $\Pi_0$ and $\Pi_1$. For every admissible prime $\wp$ of $E$ in the sense of Definition \ref{de:admissible} and every free $\wp$-ordinary anticyclotomic extension $\cF/F$, the following holds:
\begin{enumerate}
  \item Suppose that $\epsilon(\Pi_0\times\Pi_1)=1$. If $\sL_\cF(\Pi_0\times\Pi_1)\neq 0$, then
     \begin{enumerate}
       \item $\sX(\cF,\rho_{\Pi_0,\wp}\otimes\rho_{\Pi_1,\wp}(n))=\sX_0(\cF,\rho_{\Pi_0,\wp}\otimes\rho_{\Pi_1,\wp}(n))$;

       \item $\sS(\cF,\rho_{\Pi_0,\wp}\otimes\rho_{\Pi_1,\wp}(n))$ vanishes;

       \item $\sL_\cF(\Pi_0\times\Pi_1)$ belongs to $\Char_\cF\(\sX(\cF,\rho_{\Pi_0,\wp}\otimes\rho_{\Pi_1,\wp}(n))\)$.\footnote{For (1c), one does not need the nonvanishing of $\sL_\cF(\Pi_0\times\Pi_1)$ since otherwise the conclusion is automatic.}
     \end{enumerate}

  \item Suppose that $\epsilon(\Pi_0\times\Pi_1)=-1$. If $\sK(\cF,\rho_{\Pi_0,\wp}\otimes\rho_{\Pi_1,\wp}(n))\neq 0$, then
     \begin{enumerate}
       \item $\sX(\cF,\rho_{\Pi_0,\wp}\otimes\rho_{\Pi_1,\wp}(n))$ has $\Lambda_{\cF,E_\wp}$-rank one;

       \item $\sS(\cF,\rho_{\Pi_0,\wp}\otimes\rho_{\Pi_1,\wp}(n))$ is a torsion free $\Lambda_{\cF,E_\wp}$-module of rank one;

       \item $\Char_\cF\(\sT(\cF,\rho_{\Pi_0,\wp}\otimes\rho_{\Pi_1,\wp}(n))\)^2$ is contained in $\Char_\cF\(\sX_0(\cF,\rho_{\Pi_0,\wp}\otimes\rho_{\Pi_1,\wp}(n))\)$.
     \end{enumerate}
\end{enumerate}
\end{theorem}

The two cases of the above theorem confirm one direction of Conjecture 6.4 and Conjecture 7.7 in \cite{Liu5}, respectively, for admissible $\wp$. Together with Theorem \ref{th:admissible}(1) below, they generalize the works of \cite{BD05} and \cites{Ber95,How03,How04}, respectively, to higher rank groups (that is, $n>1$). On the other hand, by taking $\cF=F$, the two cases of the above theorem together with Remark \ref{re:kappa_special} recover Theorem 1.1.5 and Theorem 1.1.9 of \cite{LTXZZ}, respectively, with the minor caveat that the notion of admissible primes of $E$ in Definition \ref{de:admissible} is slightly different from the one in \cite{LTXZZ}*{Definition~8.1.1}. Nevertheless, we have the following result, whose proof will be given after the statement of Theorem \ref{th:iwasawa}.

\begin{theorem}\label{th:admissible}
Let $\Pi_0$ and $\Pi_1$ be two relevant representations of $\GL_{n_0}(\dA_F)$ and $\GL_{n_1}(\dA_F)$, respectively, and $E\subseteq\dC$ a strong coefficient field of both $\Pi_0$ and $\Pi_1$. If we are in the following three cases, then all but (effectively) finitely many primes of $E$ are admissible in the sense of Definition \ref{de:admissible}:
\begin{enumerate}
  \item $\Pi_0=(\Sym^{n_0-1}\pi_{A_0})_F$ and $\Pi_1=(\Sym^{n_1-1}\pi_{A_1})_F$ for modular elliptic curves $A_0$ and $A_1$ over $F^+$ (with the associated automorphic representations $\pi_{A_0}$ and $\pi_{A_1}$ of $\GL_2(\dA_{F^+})$, respectively) satisfying
      \begin{enumerate}[label=(\alph*)]
        \item $\End(A_{0\ol{F}})=\End(A_{1\ol{F}})=\dZ$;

        \item $A_{0\ol{F}}$ and $A_{1\ol{F}}$ are not isogenous to each other.
      \end{enumerate}

  \item $\Pi_0$ and $\Pi_1$ satisfy
      \begin{enumerate}[label=(\alph*)]
        \item there exists a very special inert prime $\fl$ of $F^+$ (Definition \ref{de:special_inert}) such that $\Pi_{0,\fl}$ is Steinberg, and $\Pi_{1,\fl}$ is unramified whose Satake parameter contains $1$ exactly once;\footnote{Note that the Satake parameter of $\Pi_{1,\fl}$ has to contain $1$ at least once by Definition \ref{de:relevant}(2).}

        \item for $\alpha=0,1$, there exists a nonarchimedean place $w_\alpha$ of $F$ such that $\Pi_{\alpha,w_\alpha}$ is supercuspidal.
      \end{enumerate}

  \item $\Pi_0$ and $\Pi_1$ satisfy
      \begin{enumerate}[label=(\alph*)]
        \item there exists a very special inert prime $\fl$ of $F^+$ (Definition \ref{de:special_inert}) such that $\Pi_{0,\fl}$ is almost unramified, and $\Pi_{1,\fl}$ is unramified whose Satake parameter contains $1$ exactly once;

        \item there exists a nonarchimedean place $w$ of $F$ such that $\Pi_{0,w}$ is Steinberg and $\Pi_{1,w}$ is unramified.

        \item for $\alpha=0,1$, there exists a nonarchimedean place $w_\alpha$ of $F$ such that $\Pi_{\alpha,w_\alpha}$ is supercuspidal.
      \end{enumerate}
\end{enumerate}
\end{theorem}

In the above theorem, case (1) and case (2) have already appeared in \cite{LTXZZ}. Case (3) is added in order to have application in \cite{DZ}*{Theorem~C}. Indeed, there are no examples satisfying the conditions in \cite{DZ}*{Theorem~C} and those in either case (1) or case (2) simultaneously when $n\geq 3$.

\begin{remark}\label{re:elliptic}
In Case (1) of Theorem \ref{th:admissible}, we give more explicit conditions for admissible primes only using arithmetics of $A_0$ and $A_1$. It suffices to consider the case $E=\dQ$ as the conditions in Definition \ref{de:admissible} are stable under extension of coefficients. Consider
\begin{description}
  \item[(S1)] $p>2(n_0+1)$;

  \item[(S2)] $p$ is unramified in $F$ and both $A_0$ and $A_1$ have good reduction at every $p$-adic place of $F^+$;

  \item[(S3)] the natural homomorphism $\Gal(\ol{F}/F_{\r{rflx}})\to\GL_{\dF_p}(A_0[p](\ol{F}))\times\GL_{\dF_p}(A_1[p](\ol{F}))$ has the largest possible image (that is, isomorphic to $\rG(\SL_2(\dF_p)\times\SL_2(\dF_p))$), where $F_{\r{rflx}}\subseteq\ol{F}$ is the reflexive closure of $F$ (Definition \ref{de:reflexive});\footnote{In particular, $F_{\r{rflx}}=F$ when $F$ is Galois or contains an imaginary quadratic field.}

  \item[(S4)] the polynomial
      \[
      XY\cdot\prod_{i=1}^{n_0}(1-(-X)^i)\cdot\prod_{(i,j)\in I\times J\setminus{\{(1,0)\}}}(X^iY^j-X)\cdot\prod_{j\in J}(Y^j+X)\in\dZ[X,Y]
      \]
      is not identically zero on $\dF_p\times\dF_p$, where $I\coloneqq\{i\in\dZ\res|i|\leq n_0-1, i\text{ odd}\}$ and $J\coloneqq\{j\in\dZ\res|j|\leq n_1-1, i\text{ even}\}$;\footnote{For example, this condition excludes $\{2,3\}$ and $\{2,3,5,7\}$ when $n=1$ and $n=2$, respectively.}

  \item[(S5)] for every prime of $w$ of $F$ whose underlying rational prime divides the conductors of either $F$, $A_0$, or $A_1$, if we let $F'_w/F_w$ be the minimal totally ramified extension over which $A_0$ has good or multiplicative reduction, then
      \begin{itemize}[label={\ding{118}}]
        \item $p\nmid[F'_w:F_w]$;

        \item $p\nmid\prod_{i=1}^{n_0}(q_w^i-1)$, where $q_w$ denotes the residue cardinality of $F_w$;

        \item when $A_0$ has good reduction over $F'_w$, the two roots $\alpha,\beta$ of the Satake polynomial of $A_0$ over $F'_w$ belong to $(\ol\dZ_{(p)})^\times$ such that the image of
            \[
            \left\{(\alpha/\beta)^{n_0-1},(\alpha/\beta)^{n_0-3},\ldots,(\alpha/\beta)^{3-n_0},(\alpha/\beta)^{1-n_0}\right\}
            \]
            in $\ol\dF_p^\times$ does not contain $q_w$;

        \item when $A_0$ has multiplicative reduction over $F'_w$, $p$ does not divide the $F_w$-valuation of $j(A_0)$.
      \end{itemize}
\end{description}
It is obvious that (S1,S2,S4,S5) together exclude only finitely many $p$; so does (S3) by the famous result of Serre \cite{Ser72}.

We claim that if $p$ satisfies the five conditions above, then it is an admissible prime in the sense of Definition \ref{de:admissible}. Indeed, (A1) is nothing but (S1); (A2) is implied by (S2); (A3) is implied by (S3); (A4) is implied by (S3); (A5) is implied by (S4); (A6) is implied by (S3) and (S5) by the proof of \cite{LTXZZ2}*{Proposition~4.1}.
\end{remark}

Theorem \ref{th:main}(1) has a standard corollary, whose proof will be given at the end of \S\ref{ss:euler}. It confirms half of the $p$-adic Beilinson--Bloch--Kato conjecture for the motive associated with the Rankin--Selberg product of $\Pi_0$ and $\Pi_1$.

\begin{corollary}\label{co:main}
Let $n\geq 1$ be an integer that is at most $2$ when $F^+=\dQ$. Let $\Pi_0$ and $\Pi_1$ be two relevant representations of $\GL_{n_0}(\dA_F)$ and $\GL_{n_1}(\dA_F)$, respectively, and $E\subseteq\dC$ a strong coefficient field of both $\Pi_0$ and $\Pi_1$. Suppose that $\epsilon(\Pi_0\times\Pi_1)=1$. Then for every admissible prime $\wp$ of $E$ in the sense of Definition \ref{de:admissible} and every free $\wp$-ordinary anticyclotomic extension $\cF/F$,
\[
\dim_{E_\wp}\rH^1_f(F,\rho_{\Pi_0,\wp}\otimes\rho_{\Pi_1,\wp}(n))
\leq\sup\left\{m\geq 0\left|\: \sL_\cF(\Pi_0\times\Pi_1)\in \fA_\cF^m\right.\right\},
\]
where $\fA_\cF$ denotes the augmentation ideal of $\Lambda_{\cF,E_\wp}$.
\end{corollary}

\begin{remark}
In \cite{Liu5}, it is conjectured that if $\rho_{\Pi_0,\wp}\otimes\rho_{\Pi_1,\wp}$ is absolutely irreducible, then $\sL_\cF(\Pi_0\times\Pi_1)$ or $\sK(\cF,\rho_{\Pi_0,\wp}\otimes\rho_{\Pi_1,\wp}(n))$, depending on whether $\epsilon(\Pi_0\times\Pi_1)=1$ or $\epsilon(\Pi_0\times\Pi_1)=-1$, is automatically nonzero as long as the following condition holds: there exists $v\in\Sigma^+_\cF$ such that the maximal extension $\cF_{(v)}/F$ contained $\cF$ satisfying $\Sigma^+_{\cF_{(v)}}=\{v\}$ has $\dZ_p$-rank $[F^+_v:\dQ_p]$. However, currently we are unable to prove this, except for the case where $n=1$, which follows from \cite{CV07}.
\end{remark}

\begin{remark}
The condition that $n\leq 2$ when $F^+=\dQ$ in Theorem \ref{th:elliptic}, Theorem \ref{th:main}, and Corollary \ref{co:main} is (only) due to Hypothesis \ref{hy:unitary_cohomology}. In other words, once Hypothesis \ref{hy:unitary_cohomology} is known for $N\geq 4$ when $F^+=\dQ$, this condition can be removed.
\end{remark}

\begin{remark}
Unlike most references concerning the Iwasawa theory for Selmer groups, we allow $\cF/F$ to ramify at more than one ($p$-adic) places, and even allow places ramified in $\cF/F$ to have infinite residue extension degree. The latter generality forces us to develop some new types of control theorems in the specialization of Iwasawa Selmer modules, which is the main focus for the second half of \S\ref{ss:selmer}.
\end{remark}

\begin{remark}
The Iwasawa's main conjectures we study in this article are the ones inverting $p$. It is natural to ask whether one can obtain similar results in the $p$-integral Iwasawa algebra. Indeed, under certain conditions on $(\Pi_0,\Pi_1)$ so that we have a good new form theory locally, it is possible to do this; and we plan to continue the investigation along this direction in the future.
\end{remark}

\subsection{Structure and strategy}
\label{ss:structure}

For simplicity of the introduction, we assume that the global root number is $1$ and that the coefficient field $E$ is $\dQ$. The fundamental ideal for proving our theorems is the same as that of Bertolini--Darmon, namely, we will construct a bipartite Euler system over the Iwasawa algebra. A bipartite Euler system comes with two formulae relating local Galois cohomology with $p$-adic period integrals (or rather sums), known as the first and the second explicit reciprocity laws (in the Iwasawa algebra).

For every integer $k\geq 1$, the Euler system modulo $p^k$ is indexed by a set $\fN_k$ consisting of finite sets of certain ``level-raising'' primes of the base field $F$ modulo $p^k$ (these primes are all inert over $F^+$ and coprime to $p$). For every element $\fn\in\fN_k$, we will have a $\dZ_p/p^k$-module $\cC^{\fn,k}$, which encodes the ``amount'' of level-raising phenomenon at places in $\fn$. The explicit reciprocity laws themselves provide maps from $\cC^{\fn,k}$ to $\cC^{\fn',k}$ when $\fn'\setminus\fn$ is a singleton, which turn out to be isomorphisms (first/second according to whether the cardinality of $\fn$ is even/odd).

Readers may notice that in our prequel \cite{LTXZZ}, the reciprocity laws already appeared but without the variation in the Iwasawa algebra. One clear task for us in this article is to upgrade them to the level of Iwasawa algebra. On the other hand, we also need to know that maps from $\cC^{\fn,k}$ to $\cC^{\fn',k}$ encoded in the reciprocity laws are \emph{isomorphisms} -- this was proved in \cite{LTXZZ} only for the first law. For the second law, using certain Galois deformation argument, we reduce the problem to the \emph{surjectivity} of the integral $p$-adic Abel--Jacobi map from the basic locus of certain unitary Shimura varieties after suitable Hecke localization, which generalizes the following famous theorem of Ribet \cite{Rib90}: the integral $p$-adic Abel--Jacobi map from the supersingular locus of the modular curve $X_0(N)_{\dF_\ell}$ over $\dF_\ell$ with $\ell\nmid pN$ is surjective after localizing at a non-Eisenstein ideal of the Hecke algebra away from $N\ell$.

The proof of Ribet's theorem is to reduce the surjectivity of the Abel--Jacobi map to the surjectivity in the Ihara lemma for modular curves, in which the latter is known. We follow the same strategy by considering unitary Shimura varieties with the Siegel parahoric level. We formulate the correct version of the Ihara type lemma for unitary Shimura varieties at inert primes that is responsible for the surjectivity of the Abel--Jacobi map, via the so-called \emph{Tate--Thompson local system}, which is analogous to the rank-$\ell$ local system over $X_0(N)_{\dQ_\ell}$ corresponding to the Steinberg representation. In order to connect with the basic locus of the special fiber, we have to compute the nearby cycles of the Tate--Thompson local system; this is one of major technical difficulties of the article. As a result, we show that the surjectivity of a certain map, which we call the \emph{boosting map}, defined in terms of the cohomology of the nearby cycles of the Tate--Thompson local system, implies the surjectivity of the Abel--Jacobi map from the basic locus, after same Hecke localizations. For the surjectivity of the boosting map itself, it can be reduced to the surjectivity of a map only involving \emph{generic} fibers (see Remark \ref{re:significance}) so that we can apply a certain ``changing prime'' trick. As a result, we are able to prove it for a large class of maximal Hecke ideals (for which we take localization) that are sufficient for our construction of the bipartite Euler system.\footnote{However, currently we are not able to prove the full Ihara type lemma under our expectation. See Remark \ref{re:ihara}.}

We start the main part of the article by collecting in \S\ref{ss:2} certain background materials from the prequel \cite{LTXZZ} that will be systematically used later. In \S\ref{ss:3}, we study the image of the basic locus of unitary Shimura varieties at good inert primes under the Abel--Jacobi map and its relation with the boosting map involving the Tate--Thompson local system. In \S\ref{ss:4}, we upgrade the two explicit reciprocity laws from \cite{LTXZZ}*{\S7} to the level of Iwasawa algebra. In \S\ref{ss:5}, we give the precise statement and the proof of our main theorems toward the Iwasawa main conjecture formulated in \cite{Liu5}, based on the (generalized) bipartite Euler system we construct.

\subsection{Notations and conventions}
\label{ss:notation}

In this subsection, we setup some common notations and conventions for the entire article, including appendices, unless otherwise specified.

\subsubsection*{Generalities:}

\begin{itemize}[label={\ding{118}}]
  \item Denote by $\dN=\{0,1,2,3,\ldots\}$ the monoid of nonnegative integers.

  \item For a set $S$, we denote by $\CF_S$ the characteristic function of $S$.

  \item The eigenvalues or generalized eigenvalues of a matrix over a field $k$ are counted with multiplicity (namely, dimension of the corresponding eigenspace or generalized eigenspace); in other words, they form a multi-subset of an algebraic extension of $k$.

  \item For every rational prime $p$, we fix an algebraic closure $\ol\dQ_p$ of $\dQ_p$ with the residue field $\ol\dF_p$. For every integer $r\geq 1$, we denote by $\dQ_{p^r}$ the subfield of $\ol\dQ_p$ that is an unramified extension of $\dQ_p$ of degree $r$, by $\dZ_{p^r}$ its ring of integers, and by $\dF_{p^r}$ its residue field.

  \item For a nonarchimedean place $v$ of a number field $K$, we write $\|v\|$ for the cardinality of the residue field of $K_v$.

  \item We use standard notations from the category theory. The category of sets is denoted by $\Set$. For a category $\fC$, we denote by $\fC^{\r{op}}$ its opposite category, and denote by $\fC_{/A}$ the category of morphisms to $A$ for an object $A$ of $\fC$. For another category $\fD$, we denote by $\Fun(\fC,\fD)$ the category of functors from $\fC$ to $\fD$. We also use the symbol $\obj$ to indicate a virtual object.

  \item All rings are commutative and unital; and ring homomorphisms preserve units. For a (topological) ring $L$, a (topological) $L$-ring is a (topological) ring $R$ together with a (continuous) ring homomorphism from $L$ to $R$. However, we use the word \emph{algebra} in the general sense, which is not necessarily commutative or unital.

  \item If a base ring is not specified in the tensor operation $\otimes$, then it is $\dZ$.

  \item For a ring $L$ and a set $S$, denote by $L[S]$ the $L$-module of $L$-valued functions on $S$ of finite support.

  \item For an algebra $A$, we denote by $\Mod(A)$ the category of left $A$-modules.
\end{itemize}

\subsubsection*{Algebraic geometry:}

\begin{itemize}[label={\ding{118}}]
  \item We denote by the category of schemes by $\Sch$ and its full subcategory of locally Noetherian schemes by $\Sch'$. For a scheme $S$ (resp. Noetherian scheme $S$), we denote by $\Sch_{/S}$ (resp.\ $\Sch'_{/S}$) the category of $S$-schemes (resp.\ locally Noetherian $S$-schemes). If $S=\Spec R$ is affine, we also write $\Sch_{/R}$ (resp.\ $\Sch'_{/R}$) for $\Sch_{/S}$ (resp.\ $\Sch'_{/S}$).

  \item The structure sheaf of a scheme $X$ is denoted by $\cO_X$.

  \item For two schemes $X,Y$ over an affine scheme $\Spec R$, we write $X\times_RY$ for $X\times_{\Spec R}Y$. When $Y=\Spec S$ for an $R$-ring $S$, we write $X\otimes_RS$ or even $X_S$ for $X\times_{\Spec R}\Spec S$.

  \item For a scheme $S$ in characteristic $p$ for some rational prime $p$, we denote by $\sigma\colon S\to S$ the absolute $p$-power Frobenius morphism. For a perfect field $\kappa$ of characteristic $p$, we denote by $W(\kappa)$ its Witt ring, and by abuse of notation, $\sigma\colon W(\kappa)\to W(\kappa)$ the canonical lifting of the $p$-power Frobenius map.

  \item For a smooth morphism $S\to T$ of schemes, we denote by $\cT_{S/T}$ the relative tangent sheaf, which is a locally free $\cO_S$-module.

  \item For a scheme $S$ and (sheaves of) $\cO_S$-modules $\cF$ and $\cG$, we denote by $\HOM(\cF,\cG)$ the quasi-coherent sheaf of $\cO_S$-linear homomorphisms from $\cF$ to $\cG$.

  \item For a positive integer $r$, we denote by $\GL_r$ the scheme of invertible $r$-by-$r$ matrices. Then $\GL_1$ is simply the multiplicative group $\bG_m\coloneqq\dZ[T,T^{-1}]$; but we will distinguish between $\GL_1$ and $\bG_m$ according to the context.

  \item For a number field $K$, a commutative group scheme $G\to S$ equipped with an action by $O_K$ over some base scheme $S$, and an ideal $\fa\subset O_K$, we denote by $G[\fa]$ the maximal closed subgroup scheme of $G$ annihilated by all elements in $\fa$.

  \item By a \emph{coefficient ring} for \'{e}tale cohomology, we mean either a finite ring, or a finite extension of $\dQ_\ell$, or the ring of integers of a finite extension of $\dQ_\ell$. In the latter two cases, we regard the \'{e}tale cohomology as the continuous one. We say that a coefficient ring $L$ is \emph{$n$-coprime} for a positive integer $n$ if $n$ is invertible in $L$ in the first case, and $\ell\nmid n$ in the latter two cases.
\end{itemize}

\subsubsection*{Ground fields:}

\begin{itemize}[label={\ding{118}}]
  \item Let $\tc\in\Aut(\dC/\dQ)$ be the complex conjugation.

  \item Throughout the article, we fix a \emph{subfield} $F\subseteq\dC$ that is a CM number field.

  \item Let $F^+\subseteq F$ be the maximal subfield on which $\tc$ acts by the identity.

  \item Let $\ol{F}$ be \emph{the} Galois closure of $F$ in $\dC$ and denote by $F^\ab$ the maximal abelian extension of $F$ in $\ol{F}$. Put $\Gamma_\cF\coloneqq\Gal(\ol{F}/\cF)$ for every subfield $\cF\subseteq\ol{F}$.

  \item Denote by $\Sigma$ and $\Sigma^+$ the sets of places of $F$ and $F^+$, respectively. For every $w\in\Sigma$, denote by $w^\tc\in\Sigma$ its complex conjugation.

  \item For every finite set $\Box$ of places of $\dQ$, denote by $\Sigma_\Box$ and $\Sigma^+_\Box$ the sets of places of $F$ and $F^+$ above $\Box$, respectively. When $\Box=\{v\}$ is a singleton, we simply write $\Sigma_v$ and $\Sigma^+_v$ for $\Sigma_{\{v\}}$ and $\Sigma^+_{\{v\}}$, respectively. Denote by $\tau_\infty\in\Sigma_\infty$ the default complex embedding of $F$.

  \item Denote by $\Sigma^+_\bad$ (resp.\ $\Sigma_\bad$) the union of $\Sigma^+_p$ (resp.\ $\Sigma_p$) for all $p$ that ramifies in $F$.

  \item For every place $v$ of $F^+$, we put $F_v\coloneqq F\otimes_{F^+}F^+_v$, fix an algebraic closure $\ol{F}^+_v$ of $F^+_v$ containing $\ol{F}$, and put $\Gamma_{F^+_v}\coloneqq\Gal(\ol{F}^+_v/F^+_v)$ as a subgroup of $\Gamma_{F^+}$.

  \item For every nonarchimedean place $w$ of $F$, we
      \begin{itemize}
        \item identify the Galois group $\Gamma_{F_w}$ with $\Gamma_{F^+_v}\cap\Gamma_F$ (resp.\ $\tc(\Gamma_{F^+_v}\cap\Gamma_F)\tc$), where $v$ is the underlying place of $F^+$, if the embedding $F\hookrightarrow\ol{F}^+_v$ induces (resp.\ does not induce) the place $w$;

        \item let $\rI_{F_w}\subseteq\Gamma_{F_w}$ be the inertia subgroup;

        \item let $\kappa_w$ be the residue field of $F_w$, and identify its Galois group $\Gamma_{\kappa_w}$ with $\Gamma_{F_w}/\rI_{F_w}$;

        \item denote by $\phi_w\in\Gamma_{F_w}$ a lifting of the \emph{arithmetic} Frobenius element in $\Gamma_{\kappa_w}$.
      \end{itemize}
\end{itemize}

\subsection*{Acknowledgements}

We would like to thank Ruiqi~Bai, Hao~Fu, David~Hansen, and Xuhua~He for useful comments and discussion, and Dongwen~Liu and Binyong~Sun for sharing their early manuscript \cite{LS}. We also like to thank Wei~Zhang and Xinwen~Zhu for the collaboration on the prequels \cites{LTXZZ,LTXZZ2} of this article. Y.~L. is supported by the National Key R\&D Program of China No.~2022YFA1005300. Y.~T. is supported by the National Natural Science Foundation of China (No.~12225112, 12231001), and CAS Project for Young Scientists in Basic Research (Grant No.~YSBR-033). L.~X. is supported by a grant from New Cornerstone Foundation, and grants from National Natural Science Foundation of China (No.~12231001, 12321001, 12071004).

\section{Recollection on hermitian structures}
\label{ss:2}

In this section, we collect certain background materials from the prequel \cite{LTXZZ} that will be systematically used in the current article as well, with some complements. Let $N\geq 1$ be an integer.

\subsection{Unitary Satake parameters}
\label{ss:hermitian_space}

Let $\Pi$ be a relevant representation of $\GL_N(\dA_F)$ (Definition \ref{de:relevant}) with the coefficient field $\dQ(\Pi)\subseteq\dC$ \cite{LTXZZ}*{Definition~3.1.1}.

\begin{definition}[Abstract Satake parameter]\label{de:satake_parameter}
Let $L$ be a ring. For a multi-subset $\balpha\coloneqq\{\alpha_1,\ldots,\alpha_N\}\subseteq L$, we put
\[
P_{\balpha}(T)\coloneqq\prod_{i=1}^N(T-\alpha_i)\in L[T].
\]
Consider a nonarchimedean place $v$ of $F^+$ not in $\Sigma^+_\bad$.
\begin{enumerate}
  \item Suppose that $v$ is inert in $F$. We define an \emph{(abstract) Satake parameter} in $L$ at $v$ of rank $N$ to be a multi-subset $\balpha\subseteq L$ of cardinality $N$. We say that $\balpha$ is \emph{unitary} if $P_{\balpha}(T)=(-T)^N\cdot P_{\balpha}(T^{-1})$.

  \item Suppose that $v$ splits in $F$. We define an \emph{(abstract) Satake parameter} in $L$ at $v$ of rank $N$ to be a pair $\balpha\coloneqq(\balpha_1;\balpha_2)$ of multi-subsets $\balpha_1,\balpha_2\subseteq L$ of cardinality $N$, indexed by the two places $w_1,w_2$ of $F$ above $v$. We say that $\balpha$ is \emph{unitary} if $P_{\balpha_1}(T)=c\cdot T^N\cdot P_{\balpha_2}(T^{-1})$ for some constant $c\in L^\times$.
\end{enumerate}
For two Satake parameters $\balpha_0$ and $\balpha_1$ in $L$ at $v$ of rank $n_0$ and $n_1$, respectively, we may form their tensor product $\balpha_0\otimes\balpha_1$ in the obvious way, which is of rank $n_0n_1$. If $\balpha_0$ and $\balpha_1$ are both unitary, then so is $\balpha_0\otimes\balpha_1$.
\end{definition}

\begin{notation}\label{no:satake}
We denote by $\Sigma_\Pi$ the smallest (finite) set of nonarchimedean places of $F$ containing $\Sigma_\bad$ such that $\Pi_w$ is unramified for every $w\in\Sigma\setminus(\Sigma_\infty\cup\Sigma_\Pi)$, and by $\Sigma^+_\Pi\subseteq\Sigma^+$ the subset underlying $\Sigma_\Pi$.
\begin{enumerate}
  \item For every $w\in\Sigma\setminus(\Sigma_\infty\cup\Sigma_\Pi)$, let
      \[
      \balpha(\Pi_w)\coloneqq\{\alpha(\Pi_w)_1,\ldots,\alpha(\Pi_w)_N\}\subseteq\dC
      \]
      be the Satake parameter of $\Pi_w$.

  \item For every $v\in\Sigma^+\setminus(\Sigma^+_\infty\cup\Sigma^+_\Pi)$,
      \begin{itemize}[label={\ding{118}}]
        \item if $v$ is inert in $F$, then we put $\balpha(\Pi_v)\coloneqq\balpha(\Pi_w)$ for the unique place $w$ of $F$ above $v$;

        \item if $v$ splits in $F$ into two places $w_1$ and $w_2$, then we put $\balpha(\Pi_v)\coloneqq(\balpha(\Pi_{w_1});\balpha(\Pi_{w_2}))$;
      \end{itemize}
      in particular, $\balpha(\Pi_v)$ is a unitary Satake parameter in $\dC$ at $v$ of rank $N$.
\end{enumerate}
\end{notation}

\begin{remark}\label{re:galois}
Let $\Pi$ be a relevant representation of $\GL_N(\dA_F)$ (Definition \ref{de:relevant}). Parts (1) and (2) below are recorded in \cite{LTXZZ}*{Proposition~3.2.4}.
\begin{enumerate}
  \item For every nonarchimedean place $w$ of $F$, $\Pi_w$ is tempered.

  \item For every rational prime $\ell$ and every isomorphism $\iota_\ell\colon\dC\xrightarrow{\sim}\ol\dQ_\ell$, there is a semisimple continuous homomorphism
      \[
      \rho_{\Pi,\iota_\ell}\colon\Gamma_F\to\GL_N(\ol\dQ_\ell),
      \]
      unique up to conjugation, satisfying that for every nonarchimedean place $w$ of $F$, the Frobenius semisimplification of the associated Weil--Deligne representation of $\rho_{\Pi,\iota_\ell}\res_{\Gamma_{F_w}}$ corresponds to the irreducible admissible representation $\iota_\ell\Pi_w|\det|_w^{\frac{1-N}{2}}$ of $\GL_N(F_w)$ under the local Langlands correspondence. Moreover, $\rho_{\Pi,\iota_\ell}^\tc$ and $\rho_{\Pi,\iota_\ell}^\vee(1-N)$ are conjugate.

  \item By \cite{CH13}*{Proposition~3.2.5}, there exists a number field $E\subseteq\dC$ (which we call a \emph{strong coefficient field} of $\Pi$ in \cite{LTXZZ}*{Definition~3.2.5}) containing $\dQ(\Pi)$ such that for every isomorphism $\iota_\ell\colon\dC\xrightarrow{\sim}\ol\dQ_\ell$, the representation $\rho_{\Pi,\iota_\ell}$ has coefficients in $E_\lambda$, where $\lambda$ is the induced $\ell$-adic place of $E$.
\end{enumerate}
\end{remark}

\begin{definition}\label{de:satake_condition}
Let $v$ be a nonarchimedean place of $F^+$ inert in $F$, and $L$ a ring in which $\|v\|$ is invertible. Let $P\in L[T]$ be a monic polynomial of degree $N$ satisfying $P(T)=(-T)^N\cdot P(T^{-1})$.
\begin{enumerate}
  \item When $N$ is odd, we say that $P$ is \emph{Tate generic at $v$} if $P'(1)$ is invertible in $L$.

  \item When $N$ is odd, we say that $P$ is \emph{intertwining generic at $v$} if $P(-\|v\|)$ is invertible in $L$.

  \item When $N$ is even, we say that $P$ is \emph{level-raising special at $v$} if $P(\|v\|)=0$ and $P'(\|v\|)$ is invertible in $L$.

  \item When $N$ is even, we say that $P$ is \emph{intertwining generic at $v$} if $P(-1)$ is invertible in $L$.
\end{enumerate}
\end{definition}

The following definition will be used in \S\ref{ss:5}.

\begin{definition}\label{de:ordinary}
Let $\wp$ be a prime of a finite extension $E/\dQ(\Pi)$ inside $\dC$ (with the underlying rational prime $p$). For every $w\in\Sigma\setminus(\Sigma_\infty\cup\Sigma_\Pi)$, we say that $\Pi_w$ is \emph{$\wp$-ordinary} if one can order $\balpha(\Pi_w)$ in a way such that for $1\leq i\leq N$, $\|w\|^{\frac{N+1}{2}-i}\alpha(\Pi_w)_i\in\ol\dZ_p^\times$ under every isomorphism $\dC\simeq\ol\dQ_p$ that induces $\wp$.
\end{definition}

\subsection{Hermitian spaces, unitary Shimura varieties and sets}
\label{ss:unitary_shimura}

\begin{definition}[Hermitian space]\label{de:hermitian_space}
Let $R$ be an $O_{F^+}[(\Sigma^+_\bad)^{-1}]$-ring. A \emph{hermitian space} over $O_F\otimes_{O_{F^+}}R$ of rank $N$ is a projective $O_F\otimes_{O_{F^+}}R$-module $\rV$ of rank $N$ together with a perfect pairing
\[
(\;,\;)_\rV\colon\rV\times\rV\to O_F\otimes_{O_{F^+}}R
\]
that is $O_F\otimes_{O_{F^+}}R$-linear in the first variable and $(O_F\otimes_{O_{F^+}}R,\tc\otimes\id_R)$-linear in the second variable, and satisfies $(x,y)_\rV=(y,x)_\rV^\tc$ for $x,y\in\rV$. We denote by $\rU(\rV)$ the group of $O_F\otimes_{O_{F^+}}R$-linear isometries of $\rV$, which is a reductive group over $R$.

Moreover, we denote by $\rV_\sharp$ the hermitian space $\rV\oplus O_F\otimes_{O_{F^+}}R\cdot 1$, where $1$ has norm $1$. For an $O_F\otimes_{O_{F^+}}R$-linear isometry $f\colon \rV\to\rV'$, we have the induced isometry $f_\sharp\colon\rV_\sharp\to\rV'_\sharp$.
\end{definition}

Let $v$ be a nonarchimedean place of $F^+$ not in $\Sigma^+_\bad$. Let $\Lambda_{N,v}$ be the unique up to isomorphism hermitian space over $O_{F_v}=O_F\otimes_{O_{F^+}}O_{F^+_v}$ of rank $N$, and $\rU_{N,v}$ its unitary group over $O_{F^+_v}$. We have the \emph{local spherical Hecke algebra}
\[
\dT_{N,v}\coloneqq\dZ[\rU_{N,v}(O_{F^+_v})\backslash\rU_{N,v}(F^+_v)/\rU_{N,v}(O_{F^+_v})],
\]
with the unit element $\CF_{\rU_{N,v}(O_{F^+_v})}$.

\begin{definition}[Abstract unitary Hecke algebra]\label{de:abstract_hecke}
For a finite subset $\Box\subseteq\Sigma^+\setminus\Sigma^+_\infty$ containing $\Sigma^+_\bad$, we define the \emph{abstract unitary Hecke algebra away from $\Sigma^+$} to be the restricted tensor product
\[
\dT_N^{\Box}\coloneqq{\bigotimes_v}'\dT_{N,v}
\]
over all places $v\in\Sigma^+\setminus(\Sigma^+_\infty\cup\Box)$ with respect to unit elements.
\end{definition}

We have the \emph{Satake homomorphism}
\[
\phi_\Pi\colon\dT_N^{\Sigma^+_\Pi}\to O_{\dQ(\Pi)}
\]
from \cite{LTXZZ}*{Construction~3.1.10}.

Now we introduce some categories of open compact subgroups, which will be used later.

\begin{definition}\label{de:neat_category}
Let $\rV$ be a hermitian space over $F$ of rank $N$. Let $\Box\subseteq\Sigma^+\setminus\Sigma^+_\infty$ be a finite subset. Recall the notion of neat open compact subgroups from \cite{LTXZZ}*{Definition~3.1.11(1)}.
\begin{enumerate}
  \item We define a category $\fK(\rV)^\Box$ whose objects are neat open compact subgroups $\rK$ of $\rU(\rV)(\dA_{F^+}^{\infty,\Box})$, and a morphism from $\rK$ to $\rK'$ is an element $g\in\rK\backslash\rU(\rV)(\dA_{F^+}^{\infty,\Box})/\rK'$ satisfying $g^{-1}\rK g\subseteq\rK'$. Denote by $\fK'(\rV)^\Box$ the subcategory of $\fK(\rV)^\Box$ that allows only identity double cosets as morphisms.

  \item We define a category $\fK(\rV)_\sp^\Box$ whose objects are pairs $\rK=(\rK_\flat,\rK_\sharp)$, where $\rK_\flat$ is an object of $\fK(\rV)^\Box$ and $\rK_\sharp$ is an object of $\fK(\rV_\sharp)^\Box$ such that $\rK_\flat\subseteq\rK_\sharp\cap\rU(\rV)(\dA_{F^+}^{\infty,\Box})$, and a morphism from $\rK=(\rK_\flat,\rK_\sharp)$ to $\rK'=(\rK'_\flat,\rK'_\sharp)$ is an element $g\in\rK_\flat\backslash\rU(\rV)(\dA_{F^+}^{\infty,\Box})/\rK'_\flat$ such that $g^{-1}\rK_\flat g\subseteq\rK'_\flat$ and $g^{-1}\rK_\sharp g\subseteq\rK'_\sharp$.\footnote{The subscript ``sp'' indicates that this notation will be related the special homomorphism of Shimura varieties later.} We have the obvious functors
      \begin{align*}
      \obj_\flat\colon\fK(\rV)_\sp^\Box\to\fK(\rV)^\Box,\quad
      \obj_\sharp\colon\fK(\rV)_\sp^\Box\to\fK(\rV_\sharp)^\Box
      \end{align*}
      sending $\rK=(\rK_\flat,\rK_\sharp)$ to $\rK_\flat$ and $\rK_\sharp$, respectively. Note that $\fK(\rV)_\sp^\Box$ is a non-full subcategory of $\fK(\rV)^\Box\times\fK(\rV_\sharp)^\Box$.
\end{enumerate}
When $\Box$ is the empty set, we suppress it from all the notations above. When $\Box=\Sigma^+_p$ for a rational prime $p$, we simply put $p$ on the various superscripts.
\end{definition}

We introduce hermitian spaces over $F$ that will be used in this article.

\begin{definition}\label{de:standard_hermitian_space}
Let $\rV$ be a hermitian space over $F$ of rank $N$.
\begin{enumerate}
  \item We say that $\rV$ is \emph{standard definite} if it has signature $(N,0)$ at every place in $\Sigma^+_\infty$.

  \item We say that $\rV$ is \emph{standard indefinite} if it has signature $(N-1,1)$ at $\tau_\infty\res_{F^+}$ and $(N,0)$ at other places in $\Sigma^+_\infty$.
\end{enumerate}
\end{definition}

For a standard indefinite hermitian space $\rV$ over $F$ of rank $N$, we have a functor
\begin{align*}
\Sh(\rV,\obj)\colon \fK(\rV)&\to\Sch_{/F} \\
\rK&\mapsto\Sh(\rV,\rK)
\end{align*}
of Shimura varieties associated with the reductive group $\Res_{F^+/\dQ}\rU(\rV)$ and the Deligne homomorphism described in \cite{LTXZZ}*{\S3.2}.

For a standard definite hermitian space $\rV$ over $F$ of rank $N$, we have a functor
\begin{align*}
\Sh(\rV,\obj)\colon \fK(\rV)&\to\Set \\
\rK&\mapsto\Sh(\rV,\rK)\coloneqq\rU(\rV)(F^+)\backslash\rU(\rV)(\dA_{F^+}^\infty)/\rK.
\end{align*}

\begin{hypothesis}\label{hy:unitary_cohomology}
For every standard indefinite hermitian space $\rV$ over $F$ of rank $N$, every discrete automorphic representation $\pi$ of $\rU(\rV)(\dA_{F^+})$ such that $\BC(\pi)$ \cite{LTXZZ}*{Definition~3.2.3} exists and is a relevant representation of $\GL_N(\dA_F)$, and every isomorphism $\iota_\ell\colon\dC\xrightarrow{\sim}\ol\dQ_\ell$, if $\rho_{\BC(\pi),\iota_\ell}$ (Remark \ref{re:galois}(2)) is irreducible, then
\[
W^{N-1}(\pi)\coloneqq\Hom_{\ol\dQ_\ell[\rU(\rV)(\dA_{F^+}^\infty)]}
\(\iota_\ell\pi^\infty,\varinjlim_{\fK'(\rV)}\rH^{N-1}_\et(\Sh(\rV,\rK)_{\ol{F}},\ol\dQ_\ell)\)
\]
is isomorphic to the underlying $\ol\dQ_\ell[\Gamma_F]$-module of $\rho_{\BC(\pi),\iota_\ell}^\tc$.
\end{hypothesis}

\begin{remark}\label{pr:unitary_cohomology}
By \cite{LTXZZ}*{Proposition~3.2.11}, Hypothesis \ref{hy:unitary_cohomology} holds for $N\leq 3$, and for $N>3$ if $F^+\neq\dQ$.
\end{remark}

\subsection{Generalized CM type and reflexive closure}
\label{ss:cm}

We denote by $\dN[\Sigma_\infty]$ the commutative monoid freely generated by the set $\Sigma_\infty$, which admits an action of $\Aut(\dC/\dQ)$ via the set $\Sigma_\infty$.

\begin{definition}\label{de:cm_type}
A \emph{generalized CM type of rank $N$} is an element
\[
\Psi=\sum_{\tau\in\Sigma_\infty}r_\tau\tau\in\dN[\Sigma_\infty]
\]
satisfying $r_\tau+r_{\tau^\tc}=N$ for every $\tau\in\Sigma_\infty$. For such $\Psi$, we define its \emph{reflex field} $F_\Psi\subseteq\dC$ to be the fixed subfield of the stabilizer of $\Psi$ in $\Aut(\dC/\dQ)$. A \emph{CM type} is simply a generalized CM type of rank $1$.
\end{definition}

\begin{definition}\label{de:reflexive}
We define the \emph{reflexive closure} of $F$, denoted by $F_{\r{rflx}}$, to be the subfield of $\dC$ generated by $F$ and the intersection of $F_\Phi$ for all CM types $\Phi$ of $F$. Put $F^+_{\r{rflx}}\coloneqq(F_{\r{rflx}})^{\tc=1}$.
\end{definition}

\begin{remark}
It is clear that $F_{\r{rflx}}$ is a CM field finite Galois over $F$; $F_{\r{rflx}}^+$ is the maximal totally real subfield of $F_{\r{rflx}}$ and is finite Galois over $F^+$. In many cases, we have $F_{\r{rflx}}=F$ and hence $F_{\r{rflx}}^+=F^+$, for example, when $F$ is Galois or contains an imaginary quadratic field.
\end{remark}

\begin{definition}\label{de:special_inert}
We say that a prime $\fp$ of $F^+$ is \emph{special inert} if the following are satisfied:
\begin{enumerate}
  \item $\fp$ is inert in $F$;

  \item the underlying rational prime $p$ of $\fp$ is odd and is unramified in $F$;

  \item $\fp$ is of degree one over $\dQ$, that is, $F^+_\fp=\dQ_p$.
\end{enumerate}
By abuse of notation, we also denote by $\fp$ for its induced prime of $F$.

We say that a special inert prime $\fp$ of $F^+$ is \emph{very special inert} if it is special inert and splits completely in $F_{\r{rflx}}^+$.\footnote{This is equivalent to that for every prime $\fq$ of $F^+$ above $p$ that is inert in $F$, $[F^+_{\fq}:\dQ_p]$ is odd.}
\end{definition}

In what follows in this article, we will often take a rational prime $p$ that is unramified in $F$, and an isomorphism $\iota_p\colon\dC\xrightarrow{\sim}\ol\dQ_p$. By composing with $\iota_p$, we regard $\Sigma_\infty$ also as the set of $p$-adic embeddings of $F$. We also regard $\dQ_p$ as a subfield of $\dC$ via $\iota_p^{-1}$.

\begin{notation}\label{no:p_notation}
We introduce the following important notations.
\begin{enumerate}
  \item In what follows, whenever we introduce some finite unramified extension $\dQ_?^?$ of $\dQ_p$, we denote by $\dZ_?^?$ its ring of integers and put $\dF_?^?\coloneqq\dZ_?^?/p\dZ_?^?$.

  \item For every $\tau\in\Sigma_\infty$, we denote by $\dQ_p^\tau\subseteq\dC$ the composition of $\tau(F)$ and $\dQ_p$, which is unramified over $\dQ_p$. For a scheme $S\in\Sch_{/\dZ_p^\tau}$ and an $\cO_S$-module $\cF$ with an action $O_F\to\End_{\cO_S}(\cF)$, we denote by $\cF_\tau$ the maximal $\cO_S$-submodule of $\cF$ on which $O_F$ acts via the homomorphism $\tau\colon O_F\to\dZ_p^\tau\to\cO_S$.

  \item We denote by $\dQ_p^\diamondsuit\subseteq\dC$ the composition of $\dQ_p^\tau$ for all $\tau\in\Sigma_\infty$, which is unramified over $\dQ_p$. We can identify $\Sigma_\infty$ with $\Hom(O_F,\dZ_p^\diamondsuit)=\Hom(O_F,\dF_p^\diamondsuit)$. In particular, the $p$-power Frobenius map $\sigma$ acts on $\Sigma_\infty$.

  \item For a generalized CM type $\Psi$ of rank $N$, we denote by $\dQ_p^\Psi\subseteq\dC$ the composition of $\dQ_p$, $F$, and $F_\Psi$, which is contained in $\dQ_p^\diamondsuit$.

  \item For a (functor in) scheme over $\dZ^?_?$ written like $\pres\varsigma\bX_?(\cdot\cdot\cdot)$, we put $\pres\varsigma\rX_?(\cdot\cdot\cdot)\coloneqq\pres\varsigma\bX_?(\cdot\cdot\cdot)\otimes_{\dZ^?_?}\dF^?_?$ and $\pres\varsigma\bX^\eta_?(\cdot\cdot\cdot)\coloneqq\pres\varsigma\bX_?(\cdot\cdot\cdot)\otimes_{\dZ^?_?}\dQ^?_?$. For a (functor in) scheme over $\dF^?_?$ written like $\pres\varsigma\rX^?_?(\cdot\cdot\cdot)$, we put $\ol{\pres\varsigma\rX}^?_?(\cdot\cdot\cdot)\coloneqq\pres\varsigma\rX^?_?(\cdot\cdot\cdot)\otimes_{\dF^?_?}\ol\dF_p$. Similar conventions are applied to morphisms as well. In practice, $\varsigma$ will be a certain similitude factor.
\end{enumerate}
\end{notation}

\subsection{Unitary abelian schemes}
\label{ss:unitary_abelian_scheme}

We first introduce some general notations about abelian schemes.

\begin{notation}
Let $A$ be an abelian scheme over a scheme $S$. We denote by $A^\vee$ the \emph{dual abelian variety} of $A$ over $S$. We denote by $\rH^\dr_1(A/S)$ (resp.\ $\Lie_{A/S}$, and $\omega_{A/S}$) for the \emph{relative de Rham homology} (resp.\ \emph{Lie algebra}, and \emph{dual Lie algebra}) of $A/S$, all regarded as locally free $\cO_S$-modules. We have the following \emph{Hodge exact sequence}
\begin{align}\label{eq:hodge_sequence}
0 \to \omega_{A^\vee/S} \to \rH^\dr_1(A/S) \to \Lie_{A/S} \to 0
\end{align}
of sheaves on $S$. When the base $S$ is clear from the context, we sometimes suppress it from the notation.
\end{notation}

\begin{definition}[Unitary abelian scheme]\label{de:unitary_abelian_scheme}
We prescribe a subring $\dP\subseteq\dQ$. Let $S$ be a scheme in $\Sch_{/\dP}$.
\begin{enumerate}
  \item An \emph{$O_F$-abelian scheme} over $S$ is a pair $(A,i)$ in which $A$ is an abelian scheme over $S$ and $i\colon O_F\to\End_S(A)\otimes\dP$ is a homomorphism of algebras sending $1$ to the identity endomorphism.

  \item A \emph{unitary $O_F$-abelian scheme} over $S$ is a triple $(A,i,\lambda)$ in which $(A,i)$ is an $O_F$-abelian scheme over $S$, and $\lambda\colon A\to A^\vee$ is a quasi-polarization such that $i(a^\tc)^\vee\circ\lambda=\lambda\circ i(a)$ for every $a\in O_F$, and there exists $c\in\dP^\times$ making $c\lambda$ a polarization.

  \item For two $O_F$-abelian schemes $(A,i)$ and $(A',i')$ over $S$, a (quasi-)homomorphism from $(A,i)$ to $(A',i')$ is a (quasi-)homomorphism $\varphi\colon A\to A'$ such that $\varphi\circ i(a)=i'(a)\circ\varphi$ for every $a\in O_F$. We will usually refer to such $\varphi$ as an $O_F$-linear (quasi-)homomorphism.
\end{enumerate}
Moreover, we will usually suppress the notion $i$ if the argument is insensitive to it.
\end{definition}

\begin{definition}[Signature type]\label{de:signature}
Let $\Psi$ be a generalized CM type of rank $N$ (Definition \ref{de:cm_type}). Consider a scheme $S\in\Sch_{/O_{F_\Psi}\otimes\dP}$. We say that an $O_F$-abelian scheme $(A,i)$ over $S$ has \emph{signature type} $\Psi$ if for every $a\in O_F$, the characteristic polynomial of $i(a)$ on $\Lie_{A/S}$ is given by
\[
\prod_{\tau\in\Sigma_\infty}(T-\tau(a))^{r_\tau}\in\cO_S[T].
\]
\end{definition}

\begin{construction}\label{cs:hermitian_structure}
Let $K$ be an $O_{F_\Psi}\otimes\dP$-ring that is an algebraically closed field. Suppose that we are given a unitary $O_F$-abelian scheme $(A_0,i_0,\lambda_0)$ over $K$ of signature type $\Phi$ that is a CM type, and a unitary $O_F$-abelian scheme $(A,i,\lambda)$ over $K$ of signature type $\Psi$. For every set $\Box$ of places of $\dQ$ containing $\infty$ and the characteristic of $K$, if not zero, we construct a hermitian space
\[
\Hom_{F\otimes_\dQ\dA^\Box}^{\lambda_0,\lambda}(\rH^\et_1(A_0,\dA^\Box),\rH^\et_1(A,\dA^\Box))
\]
over $F\otimes_\dQ\dA^\Box=F\otimes_{F^+}(F^+\otimes_\dQ\dA^\Box)$, with the underlying $F\otimes_\dQ\dA^\Box$-module
\[
\Hom_{F\otimes_\dQ\dA^\Box}(\rH^\et_1(A_0,\dA^\Box),\rH^\et_1(A,\dA^\Box))
\]
equipped with the pairing
\[
(x,y)\coloneqq i_0^{-1}\((\lambda_{0*})^{-1}\circ y^\vee\circ\lambda_*\circ x\)\in i_0^{-1}\End_{F\otimes_\dQ\dA^\Box}(\rH^\et_1(A_0,\dA^\Box))=F\otimes_\dQ\dA^\Box.
\]
\end{construction}

Now we take a rational prime $p$ that is unramified in $F$, and take the prescribed subring $\dP$ in Definition \ref{de:unitary_abelian_scheme} to be $\dZ_{(p)}$. We also choose an isomorphism $\iota_p\colon\dC\xrightarrow\sim\ol\dQ_p$ and adopt Notation \ref{no:p_notation}.

\begin{definition}\label{de:p_quasi}
Let $A$ and $B$ be two abelian schemes over a scheme $S\in\Sch_{/\dZ_{(p)}}$. We say that a quasi-homomorphism (resp.\ quasi-isogeny) $\varphi\colon A\to B$ is a \emph{quasi-$p$-homomorphism} (resp.\ \emph{quasi-$p$-isogeny}) if there exists some $c\in\dZ_{(p)}^\times$ such that $c\varphi$ is a homomorphism (resp.\ isogeny). A quasi-isogeny $\varphi$ is \emph{prime-to-$p$} if both $\varphi$ and $\varphi^{-1}$ are quasi-$p$-isogenies. We say that a quasi-polarization $\lambda$ of $A$ is \emph{$p$-principal} if $\lambda$ is a prime-to-$p$ quasi-isogeny.
\end{definition}

Note that for a unitary $O_F$-abelian scheme $(A,i,\lambda)$, the quasi-polarization $\lambda$ is a quasi-$p$-isogeny. To continue, take a generalized CM type $\Psi=\sum_{\tau\in\Sigma_\infty}r_\tau\tau$ of rank $N$.

\begin{remark}\label{re:hodge_sequence}
Let $A$ be an $O_F$-abelian scheme of signature type $\Psi$ over a scheme $S\in\Sch_{/\dZ_p^\tau}$ for some $\tau\in\Sigma_\infty$. Then \eqref{eq:hodge_sequence} induces a short exact sequence
\[
0 \to \omega_{A^\vee/S,\tau} \to \rH^\dr_1(A/S)_\tau \to \Lie_{A/S,\tau} \to 0
\]
of locally free $\cO_S$-modules of ranks $N-r_\tau$, $N$, and $r_\tau$, respectively. If $S$ belongs to $\Sch_{/\dZ_p^\diamondsuit}$, then we have decompositions
\begin{align*}
\rH^\dr_1(A/S)&=\bigoplus_{\tau\in\Sigma_\infty}\rH^\dr_1(A/S)_\tau,\\
\Lie_{A/S}&=\bigoplus_{\tau\in\Sigma_\infty}\Lie_{A/S,\tau},\\
\omega_{A/S}&=\bigoplus_{\tau\in\Sigma_\infty}\omega_{A/S,\tau}
\end{align*}
of locally free $\cO_S$-modules.
\end{remark}

\begin{notation}\label{no:weil_pairing}
Take $\tau\in\Sigma_\infty$. Let $(A,\lambda)$ be a unitary $O_F$-abelian scheme of signature type $\Psi$ over a scheme $S\in\Sch_{/\dZ_p^\tau}$. We denote
\[
\langle\;,\;\rangle_{\lambda,\tau}\colon\rH^\dr_1(A/S)_\tau\times\rH^\dr_1(A/S)_{\tau^\tc}\to\cO_S
\]
the $\cO_S$-bilinear pairing induced by the quasi-polarization $\lambda$, which is perfect if and only if $\lambda$ is $p$-principal. Moreover, for an $\cO_S$-submodule $\cF\subseteq\rH^\dr_1(A/S)_\tau$, we denote by $\cF^\perp\subseteq\rH^\dr_1(A/S)_{\tau^\tc}$ its (right) orthogonal complement under the above pairing, if $\lambda$ is clear from the context.
\end{notation}

We now review some notation for abelian schemes in characteristic $p$.

\begin{notation}\label{no:frobenius_verschiebung}
Let $A$ be an abelian scheme over a scheme $S\in\Sch_{/\dF_p}$. Put
\[
A^{(p)}\coloneqq A\times_{S,\sigma}S,
\]
where $\sigma$ is the absolute Frobenius morphism of $S$.
Then we
\begin{enumerate}
  \item have a canonical isomorphism $\rH^\dr_1(A^{(p)}/S)\simeq\sigma^*\rH^\dr_1(A/S)$ of $\cO_S$-modules;

  \item have the Frobenius homomorphism $\Fr_A\colon A\to A^{(p)}$ which induces the \emph{Verschiebung map}
      \[
      \tV_A\coloneqq(\Fr_A)_*\colon\rH^\dr_1(A/S)\to\rH^\dr_1(A^{(p)}/S)
      \]
      of $\cO_S$-modules;

  \item have the Verschiebung homomorphism $\Ver_A\colon A^{(p)}\to A$ which induces the \emph{Frobenius map}
      \[
      \tF_A\coloneqq(\Ver_A)_*\colon\rH^\dr_1(A^{(p)}/S)\to\rH^\dr_1(A/S)
      \]
      of $\cO_S$-modules;

  \item define $\nu_{A/S}\subseteq\rH^\dr_1(A/S)$ to be the image of $\tF_A$.
\end{enumerate}
For a subbundle $H$ of $\rH^\dr_1(A/S)$, we denote by $H^{(p)}$ the subbundle of $\rH^\dr_1(A^{(p)}/S)$ that corresponds to $\sigma^*H$ under the isomorphism in (1). In what follows, we will suppress $A$ in the notations $\tF_A$ and $\tV_A$ if the reference to $A$ is clear.
\end{notation}

In Notation \ref{no:frobenius_verschiebung}, we have $\Ker\tF=\IM\tV=\omega_{A^{(p)\vee}/S}$ and $\Ker\tV=\IM\tF$. Take $\tau\in\Sigma_\infty$. For a scheme $S\in\Sch_{/\dF_p^\tau}$ and an $O_F$-abelian scheme $A$ over $S$, we have $(\rH^\dr_1(A/S)_\tau)^{(p)}=\rH^\dr_1(A^{(p)}/S)_{\sigma\tau}$ under Notations \ref{no:p_notation} and \ref{no:frobenius_verschiebung}.

\begin{notation}\label{re:frobenius_verschiebung}
Suppose that $S=\Spec R$ for a ring $R$ of characteristic $p$. Then we have a canonical isomorphism $\rH^\dr_1(A^{(p)}/R)\simeq\rH^\dr_1(A/R)\otimes_{R,\sigma}R$.
\begin{enumerate}
  \item By abuse of notation, we have
     \begin{itemize}[label={\ding{118}}]
        \item the $(R,\sigma)$-linear Frobenius map $\tF\colon\rH^\dr_1(A/R)\to\rH^\dr_1(A/R)$ and

        \item if $R$ is perfect, the $(R,\sigma^{-1})$-linear Verschiebung map $\tV\colon\rH^\dr_1(A/R)\to\rH^\dr_1(A/R)$.
     \end{itemize}

  \item When $R=\kappa$ is a perfect field, recall that we have the \emph{covariant} Dieudonn\'{e} module $\cD(A)$ associated with the $p$-divisible group $A[p^\infty]$, which is a free $W(\kappa)$-module, such that $\cD(A)/p\cD(A)$ is canonically isomorphic to $\rH^\dr_1(A/\kappa)$. Again by abuse of notation, we have
     \begin{itemize}[label={\ding{118}}]
       \item the $(W(\kappa),\sigma)$-linear Frobenius map $\tF\colon\cD(A)\to\cD(A)$ lifting the one above, and

       \item the $(W(\kappa),\sigma^{-1})$-linear Verschiebung map $\tV\colon\cD(A)\to\cD(A)$ lifting the one above,
     \end{itemize}
     respectively, satisfying $\tF\circ\tV=\tV\circ\tF=p$.

  \item When $R=\kappa$ is a perfect field and contains $\dF_p^\tau$ for some $\tau\in\Sigma_\infty$, applying Notation \ref{no:p_notation} to the $W(\kappa)$-module $\cD(A)$, we obtain $W(\kappa)$-submodules $\cD(A)_{\sigma^i\tau}\subseteq\cD(A)$ for every $i\in\dZ$. Thus, we obtain
     \begin{itemize}[label={\ding{118}}]
       \item the $(W(\kappa),\sigma)$-linear Frobenius map $\tF\colon\cD(A)_\tau\to\cD(A)_{\sigma\tau}$ and

       \item the $(W(\kappa),\sigma^{-1})$-linear Verschiebung map $\tV\colon\cD(A)_\tau\to\cD(A)_{\sigma^{-1}\tau}$
     \end{itemize}
     by restriction. We have canonical isomorphisms and inclusions:
     \[
     \tV\cD(A)_{\sigma\tau}/p\cD(A)_\tau\simeq\omega_{A^\vee,\tau}\subseteq
     \cD(A)_\tau/p\cD(A)_\tau\simeq\rH^\dr_1(A)_\tau.
     \]
\end{enumerate}
\end{notation}

To end this subsection, we propose the following definition analogous to Definition \ref{de:unitary_abelian_scheme}, which will be used later.

\begin{definition}[Unitary $p$-truncated group]\label{de:unitary_divisible}
We call a finite flat group scheme that is annihilated by $p$ a \emph{$p$-truncated group}. Let $S$ be a scheme in $\Sch_{/\dZ_{p^2}}$.
\begin{enumerate}
  \item A \emph{unitary $p$-truncated group} over $S$ is a tripe $(G,i,\lambda)$ in which
     \begin{itemize}[label={\ding{118}}]
       \item $G$ is a $p$-truncated group over $S$;

       \item $i\colon\dF_{p^2}\to\End_S(G)$ is an algebra homomorphism sending $1$ to the identity endomorphism;

       \item $\lambda\colon G\to G^\vee$ is a \emph{principal} polarization such that $i(a^\tc)^\vee\circ\lambda=\lambda\circ i(a)$ for every $a\in\dF_{p^2}$.
     \end{itemize}

  \item A unitary $p$-truncated group over $S$ is \emph{superspecial} if it is isomorphic to a copy of supersingular unitary $p$-truncated groups of height $2$.

  \item A unitary $p$-truncated group over $S$ has signature $(r,s)$ if for every $a\in\dF_{p^2}$, the characteristic polynomial of $i(a)$ on $\Lie_{G/S}$ is given by $(T-a)^r(T-a^\tc)^s\in\cO_S[T]$.
\end{enumerate}
Moreover, we will usually suppress the notion $i$ if the argument is insensitive to it.
\end{definition}

\subsection{A CM moduli scheme}
\label{ss:cm_moduli}

In this subsection, we introduce an auxiliary moduli scheme parameterizing certain CM abelian varieties.

\begin{definition}
Let $R$ be a $\dZ[(\disc F)^{-1}]$-ring.
\begin{enumerate}
  \item A \emph{rational skew-hermitian space} over $O_F\otimes R$ of rank $N$ is a free $O_F\otimes R$-module $\rW$ of rank $N$ together with an $R$-bilinear skew-symmetric perfect pairing
      \[
      \langle\;,\;\rangle_\rW\colon\rW\times\rW\to R
      \]
      satisfying $\langle ax,y\rangle_\rW=\langle x,a^\tc y\rangle_\rW$ for every $a\in O_F\otimes R$ and $x,y\in\rW$.

  \item Let $\rW$ and $\rW'$ be two rational skew-hermitian spaces over $O_F\otimes R$, a map $f\colon\rW\to\rW'$ is a \emph{similitude} if $f$ is an $O_F\otimes R$-linear isomorphism such that there exists some $c(f)\in R^\times$ satisfying $\langle f(x),f(y)\rangle_{\rW'}=c(f)\langle x,y\rangle_\rW$ for every $x,y\in\rW$.

  \item Two rational skew-hermitian spaces over $O_F\otimes R$ are \emph{similar} if there exists a similitude between them.

  \item For a rational skew-hermitian space $\rW$ over $O_F\otimes R$, we denote by $\GU(\rW)$ its \emph{group of similitude} as a reductive group over $R$; it satisfies that for every ring $R'$ over $R$, $\GU(\rW)(R')$ is the set of self-similitude of the rational skew-hermitian space $\rW\otimes_RR'$ over $O_F\otimes R'$.
\end{enumerate}
\end{definition}

We define a subtorus $\rT_0\subseteq(\Res_{O_F/\dZ}\bG_m)\otimes\dZ[(\disc F)^{-1}]$ such that for every $\dZ[(\disc F)^{-1}]$-ring $R$, we have
\[
\rT_0(R)=\{a\in O_F\otimes R\res \Nm_{F/{F^+}}a\in R^\times\}.
\]

Now we take a rational prime $p$ that is unramified in $F$. We take the prescribed subring $\dP$ in Definition \ref{de:unitary_abelian_scheme} to be $\dZ_{(p)}$.

\begin{remark}\label{re:hasse_principle}
Let $\rW_0$ be a rational skew-hermitian space over $O_F\otimes\dZ_{(p)}$ of rank $1$. Then $\GU(\rW_0)$ is identified with $\rT_0\otimes_{\dZ[(\disc F)^{-1}]}\dZ_{(p)}$. Moreover, the set of similarity classes of rational skew-hermitian spaces $\rW'_0$ over $O_F\otimes\dZ_{(p)}$ of rank $1$ such that $\rW'_0\otimes_{\dZ_{(p)}}\dA$ is similar to $\rW_0\otimes_{\dZ_{(p)}}\dA$ is parameterized by
\[
\Ker^1(\rT_0)\coloneqq\Ker\(\rH^1(\dQ,\rT_0)\to\prod_{v\leq\infty}\rH^1(\dQ_v,\rT_0)\),
\]
which is a finite abelian group.
\end{remark}

\begin{definition}\label{de:skew_hermitian_type}
Let $\Phi$ be a CM type. We say that a rational skew-hermitian space $\rW_0$ over $O_F\otimes\dZ_{(p)}$ of rank $1$ has \emph{type $\Phi$} if for every $x\in\rW_0$ and every totally imaginary element $a\in F^\times$ satisfying $\IP\tau(a)>0$ for all $\tau\in\Phi$, we have $\langle a x,x\rangle_{\rW_0}\geq 0$.
\end{definition}

\begin{definition}
For a rational skew-hermitian space $\rW_0$ over $O_F\otimes\dZ_{(p)}$ of rank $1$ and type $\Phi$ and an open compact subgroup $\rK_0^p\subseteq\rT_0(\dA^{\infty,p})$, we define a presheaf $\bT^1_p(\rW_0,\rK_0^p)$ on $\Sch'_{/O_{F_\Phi}\otimes\dZ_{(p)}}$ as follows: for every $S\in\Sch'_{/O_{F_\Phi}\otimes\dZ_{(p)}}$, we let $\bT^1_p(\rW_0,\rK_0^p)(S)$ be the set of equivalence classes of triples $(A_0,\lambda_0,\eta_0^p)$, where
\begin{itemize}[label={\ding{118}}]
  \item $(A_0,\lambda_0)$ is a unitary $O_F$-abelian scheme of signature type $\Phi$ over $S$ such that $\lambda_0$ is $p$-principal;

  \item $\eta_0^p$ is a $\rK_0^p$-level structure, that is, for a chosen geometric point $s$ on every connected component of $S$, a $\pi_1(S,s)$-invariant $\rK_0^p$-orbit of similitude
      \[
      \eta_0^p\colon\rW_0\otimes_{\dZ_{(p)}}\dA^{\infty,p}\to\rH^\et_1(A_{0s},\dA^{\infty,p})
      \]
      of rational skew-hermitian spaces over $F\otimes_\dQ\dA^{\infty,p}$, where $\rH^\et_1(A_{0s},\dA^{\infty,p})$ is equipped with the rational skew-hermitian form induced by $\lambda_0$.
\end{itemize}
Two triples $(A_0,\lambda_0,\eta_0^p)$ and $(A'_0,\lambda'_0,\eta_0^{p\prime})$ are equivalent if there exists a prime-to-$p$ $O_F$-linear quasi-isogeny $\varphi_0\colon A_0\to A'_0$ carrying $(\lambda_0,\eta_0^p)$ to $(c\lambda'_0,\eta_0^{p\prime})$ for some $c\in\dZ_{(p)}^\times$.
\end{definition}

For an object $(A_0,\lambda_0,\eta_0^p)\in\bT^1_p(\rW_0,\rK_0^p)(\dC)$, its first homology $\rH_1(A_0(\dC),\dZ_{(p)})$ is a rational skew-hermitian space over $O_F\otimes\dZ_{(p)}$ induced by $\lambda_0$, which is of rank $1$ and type $\Phi$, and is everywhere locally similar to $\rW_0$. Thus, by Remark \ref{re:hasse_principle}, we obtain a map
\[
\tw\colon\bT^1_p(\rW_0,\rK_0^p)(\dC)\to\Ker^1(\rT_0)
\]
sending $(A_0,\lambda_0,\eta_0^p)\in\bT^1_p(\rW_0,\rK_0^p)(\dC)$ to the similarity class of $\rH_1(A_0(\dC),\dZ_{(p)})$.

It is known that when $\rK_0^p$ is neat, $\bT^1_p(\rW_0,\rK_0^p)$ is represented by a scheme finite \'{e}tale over $O_{F_\Phi}\otimes\dZ_{(p)}$. We define $\bT_p(\rW_0,\rK_0^p)$ to be the minimal open and closed subscheme of $\bT^1_p(\rW_0,\rK_0^p)$ containing $\tw^{-1}(\rW_0)$. The group $\rT_0(\dA^{\infty,p})$ acts on $\bT_p(\rW_0,\rK_0^p)$ via the formula
\[
a\cdot(A_0,\lambda_0,\eta_0^p)=(A_0,\lambda_0,\eta_0^p\circ a)
\]
whose stabilizer is $\rT_0(\dZ_{(p)})\rK_0^p$. In fact, $\rT_0(\dA^{\infty,p})/\rT_0(\dZ_{(p)})\rK_0^p$ is the Galois group of the Galois morphism
\[
\bT_p(\rW_0,\rK_0^p)\to\Spec (O_{F_\Phi}\otimes\dZ_{(p)}).
\]

\begin{definition}\label{de:groupoid}
We denote by $\fT$ the groupoid of $\rT_0(\dA^{\infty,p})/\rT_0(\dZ_{(p)})\rK_0^p$, that is, a category with a single object $\ast$ with $\Hom(\ast,\ast)=\rT_0(\dA^{\infty,p})/\rT_0(\dZ_{(p)})\rK_0^p$.
\end{definition}

\begin{remark}\label{re:groupoid}
As $\bT_p(\rW_0,\rK_0^p)$ is an object in $\Sch_{/O_{F_\Phi}\otimes\dZ_{(p)}}$ with an action by $\rT_0(\dA^{\infty,p})/\rT_0(\dZ_{(p)})\rK_0^p$, it induces a functor from $\fT$ to $\Sch_{/O_{F_\Phi}\otimes\dZ_{(p)}}$, which we still denote by $\bT_p(\rW_0,\rK_0^p)$. In what follows, we may often have another category $\fC$ and will regard $\bT_p(\rW_0,\rK_0^p)$ as a functor from $\fC\times\fT$ to $\Sch_{/O_{F_\Phi}\otimes\dZ_{(p)}}$ as the composition of the projection functor $\fC\times\fT\to\fT$ and the functor $\bT_p(\rW_0,\rK_0^p)\colon\fT\to\Sch_{/O_{F_\Phi}\otimes\dZ_{(p)}}$.
\end{remark}

\begin{notation}\label{no:groupoid}
For a functor $X\colon\fT\to\Sch$ and a coefficient ring $L$, we denote
\[
\rH^i_\fT(X,L(j))\subseteq\rH^i_\et(X(\ast),L(j)),\quad
\rH^i_{\fT,c}(X,L(j))\subseteq\rH^i_{\et,c}(X(\ast),L(j))
\]
the maximal $L$-submodules, respectively, on which $\rT_0(\dA^{\infty,p})/\rT_0(\dZ_{(p)})\rK_0^p$ acts trivially.
\end{notation}

\subsection{Unitary Deligne--Lusztig varieties}

Let $p$ be a rational prime and $\kappa$ a field containing $\dF_{p^2}$. Consider a pair $(\sV,\{\;,\;\})$ in which $\sV$ is a finite dimensional $\kappa$-linear space, and $\{\;,\;\}\colon\sV\times\sV\to\kappa$ is a (not necessarily nondegenerate) pairing that is $(\kappa,\sigma)$-linear in the first variable and $\kappa$-linear in the second variable. For every $\kappa$-scheme $S$, put $\sV_S\coloneqq\sV\otimes_\kappa\cO_S$. Then there is a unique pairing $\{\;,\;\}_S\colon\sV_S\times\sV_S\to\cO_S$ extending $\{\;,\;\}$ that is $(\cO_S,\sigma)$-linear in the first variable and $\cO_S$-linear in the second variable. For a subbundle $H\subseteq\sV_S$, we denote by $H^\dashv\subseteq\sV_S$ its \emph{right} orthogonal complement under $\{\;,\;\}_S$.

\begin{definition}\label{de:dl_admissible}
We say that a pair $(\sV,\{\;,\;\})$ is \emph{admissible} if there exists an $\dF_{p^2}$-linear subspace $\sV_0\subseteq\sV_{\ol\kappa}$ such that the induced map $\sV_0\otimes_{\dF_{p^2}}\ol\kappa\to\sV_{\ol\kappa}$ is an isomorphism, and $\{x,y\}=-\{y,x\}^\sigma$ for every $x,y\in\sV_0$.
\end{definition}

\begin{definition}\label{de:dl}
For a pair $(\sV,\{\;,\;\})$ and an integer $h$, we define a presheaf
\[
\DL(\sV,\{\;,\;\},h)
\]
on $\Sch_{/\kappa}$ such that for every $S\in\Sch_{/\kappa}$, $\DL(\sV,\{\;,\;\},h)(S)$ is the set of subbundles $H$ of $\sV_S$ of rank $h$ such that $H^\dashv\subseteq H$. We call $\DL(\sV,\{\;,\;\},h)$ the \emph{(unitary) Deligne--Lusztig variety} (see Proposition \ref{pr:dl} below) attached to $(\sV,\{\;,\;\})$ of rank $h$.
\end{definition}

\begin{proposition}\label{pr:dl}
Consider an admissible pair $(\sV,\{\;,\;\})$. Put $N\coloneqq\dim_\kappa\sV$ and $d\coloneqq\dim_\kappa\sV^\dashv$.
\begin{enumerate}
  \item If $2h<N+d$ or $h>N$, then $\DL(\sV,\{\;,\;\},h)$ is empty.

  \item If $N+d\leq 2h\leq 2N$, then $\DL(\sV,\{\;,\;\},h)$ is represented by a projective smooth scheme over $\kappa$ of dimension $(2h-N-d)(N-h)$ with a canonical isomorphism for its tangent sheaf
      \[
      \cT_{\DL(\sV,\{\;,\;\},h)/\kappa}\simeq\HOM\(\cH/\cH^\dashv,\sV_{\DL(\sV,\{\;,\;\},h)}/\cH\),
      \]
      where $\cH\subseteq\sV_{\DL(\sV,\{\;,\;\},h)}$ is the universal subbundle.

  \item If $N+d< 2h\leq 2N$, then $\DL(\sV,\{\;,\;\},h)$ is geometrically irreducible.
\end{enumerate}
\end{proposition}

\begin{proof}
This is \cite{LTXZZ}*{Proposition~A.1.3}.
\end{proof}

\begin{lem}\label{le:intersection}
Let $(\sV,\{\;,\;\})$ be an admissible pair with $\dim\sV=2r-1$ for a positive integer $r$ and $d=0$ (so that $\DL(\sV,\{\;,\;\},r)$ is smooth projective of dimension $r-1$). We have
\[
\int_{\DL(\sV,\{\;,\;\},r)}c_{r-1}(\cH^\dashv)=(-p)^{r-1}(1+p)(1+p^3)\cdots(1+p^{2r-3})
\]
(which is interpreted as $1$ when $r=1$), where $\cH\subseteq\sV_{\DL(\sV,\{\;,\;\},r)}$ is the universal subbundle.
\end{lem}

\begin{proof}
Since $\cH^\dashv$ is isomorphic to the $p$-Frobenius twist of the dual bundle of $(\sV\otimes_\kappa\cO_{\DL(\sV,\{\;,\;\},r)})/\cH$, it suffices to show that
\begin{align}\label{eq:intersection}
\int_{\DL(\sV,\{\;,\;\},r)}c_{r-1}(\sV\otimes_\kappa\cO_{\DL(\sV,\{\;,\;\},r)}/\cH)=(1+p)(1+p^3)\cdots(1+p^{2r-3}).
\end{align}

It is clear that $(\sV\otimes_\kappa\cO_{\DL(\sV,\{\;,\;\},r)})/\cH$ is globally generated by the constant bundle $\sV\otimes_\kappa\cO_{\DL(\sV,\{\;,\;\},r)}$. Choose an $\dF_{p^2}$-subspace $\sV_0$ of $\sV$ as in Definition \ref{de:admissible} and an anisotropic vector $v_0\in\sV_0$, which is possible as $\sV$ is nondegenerate. Let $\sV'$ be the orthogonal complement of $v_0$ in $\sV$; it is again admissible and nondegenerate with respect to the restriction $\{\;,\;\}'$ of $\{\;,\;\}$ to $\sV'$. Note that we have a natural morphism $\DL(\sV',\{\;,\;\}',r-1)\to\DL(\sV,\{\;,\;\},r)$ sending a point $H'\in\DL(\sV',\{\;,\;\}',r-1)(S)$ to $H'\oplus\cO_Sv_0$, which is a closed embedding. Then the zero locus $Z(v_0)$ of $v_0$, regarded as a global section of $(\sV\otimes_\kappa\cO_{\DL(\sV,\{\;,\;\},r)})/\cH$, is identified with $\DL(\sV',\{\;,\;\}',r-1)$ via the assignment $H\in Z(v_0)(S)$ to $H\cap\sV'\otimes_\kappa\cO_S$. As $\DL(\sV',\{\;,\;\}',r-1)$ is a finite scheme over $\kappa$ of length $(1+p)(1+p^3)\cdots(1+p^{2r-3})$, we obtain \eqref{eq:intersection}.

The lemma is proved.
\end{proof}

\section{Abel--Jacobi image of basic locus}
\label{ss:3}

In this section, we study the image of the basic locus of unitary Shimura varieties at good inert primes under the Abel--Jacobi map and its relation with a certain map involving a local system coming from the Siegel Iwahori subgroup (which we call the Tate--Thompson local system), generalizing the early study of Ribet in the case of modular curves.

\subsection{Initial setup}
\label{ss:qs_initial}

We fix a special inert prime (Definition \ref{de:special_inert}) $\fp$ of $F^+$ (with the underlying rational prime $p$). We take the prescribed subring $\dP$ in Definition \ref{de:unitary_abelian_scheme} to be $\dZ_{(p)}$. We choose the following data
\begin{itemize}[label={\ding{118}}]
  \item a CM type $\Phi$ containing $\tau_\infty$;

  \item a rational skew-hermitian space $\rW_0$ over $O_F\otimes\dZ_{(p)}$ of rank $1$ and type $\Phi$ (Definition \ref{de:skew_hermitian_type});

  \item a neat open compact subgroup $\rK_0^p\subseteq\rT_0(\dA^{\infty,p})$;

  \item an isomorphism $\iota_p\colon\dC\xrightarrow\sim\ol\dQ_p$ such  that $\iota_p\circ\tau_\infty\colon F^+\hookrightarrow \ol\dQ_p$ induces the place $\fp$ of $F^+$;

  \item an element $\varpi\in F^+$ that is totally positive and satisfies $\val_\fp(\varpi)=1$, and $\val_\fq(\varpi)=0$ for every prime $\fq\neq\fp$ of $F^+$ above $p$.
\end{itemize}
We adopt Notation \ref{no:p_notation}. In particular, $\dF^\Phi_p$ contains $\dF_{p^2}$. Put $\bT\coloneqq\bT_p(\rW_0,\rK_0^p)\otimes_{O_{F_\Phi}\otimes\dZ_{(p)}}\dZ^\Phi_p$.

We fix a standard \emph{indefinite} hermitian space (Definition \ref{de:standard_hermitian_space}) $\rV$ over $F$ of rank $N\geq 1$ that admits a self-dual lattice at every $p$-adic place of $F^+$.

\begin{itemize}[label={\ding{118}}]
  \item Denote by $\fS$ the \emph{similitude group} consisting of totally positive elements $\varsigma$ in $F^+$ such that $N\val_\fq(\varsigma)$ is even for every $\fq\mid p$.

  \item For every $\varsigma\in\fS$, we denote by $\pres\varsigma\rV$ the hermitian space obtained from $\rV$ via rescaling the hermitian form by $\varsigma$. In particular, $\pres\varsigma\rV$ remains a standard indefinite hermitian space that admits a self-dual lattice at every $p$-adic place of $F^+$.

  \item The groups $\rU(\pres\varsigma\rV)$ are independent of $\varsigma\in\fS$, hence the categories $\fK(\pres\varsigma\rV)^\Box$ (Definition \ref{de:neat_category}) are canonically identified, which we denote as $\fK^\Box$.
\end{itemize}

\begin{notation}\label{no:numerical1}
For every integer $q\geq 2$ and $j\in\dZ$, we put
\[
\tb_{N,q}\coloneqq q\prod_{i=1}^{N}(1-(-q)^i),\quad
\tl_{N,q}^j\coloneqq\frac{-(-q)^{N+1-j}-(-q)^j-q+1}{q^2-1},
\]
which are all integers. Note that $\tl_{N,q}^j=\tl_{N,q}^{N+1-j}$.
\end{notation}

\subsection{Unitary moduli schemes at hyperspecial level}

In this subsection, we recall the construction of unitary moduli schemes at the hyperspecial level and their basic locus by summarizing \cite{LTXZZ}*{\S4} with a slight change of notation.

Take an element $\varsigma\in\fS$. We have the functor\footnote{The functor $\pres\varsigma\bM_N$ below is nothing but $\bM_\fp(\pres\varsigma\rV,\obj)$ in \cite{LTXZZ}*{\S4}, and similarly for $\rB$ and $\rS$ later.}
\begin{align*}
\pres\varsigma\bM_N\colon\fK^p\times\fT &\to\Sch'_{/\dZ^\Phi_p} \\
\rK^p &\mapsto \pres\varsigma\bM(\rK^p)
\end{align*}
such that for every $S\in\Sch'_{/\dZ^\Phi_p}$, $\pres\varsigma\bM_N(\rK^p)(S)$ is the set of equivalence classes of sextuples $(A_0,\lambda_0,\eta_0^p;A,\lambda,\eta^p)$, where
\begin{itemize}[label={\ding{118}}]
  \item $(A_0,\lambda_0,\eta_0^p)$ is an element in $\bT(S)$;

  \item $(A,\lambda)$ is a unitary $O_F$-abelian scheme of signature type $N\Phi-\tau_\infty+\tau_\infty^\tc$ over $S$ (Definitions \ref{de:unitary_abelian_scheme} and \ref{de:signature}) such that $\lambda$ is $p$-principal;

  \item $\eta^p$ is a $\rK^p$-level structure, that is, for a chosen geometric point $s$ on every connected component of $S$, a $\pi_1(S,s)$-invariant $\rK^p$-orbit of isomorphisms
      \[
      \eta^p\colon\pres\varsigma\rV\otimes_\dQ\dA^{\infty,p}\to
      \Hom_{F\otimes_\dQ\dA^{\infty,p}}^{\lambda_0,\lambda}(\rH^\et_1(A_{0s},\dA^{\infty,p}),\rH^\et_1(A_s,\dA^{\infty,p}))
      \]
      of hermitian spaces over $F\otimes_\dQ\dA^{\infty,p}=F\otimes_{F^+}\dA_{F^+}^{\infty,p}$. See Construction \ref{cs:hermitian_structure} (with $\Box=\{\infty,p\}$) for the right-hand side.
\end{itemize}
Two sextuples $(A_0,\lambda_0,\eta_0^p;A,\lambda,\eta^p)$ and $(A'_0,\lambda'_0,\eta_0^{p\prime};A',\lambda',\eta^{p\prime})$ are equivalent if there are prime-to-$p$ $O_F$-linear quasi-isogenies $\varphi_0\colon A_0\to A'_0$ and $\varphi\colon A\to A'$ such that
\begin{itemize}[label={\ding{118}}]
  \item $\varphi_0$ carries $\eta_0^p$ to $\eta_0^{p\prime}$;

  \item there exists $c\in\dZ_{(p)}^\times$ such that $\varphi_0^\vee\circ\lambda'_0\circ\varphi_0=c\lambda_0$ and $\varphi^\vee\circ\lambda'\circ\varphi=c\lambda$; and

  \item the $\rK^p$-orbit of maps $v\mapsto\varphi_*\circ\eta^p(v)\circ(\varphi_{0*})^{-1}$ for $v\in\rV\otimes_\dQ\dA^{\infty,p}$ coincides with $\eta^{p\prime}$.
\end{itemize}
On the level of morphisms,
\begin{itemize}[label={\ding{118}}]
  \item a morphism $g\in\rK^p\backslash\rU(\pres\varsigma\rV)(\dA_F^{\infty,p})/\rK^{p\prime}$ of $\fK^p$ maps $\pres\varsigma\bM_N(\rK^p)(S)$ to $\pres\varsigma\bM_N(\rK^{p\prime})(S)$ by changing $\eta^p$ to $\eta^p\circ g$; and

  \item a morphism $a$ of $\fT$ acts on $\pres\varsigma\bM_N(\rK^p)(S)$ by changing $\eta_0^p$ to $\eta_0^p\circ a$.
\end{itemize}
The image of $\pres\varsigma\bM_N$ is contained in the subcategory spanned by quasi-projective smooth schemes of relative dimension $N-1$. Note that for every object $(A_0,\lambda_0,\eta_0^p;A,\lambda,\eta^p)\in\pres\varsigma\bM_N(\rK^p)(S)$, the pair $(A[\fp],\lambda[\fp])$ is a unitary $p$-truncated group of signature $(N-1,1)$ over $S$ (Definition \ref{de:unitary_divisible}).

\begin{remark}\label{re:square}
For two elements $\varsigma,\varsigma'\in\fS$ whose quotient belongs to $\fS^2$, the functors $\pres\varsigma\bM_N$ and $\pres{\varsigma'}\bM_N$ are canonically identified via scaling the level structure by the unique square root of $\varsigma/\varsigma'$ in $\fS$.
\end{remark}

We then review the basic correspondence for the special fiber $\pres\varsigma\rM_N$.

We have the functor
\begin{align*}
\pres\varsigma\rS_N\colon\fK^p\times\fT &\to\Sch'_{/\dF^\Phi_p} \\
\rK^p &\mapsto \pres\varsigma\rS_N(\rK^p)
\end{align*}
such that for every $S\in\Sch'_{/\dF^\Phi_p}$, $\pres\varsigma\rS_N(\rK^p)(S)$ is the set of equivalence classes of sextuples $(A_0,\lambda_0,\eta_0^p;A^\star,\lambda^\star,\eta^{p\star})$, where
\begin{itemize}[label={\ding{118}}]
  \item $(A_0,\lambda_0,\eta_0^p)$ is an element in $\rT(S)$;

  \item $(A^\star,\lambda^\star)$ is a unitary $O_F$-abelian scheme of signature type $N\Phi$ over $S$ such that $\Ker\lambda^\star[p^\infty]$ is trivial (resp.\ contained in $A^\star[\fp]$ of rank $p^2$) if $N$ is odd (resp.\ even);

  \item $\eta^{p\star}$ is, for a chosen geometric point $s$ on every connected component of $S$, a $\pi_1(S,s)$-invariant $\rK^p$-orbit of isomorphisms
      \[
      \eta^{p\star}\colon\pres\varsigma\rV\otimes_\dQ\dA^{\infty,p}\to\Hom_{F\otimes_\dQ\dA^{\infty,p}}^{\varpi\lambda_0,\lambda^\star}
      (\rH^\et_1(A_{0s},\dA^{\infty,p}),\rH^\et_1(A^\star_s,\dA^{\infty,p}))
      \]
      of hermitian spaces over $F\otimes_\dQ\dA^{\infty,p}=F\otimes_{F^+}\dA_{F^+}^{\infty,p}$.
\end{itemize}
The definitions of the equivalence relation and the action of morphisms in $\fK^p\times\fT$ are omitted.

We have the functor
\begin{align*}
\pres\varsigma\rB_N\colon\fK^p\times\fT &\to\Sch'_{/\dF^\Phi_p} \\
\rK^p &\mapsto \pres\varsigma\rB_N(\rK^p)
\end{align*}
such that for every $S\in\Sch'_{/\dF^\Phi_p}$, $\pres\varsigma\rB_N(\rK^p)(S)$ is the set of equivalence classes of decuples
\[
(A_0,\lambda_0,\eta_0^p;A,\lambda,\eta^p;A^\star,\lambda^\star,\eta^{p\star};\alpha),
\]
where
\begin{itemize}[label={\ding{118}}]
  \item $(A_0,\lambda_0,\eta_0^p;A,\lambda,\eta^p)$ is an element of $\pres\varsigma\rM_N(\rK^p)(S)$;

  \item $(A_0,\lambda_0,\eta_0^p;A^\star,\lambda^\star,\eta^{p\star})$ is an element of $\pres\varsigma\rS_N(\rK^p)(S)$; and

  \item $\alpha\colon A\to A^\star$ is an $O_F$-linear quasi-$p$-isogeny (Definition \ref{de:p_quasi}) such that
  \begin{enumerate}[label=(\alph*)]
    \item $\Ker\alpha[p^\infty]$ is contained in $A[\fp]$;

    \item we have $\varpi\cdot\lambda=\alpha^\vee\circ\lambda^\star\circ\alpha$; and

    \item the $\rK^p$-orbit of maps $v\mapsto\alpha_*\circ\eta^p(v)$ for $v\in\rV\otimes_\dQ\dA^{\infty,p}$ coincides with $\eta^{p\star}$.
  \end{enumerate}
\end{itemize}
The definitions of the equivalence relation and the action of morphisms in $\fK^p\times\fT$ are omitted.

We obtain in the obvious way a correspondence
\begin{align}\label{eq:qs_basic_correspondence}
\xymatrix{
\pres\varsigma\rS_N  &
\pres\varsigma\rB_N \ar[r]^-{\iota}\ar[l]_-{\pi} &
\pres\varsigma\rM_N
}
\end{align}
in $\Fun(\fK^p\times\fT,\Sch'_{/\dF_p^\Phi})_{/\rT}$, which we call the \emph{basic correspondence} on $\pres\varsigma\rM_N$, with $\pres\varsigma\rS_N$ being the \emph{source} of the basic correspondence.

Take a point $s^\star=(A_0,\lambda_0,\eta_0^p;A^\star,\lambda^\star,\eta^{p\star})\in\pres\varsigma\rS_N(\rK^p)(\kappa)$, where $\kappa$ is a field containing $\dF^\Phi_p$. Then $A^\star_{\ol\kappa}[\fp^\infty]$ is a supersingular $p$-divisible group by the signature condition and the fact that $\fp$ is inert in $F$. From Notation \ref{re:frobenius_verschiebung}(1), we have the $(\kappa,\sigma)$-linear Frobenius map
\[
\tF\colon\rH^\dr_1(A^\star/\kappa)_{\tau_\infty}\to\rH^\dr_1(A^\star/\kappa)_{\sigma\tau_\infty}
=\rH^\dr_1(A^\star/\kappa)_{\tau_\infty^\tc}.
\]
We define a pairing
\[
\{\;,\;\}_{s^\star}\colon\rH^\dr_1(A^\star/\kappa)_{\tau_\infty}\times\rH^\dr_1(A^\star/\kappa)_{\tau_\infty}\to\kappa
\]
by the formula $\{x,y\}_{s^\star}\coloneqq\langle\tF x,y\rangle_{\lambda^\star,\tau_\infty^\tc}$ (Notation \ref{no:weil_pairing}). To ease notation, we put
\[
\sV_{s^\star}\coloneqq\rH^\dr_1(A^\star/\kappa)_{\tau_\infty}.
\]
The pair $(\sV_{s^\star},\{\;,\;\}_{s^\star})$ is admissible of rank $N$ (Definition \ref{de:dl_admissible}). In particular, the Deligne--Lusztig variety $\DL_{s^\star}\coloneqq\DL(\sV_{s^\star},\{\;,\;\}_{s^\star},\ceil{\tfrac{N+1}{2}})$ (Definition \ref{de:dl}) is a geometrically irreducible projective smooth scheme in $\Sch_{/\kappa}$ of dimension $\floor{\tfrac{N-1}{2}}$.

\begin{lem}\label{le:qs_basic_correspondence}
In the diagram \eqref{eq:qs_basic_correspondence}, take a point
\[
s^\star=(A_0,\lambda_0,\eta_0^p;A^\star,\lambda^\star,\eta^{p\star})\in\pres\varsigma\rS_N(\rK^p)(\kappa)
\]
where $\kappa$ is a field containing $\dF^\Phi_p$. Put $\rB_{s^\star}\coloneqq\pi^{-1}(s^\star)$, and denote by $(\cA,\lambda,\eta^p;\alpha)$ the universal object over the fiber $\rB_{s^\star}$.
\begin{enumerate}
  \item The fiber $\rB_{s^\star}$ is a smooth scheme over $\kappa$, with a canonical isomorphism for its tangent bundle
        \[
        \cT_{\rB_{s^\star}/\kappa}\simeq\HOM\(\omega_{\cA^\vee,\tau_\infty},\Ker\alpha_{*,\tau_\infty}/\omega_{\cA^\vee,\tau_\infty}\).
        \]

  \item The restriction of $\iota$ to $\rB_{s^\star}$ is locally on $\rB_{s^{\star}}$ a  closed immersion, with a canonical isomorphism for its normal bundle
        \[
        \cN_{\iota\res\rB_{s^\star}}\simeq\HOM\(\omega_{\cA^\vee,\tau_\infty},\IM\alpha_{*,\tau_\infty}\).
        \]

  \item The assignment sending a point
        $(A_0,\lambda_0,\eta_0^p;A,\lambda,\eta^p;A^\star,\lambda^\star,\eta^{p\star};\alpha)\in\rB_{s^\star}(S)$  for every  $S\in \Sch'_{/\kappa}$ to the subbundle
        \[
        H\coloneqq(\breve\alpha_{*,\tau_\infty})^{-1}\omega_{A^\vee/S,\tau_\infty}\subseteq\rH^\dr_1(A^\star/S)_{\tau_\infty}=
        \rH^\dr_1(A^\star/\kappa)_{\tau_\infty}\otimes_\kappa\cO_S=(\sV_{s^\star})_S,
        \]
        where $\breve\alpha\colon A^\star\to A$ is the (unique) $O_F$-linear quasi-$p$-isogeny such that $\breve\alpha\circ\alpha=\varpi\cdot\id_A$, induces an isomorphism
        \[
        \zeta_{s^\star}\colon\rB_{s^\star}\xrightarrow{\sim}\DL_{s^\star}=\DL(\sV_{s^\star},\{\;,\;\}_{s^\star},\ceil{\tfrac{N+1}{2}})
        \]
        (Definition \ref{de:dl}). In particular, $\rB_{s^\star}$ is a geometrically irreducible projective smooth scheme in $\Sch_{/\kappa}$ of dimension $\floor{\frac{N-1}{2}}$.
\end{enumerate}
\end{lem}

\begin{proof}
This is \cite{LTXZZ}*{Theorem~4.3.5}.
\end{proof}

To end this subsection, we introduce the notion of essential Frobenius morphisms, which will be used later.

\begin{definition}\label{de:frobenius}
Let $A$ be an abelian scheme over a scheme $S\in\Sch_{/\dF_p}$. We have the Frobenius twist $A^{(p)}$ of $A$ from Notation \ref{no:frobenius_verschiebung}. Put
\[
A^{(p^2)}\coloneqq(A^{(p)})^{(p)},\quad
F_A\coloneqq\Ker\(A\to A^{(p^2)}\),\quad
A^\phi\coloneqq A/F_A[\fp^\infty].
\]
For every unitary $O_F$-abelian scheme $(A,i,\lambda)$ over $S$, we have the induced triple $(A^\phi,i^\phi,\lambda^\phi)$, which does not change the signature type by \cite{LTXZZ}*{Lemma~3.4.12(2)}.

We define the \emph{essential Frobenius morphism} of $\pres\varsigma\rM_N$ to be the morphism
\[
\phi_\rM\colon\pres\varsigma\rM_N\to\pres\varsigma\rM_N
\]
in $\Fun(\fK^p\times\fT,\Sch'_{/\dF_p^\Phi})_{/\rT}$ that sends $(A_0,\lambda_0,\eta_0^p;A,\lambda,\eta^p)$ to $(A_0,\lambda_0,\eta_0^p;A^\phi,\lambda^\phi,\eta^{p\phi})$, where $\eta^{p\phi}=\eta^p$ under the canonical identification $\rH^\et_1(A_s,\dA^{\infty,p})=\rH^\et_1(A^\phi_s,\dA^{\infty,p})$.

Similarly, we have essential Frobenius morphisms $\phi_\rB$ and $\phi_\rS$ of $\pres\varsigma\rB_N$ and $\pres\varsigma\rS_N$, respectively, so that the following diagram
\begin{align}\label{eq:frobenius}
\xymatrix{
\pres\varsigma\rS_N \ar[d]_-{\phi_\rS} & \pres\varsigma\rB_N \ar[d]^-{\phi_\rB}\ar[l]_-{\pi} \ar[r]^-{\iota} & \pres\varsigma\rM_N \ar[d]^-{\phi_\rM} \\
\pres\varsigma\rS_N & \pres\varsigma\rB_N \ar[l]_-{\pi} \ar[r]^-{\iota} & \pres\varsigma\rM_N
}
\end{align}
in $\Fun(\fK^p\times\fT,\Sch'_{/\dF_p^\Phi})_{/\rT}$ commutes.
\end{definition}

\begin{lem}\label{le:frobenius}
Write $d=[\dF_p^\Phi:\dF_{p^2}]$. Then $\phi_\rM^d$, $\phi_\rB^d$ and $\phi_\rS^d$ coincide with the absolute $p^{2d}$-Frobenius morphisms on $\pres\varsigma\rM_N$, $\pres\varsigma\rB_N$ and $\pres\varsigma\rS_N$, respectively, up to a translation by a common morphism $a\in\fT$. In particular, $\phi_\rM$, $\phi_\rB$ and $\phi_\rS$ are flat universal homeomorphisms of degree $p^{2N-2}$, $p^{2\floor{\frac{N-1}{2}}}$ and $1$, respectively.
\end{lem}

\begin{proof}
The second statement follows from the first one. In fact, the first statement implies that $\phi_\rM$, $\phi_\rB$ and $\phi_\rS$ are all universal homeomorphisms, which are in particular finite. Since their targets and sources are regular of the same dimension, they are flat by the miracle flatness criterion over regular local rings \cite{Mat89}*{Theorem~23.1}. The degree part follows since $\pres\varsigma\rM_N$, $\pres\varsigma\rB_N$ and $\pres\varsigma\rS_N$ are of pure dimensions $N-1$, $\floor{\frac{N-1}{2}}$ and $0$, respectively.

We now show the first statement with $a$ a morphism in $\fT$ such that its action on $\rT$ gives the inverse of the absolute $p^{2d}$-Frobenius morphism. We only consider $\phi_\rM$ as the other two cases are similar. Since for every $p$-adic place $\fq$ of $F$ other than $\fp$, $F_A[\fq^\infty]=A[\fq]$, the data $(A_0,\lambda_0,\eta_0^p;A^\phi,\lambda^\phi,\eta^{p\phi})$ and $(A_0,\lambda_0,\eta_0^p;A^{(p^2)},\lambda^{(p^2)},\eta^{p(p^2)})$ represent the same object. The statement then follows.
\end{proof}

\subsection{Ekedahl--Oort stratification}
\label{ss:eo}

In this subsection, we study the Ekedahl--Oort stratification of the special fibers of the unitary moduli schemes introduced in the previous subsection, following the work \cite{BW06}.

Let $\kappa$ be an algebraically closed field of characteristic $p$. For every $W(\kappa)$-ring $W$ with a compatible Frobenius map $\sigma$, consider the $W\times W$-algebra
\[
\tD_W\coloneqq(W\times W)[\tF,\tV]\left/\(\tF(w_1,w_2)-(w_2^\sigma,w_1^\sigma)\tF,(w_1,w_2)\tV-\tV(w_2^\sigma,w_1^\sigma),\tF\tV-p,\tV\tF-p\)\right.
\]
in which $w_1,w_2$ are taken over all elements in $W$. For a $\tD_W$-module $\cD$,
\begin{itemize}[label={\ding{118}}]
  \item we put $\cD^\leftarrow\coloneqq(1,0)\cD$ and $\cD^\rightarrow\coloneqq(0,1)\cD$;

  \item we say that $\cD$ has signature $(r^\leftarrow,r^\rightarrow)$ if $\cD^\leftarrow/\tV\cD^\rightarrow$ and $\cD^\rightarrow/\tV\cD^\leftarrow$ are locally free $W\otimes\dF_p$-modules of rank $r^\leftarrow$ and $r^\rightarrow$, respectively.

  \item we say that $\cD$ is \emph{left-leaning}, \emph{right-leaning}, or \emph{neutral} if $\cD$ is finite free over $W$ satisfying that $\rank_W\cD^\leftarrow\geq\rank_W\cD^\rightarrow$, $\rank_W\cD^\leftarrow\leq\rank_W\cD^\rightarrow$, or $\rank_W\cD^\leftarrow=\rank_W\cD^\rightarrow$, respectively.
\end{itemize}

We recall the following definitions from \cite{BW06}.
\begin{itemize}[label={\ding{118}}]
  \item A \emph{Dieudonn\'{e} module} (over $\kappa$) is a neutral $\tD_{W(\kappa)}$-module. A \emph{quasi-unitary Dieudonn\'{e} module} (over $\kappa$) is a Dieudonn\'{e} module $\cD$ together with an alternating pairing $\langle\;,\;\rangle\colon\cD\times\cD\to W(\kappa)$ that is perfect after tensoring with $\dQ$, satisfies $\langle\tF x,y\rangle=\langle x,\tV y\rangle^\sigma$ and such that both $\cD^\leftarrow$ and $\cD^\rightarrow$ are totally isotropic with respect to $\langle\;,\;\rangle$.

  \item A \emph{Dieudonn\'{e} space} (over $\kappa$) is a neutral $\tD_\kappa$-module satisfying $\Ker\tV=\IM\tF$ and $\Ker\tF=\IM\tV$. A \emph{quasi-unitary Dieudonn\'{e} space} (over $\kappa$) is a Dieudonn\'{e} space $\cD$ together with an alternating pairing $\langle\;,\;\rangle\colon\cD\times\cD\to\kappa$ that satisfies $\langle\tF x,y\rangle=\langle x,\tV y\rangle^\sigma$ and such that both $\cD^\leftarrow$ and $\cD^\rightarrow$ are totally isotropic with respect to $\langle\;,\;\rangle$.

  \item A quasi-unitary Dieudonn\'{e} module or a quasi-unitary Dieudonn\'{e} space is \emph{unitary} if the pairing $\langle\;,\;\rangle$ is perfect. It is clear that if $\cD$ is a (quasi-unitary, unitary) Dieudonn\'{e} module, then the reduction $\cD\otimes\dF_p$ is a (quasi-unitary, unitary) Dieudonn\'{e} space.

  \item Let $\cB(m)$ for $m\geq 1$ and $\cS$ be the reduction of $\ul{B}(m)$ and $\ul{S}$ in \cite{BW06}*{\S3.2}, which are unitary Dieudonn\'{e} spaces of signatures $(m-1,1)$ and $(1,0)$, respectively.
\end{itemize}

\begin{definition}\label{de:lagrangian}
Let $\cD$ be a Dieudonn\'{e} space (over $\kappa$) and $\cC$ a $\tD_\kappa$-submodule of $\cD$.
\begin{enumerate}
  \item We define the \emph{$\tF$-length} of a $\tD_\kappa$-submodule $\cC$ of $\cD$ to be the maximal integer $l\geq 0$ such that one can find a sequence of nonzero elements $x_1,\ldots,x_l\in\cC^\leftarrow$ such that $\tF x_1=0$ and $\tV x_i=\tF x_{i+1}$ for $1\leq i<l$.

  \item We define the \emph{$\tV$-length} of a $\tD_\kappa$-submodule $\cC$ of $\cD$ to be the maximal integer $l\geq 0$ such that one can find a sequence of nonzero elements $x_1,\ldots,x_l\in\cC^\leftarrow$ such that $\tV x_1=0$ and $\tF x_i=\tV x_{i+1}$ for $1\leq i<l$.

  \item When $\cD$ is unitary, we denote by $\cC^\perp_\cD$ (or simply $\cC^\perp$ when $\cD$ is clear from the context) its annihilator under the (perfect) alternating pairing on $\cD$, which is also a $\tD_\kappa$-submodule of $\cD$.

  \item When $\cD$ is unitary, we say that a $\tD_\kappa$-submodule $\cC$ of $\cD$ is \emph{Lagrangian} if $\cC=\cC^\perp$.
\end{enumerate}
\end{definition}

\begin{remark}\label{re:length}
It is clear that a $\tD_\kappa$-submodule $\cC$ of $\cD$ has positive $\tF$-length (resp.\ $\tV$-length) if and only if the intersection of $\cC^\leftarrow$ with $\Ker\tF$ (resp.\ $\Ker\tV$) is nonzero.
\end{remark}

The following lemma summarizes several properties concerning $\cB(m)$ and $\cS$.

\begin{lem}\label{le:eo}
Let $m\geq 1$ be an integer.
\begin{enumerate}
  \item Every Dieudonn\'{e} space of signature $(m,0)$ is isomorphic to (the underlying $\tD_\kappa$-module of) $\cS^m$.

  \item The composition map $\Hom_{\tD_\kappa}(\cS,\cB(m))\times\Hom_{\tD_\kappa}(\cB(m),\cS)\to\End_{\tD_\kappa}(\cS)$ is trivial.

  \item When $m$ is odd, $\cB(m)$ has a unique left-leaning totally isotropic $\tD_\kappa$-submodule $\cC$ that is \emph{not} neutral. Moreover, $\cC$ is Lagrangian, satisfies $\dim_\kappa\cC^\leftarrow-\dim_\kappa\cC^\rightarrow=1$, and has both $\tF$-length and $\tV$-length $\frac{m+1}{2}$.

  \item When $m$ is even, every left-leaning totally isotropic $\tD_\kappa$-submodule of $\cB(m)$ is neutral.

  \item When $m$ is even, the set of neutral Lagrangian $\tD_\kappa$-submodules of $\cB(m)$ can be written as $\{\cC(m)^l\res 0\leq l\leq \tfrac{m}{2}\}$ in which $\cC(m)^l$ has $\tF$-length $l$ and $\tV$-length $\frac{m}{2}-l$. Moreover, $\cC(m)^l$ has only finitely many neutral $\tD_\kappa$-submodules; and for every such submodule $\cC$, $\Hom_{\tD_\kappa}(\cS,\cC(m)^l/\cC)=\{0\}$.

  \item If $\cC$ is a neutral totally isotropic $\tD_\kappa$-submodule of $\cB(m)$ such that the induced map $\tV\colon(\cC^\perp/C)^\leftarrow\to(\cC^\perp/C)^\rightarrow$ is injective, then $\cC$ is Lagrangian (in particular, $m$ is even).
\end{enumerate}
\end{lem}

\begin{proof}
For (1), let $\cD$ be a Dieudonn\'{e} space of signature $(m,0)$. Then $\tF\colon\cD^\leftarrow\to\cD^\rightarrow$ and $\tV\colon\cD^\leftarrow\to\cD^\rightarrow$ are $(\kappa,\sigma)$-linear and $(\kappa,\sigma^{-1})$-linear isomorphisms, respectively. In particular, we have a $(\kappa,\sigma^2)$-linear isomorphism $\tV^{-1}\circ\tF\colon\cD^\leftarrow\to\cD^\leftarrow$. We may choose a nonzero element $e\in\cD^\leftarrow$ such that $\tV^{-1}\tF e=-e$. Then $\cD_1\coloneqq\kappa e\oplus\kappa\tF e$ is a $\tD_\kappa$-submodule of $\cD$ that is isomorphic to $\cS$. Moreover, the quotient $\cD/\cD_1$ is a Dieudonn\'{e} space of signature $(m-1,0)$. By induction on $m$, it suffices to show that $\Ext_{\tD_\kappa}(\cS,\cS)=\{0\}$, which is an easy exercise in linear algebra.

For (2), since $\End_{\tD_\kappa}(\cS)=\kappa^{\sigma^2=1}$, it remains to show that $\cS$ is not a direct summand of $\cB(m)$ as a $\tD_\kappa$-module. Indeed, the canonical permutation of $\cB(m)$ (see \cite{Moo01}*{Lemma~4.5}), which in our case is a permutation on the set $\{\leftarrow,\rightarrow\}\times\{1,\ldots,m\}$, is a cycle of length $2m$ (resp. a product of two cycles of length $m$) when $m$ is odd (resp.\ even). Thus, $\cS$ cannot be a direct summand of $\cB(m)$ when $m\geq 3$. When $m=1,2$, it is easy to rule out the possibility.

For (3--6), we choose bases $\cB(m)^\leftarrow=\kappa e_1\oplus\cdots\oplus \kappa e_m$ and $\cB(m)^\rightarrow=\kappa f_1\oplus\cdots\oplus \kappa f_m$ under which $\langle e_i,f_j\rangle=(-1)^i\delta_{ij}$ and $\tF$ and $\tV$ are determined uniquely by the relation
\begin{align*}
\tF e_2=f_1,\quad \tF e_3=f_2,\quad &\cdots \quad \tF e_m=f_{m-1},\quad \tF f_1=(-1)^me_m; \\
\tV e_1=f_2,\quad \tV e_2=f_3,\quad &\cdots \quad \tV e_{m-1}=f_m,\quad \tV f_m=e_1.
\end{align*}
Then $\Ker\tF\res_{\cB(m)^\leftarrow}=\tV(\cB(m)^\rightarrow)=\kappa e_1$ and $\Ker\tV\res_{\cB(m)^\leftarrow}=\tF(\cB(m)^\rightarrow)=\kappa e_m$.
\begin{itemize}[label={\ding{118}}]
  \item When $m$ is odd, we consider a complete flag $\cF_\bullet$ with\footnote{In fact, $\cF_\bullet$ is the canonical filtration of $\cB(m)$.}
     \begin{align*}
     \cF_0=\{0\},\quad &\cF_1=\kappa e_m,\quad \cF_2=\cF_1\oplus\kappa f_{m-1},\quad
     \cF_3=\cF_2\oplus \kappa e_{m-2},\quad\cdots,\quad
     \cF_m=\cF_{m-1}\oplus\kappa e_1, \\
     &\cF_{m+1}=\cF_m\oplus\kappa f_m,\quad
     \cF_{m+2}=\cF_{m+1}\oplus\kappa e_{m-1},\quad\cdots,\quad
     \cF_{2m}=\cB(m).
     \end{align*}

  \item When $m$ is even, we consider two complete flags $\cF_{1,\bullet}$ and $\cF_{2,\bullet}$ with
     \begin{align*}
     &\cF_{1,0}=\{0\},\quad \cF_{1,1}=\kappa e_m,\quad \cF_{1,2}=\cF_{1,1}\oplus\kappa f_{m-1},\quad
     \cF_{1,3}=\cF_{1,2}\oplus \kappa e_{m-2},\quad\cdots,\quad \cF_{1,m}=\cF_{1,m-1}\oplus\kappa f_1; \\
     &\cF_{2,0}=\{0\},\quad \cF_{2,1}=\kappa e_1,\quad \cF_{2,2}=\cF_{2,1}\oplus\kappa f_2,\quad
     \cF_{2,3}=\cF_{2,2}\oplus \kappa e_3,\quad\cdots,\quad \cF_{2,m}=\cF_{2,m-1}\oplus\kappa f_m.
     \end{align*}
\end{itemize}

For (3), we have $\dim_\kappa\cC^\leftarrow>\dim_\kappa\cC^\rightarrow$ since $\cC$ is left-leaning but not neutral. As $\cC^\rightarrow$ contains $\tV\cC^\leftarrow$, it follows that $\cC^\leftarrow$ contains $\Ker\tV\res_{\cB(m)^\leftarrow}=\cF_1$ and hence $\cF_2=\cF_1+\tF\cF_1\subseteq\cC$. Now the map $\tV\colon\cC^\leftarrow/\cF_1\to\cC^\rightarrow$ is a $(\kappa,\sigma^{-1})$-linear isomorphism, which implies that $\cF_3\subseteq\cC$ and hence $\cF_4=\cF_3+\tF\cF_3\subseteq\cC$. By induction, we have $\cF_m\subseteq\cC$. Since $\cC$ is totally isotropic, we must have $\cC=\cF_m$. It is clear that $\cF_m$ is Lagrangian, satisfies $\dim_\kappa\cF_m^\leftarrow-\dim_\kappa\cF_m^\rightarrow=1$, and has both $\tF$-length and $\tV$-length $\frac{m+1}{2}$. Part (3) follows.

For (4), suppose that $\cC$ is a left-leaning totally isotropic $\tD_\kappa$-submodule of $\cB(m)$ that is not neutral. By the similar discussion in (3), we must have $\cF_{1,m}\subseteq\cC$. Since $\cC$ is totally isotropic, we have $\cC=\cF_{1,m}$ which is neutral. Part (4) follows.

For (5), let $\cC$ be a neutral totally isotropic $\tD_\kappa$-submodule of $\cB(m)$. Put $\cC_1\coloneqq\cC\cap\cF_{1,m}$ and denote by $\cC_2$ the image of $\cC$ in $\cF_{2,m}$ under the natural projection $\cB(m)\to\cF_{2,m}$. Since $\tD_\kappa$-submodules of $\cF_{1,m}$ and $\cF_{2,m}$ are all right-leaning but $\cC$ is neutral, both $\cC_1$ and $\cC_2$ have to be neutral. It is clear that neutral $\tD_\kappa$-submodules of $\cF_{1,m}$ and $\cF_{2,m}$ have the forms $\cF_{1,2l}$ and $\cF_{2,2l'}$, respectively. Since $\cC$ is totally isotropic, there exist unique integers $0\leq l\leq l'\leq \tfrac{m}{2}$ such that $\cC_1=\cF_{1,2l}$ and $\cC_2=\cF_{2,m-2l'}$. We claim that $\cC\cap\cF_{2,m-2l'}=\cC_2$. Assuming this, $\cC$ must have the form $\cF_{1,2l}\oplus\cF_{2,m-2l'}$. Part (5) follows by taking $\cC(m)^l\coloneqq\cF_{1,m-2l}\oplus\cF_{2,2l}$ and the easy fact that $\Hom_{\tD_\kappa}(\cS,\cF_{1,2l}/\cF_{1,2l'})=\Hom_{\tD_\kappa}(\cS,\cF_{2,2l}/\cF_{2,2l'})=\{0\}$ for $l'\leq l$.

For the claim, it is trivial when $l=\tfrac{m}{2}$. Now assume $0\leq l<\tfrac{m}{2}$. Then $\cC$ contains an element of the form $f'+f_{m-2l}$ with $f'\in\oplus_{i<m-2l}\kappa f_i$. Now, every element in $f_{m-2l}+\oplus_{i<m-2l}\kappa f_i$ cannot be in the image of $\tF\res_{\cC^\leftarrow}$ since otherwise $\cC$ contains an element that is not perpendicular to $f_{m-2l+1}\in\cC$. In particular, $\tF\colon\cC^\leftarrow\to\cC^\rightarrow$ is not an isomorphism, which implies that $e_1\in\cC$. Thus, $f_2\in\cC$. If $2<m-2l$, then $f_2$ must be in the image of $\tF\res_{\cC^\leftarrow}$ since otherwise the cokernel of $\tF\colon\cC^\leftarrow\to\cC^\rightarrow$ has rank at least two, which is impossible. Thus, $e_3\in\cC$. The claim follows by induction.

Now we prove (6).

First assume $m$ even. By the same notation and argument in the proof of (5), $\cC_1=\cF_{1,2l}$ and $\cC_2=\cF_{2,m-2l'}$ for integers $0\leq l\leq l'\leq\frac{m}{2}$. If $l=l'$, then $\dim_\kappa\cC=m$ hence $\cC$ is Lagrangian. If $l<l'$, then $\cC^\perp\cap\cF_{1,m}=\cF_{1,2l'}$. In particular, $e_{m-2l}\in\cC^\perp\setminus\cC$ but $\tV e_{m-2l}\in\cC$, which contradicts with the fact that $\tV\colon(\cC^\perp/C)^\leftarrow\to(\cC^\perp/C)^\rightarrow$ is injective.

Second assume $m$ odd. Put $\cG\coloneqq\kappa\langle f_1,e_2,f_3,\ldots\rangle$ so that $\cB(m)=\cF_m\oplus\cG$. Put $\cC_1\coloneqq\cC\cap\cF_m$ and denote by $\cC_2$ the image of $\cC$ in $\cG$ under the projection $\cB(m)\to\cG$ given by the previous orthogonal decomposition. We claim that $\cC_2=0$. Otherwise, $\cC_2$ is right-leaning but not neutral, which implies that $\cC_1$ is left-leaning but not neutral. By (3), we must have $\cC_1=\cF_m$, which implies that $\cC_2=0$. Thus, we have $\cC\subseteq\cF_m\subseteq\cC^\perp$. Since $\cC$ is neutral, $\cC\neq\cF_m$. Let $i$ be the smallest integer such that $\cF_i$ is not contained in $\cC$. Then $i$ has to be odd. In particular, $e_{m+1-i}\in\cC^\perp\setminus\cC$ but $\tV e_{m+1-i}\in\cC$, which contradicts with the fact that $\tV\colon(\cC^\perp/C)^\leftarrow\to(\cC^\perp/C)^\rightarrow$ is injective. Thus, the case that $m$ is odd cannot happen.

\end{proof}

\begin{notation}\label{no:eo}
Take a point $x=(A_0,\lambda_0,\eta_0^p;A,\lambda,\eta^p)\in\pres\varsigma\rM_N(\rK^p)(\kappa)$, where $\rK^p\in\fK^p$ and $\kappa$ is an algebraically closed field containing $\dF^\Phi_p$.
\begin{enumerate}
  \item We obtain a unitary Dieudonn\'{e} module $\cD_x$ of signature $(N-1,1)$, whose underlying $\tD_{W(\kappa)}$-module is $\cD(A)_{\tau_\infty}\oplus\cD(A)_{\tau_\infty^\tc}$ and the pairing is $\langle\;,\;\rangle_{\lambda,\tau_\infty}\oplus\langle\;,\;\rangle_{\lambda,\tau_\infty^\tc}$. In particular, $\cD_x\otimes\dF_p=\rH^\dr_1(A/\kappa)_{\tau_\infty}\oplus\rH^\dr_1(A/\kappa)_{\tau_\infty^\tc}$.

  \item By \cite{BW06}*{Proposition~3.6} and Lemma \ref{le:eo}(2), we have a unique integer $m_x\geq 1$ and a unique orthogonal decomposition $\cD_x\otimes\dF_p=\cB_x\oplus\cS_x$ of unitary Dieudonn\'{e} spaces such that $\cB_x\simeq\cB(m_x)$ and $\cS_x\simeq\cS^{N-m_x}$.

  \item By Lemma \ref{le:eo}(3), when $m_x$ is odd, we have a unique left-leaning Lagrangian $\tD_\kappa$-submodule $\cC_x$ of $\cB_x$, which satisfies $\dim_\kappa\cC_x^\leftarrow-\dim_\kappa\cC_x^\rightarrow=1$ and contains both $\nu_{A,\tau_\infty}$ (Notation \ref{no:frobenius_verschiebung}(4)) and $\omega_{A^\vee,\tau_\infty}$ by Remark \ref{re:length}.

  \item By Lemma \ref{le:eo}(5), when $m_x$ is even, we have for $0\leq l\leq \tfrac{m_x}{2}$ a unique neutral Lagrangian $\tD_\kappa$-submodule $\cC_x^l$ of $\cB_x$ that has $\tF$-length $l$ and $\tV$-length $\frac{m_x}{2}-l$. By Remark \ref{re:length}, $\cC_x^l$ contains $\nu_{A,\tau_\infty}$ (resp.\ $\omega_{A^\vee,\tau_\infty}$) if and only if $l\neq\tfrac{m}{2}$ (resp.\ $l\neq 0$).
\end{enumerate}
\end{notation}

\begin{notation}\label{no:strata}
Using $m_x$, we obtain the Ekedahl--Oort stratification
\begin{align*}
\pres\varsigma\rM_N(\rK^p)=\bigcup_{1\leq m\leq N}\pres\varsigma\rM^{(m)}_N(\rK^p)
\end{align*}
for every object $\rK^p\in\fK^p$, where $\pres\varsigma\rM^{(m)}_N(\rK^p)$ is the locally closed reduced subscheme such that $m_x=m$ for every geometric point $x$ of it. Furthermore, we
\begin{itemize}[label={\ding{118}}]
  \item adopt the convention that $\pres\varsigma\rM^{(m)}_N(\rK^p)=\emptyset$ for an integer $m$ that is not in the range $1\leq m\leq N$;

  \item denote by $\pres\varsigma\rM^{[m]}_N(\rK^p)$ the Zariski closure of $\pres\varsigma\rM^{(m)}_N(\rK^p)$ in $\pres\varsigma\rM_N(\rK^p)$;

  \item put $\pres\varsigma\bM^{]m[}_N(\rK^p)\coloneqq\pres\varsigma\bM_N(\rK^p)\setminus\pres\varsigma\rM^{[m]}_N(\rK^p)$;

  \item denote by $\pres\varsigma\rM^\rb_N(\rK^p)\subseteq\pres\varsigma\rM_N(\rK^p)$ the basic locus, that is, the locally closed reduced subscheme on which $A[\fp^\infty]$ is supersingular;

  \item write $\pres\varsigma\rM^\no_N(\rK^p)$ for $\pres\varsigma\rM^{[4]}_N(\rK^p)$ (resp.\ $\pres\varsigma\rM^\rb_N(\rK^p)$) when $N\geq 4$ (resp.\ $N\leq 3$) as the ``non-ordinary'' locus.
\end{itemize}
\end{notation}

\begin{proposition}\label{pr:eo}
Take an object $\rK^p\in\fK^p$.
\begin{enumerate}
  \item For every integer $1\leq m\leq N$,
      \[
      \pres\varsigma\rM^{[m]}_N(\rK^p)=
      \begin{dcases}
      \bigcup_{\substack{m'\geq m \\ m'\text{ even}}}
      \pres\varsigma\rM^{(m')}_N(\rK^p)
      \cup\bigcup_{m'\text{ odd}}\pres\varsigma\rM^{(m')}_N(\rK^p), & m\text{ even},\\
      \bigcup_{\substack{m'\leq m \\ m'\text{ odd}}}\pres\varsigma\rM^{(m')}_N(\rK^p), & m\text{ odd}.
      \end{dcases}
      \]

  \item We have $\pres\varsigma\rM^\rb_N(\rK^p)=\bigcup_{m\text{ odd}}\pres\varsigma\rM^{(m)}_N(\rK^p)$. Moreover,
      \[
      \pres\varsigma\rM_N(\rK^p)=\(\bigcup_{\substack{1\leq m\leq N \\ m\text{ even}}}
      \pres\varsigma\rM^{(m)}_N(\rK^p)\)\cup\pres\varsigma\rM^\rb_N(\rK^p)
      \]
      is the Newton stratification of $\pres\varsigma\rM_N(\rK^p)$.

  \item For every integer $1\leq m\leq N$, $\pres\varsigma\rM^{(m)}_N(\rK^p)$ is of pure dimension $\tfrac{m-1}{2}$ (resp.\ $N-\tfrac{m}{2}$) when $m$ is odd (resp.\ even).
\end{enumerate}
\end{proposition}

\begin{proof}
They are proved in \cite{BW06}*{\S5.4}.
\end{proof}

The following lemma and definition will be used later.

\begin{lem}\label{le:eo1}
Suppose that $N=2r$ is even. Take a point $x=(A_0,\lambda_0,\eta_0^p;A,\lambda,\eta^p)\in\pres\varsigma\rM_N(\rK^p)(\kappa)$, where $\kappa$ is an algebraically closed field containing $\dF^\Phi_p$.
\begin{enumerate}
  \item For every neutral Lagrangian $\tD_\kappa$-submodule $\cE$ of $\cD_x\otimes\dF_p$, we have $\cE=(\cE\cap\cB_x)\oplus(\cE\cap\cS_x)$ in which $\cE\cap\cB_x$ and $\cE\cap\cS_x$ are left-leaning and right-leaning Lagrangian $\tD_\kappa$-submodules of $\cB_x$ and $\cS_x$, respectively.

  \item When $m_x$ is odd, if $\cE$ is a neutral Lagrangian $\tD_\kappa$-submodule of $\cD_x\otimes\dF_p$, then $\cE\cap\cB_x=\cC_x$ (Notation \ref{no:eo}(3)).

  \item When $m_x$ is even, there are only finitely many neutral Lagrangian $\tD_\kappa$-submodules of $\cD_x\otimes\dF_p$.
\end{enumerate}
\end{lem}

\begin{proof}
We denote by $\cE'$ the image of $\cE$ in $\cS_x$, which is right-leaning. Thus, $\cE\cap\cB_x$ is left-leaning and totally isotropic. Since $\cE$ is totally isotropic, we have a short exact sequence
\begin{align}\label{eq:eo1}
0 \to \cE\cap\cS_x \to \cE' \to \frac{(\cE\cap\cB_x)^\perp_{\cB_x}}{\cE\cap\cB_x}
\end{align}
of $\tD_\kappa$-modules.

For (1), we first assume that $\cE'$ is \emph{not} neutral. Then $\cE\cap\cB_x$ is not neutral as well, which implies that $m_x$ is odd and $\cE\cap\cB_x=\cC_x$ is Lagrangian by Lemma \ref{le:eo}(3,4). By \eqref{eq:eo1}, we have $\cE'=\cE\cap\cS_x$, that is, $\cE=(\cE\cap\cB_x)\oplus(\cE\cap\cS_x)$, satisfying that $\cE\cap\cB_x$ and $\cE\cap\cS_x$ are left-leaning and right-leaning Lagrangian $\tD_\kappa$-submodules of $\cB_x$ and $\cS_x$, respectively.

We then assume that $\cE'$ is neutral. Then $\cE\cap\cB_x$ is neutral as well, whose rank over $\kappa\times\kappa$ we denote by $l$. Then the rank of $\cE'$ over $\kappa\times\kappa$ is $r-l$. Since $\cE$ is totally isotopic, we have $\cE\cap\cS_x\subseteq(\cE')^\perp_{\cS_x}$. Then the sequence \eqref{eq:eo1} implies that $(r-l)-(2r-m_x-(r-l))\leq (m_x-l)-l$, which is an equality. Thus, we must have $\cE\cap\cS_x=(\cE')^\perp_{\cS_x}$ and that the induced map
\[
\frac{\cE'}{(\cE')^\perp_{\cS_x}}\to \frac{(\cE\cap\cB_x)^\perp_{\cB_x}}{\cE\cap\cB_x}
\]
is an isomorphism of $\tD_\kappa$-modules. Since $\cE'/(\cE')^\perp_{\cS_x}$ is a Dieudonn\'{e} space of signature $(m_x-2l,0)$, we must have $m_x=2l$ by Lemma \ref{le:eo}(6). In particular, $m_x$ is even; and $\cE=(\cE\cap\cB_x)\oplus(\cE\cap\cS_x)$ in which $\cE\cap\cB_x$ and $\cE\cap\cS_x$ are neutral Lagrangian $\tD_\kappa$-submodules of $\cB_x$ and $\cS_x$, respectively. Part (1) is proved.

Part (2) follows from (1) and the property of $\cC_x$ in Notation \ref{no:eo}(3).

Part (3) follows from (1), Lemma \ref{le:eo}(5), and the easy fact that there are only finitely many neutral Lagrangian $\tD_\kappa$-submodules of $\cS_x\simeq\cS^{2r-m_x}$.\footnote{The number of such submodules is $(1+p)(1+p^3)\cdots(1+p^{2r-m_x-1})$.}
\end{proof}

\begin{definition}\label{de:neutral}
Let $(G,\lambda)$ be a unitary $p$-truncated group over a scheme $S\in\Sch_{/\dZ_{p^2}}$ (Definition \ref{de:unitary_divisible}). We say that an $\dF_{p^2}$-linear subgroup $C$ of $G$ is
\begin{enumerate}
  \item \emph{Lagrangian} if $C$ is finite flat such that $\rank G=(\rank C)^2$ and that the composite homomorphism $C\to G\xrightarrow\lambda G^\vee\to C^\vee$ vanishes;

  \item \emph{neutral} if for every geometric point $x\in S(\kappa)$ of characteristic $p$, the $\tD_\kappa$-module given by $C$ is neutral.
\end{enumerate}
\end{definition}

\subsection{Unitary moduli schemes at Siegel parahoric level}

In this subsection, we define unitary moduli schemes at Siegel parahoric level, which is a key geometric object to define and study the Tate--Thompson local system.

From this point to the end of this section, we assume $N=2r$ \emph{even}. In particular, $\fS$ is simply the group of totally positive elements $\varsigma$ in $F^+$.

\begin{definition}\label{de:siegel_parahoric}
For every $\varsigma\in\fS$, we define a functor
\begin{align*}
\pres\varsigma\bP_N\colon\fK^p\times\fT &\to\Sch'_{/\dZ^\Phi_p} \\
\rK^p &\mapsto \pres\varsigma\bP_N(\rK^p)
\end{align*}
such that for every $S\in\Sch'_{/\dZ^\Phi_p}$, $\pres\varsigma\bP_N(\rK^p)(S)$ is the set of equivalence classes of decuples
\[
(A_0,\lambda_0,\eta_0^p;A^\triangle,\lambda^\triangle,\eta^{p\triangle};
A^\blacktriangle,\lambda^\blacktriangle,\eta^{p\blacktriangle};\psi^\triangle),
\]
where
\begin{itemize}[label={\ding{118}}]
  \item $(A_0,\lambda_0,\eta_0^p;A^\triangle,\lambda^\triangle,\eta^{p\triangle})\in\pres\varsigma\bM_N(\rK^p)(S)$ and $(A_0,\lambda_0,\eta_0^p;A^\blacktriangle,\lambda^\blacktriangle,\eta^{p\blacktriangle})\in\pres{\varsigma\varpi}\bM_N(\rK^p)(S)$;

  \item $\psi^\triangle\colon A^\triangle\to A^\blacktriangle$ is an $O_F$-linear quasi-$p$-isogeny such that
    \begin{enumerate}[label=(\alph*)]
      \item $\Ker\psi^\triangle[p^\infty]$ is contained in $A^\triangle[\fp]$;

      \item we have $\varpi\cdot\lambda^\triangle=\psi^{\triangle\vee}\circ\lambda^\blacktriangle\circ\psi^\triangle$; and

      \item the $\rK^p$-orbit of maps $v\mapsto\psi^\triangle_*\circ\eta^{p\triangle}(v)$ for $v\in\rV\otimes_\dQ\dA^{\infty,p}$ coincides with $\eta^{p\blacktriangle}$.
  \end{enumerate}
\end{itemize}
The definitions of the equivalence relation and the action of morphisms in $\fK^p\times\fT$ are similar to \cite{LTXZZ}*{Definition~4.3.3}. In what follows, for an object of $\bP_N(\rK^p)(S)$ as above, we denote by $\psi^\blacktriangle\colon A^\blacktriangle\to A^\triangle$ the $O_F$-linear quasi-$p$-isogeny such that $\psi^\blacktriangle\circ\psi^\triangle=\varpi\cdot\id_{A^\triangle}$.
\end{definition}

In $\Fun(\fK^p\times\fT,\Sch'_{/\dZ_p^\Phi})_{/\bT}$, we have the morphism
\begin{align}\label{eq:siegel_parahoric}
\pres\varsigma\bff\colon\pres\varsigma\bP_N \to \pres\varsigma\bM_N
\end{align}
sending $(A_0,\lambda_0,\eta_0^p;A^\triangle,\lambda^\triangle,\eta^{p\triangle};
A^\blacktriangle,\lambda^\blacktriangle,\eta^{p\blacktriangle};\psi^\triangle)$ to $(A_0,\lambda_0,\eta_0^p;A^\triangle,\lambda^\triangle,\eta^{p\triangle})$ and the morphism
\begin{align*}\label{eq:siegel_involution}
\pres\varsigma\bi\colon \pres\varsigma\bP_N \to \pres{\varsigma\varpi}\bP_N
\end{align*}
sending $(A_0,\lambda_0,\eta_0^p;A^\triangle,\lambda^\triangle,\eta^{p\triangle};
A^\blacktriangle,\lambda^\blacktriangle,\eta^{p\blacktriangle};\psi^\triangle)$ to $(A_0,\lambda_0,\eta_0^p;A^\blacktriangle,\lambda^\blacktriangle,\eta^{p\blacktriangle};
A^\triangle,\lambda^\triangle,\varpi\eta^{p\triangle};\psi^\blacktriangle)$.

Similar to Remark \ref{re:square}, $\pres\varsigma\bP_N$ and $\pres{\varsigma\varpi^2}\bP_N$ are canonically identified. It is clear that under such identification, $\pres{\varsigma\varpi}\bi\circ\pres\varsigma\bi$ and $\pres\varsigma\bi\circ\pres{\varsigma\varpi}\bi$ coincide with the identity morphisms on $\pres\varsigma\bP_N$ and $\pres{\varsigma\varpi}\bP_N$, respectively.

\begin{remark}\label{re:siegel_parahoric}
By \cite{LTXZZ}*{Lemma~3.4.12(2)}, $\pres\varsigma\bP_N(\rK^p)(S)$ is also the set of equivalence classes of septuples $(A_0,\lambda_0,\eta_0^p;A^\triangle,\lambda^\triangle,\eta^{p\triangle};C)$ in which $C$ (corresponding to $\Ker\psi^\triangle[\fp^\infty]$) is an $\dF_{p^2}$-linear subgroup of $A^\triangle[\fp]$ that is Lagrangian and neutral (Definition \ref{de:neutral}).
\end{remark}

\begin{definition}\label{de:siegel_parahoric_1}
For every $\rK^p\in\fK^p$, let $(\cA_0,\lambda_0,\eta_0^p;\cA^\triangle,\lambda^\triangle,\eta^{p\triangle};
\cA^\blacktriangle,\lambda^\blacktriangle,\eta^{p\blacktriangle};\psi^\triangle)$ be the universal object over $\pres\varsigma\bP_N(\rK^p)$. We define
\begin{enumerate}
  \item $\pres\varsigma\rP^\triangle_N(\rK^p)$ to be the reduced closed subscheme of $\pres\varsigma\rP_N(\rK^p)$ on which $\psi^\blacktriangle_{*,\tau_\infty}(\omega_{\cA^{\blacktriangle\vee},\tau_\infty})=0$;

  \item $\pres\varsigma\rP^\blacktriangle_N(\rK^p)$ to be the reduced closed subscheme of $\pres\varsigma\rP_N(\rK^p)$ on which $\psi^\triangle_{*,\tau_\infty}(\omega_{\cA^{\triangle\vee},\tau_\infty})=0$;

  \item $\pres\varsigma\rP^\ddag_N(\rK^p)$ to be $\pres\varsigma\rP^\triangle_N(\rK^p)\bigcap\pres\varsigma\rP^\blacktriangle_N(\rK^p)$;

  \item $\pres\varsigma\rP^?_N(\rK^p)$ to be $\pres\varsigma\rM^?_N(\rK^p)\times_{\pres\varsigma\rM_N(\rK^p)}\pres\varsigma\rP_N(\rK^p)$ for $?\in\{(m),[m],\rb\}$;

  \item $\pres\varsigma\bP^{]m[}_N(\rK^p)$ to be $\pres\varsigma\bM^{]m[}_N(\rK^p)\times_{\pres\varsigma\bM_N(\rK^p)}\pres\varsigma\bP_N(\rK^p)$.
\end{enumerate}
For $?\in\{\triangle,\blacktriangle,\ddag,(m),[m],\rb\}$, we denote by $\pres\varsigma\rf^?$ the restrictions of $\pres\varsigma\rf$ in \eqref{eq:siegel_parahoric} to $\pres\varsigma\rP^?_N$ and by $\pres\varsigma\rp^?\colon\pres\varsigma\rP^?_N\to\pres\varsigma\rP_N$ the natural inclusion morphism, and similarly for $\pres\varsigma\bff^{]m[}$.
\end{definition}

\begin{theorem}\label{th:siegel_parahoric}
For every $\rK^p\in\fK^p$, we have
\begin{enumerate}
  \item The scheme $\pres\varsigma\bP_N(\rK^p)$ is quasi-projective and strictly semistable over $\bT$ of relative dimension $N-1$, such that
      \[
      \pres\varsigma\rP_N(\rK^p)=\pres\varsigma\rP^\triangle_N(\rK^p)\bigcup\pres\varsigma\rP^\blacktriangle_N(\rK^p).
      \]
      Moreover, $\pres\varsigma\rP^\triangle_N(\rK^p)$, $\pres\varsigma\rP^\blacktriangle_N(\rK^p)$ and $\pres\varsigma\rP^\ddag_N(\rK^p)$ are all smooth over $\rT$.

  \item The morphism $\pres\varsigma\bff$ is proper.

  \item The relative tangent sheaf $\cT_{\pres\varsigma\rP^\triangle_N(\rK^p)/\rT}$ fits canonically into an exact sequence
      \[
      \resizebox{\hsize}{!}{
      \xymatrix{
      0 \ar[r]& \HOM\(\omega_{\cA^{\triangle\vee},\tau_\infty},(\psi^\triangle_{*,\tau_\infty})^{-1}
      \omega_{\cA^{\blacktriangle\vee},\tau_\infty}/\omega_{\cA^{\triangle\vee},\tau_\infty}\)
      \ar[r]& \cT_{\pres\varsigma\rP^\triangle_N(\rK^p)/\rT}
      \ar[r]& \HOM\(\omega_{\cA^{\blacktriangle\vee},\tau_\infty},
      \Ker\psi^\blacktriangle_{*,\tau_\infty}/\omega_{\cA^{\blacktriangle\vee},\tau_\infty}\) \ar[r]& 0
      }
      }
      \]
      of coherent sheaves over $\pres\varsigma\rP^\triangle_N(\rK^p)$.

  \item The relative tangent sheaf $\cT_{\pres\varsigma\rP^\blacktriangle_N(\rK^p)/\rT}$ fits canonically into an exact sequence
      \[
      \resizebox{\hsize}{!}{
      \xymatrix{
      0 \ar[r]& \HOM\(\omega_{\cA^{\blacktriangle\vee},\tau_\infty},(\psi^\blacktriangle_{*,\tau_\infty})^{-1}
      \omega_{\cA^{\triangle\vee},\tau_\infty}/\omega_{\cA^{\blacktriangle\vee},\tau_\infty}\)
      \ar[r]& \cT_{\pres\varsigma\rP^\blacktriangle_N(\rK^p)/\rT}
      \ar[r]& \HOM\(\omega_{\cA^{\triangle\vee},\tau_\infty},
      \Ker\psi^\triangle_{*,\tau_\infty}/\omega_{\cA^{\triangle\vee},\tau_\infty}\) \ar[r]& 0
      }
      }
      \]
      of coherent sheaves over $\pres\varsigma\rP^\blacktriangle_N(\rK^p)$.

  \item The natural map $\cT_{\pres\varsigma\rP^\ddag_N(\rK^p)/\rT}\to\cT_{\pres\varsigma\rP^\triangle_N(\rK^p)/\rT}\res_{\pres\varsigma\rP^\ddag_N(\rK^p)}
      \bigoplus\cT_{\pres\varsigma\rP^\blacktriangle_N(\rK^p)/\rT}\res_{\pres\varsigma\rP^\ddag_N(\rK^p)}$ between relative tangent sheaves induces an isomorphism
      \[
      \cT_{\pres\varsigma\rP^\ddag_N(\rK^p)/\rT}\simeq
      \HOM\(\omega_{\cA^{\blacktriangle\vee},\tau_\infty},
      \Ker\psi^\blacktriangle_{*,\tau_\infty}/\omega_{\cA^{\blacktriangle\vee},\tau_\infty}\)
      \bigoplus
      \HOM\(\omega_{\cA^{\triangle\vee},\tau_\infty},
      \Ker\psi^\triangle_{*,\tau_\infty}/\omega_{\cA^{\triangle\vee},\tau_\infty}\)
      \]
      of coherent sheaves over $\pres\varsigma\rP^\ddag_N(\rK^p)$ under the two exact sequences in (3,4). In particular, the two exact sequences in (3,4) both split over $\pres\varsigma\rP^\ddag_N(\rK^p)$.

  \item The morphism $\pres\varsigma\rf^\triangle$ is \'{e}tale away from $\pres\varsigma\rP^\ddag_N(\rK^p)$.

  \item For an $S$-point $(A_0,\lambda_0,\eta_0^p;A^\triangle,\lambda^\triangle,\eta^{p\triangle};
      A^\blacktriangle,\lambda^\blacktriangle,\eta^{p\blacktriangle};\psi^\triangle)\in\pres\varsigma\rP_N(\rK^p)(S)$ with $S\in\Sch'_{/\dF_p^\Phi}$ regular, it belongs to $\pres\varsigma\rP^\triangle_N(\rK^p)$ (resp.\ $\pres\varsigma\rP^\blacktriangle_N(\rK^p)$) if and only if $\psi^\triangle_{*,\tau_\infty}(\nu_{A^\triangle/S,\tau_\infty})=0$ (resp.\ $\psi^\blacktriangle_{*,\tau_\infty}(\nu_{A^\blacktriangle/S,\tau_\infty})=0$) (Notation \ref{no:frobenius_verschiebung}(4)).

  \item Under the isomorphism in (5), we have
      \[
      \cT_{\pres\varsigma\rP^\ddag_N(\rK^p)\cap\pres\varsigma\rP^{(1)}_N(\rK^p)/\rT}\bigcap\HOM\(\omega_{\cA^{\triangle\vee},\tau_\infty},
      \Ker\psi^\triangle_{*,\tau_\infty}/\omega_{\cA^{\triangle\vee},\tau_\infty}\)=\{0\}.
      \]
\end{enumerate}
\end{theorem}

\begin{proof}
For (1), the quasi-projectiveness follows from the fact that the induced morphism
\[
\pres\varsigma\bff\times(\pres{\varsigma\varpi}\bff\circ\pres\varsigma\bi)\colon
\pres\varsigma\bP_N(\rK^p)\to\pres\varsigma\bM_N(\rK^p)\times_{\bT}\pres{\varsigma\varpi}\bM_N(\rK^p)
\]
is a closed immersion and that the target is quasi-projective over $\bT$ by \cite{LTXZZ}*{Theorem~4.2.3}.

For the remaining part of (1), take a point
\[
x=(A_0,\lambda_0,\eta_0^p;A^\triangle,\lambda^\triangle,\eta^{p\triangle};
A^\blacktriangle,\lambda^\blacktriangle,\eta^{p\blacktriangle};\psi^\triangle)\in\pres\varsigma\bP_N(\rK^p)(\kappa)
\]
for a perfect field $\kappa$ containing $\dF_p^\Phi$, and denote by $\cO_x$ the completed local ring of $\pres\varsigma\bP_N(\rK^p)$ at $x$. For every $\tau\in\Sigma_\infty$, we may identify $\cD(A^\triangle)_\tau$ as a $W(\kappa)$-submodule of $\cD(A^\blacktriangle)_\tau$ so that $p\cD(A^\blacktriangle)_\tau\subseteq\cD(A^\triangle)_\tau\subseteq\cD(A^\blacktriangle)_\tau$ for $\tau\in\{\tau_\infty,\tau_\infty^\tc\}$ and $\cD(A^\triangle)_\tau=\cD(A^\blacktriangle)_\tau$ for $\tau\in\Sigma_\infty\setminus\{\tau_\infty,\tau_\infty^\tc\}$.

By \cite{LTXZZ}*{Proposition~3.4.8} and the fact that $\lambda^\triangle$ and $\lambda^\blacktriangle$ are $p$-principal, for every Artinian $W(\kappa)$-ring $R$ that is a quotient of $\cO_x$, $\Hom_{W(\kappa)}(\cO_x,R)$ is the set of pairs of $R$-subbundles
\[
M^\triangle_{\tau_\infty}\subseteq\cD(A^\triangle)_{\tau_\infty}\otimes_{W(\kappa)}R,\quad
M^\blacktriangle_{\tau_\infty}\subseteq\cD(A^\blacktriangle)_{\tau_\infty}\otimes_{W(\kappa)}R
\]
of rank $1$ lifting $\omega_{A^{\triangle\vee}/\kappa,\tau_\infty}$ and $\omega_{A^{\blacktriangle\vee}/\kappa,\tau_\infty}$, respectively, such that
\begin{itemize}[label={\ding{118}}]
  \item the image of $M^\triangle_{\tau_\infty}$ is contained in $M^\blacktriangle_{\tau_\infty}$ under the natural map $\cD(A^\triangle)_{\tau_\infty}\otimes_{W(\kappa)}R\to\cD(A^\blacktriangle)_{\tau_\infty}\otimes_{W(\kappa)}R$;

  \item the image of $M^\blacktriangle_{\tau_\infty}$ is contained in $p^{-1}M^\triangle_{\tau_\infty}$ under the natural map $\cD(A^\blacktriangle)_{\tau_\infty}\otimes_{W(\kappa)}R\to p^{-1}\cD(A^\triangle)_{\tau_\infty}\otimes_{W(\kappa)}R$.
\end{itemize}
By \cite{LTXZZ}*{Lemma~3.4.12}, for $\tau\in\{\tau_\infty,\tau_\infty^\tc\}$, $\cD(A^\blacktriangle)_\tau/\cD(A^\triangle)_\tau$ is a $\kappa$-vector space of dimension $r$. In particular, we may assume
\[
\cD(A^\triangle)_{\tau_\infty}=pW(\kappa)\oplus W(\kappa)\oplus pW(\kappa)\oplus W(\kappa)\oplus\cdots,\quad
\cD(A^\blacktriangle)_{\tau_\infty}=W(\kappa)^{\oplus N}.
\]
There are three possible cases.
\begin{enumerate}[label=(\roman*)]
  \item If $\psi^\blacktriangle_{*,\tau_\infty}(\omega_{A^{\blacktriangle\vee},\tau_\infty})=0$ but $\psi^\triangle_{*,\tau_\infty}(\omega_{A^{\triangle\vee},\tau_\infty})\neq 0$, then $M^\triangle_{\tau_\infty}$ determines $M^\blacktriangle_{\tau_\infty}$. Without loss of generality, we may assume that $\omega_{A^{\triangle\vee},\tau_\infty}$ is generated by by $(0,1,0,\ldots,0)$. It follows that $\cO_x\simeq W(\kappa)[[x_1,x_3,x_4,\ldots,x_N]]$

  \item If $\psi^\triangle_{*,\tau_\infty}(\omega_{A^{\triangle\vee},\tau_\infty})=0$ but $\psi^\blacktriangle_{*,\tau_\infty}(\omega_{A^{\blacktriangle\vee},\tau_\infty})\neq 0$, then $M^\blacktriangle_{\tau_\infty}$ determines $M^\triangle_{\tau_\infty}$. Without loss of generality, we may assume that $\omega_{A^{\blacktriangle\vee},\tau_\infty}$ is generated by $(1,0,\ldots,0)$. It follows that $\cO_x\simeq W(\kappa)[[x_2,\ldots,x_N]]$.

  \item If both $\psi^\triangle_{*,\tau_\infty}(\omega_{A^{\triangle\vee},\tau_\infty})=0$ and $\psi^\blacktriangle_{*,\tau_\infty}(\omega_{A^{\blacktriangle\vee},\tau_\infty})=0$, then without loss of generality, we may assume that $\omega_{A^{\triangle\vee},\tau_\infty}$ is generated by $(p,0,\ldots,0)$ and $\omega_{A^{\blacktriangle\vee},\tau_\infty}$ is generated by $(0,1,0,\ldots,0)$. Then there are unique generators of $M^\triangle_{\tau_\infty}$ and $M^\blacktriangle_{\tau_\infty}$ of the form $(p,x_2,px_3,x_4,\ldots)$ and $(x'_1,1,x'_3,x'_4\ldots)$ with indeterminants in the maximal ideal of $R$. Then the two compatibility conditions between $M^\triangle_{\tau_\infty}$ and $M^\blacktriangle_{\tau_\infty}$ imply that $\cO_x\simeq W(\kappa)[[x'_1,x_2,x_3,x'_4,x_5,x'_6,\cdots]]/(x'_1x_2-p)$.
\end{enumerate}
Thus, $\pres\varsigma\bP_N(\rK^p)$ is strictly semistable over $\bT$ of relative dimension $N-1$. Moreover, $\pres\varsigma\rP^\triangle_N(\rK^p)$ is the locus where (i) or (iii) happens; and $\pres\varsigma\rP^\blacktriangle_N(\rK^p)$ is the locus where (ii) or (iii) happens. Thus, (1) follows.

For (2), the properness of $\pres\varsigma\bff$ follows from Remark \ref{re:siegel_parahoric}.

For (3--5), we will use deformation theory. For common use, we consider a closed immersion $S\hookrightarrow\hat{S}$ in $\Sch'_{/\rT}$ defined by an ideal sheaf $\cI$ with $\cI^2=0$. Take an $S$-point $(A_0,\lambda_0,\eta_0^p;A^\triangle,\lambda^\triangle,\eta^{p\triangle};
A^\blacktriangle,\lambda^\blacktriangle,\eta^{p\blacktriangle};\psi^\triangle)$ in various schemes we will consider. By \cite{LTXZZ}*{Proposition~3.4.8}, we need to lift $\omega_{A^{\triangle\vee},\tau_\infty}$ and $\omega_{A^{\blacktriangle\vee},\tau_\infty}$ to subbundles $\hat\omega_{A^{\triangle\vee},\tau_\infty}\subseteq\rH^\cris_1(A^\triangle/\hat{S})_{\tau_\infty}$ and $\hat\omega_{A^{\blacktriangle\vee},\tau_\infty}\subseteq\rH^\cris_1(A^\blacktriangle/\hat{S})_{\tau_\infty}$, respectively, satisfying $\psi^\triangle_{*,\tau_\infty}(\hat\omega_{A^{\triangle\vee},\tau_\infty})
\subseteq\hat\omega_{A^{\blacktriangle\vee},\tau_\infty}$ and $\psi^\blacktriangle_{*,\tau_\infty}(\hat\omega_{A^{\blacktriangle\vee},\tau_\infty})
\subseteq\hat\omega_{A^{\triangle\vee},\tau_\infty}$.

For (3), we need to first find lifting $\hat\omega_{A^{\blacktriangle\vee},\tau_\infty}$ that is contained in $\Ker\psi^\blacktriangle_{*,\tau_\infty}$ by definition; and then find lifting $\hat\omega_{A^{\triangle\vee},\tau_\infty}$ that is contained in $(\psi^\triangle_{*,\tau_\infty})^{-1}\omega_{\cA^{\blacktriangle\vee},\tau_\infty}$. Thus, (3) follows. By a symmetric argument, (4) follows as well.

For (5), we need to find liftings $\hat\omega_{A^{\blacktriangle\vee},\tau_\infty}$ and $\hat\omega_{A^{\triangle\vee},\tau_\infty}$ that are contained in $\Ker\psi^\blacktriangle_{*,\tau_\infty}$ and $\Ker\psi^\triangle_{*,\tau_\infty}$, respectively. Thus, (5) follows.

For (6), it suffices to show that the natural map
$\pres\varsigma\rf^\triangle_*\cT_{\pres\varsigma\rP^\triangle_N(\rK^p)/\rT}\to\cT_{\pres\varsigma\rM_N(\rK^p)/\rT}$ is an isomorphism away from the image of $\pres\varsigma\rf^\ddag$. Indeed, this follows from the fact that over $\pres\varsigma\rP^\triangle_N(\rK^p)\setminus\pres\varsigma\rP^\ddag_N(\rK^p)$, the lifting $\hat\omega_{A^{\blacktriangle\vee},\tau_\infty}$ is determined by $\hat\omega_{A^{\triangle\vee},\tau_\infty}$ as $\psi^\triangle_{*,\tau_\infty}(\hat\omega_{A^{\triangle\vee},\tau_\infty})$.

For (7), we show the statement for $\pres\varsigma\rP^\triangle_N(\rK^p)$ and the other one follows by symmetry. We have the commutative diagram
\[
\xymatrix{
\rH^\dr_1(A^\triangle/S) \ar[r]^-{\psi^\triangle_*} \ar[d]^-{\tV_{A^\triangle}} & \rH^\dr_1(A^\blacktriangle/S) \ar[r]^-{\psi^\blacktriangle_*}
\ar[d]^-{\tV_{A^\blacktriangle}} & \rH^\dr_1(A^\triangle/S) \ar[d]^-{\tV_{A^\triangle}} \\
\rH^\dr_1(A^{\triangle(p)}/S) \ar[r]^-{\psi^{\triangle(p)}_*} \ar[d]^-{\tF_{A^\triangle}} & \rH^\dr_1(A^{\blacktriangle(p)}/S) \ar[r]^-{\psi^{\blacktriangle(p)}_*} \ar[d]^-{\tF_{A^\blacktriangle}}
& \rH^\dr_1(A^{\triangle(p)}/S) \ar[d]^-{\tF_{A^\triangle}}  \\
\rH^\dr_1(A^\triangle/S) \ar[r]^-{\psi^\triangle_*} & \rH^\dr_1(A^\blacktriangle/S) \ar[r]^-{\psi^\blacktriangle_*}
& \rH^\dr_1(A^\triangle/S)
}
\]
of $\cO_S$-modules. Since $S$ is regular, the absolute Frobenius morphism of $S$ is faithfully flat. It follows that $\psi^\blacktriangle_{*,\tau_\infty}(\omega_{\cA^{\blacktriangle\vee}/S,\tau_\infty})=0$ is equivalent to that $\psi^{\blacktriangle(p)}_{*,\tau_\infty^\tc}(\omega_{\cA^{\blacktriangle(p)\vee}/S,\tau_\infty^\tc})=0$, which is equivalent to that the induce map $\psi^{\blacktriangle(p)}_*\circ\tV_{A^\blacktriangle}
\colon\rH^\dr_1(A^\blacktriangle/S)_{\tau_\infty}\to\rH^\dr_1(A^{\triangle(p)}/S)_{\tau_\infty^\tc}$ vanishes. This is further equivalent to the vanishing of its dual map, which is canonically identified with $\tF_{A^\blacktriangle}\circ\psi^{\triangle(p)}_*\colon
\rH^\dr_1(A^{\triangle(p)}/S)_{\tau_\infty}\to\rH^\dr_1(A^\blacktriangle/S)_{\tau_\infty^\tc}$. Thus, $\psi^\blacktriangle_{*,\tau_\infty}(\omega_{\cA^{\blacktriangle\vee}/S,\tau_\infty})=0$ if and only if $\psi^\triangle_{*,\tau_\infty}(\nu_{A^\triangle/S,\tau_\infty})=0$.

For (8), since $\pres\varsigma\rM^{(1)}_N$ is smooth over $\rT$ of dimension zero, $\cT_{\pres\varsigma\rP^\ddag_N(\rK^p)\cap\pres\varsigma\rP^{(1)}_N(\rK^p)/\rT}$ is contained in the kernel of the induced map
\[
(\pres\varsigma\rf^\ddag)_*\colon\cT_{\pres\varsigma\rP^\ddag_N(\rK^p)/\rT}\to\cT_{\pres\varsigma\rM_N(\rK^p)/\rT}
\]
on the tangent bundle, in which the target is canonically isomorphic to $\HOM\(\omega_{\cA^{\triangle\vee},\tau_\infty},\rH^\dr_1(\cA^\triangle)_{\tau_\infty}/\omega_{\cA^{\triangle\vee},\tau_\infty}\)$ by \cite{LTXZZ}*{Theorem~4.2.3}. It is straightforward to check that the restriction of $(\pres\varsigma\rf^\ddag)_*$ to the second direct summand $\HOM\(\omega_{\cA^{\triangle\vee},\tau_\infty},\Ker\psi^\triangle_{*,\tau_\infty}/\omega_{\cA^{\triangle\vee},\tau_\infty}\)$ in (5) is the inclusion map
\[
\HOM\(\omega_{\cA^{\triangle\vee},\tau_\infty},\Ker\psi^\triangle_{*,\tau_\infty}/\omega_{\cA^{\triangle\vee},\tau_\infty}\)
\hookrightarrow\HOM\(\omega_{\cA^{\triangle\vee},\tau_\infty},\rH^\dr_1(\cA^\triangle)_{\tau_\infty}/\omega_{\cA^{\triangle\vee},\tau_\infty}\).
\]
Thus, (8) follows.

The theorem is proved.
\end{proof}

\begin{notation}\label{no:fiber}
For every point $x=(A_0,\lambda_0,\eta_0^p;A,\lambda,\eta^p)\in\pres\varsigma\rM_N(\rK^p)(\kappa)$, where $\kappa$ is an algebraically closed field containing $\dF^\Phi_p$, we denote by $\rP_x$ the induced reduced closed subscheme of (the scheme-theoretic fiber) $\pres\varsigma\rf^{-1}(x)\subseteq\pres\varsigma\rP_N(\rK^p)_\kappa$.
\end{notation}

\begin{proposition}\label{pr:siegel_parahoric}
For every $\rK^p\in\fK^p$, we have
\begin{enumerate}
  \item The subscheme $\pres\varsigma\rP^\rb_N(\rK^p)$ is set-theoretically contained in $\pres\varsigma\rP^\ddag_N(\rK^p)$.

  \item The morphism
     \[
     \pres\varsigma\bff^{]2r-3[}\colon\pres\varsigma\bP^{]2r-3[}_N(\rK^p)\to\pres\varsigma\bM^{]2r-3[}_N(\rK^p)
     \]
     is finite flat.

  \item For every point $x\in\pres\varsigma\rM^\rb_N(\rK^p)(\kappa)$, where $\kappa$ is an algebraically closed field containing $\dF^\Phi_p$, $\rP_x$ is an irreducible proper scheme over $\kappa$ of dimension $\tfrac{N-m_x-1}{2}$.

  \item The morphisms
     \[
     \pres\varsigma\rf^\triangle\colon\pres\varsigma\rP^\triangle_N(\rK^p)\to\pres\varsigma\rM_N(\rK^p),\quad
     \pres\varsigma\rf^\blacktriangle\colon\pres\varsigma\rP^\blacktriangle_N(\rK^p)\to\pres\varsigma\rM_N(\rK^p)
     \]
     are both small.

  \item The image of $\pres\varsigma\rf^\ddag$ is disjoint from $\pres\varsigma\rM^{(2)}_N(\rK^p)$ and the induced morphism
     \[
     \pres\varsigma\rf^\ddag\colon\pres\varsigma\rP^\ddag_N(\rK^p)\to\pres\varsigma\rM_N^\no(\rK^p)
     \]
     is semismall with respect to the Ekedahl--Oort stratification (of the target), under which the relevant strata are $\pres\varsigma\rM^{(4)}_N(\rK^p)$ (which is empty when $N=2$) and $\pres\varsigma\rM^{(1)}_N(\rK^p)$.
\end{enumerate}
\end{proposition}

\begin{proof}
For (1), take a geometric point $x=(A_0,\lambda_0,\eta_0^p;A^\triangle,\lambda^\triangle,\eta^{p\triangle};
A^\blacktriangle,\lambda^\blacktriangle,\eta^{p\blacktriangle};\psi^\triangle)\in\pres\varsigma\rP^\rb_N(\rK^p)(\kappa)$. By Lemma \ref{le:eo1}(1,2), $\Ker\psi^\triangle_{*,\tau_\infty}$ contains both $\nu_{A^\triangle,\tau_\infty}$ and $\omega_{A^{\triangle\vee},\tau_\infty}$. Then (1) follows by Theorem \ref{th:siegel_parahoric}(7).

For (2), since $\pres\varsigma\bff$ is a proper morphism between regular schemes of same (pure) dimension by Theorem \ref{th:siegel_parahoric}(1,2), it suffices to show that for every geometric point $x\in\pres\varsigma\bM^{]2r-3[}_N(\rK^p)(\kappa)$, the underlying set of $\pres\varsigma\bff^{-1}(x)$ is finite by the miracle flatness criterion over regular local rings \cite{Mat89}*{Theorem~23.1}. When $\kappa$ has characteristic $0$, this follows from the fact that $\pres\varsigma\bff^\eta$ is finite \'{e}tale. When $\kappa$ has characteristic $p$, the underlying set $\pres\varsigma\bff^{-1}(x)$ is the same as the set of neutral Lagrangian $\tD_\kappa$-submodules of $\cD_{\rf(x)}\otimes\dF_p$, which is finite by Proposition \ref{pr:eo}(1), Lemma \ref{le:eo1}(3) and (3) (for $m_x=N-1$). Part (2) follows.

For (3), write $x=(A_0,\lambda_0,\eta_0^p;A,\lambda,\eta^p)$. We define a pairing $\{\;,\;\}_x\colon\cS_x^\rightarrow\times\cS_x^\rightarrow\to\kappa$ by the formula $\{u,v\}_x\coloneqq\langle\tV^{-1}u,v\rangle$. Then the pair $(\cS_x^\rightarrow,\{\;,\;\}_x)$ is admissible of rank $N-m_x$ (Definition \ref{de:dl_admissible}). By Lemma \ref{le:eo1}(1) and the fact that $\rP_x$ is reduced, the intersection of $\Ker\psi^\triangle_{*,\tau_\infty}$ (for the universal object over $\rP_x$) and $\cS_x^\rightarrow\otimes_\kappa\cO_{\rP_x}$ in $\rH^\dr_1(A/\kappa)_{\tau_\infty^\tc}\otimes_\kappa\cO_{\rP_x}=\rH^\dr_1(\cA/\rP_x)_{\tau_\infty^\tc}$ defines an element in $\DL(\cS_x^\rightarrow,\{\;,\;\}_x,\tfrac{N-m_x+1}{2})(\rP_x)$ (Definition \ref{de:dl}), which gives rise to a morphism $\zeta_x\colon\rP_x\to\DL(\cS_x^\rightarrow,\{\;,\;\}_x,\tfrac{N-m_x+1}{2})$. Using Lemma \ref{le:eo1}(1,2), it is straightforward to check that for every algebraically closed field $\kappa'$ containing $\kappa$, $\zeta_x(\kappa')$ is a bijection. It follows that $\zeta_x$ is a universal homeomorphism, which implies (3) by Proposition \ref{pr:dl}.

For (4), we show the statement for $\pres\varsigma\rf^\triangle$ and the case for $\pres\varsigma\rf^\blacktriangle$ is similar. Consider the stratification $\pres\varsigma\rM^\rb_N(\rK^p)=\bigcup_{m\text{ odd}}\pres\varsigma\rM^{(m)}_N(\rK^p)$ from Proposition \ref{pr:eo}(2). For every odd positive integer $m$ and every geometric point $x\in\pres\varsigma\rM^{(m)}_N(\rK^p)(\kappa)$, we have $\dim{}(\pres\varsigma\rf^\triangle)^{-1}(x)=\tfrac{2r-m-1}{2}$ by (1,3). By Proposition \ref{pr:eo}(3), we have $\dim\pres\varsigma\rM^{(m)}_N(\rK^p)=\tfrac{m-1}{2}$. Together, we have
\[
2\dim{}(\pres\varsigma\rf^\triangle)^{-1}(x)+\dim\pres\varsigma\rM^{(m)}_N(\rK^p)=2r-m-1+\frac{m-1}{2}<2r-1=\dim\pres\varsigma\rM_N(\rK^p).
\]
Thus, combining with (2), $\pres\varsigma\rf^\triangle$ is small. Part (4) follows.

For (5), the claim that the image of $\pres\varsigma\rf^\ddag$ is disjoint from $\pres\varsigma\rM^{(2)}_N(\rK^p)$ follows from Theorem \ref{th:siegel_parahoric}(7), Lemma \ref{le:eo1}(1), and Lemma \ref{le:eo}(4,5). Then as $\pres\varsigma\rP^\ddag_N(\rK^p)$ is reduced, we obtain the induced morphism $\pres\varsigma\rf^\ddag\colon\pres\varsigma\rP^\ddag_N(\rK^p)\to\pres\varsigma\rM_N^\no(\rK^p)$.

To show that the above induced morphism is semismall with respect to the Ekedahl--Oort stratification, we consider the filtration
\[
\pres\varsigma\rM_N^\no(\rK^p)=\bigcup_{m\neq 2}\pres\varsigma\rM^{(m)}_N(\rK^p).
\]

For every even positive integer $m\geq 4$ and every geometric point $x\in\pres\varsigma\rM^{(m)}_N(\rK^p)(\kappa)$, we have $\dim{}(\pres\varsigma\rf^\ddag)^{-1}(x)=0$ by (2). By Proposition \ref{pr:eo}(3), we have $\dim\pres\varsigma\rM^{(m)}_N(\rK^p)=2r-\tfrac{m}{2}$. Together, we have
\[
2\dim{}(\pres\varsigma\rf^\ddag)^{-1}(x)+\dim\pres\varsigma\rM^{(m)}_N(\rK^p)=0+2r-\frac{m}{2}\leq 2r-2=\dim\pres\varsigma\rM_N^\no(\rK^p),
\]
in which the equality holds if and only if $m=4$.

For every odd positive integer $m$ and every geometric point $x\in\pres\varsigma\rM^{(m)}_N(\rK^p)(\kappa)$, we have $\dim{}(\pres\varsigma\rf^\ddag)^{-1}(x)=\tfrac{2r-m-1}{2}$ by (1,3). By Proposition \ref{pr:eo}(3), we have $\dim\pres\varsigma\rM^{(m)}_N(\rK^p)=\tfrac{m-1}{2}$. Together, we have
\[
2\dim{}(\pres\varsigma\rf^\ddag)^{-1}(x)+\dim\pres\varsigma\rM^{(m)}_N(\rK^p)=2r-m-1+\frac{m-1}{2}\leq 2r-2=\dim\pres\varsigma\rM_N^\no(\rK^p),
\]
in which the equality holds if and only if $m=1$.

In particular, the relevant strata are $\pres\varsigma\rM^{(4)}_N(\rK^p)$ and $\pres\varsigma\rM^{(1)}_N(\rK^p)$. Part (5) follows.
\end{proof}

The lemma below shows that the ``involution'' $\pres\varsigma\ri$ sends the fibers over $\pres\varsigma\rM_N^{(1)}$ to a section of the basic locus of $\pres{\varsigma\varpi}\rM_N$. Take an object $(A_0,\lambda_0,\eta_0^p;A^\star,\lambda^\star,\eta^{p\star})\in\pres{\varsigma\varpi}\rS_N(\rK^p)(S)$. Put $A^{\star\prime}\coloneqq A^\star/I_{A^\star}$, where $I_{A^\star}$ denotes the kernel of the relative Frobenius morphism $\Ker\lambda^\star[\fp^\infty]\to(\Ker\lambda^\star[\fp^\infty])^{(p)}$. Let $\lambda^{\star\prime}\colon A^{\star\prime}\to(A^{\star\prime})^\vee$ be the unique $O_F$-linear quasi-polarization satisfying $\lambda^\star=\alpha^\vee\circ\lambda^{\star\prime}\circ\alpha$, where $\alpha\colon A^\star\to A^{\star\prime}$ is the natural quotient morphism, which is indeed $p$-principal.

\begin{lem}\label{le:involution}
We have
\begin{enumerate}
  \item The assignment $(A_0,\lambda_0,\eta_0^p;A^\star,\lambda^\star,\eta^{p\star})
      \mapsto(A_0,\lambda_0,\eta_0^p;A^{\star\prime},\lambda^{\star\prime},\eta^{p\star\prime})$, where $\eta^{p\star\prime}=\eta^{p\star}$ under the canonical identification $\rH^\et_1(A^\star_s,\dA^{\infty,p})=\rH^\et_1(A^{\star\prime}_s,\dA^{\infty,p})$, gives an isomorphism $\pres\varsigma\ri_\rS\colon\pres{\varsigma\varpi}\rS_N\xrightarrow\sim\pres\varsigma\rM^{(1)}_N$.

  \item There exists a unique lifting $\pres\varsigma\ri_\rB$ as the dashed arrow in the diagram
      \[
      \xymatrix{
      & \pres\varsigma\rP^{(1)}_N \ar[d]^-{(\pres{\varsigma\varpi}\rf\circ\pres\varsigma\ri,\pres\varsigma\ri_\rS^{-1}\circ\pres\varsigma\rf^{\langle1\rangle})} \\
      \pres{\varsigma\varpi}\rB_N \ar[r]^-{(\iota,\pi)} \ar@{-->}[ur]^-{\pres\varsigma\ri_\rB} & \pres{\varsigma\varpi}\rM_N\times_{\rT}\pres{\varsigma\varpi}\rS_N
      }
      \]
      which induces an isomorphism from $\pres{\varsigma\varpi}\rB_N$ to the induced reduced closed subscheme of $\pres\varsigma\rP_N^{(1)}$.
\end{enumerate}
\end{lem}

\begin{proof}
Without loss of generality, we assume $\varsigma=1$ and suppress it from all the left superscripts. Fix an object $\rK^p\in\fK^p$. Put $\kappa\coloneqq\ol\dF_p$.

For (1), since both $\pres\varpi\rS_N(\rK^p)$ and $\rM_N^{(1)}(\rK^p)$ are smooth schemes over $\dF_p^\Phi$ of dimension $0$, it suffices to show that $\ri_\rS$ induces a bijection $\pres\varpi\rS_N(\rK^p)(\kappa)\xrightarrow\sim\rM_N^{(1)}(\rK^p)(\kappa)$. We first explain that $\ri_\rS$ does have the expected image. Indeed, take a point $s^\star=(A_0,\lambda_0,\eta_0^p;A^\star,\lambda^\star,\eta^{p\star})\in\pres\varpi\rS_N(\rK^p)(\kappa)$. Then we may write $A^\star[\fp^\infty]=\bigoplus_{i=1}^NG_i$ with $G_i$ supersingular of signature $(1,0)$ such that $\Ker\lambda^\star[\fp^\infty]=G_1[\fp]$. Then $A^{\star\prime}=G_1^{(p)}\oplus G_2\oplus\cdots\oplus G_N$, which is superspecial of signature $(N-1,1)$. In particular, $\ri_\rS(s^\star)\in\rM_N^{(1)}(\rK^p)(\kappa)$. We now construct an inverse (on the level of $\kappa$-points). Take a point $x=(A_0,\lambda_0,\eta_0^p;A,\lambda,\eta^p)\in\rM_N^{(1)}(\rK^p)(\kappa)$. We have the $\tD_\kappa$-module $\cD_x\otimes\dF_p$ from Notation \ref{no:eo}(1). Let $J_A\subseteq A[\fp]$ be the subgroup given by the $\tD_\kappa$-submodule $(\cD_x\otimes\dF_p)^\leftarrow\oplus\tF(\cD_x\otimes\dF_p)^\leftarrow$ (which equals $(\cD_x\otimes\dF_p)^\leftarrow\oplus\tV(\cD_x\otimes\dF_p)^\leftarrow$). Put $A^\star\coloneqq A/J_A$ with the induced polarization $\lambda^\star$ and level structure $\eta^{p\star}$. It is straightforward to check that the assignment $x\mapsto(A_0,\lambda_0,\eta_0^p;A^\star,\lambda^\star,\eta^{p\star})$ given in this way is inverse to $\ri_\rS$.

For (2), since $\dF_p^\Phi$ is a perfect field, it suffices to show the statement over an arbitrary $\kappa$-point, say $s^\star=(A_0,\lambda_0,\eta_0^p;A^\star,\lambda^\star,\eta^{p\star})$, of $\pres\varpi\rS_N(\rK^p)$. Then the diagram in (2) becomes
\begin{align}\label{eq:involution}
\xymatrix{
& \rf^{-1}(\ri_\rS(s^\star)) \ar[d]^-{\pres\varpi\rf\circ\ri} \\
\rB_{s^\star} \ar[r]^-{\iota} \ar@{-->}[ur]^-{\ri_\rB} & \pres\varpi\rM_N(\rK^p)_\kappa
}
\end{align}
Suppose that we construct a morphism $\ri_\rB$ rendering the diagram \eqref{eq:involution} commute, which turns out to be an isomorphism. Then the uniqueness of the lifting follows since $\iota$ hence $\pres\varpi\rf\circ\ri$ (in \eqref{eq:involution}) are generically immersions.

It remains to construct such $\ri_\rB$. In what follows, we still write $(A_0,\lambda_0,\eta_0^p;A^\star,\lambda^\star,\eta^{p\star})$ for its base change along any morphism $S\to s^\star$. Take a point $x=(A_0,\lambda_0,\eta_0^p;A,\lambda,\eta^p;A^\star,\lambda^\star,\eta^{p\star};\alpha)\in\rB_{s^\star}(S)$. We obtain the data $y=(A_0,\lambda_0,\eta_0^p;A,\lambda,\eta^p;A^{\star\prime},\lambda^{\star\prime},\eta^{p\star\prime};\psi)$, where $\psi$ is the composition $A\xrightarrow{\alpha}A^\star\to A^{\star\prime}$. It is easy to see that $y$ is an element of $\pres\varpi\rP_N(S)$. We now define $\ri_\rB(x)=\pres\varpi\ri(y)$, which belongs to $\rf^{-1}(\ri_\rS(s^\star))$. Since $\rB_{s^\star}$ is reduced, $\ri_\rB$ factors through the induced reduced fiber $\rP_{\ri_\rS(s^\star)}$; and that the diagram \eqref{eq:involution} commutes by construction. We now construct a morphism $\rP_{\ri_\rS(s^\star)}\to\rB_{s^\star}$ that is inverse to $\ri_\rB$. Let $(A_0,\lambda_0,\eta_0^p;A^{\star\prime},\lambda^{\star\prime},\eta^{p\star\prime};A,\lambda,\eta^p;\psi')$ be the universal object over $\rP_{\ri_\rS(s^\star)}$. We claim that $\Ker\psi'[\fp^\infty]$ is contained in $J_{A^{\star\prime}}$ (appeared in the proof of (1)). Indeed, since $\rP_{\ri_\rS(s^\star)}$ is reduced, this can be easily checked at all geometric points. Thus, we obtain an $O_F$-linear quasi-$p$-isogeny $\alpha\colon A\to A^\star$ hence a point $(A_0,\lambda_0,\eta_0^p;A,\lambda,\eta^p;A^\star,\lambda^\star,\eta^{p\star};\alpha)\in\rB_{s^\star}(\rP_{\ri_\rS(s^\star)})$. It is clear that the corresponding morphism is inverse to $\ri_\rB$.

The lemma is proved.
\end{proof}

\begin{proposition}\label{pr:intersection}
Let $x\in\pres\varsigma\rM_N^{(1)}(\rK^p)(\kappa)$ be a geometric point, where $\kappa$ is an algebraically closed field containing $\dF^\Phi_p$.
\begin{enumerate}
  \item The fiber $(\pres\varsigma\rf^\ddag)^{-1}(x)$ is a smooth irreducible proper scheme over $\kappa$ of dimension $r-1$. In particular, $(\pres\varsigma\rf^\ddag)^{-1}(x)=\rP_x$ under Notation \ref{no:fiber}.

  \item The self-intersection number of $(\pres\varsigma\rf^\ddag)^{-1}(x)$ in $\pres\varsigma\rP_N^\ddag\otimes_{\dF^\Phi_p}\kappa$ equals $(-p)^{r-1}(1+p)(1+p^3)\cdots(1+p^{2r-3})$.
\end{enumerate}
\end{proposition}

\begin{proof}
For (1), by Theorem \ref{th:siegel_parahoric}(8), the dimension of $\cT_{(\pres\varsigma\rf^\ddag)^{-1}(x)/\kappa}$ is at most $r-1$. Thus, (1) follows from Proposition \ref{pr:siegel_parahoric}(1,3). As a consequence, the normal bundle of $(\pres\varsigma\rf^\ddag)^{-1}(x)$ in $\pres\varsigma\rP_N^\ddag\otimes_{\dF^\Phi_p}\kappa$ is canonically isomorphic to $\HOM\(\omega_{\cA^{\triangle\vee},\tau_\infty},\Ker\psi^\triangle_{*,\tau_\infty}/\omega_{\cA^{\triangle\vee},\tau_\infty}\)$.

For (2), let $s^\star=(A_0,\lambda_0,\eta_0^p;A^\star,\lambda^\star,\eta^{p\star})\in\pres{\varsigma\varpi}\rS_N(\kappa)$ be the unique point whose image under $\pres\varsigma\ri_\rS$ is $x$ in view of Lemma \ref{le:involution}(1). Denote by $(\cA,\lambda,\eta^p;\alpha)$ the universal object over the fiber $\rB_{s^\star}$. Denote by $\gamma\colon A^\star\to A^{\star\prime}$ the natural quotient homomorphism and put $\psi^\blacktriangle\coloneqq\gamma\circ\alpha\colon\cA\to A^{\star\prime}\otimes_\kappa\rB_{s^\star}$. Let $\psi^\triangle\colon A^{\star\prime}\otimes_\kappa\rB_{s^\star}\to\cA$ be the unique homomorphism such that $\psi^\triangle\circ\psi^\blacktriangle$ and $\psi^\blacktriangle\circ\psi^\triangle$ are both the multiplication by $\varpi$. Put $\breve\alpha\coloneqq\psi^\triangle\circ\gamma\colon A^\star\otimes_\kappa\rB_{s^\star}\to \cA$, so that both $\breve\alpha\circ\alpha$ and $\alpha\circ\breve\alpha=\varpi$ are the multiplication by $\varpi$. In other words, we have the diagram
\[
\xymatrix{
\cA \ar[r]^-{\alpha}\ar@/_2.2pc/[rr]^-{\psi^\blacktriangle} & A^\star\otimes_\kappa\rB_{s^\star} \ar[r]^-{\gamma}\ar@/^2.2pc/[rr]^-{\breve\alpha} & A^{\star\prime}\otimes_\kappa\rB_{s^\star}
\ar[r]^-{\psi^\triangle} & \cA
}
\]
of abelian schemes over $\rB_{s^\star}$. Then under the isomorphism $\rB_{s^\star}\xrightarrow\sim(\pres\varsigma\rf^\ddag)^{-1}(x)$ in Lemma \ref{le:involution}(2), the universal object over $(\pres\varsigma\rf^\ddag)^{-1}(x)$ is
\[
(A_0,\lambda_0,\eta_0^p;
A^{\star\prime}\otimes_\kappa\rB_{s^\star},\lambda^{\star\prime}\otimes_\kappa\rB_{s^\star},\eta^{p\star\prime}\otimes_\kappa\rB_{s^\star};
\cA,\lambda,\eta^p;\psi^\triangle).
\]
Moreover, by the computation in (1), the pullback of the normal bundle of $(\pres\varsigma\rf^\ddag)^{-1}(x)$ in $\pres\varsigma\rP_N^\ddag\otimes_{\dF^\Phi_p}\kappa$ to $\rB_{s^\star}$ is
\[
\HOM\(\omega_{A^{\star\prime\vee},\tau_\infty}\otimes_\kappa\cO_{\rB_{s^\star}},
\Ker\psi^\triangle_{*,\tau_\infty}/(\omega_{A^{\star\prime\vee},\tau_\infty}\otimes_\kappa\cO_{\rB_{s^\star}})\).
\]
To further compute the above bundle, we use the isomorphism $\zeta_{s^\star}\colon\rB_{s^\star}\xrightarrow{\sim}\DL_{s^\star}=\DL(\sV_{s^\star},\{\;,\;\}_{s^\star},r+1)$ from Lemma \ref{le:qs_basic_correspondence}(3). By construction, $\sV_{s^\star}^\dashv=\Ker\gamma_{*,\tau_\infty}$ as subspaces of $\sV_{s^\star}=\rH^\dr_1(A^\star/\kappa)_{\tau_\infty}$. Put $\sV\coloneqq\sV_{s^\star}/\sV_{s^\star}^\dashv$ with the induced pairing $\{\;,\;\}$, which is again admissible but now nondegenerate. Taking quotient induces an isomorphism $\DL(\sV_{s^\star},\{\;,\;\}_{s^\star},r+1)\xrightarrow\sim\DL(\sV,\{\;,\;\},r)$ of Deligne--Lusztig varieties, whose universal objects we denote by $\cH_{s^\star}$ and $\cH$, respectively. Now as indicated in the proof of \cite{LTXZZ}*{Theorem~4.3.5(3)}, $\cH_{s^\star}^\dashv=\Ker\breve\alpha_{*,\tau_\infty}$, which implies that $\Ker\psi^\triangle_{*,\tau_\infty}/(\omega_{A^{\star\prime\vee},\tau_\infty}\otimes_\kappa\cO_{\rB_{s^\star}})\simeq\cH^\dashv$ as bundles over $\DL(\sV,\{\;,\;\},r)$. As $\omega_{A^{\star\prime\vee},\tau_\infty}\otimes_\kappa\cO_{\rB_{s^\star}}$ is a trivial line bundle, we have
\[
\int_{(\pres\varsigma\rf^\ddag)^{-1}(x)}c_{r-1}
\(\HOM\(\omega_{\cA^{\triangle\vee},\tau_\infty},\Ker\psi^\triangle_{*,\tau_\infty}/\omega_{\cA^{\triangle\vee},\tau_\infty}\)\)
=\int_{\DL(\sV,\{\;,\;\},r)}c_{r-1}(\cH_{s^\star}^\dashv),
\]
in which the left-hand side is nothing but the self-intersection number of $(\pres\varsigma\rf^\ddag)^{-1}(x)$ in $\pres\varsigma\rP_N^\ddag\otimes_{\dF^\Phi_p}\kappa$ and the right-hand side equals $(-p)^{r-1}(1+p)(1+p^3)\cdots(1+p^{2r-3})$ by Lemma \ref{le:intersection}.

The proposition is proved.
\end{proof}

\subsection{Tate--Thompson local system}

In this subsection, we introduce the so-called Tate--Thompson local system, which will have a deep connection with the Abel--Jacobi map; it can be thought as the analogue of the ``Steinberg'' local system in the modular curve case. From now to the end of \S\ref{ss:boosting}, we will suppress $\varsigma$ from all the left superscripts when $\varsigma=1$.

Take a $\tb_{N,p}$-coprime coefficient ring $L$ (see Notation \ref{no:numerical1} for $\tb_{N,p}$). The Tate--Thompson local system (with coefficients in $L$) will be a direct factor of $\bff^\eta_*L$ that is the local system analogue of the Tate--Thompson representation $\Omega_{N,L}^1$ in \S\ref{ss:complement}. In particular, we need to find a (geometric) correspondence that encodes the element $\tQ_{\rY,1}$ in that subsection. Moreover, since we eventually need to study nearby cycles of the Tate--Thompson local system, we need such correspondence to have a flat extension away from basic locus. It turns out that we are able to achieve this not for $\tQ_{\rY,1}$, but for $\tQ_{\rY,1}$ plus some copies of $\tQ_{\rY,0}$ -- the identity correspondence (see Remark \ref{re:correspondence}), which is sufficient for our purpose.

\begin{definition}\label{de:moduli_n}
We define a functor
\begin{align*}
\bN_N\colon\fK^p\times\fT &\to\Sch'_{/\dZ^\Phi_p} \\
\rK^p &\mapsto \bN_N(\rK^p)
\end{align*}
such that for every $S\in\Sch'_{/\dZ^\Phi_p}$, $\bN_N(\rK^p)(S)$ is the set of equivalence classes of septuples
\[
(A_0,\lambda_0,\eta_0^p;A,\lambda,\eta^p;H)
\]
where
\begin{itemize}[label={\ding{118}}]
  \item $(A_0,\lambda_0,\eta_0^p;A,\lambda,\eta^p)$ is an element of $\bM_N(\rK^p)(S)$;

  \item $H\subseteq A[\fp]$ is an $O_F$-stable finite locally free subgroup scheme of rank $p^{2(r-1)}$ satisfying
    \begin{enumerate}[label=(\alph*)]
      \item $H$ is isotropic under the Weil pairing on $A[\fp]$ induced by the polarization $\lambda$; and

      \item Let $\pi\colon A\to A'\coloneqq A/H$ be the quotient isogeny and $\lambda'\colon A'\to A'^\vee$ the quasi-polarization on $A'$ satisfying $\varpi\lambda=\pi^\vee\circ\lambda'\circ\pi$. If we equip $A'$ with the canonically induced $O_F$-action, then $(A',\lambda')$ is a unitary $O_F$-abelian scheme over $S$ of the same signature $N\Phi-\tau_\infty+\tau_\infty^\tc$ as $(A,\lambda)$.
  \end{enumerate}
\end{itemize}
The definitions of the equivalence relation and the action of morphisms in $\fK^p\times\fT$ are omitted.
\end{definition}

\begin{definition}\label{de:moduli_q}
We define a functor
\begin{align*}
\bR_N\colon\fK^p\times\fT &\to\Sch'_{/\dZ^\Phi_p} \\
\rK^p &\mapsto \bR_N(\rK^p)
\end{align*}
such that for every $S\in\Sch'_{/\dZ^\Phi_p}$, $\bR_N(\rK^p)(S)$ is the set of equivalence classes of undecuples
\[
(A_0,\lambda_0,\eta_0^p;A^\triangle,\lambda^\triangle,\eta^{p\triangle};
A^\blacktriangle,\lambda^\blacktriangle,\eta^{p\blacktriangle};\psi^\triangle;H),
\]
where
\begin{itemize}[label={\ding{118}}]
  \item $(A_0,\lambda_0,\eta_0^p;A^\triangle,\lambda^\triangle,\eta^{p\triangle};
      A^\blacktriangle,\lambda^\blacktriangle,\eta^{p\blacktriangle};\psi^\triangle)$ is an element of $\bP_N(\rK^p)(S)$;

  \item $(A_0,\lambda_0,\eta_0^p;A^\triangle,\lambda^\triangle,\eta^{p\triangle};H)$ is an element of $\bN_N(\rK^p)(S)$ satisfying $H\subseteq\Ker(\psi^\triangle[\fp])$.
\end{itemize}
The definitions of the equivalence relation and the action of morphisms in $\fK^p\times\fT$ are omitted.
\end{definition}

By a similar argument as for Theorem \ref{th:siegel_parahoric}(1), both $\bN_N(\rK^p)$ and $\bR_N(\rK^p)$ are quasi-projective strictly semistable schemes over $\bT$ of relative dimension $N-1$. By definition, we have a natural commutative diagram
\begin{align}\label{eq:auxiliary}
\xymatrix{
& \bR_N \ar[ld]_-{\bg_1} \ar[rd]^-{\bg_0} \\
\bN_N \ar[rd] && \bP_N \ar[ld]^-{\bff} \\
& \bM_N
}
\end{align}
in $\Fun(\fK^p\times\fT,\Sch'_{/\dZ^\Phi_p})_{/\bT}$, where all arrow are natural forgetful morphisms.

We put
\[
\bQ_N\coloneqq\bR_N\times_{\bN_N}\bR_N
\]
so that for every $S\in\Sch'_{/\dZ^\Phi_p}$, $\bQ_N(\rK^p)(S)$ is the set of equivalence classes of quindecuples
\[
(A_0,\lambda_0,\eta_0^p;A^\triangle,\lambda^\triangle,\eta^{p\triangle};
A^\blacktriangleleft,\lambda^\blacktriangleleft,\eta^{p\blacktriangleleft};\psi^\vartriangleleft;
A^\blacktriangleright,\lambda^\blacktriangleright,\eta^{p\blacktriangleright};\psi^\vartriangleright;H)
\]
in which both $(A_0,\lambda_0,\eta_0^p;A^\triangle,\lambda^\triangle,\eta^{p\triangle};
A^\blacktriangleleft,\lambda^\blacktriangleleft,\eta^{p\blacktriangleleft};\psi^\vartriangleleft)$ and $(A_0,\lambda_0,\eta_0^p;A^\triangle,\lambda^\triangle,\eta^{p\triangle};
A^\blacktriangleright,\lambda^\blacktriangleright,\eta^{p\blacktriangleright};\psi^\vartriangleright)$ are elements of $\bR_N(\rK^p)$; $(A_0,\lambda_0,\eta_0^p;A^\triangle,\lambda^\triangle,\eta^{p\triangle};H)$ is an element of $\bN_N(\rK^p)$ satisfying $H\subseteq\Ker(\psi^\vartriangleleft[\fp])\cap\Ker(\psi^\vartriangleright[\fp])$.

Let $\bh_<,\bh_>\colon\bQ_N\to\bR_N$ be the two morphisms forgetting $(A^\blacktriangleright,\lambda^\blacktriangleright,\eta^{p\blacktriangleright};\psi^\vartriangleright)$ and $(A^\blacktriangleleft,\lambda^\blacktriangleleft,\eta^{p\blacktriangleleft};\psi^\vartriangleleft)$, respectively. Put $\bg_<\coloneqq\bg_0\circ\bh_<$ and $\bg_>\coloneqq\bg_0\circ\bh_>$. Then we have a commutative diagram
\begin{align}\label{eq:correspondence}
\xymatrix{
& \bQ_N \ar[ld]_-{\bg_<} \ar[rd]^-{\bg_>} \\
\bP_N \ar[rd]_-{\bff} && \bP_N \ar[ld]^-{\bff} \\
& \bM_N
}
\end{align}
in $\Fun(\fK^p\times\fT,\Sch'_{/\dZ^\Phi_p})_{/\bT}$.

\begin{notation}\label{no:correspondence}
For pairs $(\rX,?)\in\{\rN,\rQ,\rR\}\times\{(m),[m],\rb\}$, we denote by $\rX^?_N$ the preimage of $\rM^?_N$ under the natural morphism $\rX_N\to\rM_N$, and by $\rg_0^?,\rg_1^?,\rh_<^?,\rh_>^?,\rg_<^?,\rg_>^?$ the corresponding restrictions of morphisms. We have similar notation for pairs in $\{\bN,\bQ,\bR\}\times\{]m[\}$.
\end{notation}

We first study the generic fiber of the diagram \eqref{eq:correspondence}.

\begin{remark}\label{re:correspondence}
Since $\bg_1^\eta$ is finite \'{e}tale, the diagonal embedding $\bR_N^\eta\to\bQ_N^\eta$ is both open and closed, whose image we denote by $\bQ_{N,0}^\eta$. Let $\bQ_{N,1}^\eta$ be the complement of $\bQ_{N,0}^\eta$ in $\bQ_N^\eta$, which is also both open and closed. Denote by $\Delta\bP_N^\eta$ the diagonal of $\bP_N^\eta\times_{\bM_N^\eta}\bP_N^\eta$, which is both open and closed. Then it is clear that the natural commutative diagram
\[
\xymatrix{
\bQ_{N,1}^\eta \ar[r]\ar[d] &
\(\bP_N^\eta\times_{\bM_N^\eta}\bP_N^\eta\)
\setminus\Delta\bP_N^\eta \ar[d] \\
\bQ_N^\eta \ar[r]^-{(\bg_<,\bg_>)}&
\bP_N^\eta\times_{\bM_N^\eta}\bP_N^\eta
}
\]
is Cartesian. Via the pair of morphisms $(\bg_<^\eta,\bg_>^\eta)$, we obtain three finite \'{e}tale correspondences of $\bP_N^\eta$ given by $\bQ_N^\eta$, $\bQ_{N,0}^\eta$, and $\bQ_{N,1}^\eta$, satisfying
\begin{itemize}[label={\ding{118}}]
  \item the correspondence given by $\bQ_N^\eta$ is the sum of those given by $\bQ_{N,0}^\eta$ and $\bQ_{N,1}^\eta$;

  \item the correspondence given by $\bQ_{N,0}^\eta$ is the identity correspondence (that is, the one given by the Hecke operator $\tQ_{\rY,0}$ in the notation of \S\ref{ss:complement}) multiplied by $\frac{p^N-1}{p^2-1}$ (that is, the degree of the map $\bg_0^\eta$);

  \item the correspondence given by $\bQ_{N,1}^\eta$ is the Hecke correspondence given by the Hecke operator $\tQ_{\rY,1}$ in the notation of \S\ref{ss:complement}.
\end{itemize}
\end{remark}

\begin{definition}\label{de:convolution}
Given a diagram
\[
\xymatrix{
& Q \ar[dl]_-{g_1} \ar[dr]^-{g_2} \\
P_1 \ar[dr]_-{f_1} & & P_2 \ar[dl]^-{f_2} \\
& M
}
\]
of schemes in which $g_2$ is finite flat, we define its induced \emph{($L$-linear) convolution map} to be the composition
\[
\rR f_{1*}L\to\rR f_{1*}\rR g_{1*}g_1^*L\xrightarrow\sim\rR f_{2*}\rR g_{2*}g_2^*L\to\rR f_{2*} L
\]
in which the first arrow is the adjunction map for $g_1$ and the last arrow is the trace map for $g_2$.
\end{definition}

\begin{definition}\label{de:local_system}
We define $\tQ_N^\eta\colon\bff^\eta_*L\to\bff^\eta_*L$ to be the convolution map induced by the generic fiber of the diagram \eqref{eq:correspondence} between $L$-linear local systems on $\bM_N^\eta$. For every integer $0\leq j\leq N+1$, put
\[
\Omega_{N,L}^{\eta,j}\coloneqq\Ker\(\tQ_N^\eta-\(\tl_{N,p}^j+\tfrac{p^N-1}{p^2-1}\)\colon\bff^\eta_*L\to\bff^\eta_*L\)
\]
(see Notation \ref{no:numerical1} for $\tl_{N,p}^j$) so that $\Omega_{N,L}^{\eta,j}=\Omega_{N,L}^{\eta,N+1-j}$. We call $\Omega_{N,L}^\eta\coloneqq\Omega_{N,L}^{\eta,1}$ the \emph{Tate--Thompson local system} in view of Lemma \ref{le:local_system} below.
\end{definition}

\begin{lem}\label{le:local_system}
The natural map
\[
\bigoplus_{j=0}^r\Omega_{N,L}^{\eta,j}\to\bff^\eta_*L
\]
is an isomorphism of $L$-linear local systems on $\bM_N^\eta$.
\end{lem}

\begin{proof}
We first construct a Galois closure of the finite \'{e}tale morphism $\bff^\eta$. Let $\rX_N$ be the unique up to isomorphism hermitian space over $\dF_{p^2}$ of rank $N$. Define the functor
\begin{align*}
\widetilde{\bM_N^\eta}\colon\fK^p\times\fT &\to\Sch'_{/\dQ^\Phi_p} \\
\rK^p &\mapsto \widetilde{\bM_N^\eta}(\rK^p)
\end{align*}
such that for $S\in\Sch'_{/\dQ^\Phi_p}$, $\widetilde{\bM_N^\eta}(\rK^p)(S)$ is the set of equivalence classes of septuples $(A_0,\lambda_0,\eta_0^p;A,\lambda,\eta^p,\eta_\fp)$, where $(A_0,\lambda_0,\eta_0^p;A,\lambda,\eta^p)\in\bM_N^\eta(\rK^p)(S)$ and
\[
\eta_\fp\colon \rX_N\to\Hom_{\dF_{p^2}}^{\lambda_0[\fp],\lambda[\fp]}(A_0[\fp],A[\fp])
\]
is an isometry of hermitian spaces over $\dF_{p^2}$ (over every connected component of $S$). Then the natural forgetful map $\widetilde{\bM_N^\eta}\to\bM_N^\eta$ is a Galois morphism with the Galois group $\rU_N\coloneqq\rU(\rX_N)(\dF_p)$.

Once fix a maximal isotropic $\dF_{p^2}$-linear subspace $\rY$ of $\rX_N$, we obtain a morphism $\widetilde{\bM_N^\eta}\to\bP_N^\eta$ sending $(A_0,\lambda_0,\eta_0^p;A,\lambda,\eta^p,\eta_\fp)$ to $(A_0,\lambda_0,\eta_0^p;A,\lambda,\eta^p,\eta_\fp;A^\blacktriangle,\lambda^\blacktriangle,\eta^{p\blacktriangle};\psi^\triangle)$, in which $A^\blacktriangle=A/C$, where $C$ is the subgroup spanned by the image of $\eta_\fp(x)$ for every $x\in\rY$. Then we have a sequence of finite \'{e}tale morphisms
\[
\widetilde{\bM_N^\eta}\to\bP_N^\eta\to\bM_N^\eta
\]
in which the first one is Galois with the Galois group $\rP_\rY$, where $\rP_\rY\subseteq\rU_N$ is the subgroup stabilizing $\rY$. Now we adopt the discussion in \S\ref{ss:complement} for $\kappa/\kappa^+=\dF_{p^2}/\dF_p$. Then the lemma follows from Remark \ref{re:correspondence} and Lemma \ref{le:siegel_decomposition}.
\end{proof}

\begin{remark}\label{re:local_system}
It is clear that $\Omega_{N,L}^{\eta,j}$ is self-dual for $0\leq j\leq N+1$.
\end{remark}

\begin{notation}
For a scheme $X$ separable and of finite type over a field $\kappa$ of characteristic $p$ of pure dimension $d$, we denote by
\begin{itemize}[label={\ding{118}}]
  \item $\rD(X,L)$ the derived category of $L$-linear \'{e}tale sheaves on $X$ with bounded constructible cohomology, and

  \item $\Perv(X,L)\subseteq\rD(X,L)$ the abelian subcategory consisting of $F$ such that $F[d]$ is (constructible) perverse.
\end{itemize}
For an open subscheme $U$ of $X$ and an element $F\in\Perv(U,L)$, we denote by $\IC(X,F)\in\Perv(X,L)$ the intermediate extension of $F$ along the inclusion $U\hookrightarrow X$, and by $\IC(\ol{X},F)\in\Perv(\ol{X},L)$ the restriction of $\IC(X,F)$ to $\ol{X}\coloneqq X\otimes_\kappa\ol\kappa$.
\end{notation}

By Theorem \ref{th:siegel_parahoric}(1), the nearby cycles sheaf $\rR\Psi L$ on $\ol\rP_N$ admits an (increasing) monodromy filtration
\[
0=\rF_{-2}\rR\Psi L\subseteq\rF_{-1}\rR\Psi L\subseteq\rF_0\rR\Psi L\subseteq\rF_1\rR\Psi L=\rR\Psi L
\]
in $\Perv(\ol\rP_N,L)$, in which $\rF_0\rR\Psi L$ is the constant sheaf $L$ on $\ol\rP_N$; and we have canonical isomorphisms
\[
\gr^\rF_1\rR\Psi L\simeq\ol\rp^\ddag_*L(-1)[-1],\quad
\gr^\rF_0\rR\Psi L\simeq\ol\rp^\triangle_*L\oplus\ol\rp^\blacktriangle_*L,\quad
\gr^\rF_{-1}\rR\Psi L\simeq\ol\rp^\ddag_*L[-1]
\]
of graded objects.\footnote{Strictly speaking, the ``canonical'' isomorphisms for $\gr^\rF_p\rR\Psi L$ with $p\neq 0$ depend on the choice of an order between $\rP_N^\triangle$ and $\rP_N^\blacktriangle$, for which we fix one so that $\rP_N^\triangle$ comes first.} Put $\rF_{\geq 0}\rR\Psi L\coloneqq\rR\Psi L/\rF_{-1}\rR\Psi L$.

Since $\bff$ is proper by Theorem \ref{th:siegel_parahoric}(2), we have a canonical isomorphism $\rR\ol\rf_*\rR\Psi L\simeq\rR\Psi(\bff^\eta_*L)$ on $\ol\rM_N$. Then by Proposition \ref{pr:siegel_parahoric}(4,5), the filtration $\rR\rf_*\rF_\bullet\rR\Psi L$ induces a filtration $\rF_\bullet\rR\Psi(\bff^\eta_*L)$ of $\rR\Psi(\bff^\eta_*L)$ in $\Perv(\ol\rM_N,L)$. Lemma \ref{le:local_system} provides us with a filtration $\rF_\bullet\rR\Psi(\Omega_{N,L}^{\eta,j})$ of the nearby cycles sheaf $\rR\Psi(\Omega_{N,L}^{\eta,j})$ again in $\Perv(\ol\rM_N,L)$ for $0\leq j\leq N+1$.

\begin{definition}
We call $\gr^\rF_0\rR\Psi(\Omega_{N,L}^{\eta,j})$ (resp.\ $\gr^\rF_1\rR\Psi(\Omega_{N,L}^{\eta,j})$ and $\gr^\rF_{-1}\rR\Psi(\Omega_{N,L}^{\eta,j})$) the \emph{central nearby cycle} (resp.\ \emph{marginal nearby cycles}) of $\Omega_{N,L}^{\eta,j}$.

We denote by $(\pres\Omega\rE^{p,q}_s,\pres\Omega\rd^{p,q}_s)$ the ($\fT$-invariant) spectral sequence converging to $\rH^{p+q}_\fT(\ol\rM_N,\rR\Psi(\Omega_{N,L}^\eta)(r))$ (Notation \ref{no:groupoid}) in the category $\Mod(L[\Gal(\ol\dF_p/\dF_p^\Phi)])$ induced from the filtration $\rF_\bullet\rR\Psi(\Omega_{N,L}^\eta)(r)$, which is canonically a direct summand of the weight spectral sequence of $\rR\Psi L(r)$.
\end{definition}

Consider the composite map
\begin{align}\label{eq:boosting}
\vartheta\colon\rH^{2r-1}_\fT(\ol\rM_N,\rF_{\geq 0}\rR\Psi(\Omega_{N,L}^\eta)(r))&\hookrightarrow
\rH^{2r-1}_\fT(\ol\rM_N,\rF_{\geq 0}\rR\Psi(\bff^\eta_*L)(r)) \\
&\xrightarrow\sim\rH^{2r-1}_\fT(\ol\rM_N,\rR\rf_*\rF_{\geq 0}\rR\Psi L(r))
=\rH^{2r-1}_\fT(\ol\rP_N,\rF_{\geq 0}\rR\Psi L(r)) \notag \\
&\xrightarrow{\pres\varpi\rf_!\circ\ri_!}\rH^{2r-1}_\fT(\ol{\pres\varpi\rM}_N,\rF_{\geq 0}\rR\Psi L(r))=\rH^{2r-1}_\fT(\ol{\pres\varpi\rM}_N,L(r)) \notag
\end{align}
of $L[\Gal(\ol\dF_p/\dF_p^\Phi)]$-modules, where the nearby cycle $\rR\Psi L(r)$ in $\rH^{2r-1}_\fT(\ol{\pres\varpi\rM_N},\rF_{\geq 0}\rR\Psi L(r))$ is the one for $\pres\varpi\bM_N$ hence is nothing but the constant sheaf $L(r)$ on $\ol{\pres\varpi\rM}_N$ as $\pres\varpi\bM_N$ is smooth over $\Spec\dZ_p^\Phi$. Define the map
\begin{align}\label{eq:boosting_1}
\theta\colon\pres\Omega\rE^{0,2r-1}_1=\rH^{2r-1}_\fT(\ol\rM_N,\gr_0^\rF\rR\Psi(\Omega_{N,L}^\eta)(r))
\to\rH^{2r-1}_\fT(\ol\rM_N,\rF_{\geq 0}\rR\Psi(\Omega_{N,L}^\eta)(r))\xrightarrow\vartheta\rH^{2r-1}_\fT(\ol{\pres\varpi\rM}_N,L(r))
\end{align}
of $L[\Gal(\ol\dF_p/\dF_p^\Phi)]$-modules, which induces a map
\begin{align}\label{eq:boosting_2}
\beta\colon\pres\Omega\rE^{-1,2r}_2\to\coker\theta,
\end{align}
which we call the \emph{boosting map}.

\begin{remark}\label{re:vartheta}
For every prime $\fq$ of $F^+$ above $p$, choose a self-dual $O_{F_\fq}$-lattice $\Lambda_\fq$ in $\rV\otimes_FF_\fq$ together with an $O_{F_\fq}$-lattice chain $\Lambda_\fq\subseteq\pres\varpi\Lambda_\fq\subseteq\varpi\Lambda_\fq$ satisfying $\pres\varpi\Lambda_\fq=\varpi(\pres\varpi\Lambda_\fq)^\vee$.\footnote{In particular, $\pres\varpi\Lambda_\fq=\Lambda_\fq$ for $\fq\neq\fp$.} Denote by $\rK_p$ and $\pres\varpi\rK_p$ the stabilizers of $\prod_{\fq\mid p}\Lambda_\fq$ and $\prod_{\fq\mid p}\pres\varpi\Lambda_\fq$ in $\prod_{\fq\mid p}\rU(\rV)(F_\fq)$, respectively.

Recall from \cite{LTXZZ}*{\S4.2} that we have \emph{canonical} ``moduli interpretation'' isomorphisms of varieties over $\dQ_p^\Phi$:
\begin{align*}
\bM^\eta_N\xrightarrow{\sim}\Sh(\rV,\obj\rK_p)\times_F\bT^\eta,\quad
\pres\varpi\bM^\eta_N\xrightarrow{\sim}\Sh(\rV,\obj\pres\varpi\rK_p)\times_F\bT^\eta
\end{align*}
(Notation \ref{no:p_notation}(5)) in $\Fun(\fK^p\times\fT,\Sch_{/\dQ_p^\Phi})_{/\bT^\eta_\fp}$, where $\fT$ acts on $\Sh(\rV,\obj\rK_p)\times_F\bT^\eta$ via the second factor. We further have the isomorphism
\[
\bP^\eta_N\xrightarrow{\sim}\Sh(\rV,\obj(\rK_p\cap\pres\varpi\rK_p))\times_F\bT^\eta
\]
under which the two morphisms $\bff^\eta$ and $\pres\varpi\bff^\eta\circ\bi^\eta$ coincide with the natural projection morphisms
\begin{align*}
\Sh(\rV,\obj(\rK_p\cap\pres\varpi\rK_p))\times_F\bT^\eta&\to\Sh(\rV,\obj\rK_p)\times_F\bT^\eta,\\
\Sh(\rV,\obj(\rK_p\cap\pres\varpi\rK_p))\times_F\bT^\eta&\to\Sh(\rV,\obj\pres\varpi\rK_p)\times_F\bT^\eta,
\end{align*}
respectively.

Similar to $\vartheta$ \eqref{eq:boosting}, we can define a map
\[
\vartheta^\eta\colon\rH^{2r-1}_{\et}(\Sh(\rV,\obj\rK_p)_{\ol{F}},\Omega_{N,L}^\eta(r))
\to\rH^{2r-1}_{\et}(\Sh(\rV,\obj\pres\varpi\rK_p)_{\ol{F}},L(r))
\]
on the level of the generic fibers, which is clearly $\Gamma_F$-equivariant. By \cite{LTXZZ}*{Lemma~4.2.4}, the specialization maps induce isomorphisms
\begin{align*}
\rH^{2r-1}_{\et}(\Sh(\rV,\obj\rK_p)_{\ol{F}},\Omega_{N,L}^\eta(r))&\simeq\rH^{2r-1}_\fT(\ol\rM_N,\rR\Psi(\Omega_{N,L}^\eta)(r)),\\ \rH^{2r-1}_{\et}(\Sh(\rV,\obj\pres\varpi\rK_p)_{\ol{F}},L(r))&\simeq\rH^{2r-1}_\fT(\ol{\pres\varpi\rM}_N,L(r)).
\end{align*}
In particular, we obtain the following commutative diagram
\[
\xymatrix{
\rH^{2r-1}_{\et}(\Sh(\rV,\obj\rK_p)_{\ol{F}},\Omega_{N,L}^\eta(r)) \ar[r]^-{\vartheta^\eta}\ar[d] &  \rH^{2r-1}_{\et}(\Sh(\rV,\obj\pres\varpi\rK_p)_{\ol{F}},L(r)) \ar[d]^-\simeq \\
\rH^{2r-1}_\fT(\ol\rM_N,\rF_{\geq 0}\rR\Psi(\Omega_{N,L}^\eta)(r)) \ar[r]^-{\vartheta} & \rH^{2r-1}_\fT(\ol{\pres\varpi\rM}_N,L(r))
}
\]
in the category $\Fun(\fK^p,L[\Gal(\ol\dF_p/\dF_p^\Phi)])$.
\end{remark}

\subsection{Nearby cycles over Newton strata}

In this subsection, we perform a preliminary study of the nearby cycles of the Tate--Thompson local system over non-basic Newton strata.

\begin{notation}\label{no:decomposition}
Take an \emph{even} integer $0< m\leq N$.
\begin{enumerate}
  \item For every integer $0\leq l\leq \tfrac{m}{2}$, we denote by $\rP_N^{(m|l)}$ the maximal open and closed subscheme of $\rP_N^{(m)}$ such that for every geometric point $x=(A_0,\lambda_0,\eta_0^p;A^\triangle,\lambda^\triangle,\eta^{p\triangle};
      A^\blacktriangle,\lambda^\blacktriangle,\eta^{p\blacktriangle};\psi^\triangle)\in\rP_N^{(m|l)}(\kappa)$, the $\tD_\kappa$-submodule of $\cD_x\otimes\dF_p$ given by $\Ker\psi^\triangle$ contains $\cC_x^l$ (Notation \ref{no:eo}). In particular, $\rP_N^{(m|0)}=\rP_N^{(m)}\setminus\rP_N^\blacktriangle$ and $\rP_N^{(m|\frac{m}{2})}=\rP_N^{(m)}\setminus\rP_N^\triangle$. Denote by $\rf^{(m|l)}\colon\rP_N^{(m|l)}\to\rM_N^{(m)}$ the restriction of $\rf$ to $\rP_N^{(m|l)}$.

  \item For every integer $0\leq l\leq \tfrac{m}{2}$, we denote by $\rR_N^{(m|l)}$ the preimage of $\rP_N^{(m|l)}$ under the morphism $\rg_0^{(m)}$, which is an open and closed subscheme of $\rR_N^{(m)}$, so that we have the induced morphism $\rg_0^{(m|l)}\colon\rR_N^{(m|l)}\to\rP_N^{(m|l)}$.
\end{enumerate}
\end{notation}

\begin{lem}\label{le:flat}
For every even integer $0< m\leq N$, the morphisms $\rg_0^{(m)}$, $\rg_1^{(m)}$, $\rg_<^{(m)}$, and $\rg_>^{(m)}$ are all finite.
\end{lem}

\begin{proof}
First, the morphisms in question are clearly all proper. Thus, it remains to show that they are quasi-finite. We claim that it suffices to show that $\rg_0^{(m)}$ is quasi-finite. Indeed, we have the following commutative diagram
\[
\xymatrix{
& \rQ_N^{(m)} \ar[rd]^{\rh_>^{(m)}} \ar[ld]_{\rh_<^{(m)}} \ar@/^2.5pc/[rrdd]^{\rg_>^{(m)}}\\
\rR_N^{(m)} \ar[rd]^{\rg_1^{(m)}} && \rR_N^{(m)} \ar[rd]^{\rg_0^{(m)}}\ar[ld]_{\rg_1^{(m)}}\\
&	\rN_N^{(m)} \ar[rd] & & \rP_N^{(m)} \ar[ld]_{\rf^{(m)}} \\
&	& \rM_N^{(m)}
}
\]
of schemes over $\dF_{p^2}$, in which the upper-left square is Cartesian. By Proposition \ref{pr:siegel_parahoric}(2), $\rf^{(m)}$ is finite. Thus, if $\rg_0^{(m)}$ is quasi-finite, then $\rg_1^{(m)}$ is quasi-finite hence so is $\rh_>^{(m)}$ by base change. It follows that $\rg_>^{(m)}$ is quasi-finite. By a similar argument, we know that $\rg_<^{(m)}$ is also quasi-finite.

For $\rg_0^{(m)}$, we have the decomposition
\[
\rP_N^{(m)}=\coprod_{l=0}^{m/2}\rP_N^{(m|l)}
\]
(Notation \ref{no:decomposition}(1)) by Lemma \ref{le:eo}(5). It suffices to show that the induced morphism $\rg_0^{(m|l)}\colon\rR_N^{(m|l)}\to\rP_N^{(m|l)}$ (Notation \ref{no:decomposition}(2)) is quasi-finite. Take a point $y\in\rP_N^{(m|l)}(\kappa)$, where $\kappa$ is an algebraically closed field containing $\dF^\Phi_p$, with $x\coloneqq\rf(y)\in\rM_N^{(m)}(\kappa)$. Then $y$ gives rise to a neutral Lagrangian $\tD_\kappa$-submodule $\cE$ of $\cD_x\otimes\dF_p$ (Remark \ref{re:siegel_parahoric}), which satisfies $\cE=(\cE\cap\cB_x)\oplus(\cE\cap\cS_x)=\cC_x^l\oplus(\cE\cap\cS_x)$ by Lemma \ref{le:eo1}(1). It suffices to show that there are only finitely many neutral $\tD_\kappa$-submodules of $\cE$ of $\kappa$-codimension $2$. Let $\cC\subseteq\cE$ be such a submodule. There are two cases.

Suppose that $\cC\cap\cB_x=\cC_x^l$. Then $\cC=(\cC\cap\cB_x)\oplus(\cC\cap\cS_x)$ in which $\cC\cap\cS_x$ is a neutral $\tD_\kappa$-submodule of $\cE\cap\cS_x$ of $\kappa$-codimension $2$. It is clear that there are only finitely many such submodules.

Suppose that $\cC\cap\cB_x$ is a proper submodule of $\cC_x^l$. Let $\cC'$ be the image of $\cC$ in $\cS_x$, which is right-leaning. Then $\cC\cap\cB_x$ is left-leaning (and totally isotropic). By Lemma \ref{le:eo}(4), $\cC\cap\cB_x$ has to be neutral hence has $\kappa$-codimension $2$ in $\cC_x^l$. By Lemma \ref{le:eo}(5), the number of such submodules of $\cC_x^l$ is finite and we must have $\cC'=\cC\cap\cS_x=\cE\cap\cS_x$. In other words, $\cC=(\cC\cap\cB_x)\oplus(\cE\cap\cS_x)$.

In both cases, the number of neutral $\tD_\kappa$-submodules of $\cE$ of $\kappa$-codimension $2$ is finite. The lemma follows.
\end{proof}

\begin{remark}
Unlike $\rf^{(2r-1)}$, the morphism $\rg_0^{(2r-1)}$ is not finite as long as $r\geq 2$.
\end{remark}

The following proposition confirms the flatness of the correspondence \eqref{eq:correspondence} over non-basic Newton strata, which is a crucial ingredient for our argument later.

\begin{proposition}\label{pr:flat}
The morphisms $\bg_0^{]2r-1[}$, $\bg_1^{]2r-1[}$, $\bg_<^{]2r-1[}$, and $\bg_>^{]2r-1[}$ are all finite flat, of degrees $\frac{p^{2r}-1}{p^2-1}$, $p+1$, $\frac{p^{2r-1}}{p-1}$, and $\frac{p^{2r-1}}{p-1}$, respectively.
\end{proposition}

\begin{proof}
We first show that $\bg_0^{]2r-1[}\colon \bR_N^{]2r-1[}\to \bP_N^{]2r-1[}$ is finite flat. Since both schemes are regular of dimension $N=2r$, according to the miracle flatness criterion over regular local rings \cite{Mat89}*{Theorem~23.1}, it suffices to show that every fiber of $\bg_0^{]2r-1[}$ is finite. Over the generic fiber, this is clear; over the special fiber, this was proved in Lemma~\ref{le:flat}. By the same argument, we see that $\bg_1^{]2r-1[}$ is also finite flat. It follows that $\bh_<^{]2r-1[}$ and $\bh_>^{]2r-1[}$ are both finite flat by base change, hence so are $\bg_<^{]2r-1[}=\bh_<^{]2r-1[}\circ\bg_0^{]2r-1[}$ and $\bg_>^{]2r-1[}=\bh_>^{]2r-1[}\circ\bg_0^{]2r-1[}$. The statement on the degrees can be checked easily over the generic fibers.
\end{proof}

Over each non-basic Newton strata $\rM_N^{(m)}$, the restriction of the map $\tQ_N$ has a decomposition we now explain.

\begin{notation}\label{no:triangle}
Take an \emph{even} integer $0< m\leq N$.
\begin{enumerate}
  \item Denote by $\tQ_N^{(m)}$ the induced convolution map of the commutative diagram
     \begin{align*}
     \xymatrix{
     & \rQ_N^{(m)} \ar[dl]_-{\rg_<^{(m)}} \ar[dr]^-{\rg_>^{(m)}} \\
     \rP_N^{(m)} \ar[dr]_-{\rf^{(m)}} & & \rP_N^{(m)} \ar[dl]^-{\rf^{(m)}} \\
     & \rM_N^{(m)}
     }
     \end{align*}
     in which all morphisms are finite flat by Proposition \ref{pr:siegel_parahoric}(2) and Proposition \ref{pr:flat}. For $0\leq j\leq N+1$, put
     \[
     \Omega_{N,L}^{(m),j}\coloneqq\Ker\(\tQ_N^{(m)}-\(\tl_{N,p}^j+\tfrac{p^N-1}{p^2-1}\)\colon\rR\rf^{(m)}_*L\to\rR\rf^{(m)}_*L\),
     \]
     and simply write $\Omega_{N,L}^{(m)}$ for $\Omega_{N,L}^{(m),1}$.

  \item For a pair of integers $(l_<,l_>)$ satisfying $0\leq l_<,l_>\leq\tfrac{m}{2}$, put
      \[
      \rQ_N^{(m|l_<,l_>)}\coloneqq\rg_<^{-1}\rP_N^{(m|l_<)}\cap\rg_>^{-1}\rP_N^{(m|l_>)}
      \]
      (Notation \ref{no:decomposition}(1)) and denote by $\rg_<^{(m|l_<,l_>)}\colon\rQ_N^{(m|l_<,l_>)}\to\rP_N^{(m|l_<)}$ and $\rg_>^{(m|l_<,l_>)}\colon\rQ_N^{(m|l_<,l_>)}\to\rP_N^{(m|l_>)}$ the corresponding restrictions.

  \item For every pair of integers $(l_<,l_>)$ as in (1), denote by $\tQ_N^{(m|l_<,l_>)}$ the induced convolution map of the commutative diagram
     \begin{align*}
     \xymatrix{
     & \rQ_N^{(m|l_<,l_>)} \ar[dl]_-{\rg_<^{(m|l_<,l_>)}} \ar[dr]^-{\rg_>^{(m|l_<,l_>)}} \\
     \rP_N^{(m|l_<)} \ar[dr]_-{\rf^{(m|l_<)}} & & \rP_N^{(m|l_>)} \ar[dl]^-{\rf^{(m|l_>)}} \\
     & \rM_N^{(m)}
     }
     \end{align*}
     in which all morphisms are finite flat by Proposition \ref{pr:siegel_parahoric}(2) and Proposition \ref{pr:flat}.
\end{enumerate}
\end{notation}

\begin{remark}\label{re:triangle}
Take an \emph{even} integer $0< m\leq N$. Under the decomposition
\[
\rR\rf^{(m)}_*L=\rR\rf^{(m|0)}_*L\oplus\rR\rf^{(m|1)}_*L\oplus\cdots\oplus\rR\rf^{(m|\frac{m}{2})}_*L,
\]
the map $\tQ_N^{(m)}$ is given by the matrix
\[
\begin{pmatrix}
\tQ_N^{(m|0,0)} & \tQ_N^{(m|1,0)} & \cdots & \tQ_N^{(m|\frac{m}{2},0)} \\
\tQ_N^{(m|0,1)} & \tQ_N^{(m|1,1)} & \cdots & \tQ_N^{(m|\frac{m}{2},1)} \\
\vdots & \vdots & \ddots & \vdots \\
\tQ_N^{(m|0,\frac{m}{2})} & \tQ_N^{(m|1,\frac{m}{2})} & \cdots & \tQ_N^{(m|\frac{m}{2},\frac{m}{2})}
\end{pmatrix}
\]
if we write elements in the column form.
\end{remark}

\begin{lem}\label{le:central_2}
Let $0< m\leq N$ be an even integer. Under the identification $\rF_0\rR\Psi L=L$ on $\ol\rP_N$, we have
\[
\rF_0\rR\Psi(\Omega_{N,L}^{\eta,j})\res_{\ol\rM_N^{(m)}}=\Omega_{N,L}^{(m),j}\res_{\ol\rM_N^{(m)}}
\]
for every $0\leq j\leq N+1$.
\end{lem}

\begin{proof}
This is a formal consequence of Proposition \ref{pr:flat}. To reduce the burden in notation, we regard $\bM_N$, $\bP_N$, and $\bQ_N$ as their base change to the ring of integers of $\ol\dQ_p$, and correspondingly for others. Put $m'\coloneqq m+2$ (resp.\ $m'\coloneqq 2r-1$) for $m<N$ (resp.\ $m=N=2r$). Consider the diagram
\[
\xymatrix{
\rQ_N^{(m)} \ar@/_0.5pc/[d]_-{\rg_<^{(m)}} \ar@/^0.5pc/[d]^-{\rg_>^{(m)}} \ar[rr]^-{i_\rQ}
&& \bQ_N^{]m'[} \ar@/_0.5pc/[d]_-{\bg_<^{]m'[}} \ar@/^0.5pc/[d]^-{\bg_>^{]m'[}}
&& \bQ_N^\eta \ar@/_0.5pc/[d]_-{\bg_<^\eta} \ar@/^0.5pc/[d]^-{\bg_>^\eta} \ar@{_(->}[ll]_-{j_\rQ} \\
\rP_N^{(m)} \ar[d]_-{\rf^{(m)}} \ar[rr]^-{i_\rP} && \bP_N^{]m'[} \ar[d]^-{\bff^{]m'[}}
&& \bP_N^\eta \ar[d]^-{\bff^\eta} \ar@{_(->}[ll]_-{j_\rP} \\
\rM_N^{(m)} \ar[rr]^-{i_\rM} && \bM_N^{]m'[} && \bM_N^\eta \ar@{_(->}[ll]_-{j_\rM}
}
\]
in which all vertical morphisms are finite flat by Proposition \ref{pr:flat}, $j_?$ is an open immersion and $i_?$ is a closed immersion for $?\in\{\rM,\rP,\rQ\}$. It induces the following commutative diagram
\[
\xymatrix{
\rR\rf^{(m)}_*L \ar[r]^-\sim \ar[d] & i_\rM^*\rR\bff^{]m'[}_*L \ar[d]  \ar[r] & i_\rM^*\rR j_{\rM*}\rR\bff^\eta_*L \ar[d]\\
\rR\rf^{(m)}_*\rR(\rg_<^{(m)})_*(\rg_<^{(m)})^*L \ar[d]_-\simeq \ar[r]^-\sim & i_\rM^*\rR\bff^{]m'[}_*\rR(\bg_<^{]m'[})_*(\bg_<^{]m'[})^*L \ar[d]_-\simeq \ar[r] & i_\rM^*\rR j_{\rM*}\rR\bff^\eta_*\rR(\bg_<^\eta)_*(\bg_<^\eta)^*L \ar[d]^-\simeq  \\
\rR\rf^{(m)}_*\rR(\rg_>^{(m)})_*(\rg_>^{(m)})^*L \ar[d] \ar[r]^-\sim & i_\rM^*\rR\bff^{]m'[}_*\rR(\bg_>^{]m'[})_*(\bg_>^{]m'[})^*L \ar[d] \ar[r] & i_\rM^*\rR j_{\rM*}\rR\bff^\eta_*\rR(\bg_>^\eta)_*(\bg_>^\eta)^*L \ar[d] \\
\rR\rf^{(m)}_*L \ar[r]^-\sim & i_\rM^*\rR\bff^{]m'[}_*L \ar[r] & i_\rM^*\rR j_{\rM*}\rR\bff^\eta_*L
}
\]
in $\Perv(\rM_N^{(m)},L)$. The lemma then follows.
\end{proof}

\subsection{Correspondences over Newton strata}

In this subsection, we study the restriction of the correspondence \eqref{eq:correspondence} defining the Tate--Thompson local system over non-basic Newton strata in view of certain finite \emph{\'{e}tale} correspondences.

\begin{construction}\label{co:filtration}
Take an \emph{even} integer $0<m\leq N$.
\begin{enumerate}
  \item Let $(A_0,\lambda_0,\eta_0^p;A,\lambda,\eta^p)$ be the universal object over $\rM_N^{(m)}$. By \cite{OZ02}*{Propositions~2.3~\&~1.5 and the remark after Definition~1.2}, the $p$-divisible group $A[\fp^\infty]$ admits a slope filtration
      \[
      0\subseteq A[\fp^\infty]^{>1/2}\subseteq A[\fp^\infty]^{\geq 1/2}\subseteq A[\fp^\infty]
      \]
      in which $A[\fp^\infty]^{>1/2}$ is the maximal (isoclinic) subgroup of slope $>1/2$ and $A[\fp^\infty]^{\geq1/2}$ is the maximal subgroup of slope $\geq 1/2$. Put
      \[
      A[\fp]^{>1/2}\coloneqq A[\fp^\infty]^{>1/2}[\fp],\quad
      A[\fp]^{\geq1/2}\coloneqq A[\fp^\infty]^{\geq1/2}[\fp],\quad
      A[\fp]^{1/2}\coloneqq A[\fp]^{\geq 1/2}/A[\fp]^{>1/2}.
      \]
      In particular, $A[\fp]^{1/2}$ has rank $p^{2(N-m)}$ and is equipped with a principal polarization $\lambda[\fp]^{1/2}$ induced from $\lambda[\fp]$. For every geometric point $x\in\rM_N^{(m)}(\kappa)$, the $\tD_\kappa$-submodule of $\cD_x\otimes\dF_p$ given by $A[\fp]^{>1/2}_x$ coincides with $\cC_x^0$ (Notation \ref{no:eo}(4)).

  \item We define $\rP_N^{((m))}$ to be the scheme over $\rM_N^{(m)}$ that parameterizes $\dF_{p^2}$-linear subgroups $C$ of $A[\fp]^{1/2}$ that are Lagrangian and neutral (Definition \ref{de:neutral}). Note that when $m=N$, $\rP_N^{((m))}=\rM_N^{(m)}$.

  \item When $m<N$, define $\rN_N^{((m))}$ to be the scheme over $\rM_N^{(m)}$ that parameterizes $\dF_{p^2}$-linear subgroups $H$ of $A[\fp]^{1/2}$ that are isotropic and neutral of rank $p^{N-m-2}$, and define $\rR_N^{((m))}$ to be the scheme over $\rM_N^{(m)}$ that parameterizes pairs $(C,H)$ for $C$ and $H$ in $\rP_N^{((m))}$ and $\rN_N^{(m)}$, respectively. When $m=N$, put $\rN_N^{((m))}=\rR_N^{((m))}=\rP_N^{((m))}=\rM_N^{(m)}$. In both cases, we have a natural commutative diagram
      \[
      \xymatrix{
      & \rR_N^{((m))} \ar[ld]_-{\rg_1^{((m))}} \ar[rd]^-{\rg_0^{((m))}} \\
      \rN_N^{((m))} \ar[rd] && \rP_N^{((m))} \ar[ld]^-{\rf^{((m))}} \\
      & \rM_N^{(m)}
      }
      \]
      similar to \eqref{eq:auxiliary}.

  \item Put
      \[
      \rQ_N^{((m))}\coloneqq\rR_N^{((m))}\times_{\rN_N^{((m))}}\rR_N^{((m))}
      \]
      so that we have a natural commutative diagram
      \[
      \xymatrix{
      & \rQ_N^{((m))} \ar[ld]_-{\rg_<^{((m))}} \ar[rd]^-{\rg_>^{((m))}} \\
      \rP_N^{((m))} \ar[rd]_-{\rf^{((m))}} && \rP_N^{((m))} \ar[ld]^-{\rf^{((m))}} \\
      & \rM_N^{(m)}
      }
      \]
      similar to \eqref{eq:correspondence}, in which all morphisms are finite \'{e}tale. As the diagonal $\Delta\rP_N^{((m))}$ of $\rP_N^{((m))}\times_{\rM_N^{(m)}}\rP_N^{((m))}$ is open and closed, we may put
      \[
      \pres\natural\rQ_N^{((m))}\coloneqq\rQ_N^{((m))}\times_{\rP_N^{((m))}\times_{\rM_N^{(m)}}\rP_N^{((m))}}
      \(\rP_N^{((m))}\times_{\rM_N^{(m)}}\rP_N^{((m))}\)\setminus\Delta\rP_N^{((m))},
      \]
      which is a closed subscheme of $\(\rP_N^{((m))}\times_{\rM_N^{(m)}}\rP_N^{((m))}\)\setminus\Delta\rP_N^{((m))}$.

  \item Denote by
      \[
      \pres\natural\tQ_N^{((m))}\colon\rf^{((m))}_*L\to\rf^{((m))}_*L
      \]
      the induced convolution map of the commutative diagram
      \[
      \xymatrix{
      & \pres\natural\rQ_N^{((m))} \ar[ld]_-{\pres\natural\rg_<^{((m))}} \ar[rd]^-{\pres\natural\rg_>^{((m))}} \\
      \rP_N^{((m))} \ar[rd]_-{\rf^{((m))}} && \rP_N^{((m))} \ar[ld]^-{\rf^{((m))}} \\
      & \rM_N^{(m)}
      }
      \]
      in which all morphisms are finite \'{e}tale. For $0\leq j\leq N-m+1$, put
      \[
      \Omega_{N,L}^{((m)),j}\coloneqq\Ker\(\pres\natural\tQ_N^{((m))}-\tl_{N-m,p}^j
      \colon\rf^{((m))}_*L\to\rf^{((m))}_*L\),
      \]
      and simply write $\Omega_{N,L}^{((m))}$ for $\Omega_{N,L}^{((m)),1}$. We regard $\Omega_{N,L}^{((m)),j}$ as zero for an integer $j$ not in the above range.

  \item Take a hermitian space $\rX_{N-m}$ over $\dF_{p^2}$ of (even) rank $N-m$ and put $\rU_{N-m}\coloneqq\rU(\rX_{N-m})(\dF_p)$. Define $\widetilde{\rM_N^{(m)}}$ to be the scheme over $\rM_N^{(m)}$ that parameterizes isometries
      \[
      \eta_\fp^{1/2}\colon \rX_{N-m}\xrightarrow\sim
      \Hom^{\lambda_0[\fp],\lambda[\fp]^{1/2}}_{\dF_{p^2}}(A_0[\fp],A[\fp]^{1/2})
      \]
      of hermitian spaces over $\dF_{p^2}$. The natural morphism $\widetilde{\rM_N^{(m)}}\to\rM_N^{(m)}$ is a Galois morphism with the Galois group $\rU_{N-m}$. Indeed, we have a homomorphism $\rU_{N-m}\to\Aut_{\rM_N^{(m)}}(\widetilde{\rM_N^{(m)}})$ given by the action on $\rX_{N-m}$ such that over every geometric point $x\in\rM_N^{(m)}(\kappa)$, the fiber of $x$ under $\widetilde{\rM_N^{(m)}}\to\rM_N^{(m)}$ is a disjoint union of $\Spec\kappa$ indexed by a $\rU_{N-m}$-principal homogeneous set.

      When $0<m<N$, fix a maximal isotropic $\dF_{p^2}$-linear subspace $\rY$ of $\rX_{N-m}$. We obtain a morphism $\widetilde{\rM_N^{(m)}}\to\rP_N^{((m))}$ sending $(A_0,\lambda_0,\eta_0^p;A,\lambda,\eta^p,\eta_\fp^{1/2})$ to $(A_0,\lambda_0,\eta_0^p;A,\lambda,\eta^p;C)$, in which $C$ is the subgroup spanned by the image of $\eta_\fp(x)$ for every $x\in\rY$. Then we have a sequence of finite \'{e}tale morphisms
      \[
      \widetilde{\rM_N^{(m)}}\to\rP_N^{((m))}\to\rM_N^{(m)}
      \]
      in which the first one is a Galois morphism with the Galois group $\rP_\rY$, where $\rP_\rY\subseteq\rU_{N-m}$ is the subgroup stabilizing $\rY$.
\end{enumerate}
\end{construction}

\begin{lem}\label{le:galois}
For every even integer $0<m\leq N$, the natural map
\[
\bigoplus_{j=0}^{r-\frac{m}{2}}\Omega_{N,L}^{((m)),j}\to\rf^{((m))}_*L
\]
is an isomorphism.
\end{lem}

\begin{proof}
The case $m=N$ is trivial. Thus, we assume $0<m<N$. The proof is similar to that of Lemma \ref{le:local_system} using Construction \ref{co:filtration}(6). We adopt the discussion in \S\ref{ss:complement} with $\kappa/\kappa^+=\dF_{p^2}/\dF_p$. Then the correspondence in Construction \ref{co:filtration}(5) defining $\pres\natural\tQ_N^{((m))}$ corresponds to the element $\tQ_{\rY,1}$. Thus, the lemma follows from Lemma \ref{le:siegel_decomposition}.
\end{proof}

\begin{lem}\label{le:central_1}
Let $0<m\leq N$ be an even integer.
\begin{enumerate}
  \item For every integer $0\leq l\leq\tfrac{m}{2}$, the morphism
      \[
      \sfp^{(m|l)}\colon\rP_N^{(m|l)}\to\rP_N^{((m))}
      \]
      induced by the assignment
      \[
      (A_0,\lambda_0,\eta_0^p;A^\triangle,\lambda^\triangle,\eta^{p\triangle};
      A^\blacktriangle,\lambda^\blacktriangle,\eta^{p\blacktriangle};\psi^\triangle)
      \mapsto(A_0,\lambda_0,\eta_0^p;A^\triangle,\lambda^\triangle,\eta^{p\triangle};C_{\psi^\triangle}),
      \]
      where $C_{\psi^\triangle}\coloneqq(\Ker\psi^\triangle[\fp^\infty]\cap A[\fp]^{\geq 1/2})/(\Ker\psi^\triangle[\fp^\infty]\cap A[\fp]^{>1/2})$, is a flat universal homeomorphism. In particular, we may identify $\rR\rf^{(m|l)}_*L$ with $\rf^{((m))}_*L$ via the adjunction isomorphism
      \[
      \rf^{((m))}_*L=\rR\rf^{((m))}_*L\xrightarrow\sim\rR\rf^{((m))}_*\rR\sfp^{(m|l)}_*L=\rR\rf^{(m|l)}_*L.
      \]

  \item The morphism $\sfp^{(m|0)}$ is an isomorphism.
\end{enumerate}
\end{lem}

\begin{proof}
For (1), since $\rf^{((m))}\circ\sfp^{(m|l)}=\rf^{(m|l)}$ and $\rf^{((m))}$ is finite \'{e}tale, the morphism $\sfp^{(m|l)}$ is finite flat by Proposition \ref{pr:siegel_parahoric}(2). It is clear that $\sfp^{(m|l)}$ induces a bijection on geometric points, which implies that $\sfp^{(m|l)}$ a flat universal homeomorphism.

For (2), we construct an inverse to $\sfp^{(m|0)}$, which sends $(A_0,\lambda_0,\eta_0^p;A^\triangle,\lambda^\triangle,\eta^{p\triangle};C)$ to
\[
(A_0,\lambda_0,\eta_0^p;A^\triangle,\lambda^\triangle,\eta^{p\triangle};
A^\blacktriangle,\lambda^\blacktriangle,\eta^{p\blacktriangle};\psi^\triangle),
\]
where $\psi^\triangle$ is the quotient map with the kernel that is the preimage of $C$ in $A^\triangle[\fp]^{\geq 1/2}$.
\end{proof}

\begin{lem}\label{le:degree}
For every even integer $0<m\leq N$, $\rQ_N^{(m|l_<,l_>)}$ is empty for integers $0\leq l_<,l_>\leq\tfrac{m}{2}$ with $|l_<-l_>|\geq 2$.
\end{lem}

\begin{proof}
This follows from the fact that in Lemma \ref{le:eo}(5), the $\kappa$-codimension of $\cC(m)^{l_<}\cap\cC(m)^{l_>}$ in $\cC(m)^{l_<}$ or $\cC(m)^{l_>}$ is $2|l_<-l_>|$.
\end{proof}

\begin{proposition}\label{pr:degree}
Take an even integer $0<m\leq N$.
\begin{enumerate}
  \item For every integer $0\leq l<\tfrac{m}{2}$, we have
      \begin{align*}
      \deg\rg^{(m|l,l+1)}_<&=\deg\rg^{(m|l+1,l)}_>=p\cdot\tfrac{p^{N}-p^{N-m+2l}}{p^2-1}, \\
      \deg\rg^{(m|l+1,l)}_<&=\deg\rg^{(m|l,l+1)}_>=\tfrac{p^{2l+2}-1}{p^2-1}.
      \end{align*}

  \item For every integer $0\leq l\leq\tfrac{m}{2}$, if we identify $\rR\rf^{(m|l)}_*L$ with $\rf^{((m))}_*L$ as in Lemma \ref{le:central_1}(1), then
      \[
      \tQ_N^{(m|l,l)}=p^{2l}\cdot\pres\natural\tQ_N^{((m))}+\tfrac{p^{2l}-1}{p+1}+\tfrac{p^N-1}{p^2-1}.
      \]
\end{enumerate}
\end{proposition}

\begin{proof}
We define $\rQ_{N,0}^{(m|l,l)}$ via the following Cartesian diagram
\[
\xymatrix{
\rQ_{N,0}^{(m|l,l)} \ar[rrr]\ar[d] &&& \Delta\rP_N^{((m))} \ar[d] \\
\rQ_N^{(m|l,l)} \ar[rrr]^-{(\sfp^{(m|l)}\circ\rg_<^{(m|l,l)},\sfp^{(m|l)}\circ\rg_>^{(m|l,l)})}
&&& \rP_N^{((m))}\times_{\rM_N^{(m)}}\rP_N^{((m))}
}
\]
so that $\rQ_{N,0}^{(m|l,l)}$ is an open and closed subscheme of $\rQ_N^{(m|l,l)}$. Denote by $\pres{0}\rg_<^{(m|l,l)}$ and $\pres{0}\rg_>^{(m|l,l)}$ the restrictions of $\rg_<^{(m|l,l)}$ and $\rg_>^{(m|l,l)}$ to $\rQ_{N,0}^{(m|l,l)}$, respectively, which are finite flat.

We define $\rQ_{N,1}^{(m|l,l)}$ via the following Cartesian diagram
\[
\xymatrix{
\rQ_{N,1}^{(m|l,l)} \ar[rrr]\ar[d] &&& \(\rP_N^{((m))}\times_{\rM_N^{(m)}}\rP_N^{((m))}\)\setminus\Delta\rP_N^{((m))} \ar[d] \\
\rQ_N^{(m|l,l)} \ar[rrr]^-{(\sfp^{(m|l)}\circ\rg_<^{(m|l,l)},\sfp^{(m|l)}\circ\rg_>^{(m|l,l)})}
&&& \rP_N^{((m))}\times_{\rM_N^{(m)}}\rP_N^{((m))}
}
\]
so that $\rQ_{N,1}^{(m|l,l)}$ is an open and closed subscheme of $\rQ_N^{(m|l,l)}$. The upper arrow clearly factors through the closed subscheme $\pres\natural\rQ_N^{((m))}$ (of the target), so that we have induced morphisms
\begin{align*}
\pres\natural\rg_<^{(m|l,l)}&\colon
\rQ_{N,1}^{(m|l,l)}\to\pres\natural\rQ_N^{((m))}\times_{\rg_<^{((m))},\rP_N^{((m))}}\rP_N^{(m|l)},\\
\pres\natural\rg_>^{(m|l,l)}&\colon
\rQ_{N,1}^{(m|l,l)}\to\pres\natural\rQ_N^{((m))}\times_{\rg_>^{((m))},\rP_N^{((m))}}\rP_N^{(m|l)},
\end{align*}
which are finite flat.

For (2), it suffices to show that
\begin{align}\label{eq:degree0}
\deg\pres{0}\rg^{(m|l,l)}_<=\deg\pres{0}\rg^{(m|l,l)}_>=\tfrac{p^{2l}-1}{p+1}+\tfrac{p^N-1}{p^2-1},
\end{align}
and
\begin{align}\label{eq:degree1}
\deg\pres\natural\rg^{(m|l,l)}_<=\deg\pres\natural\rg^{(m|l,l)}_>=p^{2l}.
\end{align}

Thus, the whole proposition is reduced to the computation of degrees of certain finite flat morphisms, for which we adopt a method via rigid tubular neighbourhoods. In general, for an admissible $p$-adic formal scheme $\fX$ with rigid generic fiber $\fX_\eta$ and a locally closed subscheme $\rY\subseteq\rX\coloneqq\fX\otimes_{\dZ_p}\dF_p$, we have the (rigid) tubular neighbourhood $\{\rY\}\subseteq\fX_\eta$ associated with $\rY$, that is, the preimage under the specialization map $\fX_\eta\to\fX\otimes_{\dZ_p}\dF_p$. For $\bX\in\{\bM,\bN,\bP,\bQ,\bR\}$, we let $\fX$ be the formal completion of $\bX_N^{]2r-1[}$ along its special fiber. Then the diagram \eqref{eq:correspondence} induces a similar commutative diagram
\[
\xymatrix{
& \{\rQ_N^{(m|l_<,l_>)}\} \ar[ld]_-{\{\rg_<^{(m|l_<,l_>)}\}} \ar[rd]^-{\{\rg_>^{(m|l_<,l_>)}\}} \\
\{\rP_N^{(m|l_<)}\} \ar[rd]_-{\{\rf^{(m|l_<)}\}} && \{\rP_N^{(m|l_>)}\} \ar[ld]^-{\{\rf^{(m|l_>)}\}} \\
& \{\rM_N^{(m)}\}
}
\]
of rigid analytic spaces over $\dQ_p^\Phi$ for integers $0\leq l_<,l_>\leq\tfrac{m}{2}$. Let $(\cA_0,\lambda_0,\eta_0^p;\cA,\lambda,\eta^p)$ be the universal object over $\{\rM_N^{(m)}\}$. By \cite{She13}*{Theorems~5.4~\&~5.7}, there is a canonical Hodge--Newton filtration by $\dF_{p^2}$-linear finite locally free group schemes over $\{\rM_N^{(m)}\}$:
\[
0\subseteq\cA[\fp]^{>1/2}\subseteq\cA[\fp]^{\geq1/2}\subseteq\cA[\fp]
\]
that lifts the canonical slope filtration on the similar group scheme over $\rM_N^{(m)}$ introduced in Construction \ref{co:filtration}(1). In particular, $\cA[\fp]^{>1/2}$ are $\cA[\fp]^{\geq 1/2}$ are $\dF_{p^2}$-linear finite locally free group schemes of ranks $p^m$ and $p^{2N-m}$, respectively, and are mutual annihilators under the (nondegenerate) hermitian pairing
\[
\langle\;,\;\rangle_\lambda\colon\cA[\fp]\times\cA[\fp]\to\mu_p\otimes_{\dF_p}\dF_{p^2}
\]
induced by $\lambda$. In particular, $\langle\;,\;\rangle_\lambda$ induces a (nondegenerate) hermitian pairing on $\cA[\fp]^{1/2}\coloneqq\cA[\fp]^{\geq 1/2}/\cA[\fp]^{>1/2}$. In what follows, $\dK$ denotes a complete algebraically closed field containing $\dQ_p^\Phi$.

First, we compute $\deg\rg^{(m|l,l+1)}_<$, which equals $\deg\rg^{(m|l+1,l)}_>$ by symmetry. By Proposition \ref{pr:flat}, $\rg^{(m|l,l+1)}_<$ and $\{\rg^{(m|l,l+1)}_<\}$ have the same degree. Take a geometric point $x=(A_0,\lambda_0,\eta_0^p;A^\triangle,\lambda^\triangle,\eta^{p\triangle};
A^\blacktriangleleft,\lambda^\blacktriangleleft,\eta^{p\blacktriangleleft};\psi^\vartriangleleft)\in\{\rP_N^{(m|l)}\}(\dK)$. Then $A^\triangle[\fp]$ is a hermitian space over $\dF_{p^2}$ of dimension $N$ such that the Lagrangian subspace $C^\vartriangleleft\coloneqq\Ker(\psi^\vartriangleleft)$ has dimension $l$ in the quotient $A^\triangle[\fp]/A^\triangle[\fp]^{\geq 1/2}$. The fiber $\{\rg^{(m|l,l+1)}_<\}^{-1}(x)$ is bijective to the set of Lagrangian subspaces $C^\vartriangleright$ of $A^\triangle[\fp]$ satisfying
\begin{itemize}
  \item[(1a)] the image of $C^\vartriangleright$ in $A^\triangle[\fp]/A^\triangle[\fp]^{\geq 1/2}$ has dimension $l+1$;

  \item[(1b)] the codimension of $C^\vartriangleleft\cap C^\vartriangleright$ in $C^\vartriangleleft$ is $1$ (so that $H=C^\vartriangleleft\cap C^\vartriangleright$).
\end{itemize}
By (1a) and (1b), the codimension of $C^\vartriangleleft\cap C^\vartriangleright\cap A^\triangle[\fp]^{>1/2}$ in $C^\vartriangleleft\cap A^\triangle[\fp]^{>1/2}$ has to be $1$. The number of hyperplanes of $C^\vartriangleleft$ satisfying the previous condition is $\tfrac{p^{N}-p^{N-m+2l}}{p^2-1}$. Once $H=C^\vartriangleleft\cap C^\vartriangleright$ is fixed, the number of $C^\vartriangleright$ that is Lagrangian is $p$. Thus, $\deg\rg^{(m|l,l+1)}_<=\deg\rg^{(m|l+1,l)}_>=\deg\{\rg^{(m|l,l+1)}_<\}=p\cdot\tfrac{p^{N}-p^{N-m+2l}}{p^2-1}$.

Second, we compute $\deg\rg^{(m|l+1,l)}_<$, which equals $\deg\rg^{(m|l,l+1)}_>$ by symmetry. By Proposition \ref{pr:flat}, $\rg^{(m|l+1,l)}_<$ and $\{\rg^{(m|l+1,l)}_<\}$ have the same degree. Take a geometric point $x=(A_0,\lambda_0,\eta_0^p;A^\triangle,\lambda^\triangle,\eta^{p\triangle};
A^\blacktriangleleft,\lambda^\blacktriangleleft,\eta^{p\blacktriangleleft};\psi^\vartriangleleft)\in\{\rP_N^{(m|l+1)}\}(\dK)$. Then $A^\triangle[\fp]$ is a hermitian space over $\dF_{p^2}$ of dimension $N$ such that the Lagrangian subspace $C^\vartriangleleft\coloneqq\Ker(\psi^\vartriangleleft)$ has dimension $l+1$ in the quotient $A^\triangle[\fp]/A^\triangle[\fp]^{\geq 1/2}$. The fiber $\{\rg^{(m|l+1,l)}_<\}^{-1}(x)$ is bijective to the set of Lagrangian subspaces $C^\vartriangleright$ of $A^\triangle[\fp]$ satisfying
\begin{itemize}
  \item[(2a)] the image of $C^\vartriangleright$ in $A^\triangle[\fp]/A^\triangle[\fp]^{\geq 1/2}$ has dimension $l$;

  \item[(2b)] the codimension of $C^\vartriangleleft\cap C^\vartriangleright$ in $C^\vartriangleleft$ is $1$ (so that $H=C^\vartriangleleft\cap C^\vartriangleright$).
\end{itemize}
It is clear that choosing such $C^\vartriangleright$ is equivalent to choosing hyperplanes in the image of $C^\vartriangleleft$ in $A^\triangle[\fp]/A^\triangle[\fp]^{\geq 1/2}$, which has $\tfrac{p^{2l+2}-1}{p^2-1}$ ways. Thus, $\deg\rg^{(m|l+1,l)}_<=\deg\rg^{(m|l,l+1)}_>=\deg\{\rg^{(m|l+1,l)}_<\}=\tfrac{p^{2l+2}-1}{p^2-1}$.

Third, we compute $\deg\pres{0}\rg^{(m|l,l)}_<$, which equals $\deg\pres{0}\rg^{(m|l,l)}_>$ by symmetry. By Proposition \ref{pr:flat}, $\deg\pres{0}\rg^{(m|l,l)}_<$ coincides with the degree of the morphism $\{\pres{0}\rg^{(m|l,l)}_<\}\colon\{\rQ_{N,0}^{(m|l,l)}\}\to\{\rP_N^{(m|l)}\}$. Take a geometric point $x=(A_0,\lambda_0,\eta_0^p;A^\triangle,\lambda^\triangle,\eta^{p\triangle};
A^\blacktriangleleft,\lambda^\blacktriangleleft,\eta^{p\blacktriangleleft};\psi^\vartriangleleft)\in\{\rP_N^{(m|l)}\}(\dK)$. Then $A^\triangle[\fp]$ is a hermitian space over $\dF_{p^2}$ of dimension $N$ such that the Lagrangian subspace $C^\vartriangleleft\coloneqq\Ker(\psi^\vartriangleleft)$ has dimension $l$ in the quotient $A^\triangle[\fp]/A^\triangle[\fp]^{\geq 1/2}$. The fiber $\{\pres{0}\rg^{(m|l,l)}_<\}^{-1}(x)$ is bijective to the set of pairs $(C^\vartriangleright,H)$ satisfying
\begin{itemize}
  \item[(3a)] $C^\vartriangleright$ is Lagrangian, whose image in $A^\triangle[\fp]/A^\triangle[\fp]^{\geq 1/2}$ has dimension $l$;

  \item[(3b)] The images of $C^\vartriangleleft\cap A^\triangle[\fp]^{\geq 1/2}$ and $C^\vartriangleright\cap A^\triangle[\fp]^{\geq 1/2}$ in $A^\triangle[\fp]^{1/2}$ coincide.

  \item[(3c)] $H$ is contained in $C^\vartriangleleft\cap C^\vartriangleright$ and has codimension $1$ in both $C^\vartriangleleft$ and $C^\vartriangleright$.
\end{itemize}
When $C^\vartriangleleft=C^\vartriangleright$, the number of $H$ in (3c) is $\tfrac{p^N-1}{p^2-1}$. When $C^\vartriangleleft\cap C^\vartriangleright$ has codimension $1$ in both $C^\vartriangleleft$ and $C^\vartriangleright$, it is an elementary exercise in linear algebra that the number of $C^\vartriangleright$ satisfying (3a) and (3b) is $(p-1)\cdot|\bP^{l-1}(\dF_{p^2})|=\tfrac{p^{2l}-1}{p+1}$. Thus, \eqref{eq:degree0} follows.

Finally, we compute $\deg\pres\natural\rg^{(m|l,l)}_<$, which equals $\deg\pres\natural\rg^{(m|l,l)}_>$ by symmetry. By Proposition \ref{pr:flat}, we have
\[
\sum_{l'=0}^{\frac{m}{2}}\deg\rg^{(m|l,l')}_<=\deg\bg_<^{]2r-1[}=\tfrac{p^{2r-1}}{p-1}.
\]
By Lemma \ref{le:degree}, we have $\deg\rg^{(m|l,l')}_<=0$ unless $l'\in\{l-1,l,l+1\}$. By (1), we have
\[
\deg\rg^{(m|l,l-1)}_<=\tfrac{p^{2l}-1}{p^2-1},\quad
\deg\rg^{(m|l,l+1)}_<=p\cdot\tfrac{p^{N}-p^{N-m+2l}}{p^2-1}.
\]
On the other hand, $\rg^{(m|l,l)}_<=\rg^{((m))}_<\circ\pres\natural\rg^{(m|l,l)}_<$, in which $\rg^{((m))}_<$ is finite \'{e}tale of degree $p\cdot\tfrac{p^{N-m}-1}{p^2-1}$. Now from the moduli interpretation it is easy to see that $\pres\natural\rg^{(m|l,l)}_<$ has constant degree. Then it follows from elementary computation that $\deg\pres\natural\rg^{(m|l,l)}_<=p^{2l}$. Thus, \eqref{eq:degree1} follows.

The proposition is proved.
\end{proof}

In the rest of this subsection, we prove some auxiliary results that will be used later.

\begin{lem}\label{le:cover}
For $m\in\{2,4\}$ and every integer $0\leq l\leq \tfrac{m}{2}$, the map $\rf^{(m|l)}\colon\rP_N^{(m|l)}\to\rM_N^{(m)}$ induces a bijection on the generic points.
\end{lem}

\begin{proof}
By Lemma \ref{le:central_1}, for every integer $0\leq l\leq \tfrac{m}{2}$, $(\sfp^{(m|0)})^{-1}\circ\sfp^{(m|l)}\colon\rP_N^{(m|l)}\to\rP_N^{(m|0)}$ is a flat universal homeomorphism such that $\rf^{(m|0)}\circ(\sfp^{(m|0)})^{-1}\circ\sfp^{(m|l)}=\rf^{(m|l)}$. Thus, the statement in lemma is mutually equivalent for different $l$.

When $m=2$, we show that $\rf^{(2|0)}$ induces a bijection on the generic points. Take a generic point $x$ of $\rM_N^{(2)}$. By \cite{Ach14}*{Lemma~3.2}, we may find a point $y$ of $\rM_N^{(2r-1)}$ that specializes $x$. By Proposition \ref{pr:siegel_parahoric}(3), $y$ has a unique preimage in $\rP_N^\triangle$, which we denote by $y'$. If there are two different generic points $x'$ and $x''$ of $\rP_N^{(2|0)}$ that map to $x$ under $\rf^{(2|0)}$, then both of them specializes to $y'$, which contradicts with Theorem \ref{th:siegel_parahoric}(1). Thus, $\rf^{(2|0)}$ induces a bijection on the generic points.

When $m=4$, we show that $\rf^{(4|1)}$ induces a bijection on the generic points. Take a generic point $x$ of $\rM_N^{(4)}$. By \cite{Ach14}*{Lemma~3.2}, we may find a point $y$ of $\rM_N^{(2r-1)}$ that specializes $x$. By Proposition \ref{pr:siegel_parahoric}(3), $y$ has a unique preimage in $\rP_N^\triangle$, which we denote by $y'$. If there are two different generic points $x'$ and $x''$ of $\rP_N^{(4|1)}$ that map to $x$ under $\rf^{(4|1)}$, then both of them specializes to $y'$. However, both $x'$ and $x''$ belong to $\rP_N^\ddag$, which contradicts with Theorem \ref{th:siegel_parahoric}(1). Thus, $\rf^{(4|1)}$ induces a bijection on the generic points.

The lemma is proved.
\end{proof}

\begin{remark}
We conjecture that the Galois morphism $\widetilde{\rM_N^{(m)}}\to\rM_N^{(m)}$ in Construction \ref{co:filtration}(6) is a Galois cover, that is, it induces a bijection on the generic points. In particular, Lemma \ref{le:cover} should hold for every even integer $0<m\leq N$.
\end{remark}

By Lemma \ref{le:central_1}, we have morphisms $\sfp^{(2|0)}\colon\rP_N^{(2|0)}\to\rP_N^{((2))}$ and $\sfp^{(2|1)}\colon\rP_N^{(2|1)}\to\rP_N^{((2))}$, in which the first is an isomorphism and the second is a flat universal homeomorphism of degree $p^{2r-1}$. Put
\begin{align}\label{eq:central_1}
\rj^{(2)}\coloneqq(\sfp^{(2|0)})^{-1}\circ\sfp^{(2|1)}\colon\rP_N^{(2|1)}\to\rP_N^{(2|0)},
\end{align}
which satisfies $\rf^{(2|0)}\circ\rj^{(2)}=\rf^{(2|1)}$.

\begin{lem}\label{le:marginal_3}
The morphism $\rj^{(2)}\colon\rP_N^{(2|1)}\to\rP_N^{(2|0)}$ \eqref{eq:central_1} extends uniquely to a morphism $\rj\colon\rP_N^\blacktriangle\to\rP_N^\triangle$. Moreover,
\begin{enumerate}
  \item $\rj$ satisfies $\rf^\triangle\circ\rj=\rf^\blacktriangle$;

  \item $\rj$ is a flat universal homeomorphism of degree $p^{2r-1}$;

  \item $\rj$ fits into the following commutative diagram
     \[
     \xymatrix{
     \pres\varpi\rP_N^\triangle \ar[r]^-{\pres\varpi\ri}\ar[d]_-{\pres\varpi\rf^\triangle} & \rP_N^\blacktriangle \ar[r]^-{\rj}
     & \rP_N^\triangle \ar[r]^-{\ri} & \pres\varpi\rP_N^\blacktriangle \ar[d]^-{\pres\varpi\rf^\blacktriangle} \\
     \pres\varpi\rM_N \ar[rrr]^-{\phi_\rM} &&& \pres\varpi\rM_N
     }
     \]
     of schemes over $\rT$.
\end{enumerate}
\end{lem}

\begin{proof}
We use the alternative modular interpretation for $\rP_N$ as in Remark \ref{re:siegel_parahoric}. Write $(A_0,\lambda_0,\eta_0^p;A^\triangle,\lambda^\triangle,\eta^{p\triangle};C)$ for the universal object over $\rP_N^\blacktriangle$. Recall that over $\rP_N^{(2|1)}$, $\rj^{(2)}$ sends $C$ to $C'\coloneqq C\cap A^\triangle[\fp]^{\geq 1/2}+ A^\triangle[\fp]^{>1/2}$.

We need to extend $C'$ to a flat subgroup of $A^\triangle[\fp]$ over the entire $\rP_N^\blacktriangle$ that is still Lagrangian. Choose a nonzero element $\jmath\in\dF_{p^2}$ such that $\Tr_{\dF_{p^2}/\dF_p}\jmath=0$. Then the composite map
\[
q\colon A^\triangle[\fp]\times_{\rP_N^\blacktriangle} A^\triangle[\fp]\xrightarrow{1\times\jmath}A^\triangle[\fp]\times_{\rP_N^\blacktriangle} A^\triangle[\fp]
\to\mu_p
\]
of $\dF_p$-linear finite flat groups over $\rP_N^\blacktriangle$, in which the second arrow is induced by the polarization of $A^\triangle$, is a nondegenerate quadratic form on $A^\triangle[\fp]$. Since $C'$ is Lagrangian, $C'\times_{\rP_N^{(2|1)}}C'$ is contained in the kernel of $q\res_{\rP_N^{(2|1)}}$. Thus, the scheme-theoretical closure $\ol{C'\times_{\rP_N^{(2|1)}}C'}$ of $C'\times_{\rP_N^{(2|1)}}C'$ in $A^\triangle[\fp]\times_{\rP_N^\blacktriangle}A^\triangle[\fp]$ is contained in the kernel of $q$. Let $\Gamma_\jmath\subseteq A^\triangle[\fp]\times_{\rP_N^\blacktriangle} A^\triangle[\fp]$ be the graph of $\jmath$. Consider the intersection $\ol{C'\times_{\rP_N^{(2|1)}}C'}\cap\Gamma_\jmath$,
which is a closed subgroup of $A^\triangle[\fp]\times_{\rP_N^\blacktriangle} A^\triangle[\fp]$, and denote by $C''\subseteq A^\triangle[\fp]$ its (isomorphic) image under the projection to the first factor. Then by construction, $C''$ is a closed subgroup of $A^\triangle[\fp]$ extending $C'$ and that every section of $C''$ is isotropic under the quadratic form $q$. Since $p$ is odd, we conclude that $C''$ itself is isotropic under $q$ hence has rank at most $p^N$ at every geometric point. As $\rP_N^\blacktriangle$ is reduced, it follows that $C''$ is a flat subgroup of $A^\triangle[\fp]$ that is still Lagrangian.

Now we define the morphism $\rj$ to be the one given by the point $(A_0,\lambda_0,\eta_0^p;A^\triangle,\lambda^\triangle,\eta^{p\triangle};C'')\in\rP_N(\rP_N^\blacktriangle)$. Since $\rP_N^\blacktriangle$ is smooth over $\dF_p^\Phi$ and $\rj$ sends all generic points to $\rP_N^\triangle$, $\rj$ indeed factors through the closed subscheme $\rP_N^\triangle$. The uniqueness of $\rj$ is clear. Now we show the three properties.

Part (1) is clear as $\rj^{(2)}$ satisfies such property.

For (2), by (1) and the properness of $\rf^\triangle$ and $\rf^\blacktriangle$ from Theorem \ref{th:siegel_parahoric}(2), $\rj$ is proper. Assuming $\rj$ finite, then it is a finite morphism between regular schemes of the same dimension hence is flat by the miracle flatness criterion over regular local rings \cite{Mat89}*{Theorem~23.1}. Now since generically $\rj$ is a flat universal homeomorphism of degree $p^{2r-1}$, so is $\rj$. Thus, it remains to show that $\rj$ is quasi-finite. Since the set of points of $\rP_N^\blacktriangle$ at which $\rj$ is quasi-finite is open, it suffices to check that for every closed point $x$ of $\rP_N^\triangle$, the set $\rj^{-1}(\{x\})$ is finite. Consider the composite morphism
\[
\pres\varpi\rP_N^\triangle \xrightarrow{\pres\varpi\ri} \rP_N^\blacktriangle
\xrightarrow{\rj}\rP_N^\triangle \xrightarrow{\ri} \pres\varpi\rP_N^\blacktriangle
\xrightarrow{\pres\varpi\rj}\pres\varpi\rP_N^\triangle
\]
in which both $\rj$ and $\pres\varpi\rj$ (the similar morphism for $\pres\varpi\rP_N$) are dominant and proper hence surjective; and both $\pres\varpi\ri$ and $\ri$ are isomorphism. It follows that $\rj$ must induce a bijection on closed points simply by counting points. Part (2) follows.

For (3), it suffices to verify it over $\pres\varpi\rP_N^{(2|0)}$. Take a geometric point
\[
x=(A_0,\lambda_0,\eta_0^p;A^\triangle,\lambda^\triangle,\eta^{p\triangle};
A^\blacktriangle,\lambda^\blacktriangle,\eta^{p\blacktriangle};\psi^\triangle)\in\pres\varpi\rP_N^{(2|0)}(\kappa).
\]
It follows that if we write $\rj(\pres\varpi\ri(x))=(A_0,\lambda_0,\eta_0^p;A^\blacktriangle,\lambda^\blacktriangle,\eta^{p\blacktriangle};
A^\triangledown,\lambda^\triangledown,\eta^{p\triangledown};\psi^\blacktriangledown)$, then $G_{A^\triangle}\coloneqq\Ker(\psi^\blacktriangledown\circ\psi^\triangle)[\fp^\infty]$ satisfies that $G_{A^\triangle}\cap A^\triangle[\fp^\infty]^{>1/2}=A^\triangle[\fp^\infty]^{>1/2}[\fp^2]$ and that the image of $G_{A^\triangle}$ in $A^\triangle[\fp^\infty]^{1/2}$ coincides with $A^\triangle[\fp^\infty]^{1/2}[\fp]$. It follows that $G_{A^\triangle}=F_{A^\triangle}$, that is, $\pres\varpi\rf^\blacktriangle(\ri(\rj(\pres\varpi\ri(x))))=\phi_\rM(\pres\varpi\rf^\triangle(x))$. Part (3) follows.
\end{proof}

\begin{notation}\label{no:jmap}
By the above lemma, the induced map $\tj\colon\rR\rf^\triangle_* L\to\rR\rf^\triangle_*\rR\rj_*L=\rR\rf^\blacktriangle_* L$ is an isomorphism. For every element $l\in L$, we have the map
\begin{align*}
\tj_l\colon\rR\rf^\triangle_* L\to\rR\rf^\triangle_* L\oplus\rR\rf^\blacktriangle_* L=\rR\rf_* L
\end{align*}
in $\Perv(\rM_N,L)$ induced by the pair $(\id,l\tj)$. We also denote by $\tj_l$ for its base change to $\ol\dF_p$ and the induced map on cohomology.
\end{notation}

\subsection{Central nearby cycle of Tate--Thompson local system}

In this subsection, we compute the central nearby cycle $\gr^\rF_0\rR\Psi(\Omega_{N,L}^\eta)$, or more generally $\gr^\rF_0\rR\Psi(\Omega_{N,L}^{\eta,j})$ for $0\leq j\leq r$, which are direct summands of
\[
\gr^\rF_0\rR\Psi(\bff^\eta_*L)=\rR\ol\rf_*\gr^\rF_0\rR\Psi L=\rR\ol\rf_*(\ol\rp^\triangle_*L\oplus\ol\rp^\blacktriangle_*L)=\rR\ol\rf_*\ol\rp^\triangle_*L\oplus\rR\ol\rf_*\ol\rp^\blacktriangle_*L
=\rR\ol\rf^\triangle_*L\oplus\rR\ol\rf^\blacktriangle_*L.
\]
By Proposition \ref{pr:siegel_parahoric}(2,4) and Lemma \ref{le:central_2}, we have
\[
\gr^\rF_0\rR\Psi(\Omega_{N,L}^{\eta,j})=\IC(\ol\rM_N,\gr^\rF_0\rR\Psi(\Omega_{N,L}^{\eta,j})\res_{\ol\rM_N^{(2)}})
=\IC(\ol\rM_N,\rF_0\rR\Psi(\Omega_{N,L}^{\eta,j})\res_{\ol\rM_N^{(2)}})
=\IC(\ol\rM_N,\Omega_{N,L}^{(2),j}\res_{\ol\rM_N^{(2)}}).
\]

On the other hand, for every integer $0\leq j\leq r$, $\Omega_{N,L}^{(2),j}$ is a direct summand of $\rR\rf_* L\res_{\rM_N^{(2)}}$ so that $\IC(\rM_N,\Omega_{N,L}^{(2),j})$ is a direct summand of $\rR\rf_* L$; for every integer $j$, $\Omega_{N,L}^{((2)),j}$ from Construction \ref{co:filtration}(5) is a direct summand of $\rR\rf^\triangle_* L\res_{\rM_N^{(2)}}=\rR\rf^{(2|0)}_* L$ via Lemma \ref{le:central_1}(2), so that $\IC(\rM_N,\Omega_{N,L}^{((2)),j})$ is a direct summand of $\rR\rf^\triangle_* L$.

\begin{proposition}\label{pr:central}
We have
\begin{align*}
\IC(\rM_N,\Omega_{N,L}^{(2),j})=
\begin{dcases}
\tj_{(-p)^{-j}}\IC(\rM_N,\Omega_{N,L}^{((2)),j})\oplus\tj_{(-p)^{j-1-2r}}\IC(\rM_N,\Omega_{N,L}^{((2)),j-2}), &0\leq j<r,\\
\tj_{(-p)^{-r}}\IC(\rM_N,\Omega_{N,L}^{((2)),r-1})\oplus\tj_{(-p)^{-r-1}}\IC(\rM_N,\Omega_{N,L}^{((2)),r-2}), &j=r.
\end{dcases}
\end{align*}
In particular,
\begin{enumerate}
  \item the natural map
     \[
     \bigoplus_{j=0}^r\IC(\rM_N,\Omega_{N,L}^{(2),j})\to\rR\rf^\triangle_* L\oplus\rR\rf^\blacktriangle_*L=\rR\rf_* L
     \]
     is an isomorphism;

  \item we have
     \[
     \gr^\rF_0\rR\Psi(\Omega_{N,L}^\eta)=\IC(\ol\rM_N,\Omega_{N,L}^{(2)})=\tj_{(-p)^{-1}}\IC(\ol\rM_N,\Omega_{N,L}^{((2))}).
     \]
\end{enumerate}
\end{proposition}

\begin{proof}
It suffices to prove the proposition after restriction to $\rM_N^{(2)}$. Denote by
\[
\tj_l^{(2)}\colon\rR\rf^{(2|0)}_*L\to\rR\rf^{(2|0)}_*L\oplus\rR\rf^{(2|1)}_*L=\rR\rf^{(2)}_*L
\]
the restriction of $\tj_l$ (Notation \ref{no:jmap}) to the open subscheme $\rM_N^{(2)}$. We will show that
\begin{align}\label{eq:central_7}
\Omega_{N,L}^{(2),j}=
\begin{dcases}
\tj^{(2)}_{(-p)^{-j}}\Omega_{N,L}^{((2)),j}\oplus\tj^{(2)}_{(-p)^{j-1-2r}}\Omega_{N,L}^{((2)),j-2}, &0\leq j<r,\\
\tj^{(2)}_{(-p)^{-r}}\Omega_{N,L}^{((2)),r-1}\oplus\tj^{(2)}_{(-p)^{-r-1}}\Omega_{N,L}^{((2)),r-2}, &j=r.
\end{dcases}
\end{align}
By Lemma \ref{le:galois} and the assumption that $\tb_{N,p}$ is invertible in $L$, the natural map
\[
\(\bigoplus_{j=0}^{r-1}\tj^{(2)}_{(-p)^{-j}}\Omega_{N,L}^{((2)),j}
\oplus\tj^{(2)}_{(-p)^{j-1-2r}}\Omega_{N,L}^{((2)),j-2}\)\oplus
\(\tj^{(2)}_{(-p)^{-r}}\Omega_{N,L}^{((2)),r-1}\oplus\tj^{(2)}_{(-p)^{-r-1}}\Omega_{N,L}^{((2)),r-2}\)\to\rR\rf^{(2)}_* L
\]
is an isomorphism. Thus, (1) follows from \eqref{eq:central_7}. Part (2) is a special case of \eqref{eq:central_7}.

To show \eqref{eq:central_7}, we use Remark \ref{re:triangle} so that
\begin{align*}
\tQ_N^{(2)}=
\begin{pmatrix}
\tQ_N^{(2|0,0)} & \tQ_N^{(2|1,0)} \\
\tQ_N^{(2|0,1)} & \tQ_N^{(2|1,1)}
\end{pmatrix}
\end{align*}
under the decomposition $\rR\rf^{(2)}_*L=\rR\rf^{(2|0)}_*L\oplus\rR\rf^{(2|1)}_*L$. The above matrix of operators are induced from the diagram
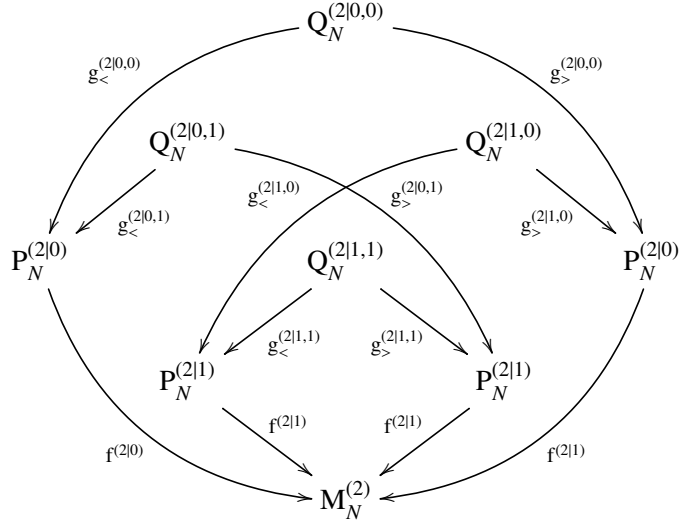
\begin{align*}
\xymatrix{
&& \rQ_N^{(2|0,0)} \ar@/_2.2pc/[ddll]_-{\rg_<^{(2|0,0)}}
\ar@/^2.2pc/[ddrr]^-{\rg_>^{(2|0,0)}} \\
& \rQ_N^{(2|0,1)} \ar[dl]^-{\rg_<^{(2|0,1)}}\ar@/^2pc/[ddrr]^-{\rg_>^{(2|0,1)}} && \rQ_N^{(2|1,0)} \ar[dr]_-{\rg_>^{(2|1,0)}}\ar@/_2pc/[ddll]_-{\rg_<^{(2|1,0)}} \\
\rP_N^{(2|0)} \ar@/_2.2pc/[ddrr]_-{\rf^{(2|0)}} && \rQ_N^{(2|1,1)}\ar[dl]^-{\rg_<^{(2|1,1)}}
\ar[dr]_-{\rg_>^{(2|1,1)}} && \rP_N^{(2|0)} \ar@/^2.2pc/[ddll]^-{\rf^{(2|1)}}\\
& \rP_N^{(2|1)} \ar[dr]^-{\rf^{(2|1)}} && \rP_N^{(2|1)} \ar[dl]_-{\rf^{(2|1)}}\\
&& \rM_N^{(2)}
}
\end{align*}
in which all morphisms are finite flat.

By Proposition \ref{pr:degree}, if we identify $\rR\rf^{(2|0)}_*L$ and $\rR\rf^{(2|1)}_*L$ with $\rf^{((2))}_*L$ as in Lemma \ref{le:central_1}(1), then
\[
\begin{pmatrix}
\tQ_N^{(2|0,0)} & \tQ_N^{(2|1,0)} \\
\tQ_N^{(2|0,1)} & \tQ_N^{(2|1,1)}
\end{pmatrix}=
\begin{pmatrix}
\pres\natural\tQ_N^{((2))}+\tfrac{p^N-1}{p^2-1} & p^{2r-1} \\
1 & p^2\pres\natural\tQ_N^{((2))}+(p-1)+\tfrac{p^N-1}{p^2-1}
\end{pmatrix}.
\]
By a straightforward calculation, we have
\begin{align*}
&\det\(
\begin{pmatrix}
\pres\natural\tQ_N^{((2))}+\tfrac{p^N-1}{p^2-1} & p^{2r-1} \\
1 & p^2\pres\natural\tQ_N^{((2))}+(p-1)+\tfrac{p^N-1}{p^2-1}
\end{pmatrix}-
\begin{pmatrix}
\tl_{N,p}^j+\tfrac{p^N-1}{p^2-1} & \\
& \tl_{N,p}^j+\tfrac{p^N-1}{p^2-1}
\end{pmatrix}
\) \\
&=
\begin{dcases}
p^2\(\pres\natural\tQ_N^{((2))}-\tl_{N-2,p}^j\)\(\pres\natural\tQ_N^{((2))}-\tl_{N-2,p}^{j-2}\), &0\leq j<r,\\
p^2\(\pres\natural\tQ_N^{((2))}-\tl_{N-2,p}^{r-1}\)\(\pres\natural\tQ_N^{((2))}-\tl_{N-2,p}^{r-2}\), & j=r.
\end{dcases}
\end{align*}
Since $\tl_{N-2,p}^j-\tl_{N-2,p}^{j'}$ is invertible in $L$ for every pair of integers $(j,j')$ satisfying $-2\leq j<j'\leq r-1$ and $j'\geq 0$, \eqref{eq:central_7} holds from easy linear algebra.

The proposition is proved.
\end{proof}

\subsection{Marginal nearby cycles of Tate--Thompson local system}

In this subsection, we compute marginal nearby cycles $\gr^\rF_1\rR\Psi(\Omega_{N,L}^\eta)$ and $\gr^\rF_{-1}\rR\Psi(\Omega_{N,L}^\eta)$, which are direct summands of $\rR\ol\rf_*\ol\rp^\ddag_*L(-1)[-1]=\rR\ol\rf^\ddag_*L(-1)[-1]$ and $\rR\ol\rf_*\ol\rp^\ddag_*L[-1]=\rR\ol\rf^\ddag_*L[-1]$, respectively.

In this and the next subsection, for a closed subscheme $Z$ of an ambient scheme, we denote by $L_Z$ the pushforward of the constant sheaf $L$ on $Z$ along the closed immersion (when the ambient scheme is clear from the context).

\begin{notation}\label{no:involution}
We denote by
\begin{itemize}[label={\ding{118}}]
  \item $\rP_N^{\langle1\rangle}$ the induced reduced closed subscheme of $\rP_N^{(1)}$, which is a closed subscheme of $\rP_N^\ddag$ by Proposition \ref{pr:siegel_parahoric}(1),

  \item $\rm^{(1)}\colon\rM_N^{(1)}\to\rM_N$ (or to $\rM_N^\no$ according to the context) the closed immersion,

  \item $\rf^{\langle1\rangle}\colon\rP_N^{\langle1\rangle}\to\rM_N^{(1)}$ the restriction of $\rf$, and

  \item $\rp^{\langle1\rangle\triangle}\colon\rP_N^{\langle1\rangle}\to\rP_N^\triangle$, $\rp^{\langle1\rangle\blacktriangle}\colon\rP_N^{\langle1\rangle}\to\rP_N^\blacktriangle$ and $\rp^{\langle1\rangle\ddag}\colon\rP_N^{\langle1\rangle}\to\rP_N^\ddag$ the corresponding closed immersions.
\end{itemize}
\end{notation}

By Proposition \ref{pr:intersection}(1), the diagram
\begin{align}\label{eq:fiber}
\xymatrix{
\rP^{\langle1\rangle}_N \ar[r]^-{\rp^{\langle1\rangle\ddag}}\ar[d]_-{\rf^{\langle1\rangle}} & \rP_N^\ddag \ar[d]^-{\rf^\ddag} \\
\rM^{(1)}_N \ar[r]^-{\rm^{(1)}} & \rM_N^\no
}
\end{align}
is Cartesian, in which $\rf^{\langle1\rangle}$ is projective smooth of dimension $r-1$.

We first compute the complex $\rR\rf^\ddag_*L$ on $\rM_N^\no$. By the Poincar\'{e} duality and the base change, we have a canonical isomorphism
\[
L_{\rP_N^{\langle1\rangle}}\simeq(\rf^{\langle1\rangle})^!L_{\rM_N^{(1)}}(1-r)[2-2r]
\]
in $\rD(\rP_N^\ddag,L)$. By adjunction, we have a canonical map $L\to L_{\rP_N^{\langle1\rangle}}$ in $\rD(\rP_N^\ddag,L)$. Together, by adjunction and the properness of $\rf^\ddag$ from Theorem \ref{th:siegel_parahoric}(2), we have a canonical map
\begin{align}\label{eq:marginal_1}
\rR\rf^\ddag_*L=\rR\rf^\ddag_!L\to L_{\rM_N^{(1)}}(1-r)[2-2r]
\end{align}
in $\rD(\rM_N^\no,L)$.

\begin{lem}
The fiber of the map \eqref{eq:marginal_1} in $\rD(\rM_N^\no,L)$ is an intermediate extension of $\rR\rf^\ddag_*L\res_{\rM_N^{(4)}}$ to $\rM_N^\no$. Moreover, the map \eqref{eq:marginal_1} splits.
\end{lem}

\begin{proof}
Denote by $F$ the fiber of \eqref{eq:marginal_1}. By Proposition \ref{pr:siegel_parahoric}(5), we know that $\rf^\ddag$ is small away from $\rM_N^{(1)}$, which implies that $F$ is an intermediate extension of $\rR\rf^\ddag_*L\res_{\rM_N^{(4)}}$ to $\rM_N^\no$ away from $\rM_N^{(1)}$. Thus, it remains to check that $\rm^{(1)*}F$ has perverse degree at most $2r-3$, and that $\rm^{(1)!}F$ has perverse degree at least $2r-1$ (note that the perversity on $\rM_N^\no$ has been shifted by the dimension, which is $2r-2$). Since $\rM_N^{(1)}$ has dimension zero, the perversity degree coincides with the usual degree over it.

By the proper base change for \eqref{eq:fiber}, the map
\[
\rm^{(1)*}\rR\rf^\ddag_*L\to\rm^{(1)*}L_{\rM_N^{(1)}}(1-r)[2-2r]
\]
coincides with a map $\rR\rf^{(1)}_!L_{\rP_N^{(1)}}\to L_{\rM_N^{(1)}}(1-r)[2-2r]$ that is a projection to the top degree relative cohomology, which implies that $\rm^{(1)*}F$ has degree at most $2r-3$.

By the proper base change together with the absolute purity for the subscheme $\rP^{\langle1\rangle}_N$ of $\rP^\ddag_N$ (of codimension $r-1$), the map
\[
\rm^{(1)!}\rR\rf^\ddag_*L\to\rm^{(1)!}L_{\rM_N^{(1)}}(1-r)[2-2r]
\]
coincides with a map $\rR\rf^{(1)}_*L_{\rP^{(1)}_N}(1-r)[2-2r]\to L_{\rM^{(1)}_N}(1-r)[2-2r]$ that is a projection to the bottom degree relative cohomology, multiplied by the self-intersection number of the fibers of $\rf^{\langle1\rangle}$ in $\rP_N^\ddag$. By Proposition \ref{pr:intersection}(2), this number is same for every fiber and is invertible in $L$. It follows that that $\rm^{(1)!}F$ has degree at least $2r-1$ and also that \eqref{eq:marginal_1} splits.

The lemma is proved.
\end{proof}

By the above lemma, we obtain a \emph{splitting} short exact sequence
\begin{align}\label{eq:decomposition}
0 \to \IC(\rM_N^\no,\rR\rf^\ddag_*L\res_{\rM_N^{(4)}}) \to \rR\rf^\ddag_*L \to L_{\rM_N^{(1)}}(1-r)[2-2r] \to 0
\end{align}
in the abelian category $\Perv(\rM_N^\no,L)$.

\if false

On the other hand, we have a canonical map
\begin{align}\label{eq:marginal_2}
\IC(\rM_N^\no,\rR\rf^\ddag_*L\res_{\rM_N^{(4)}})\to\rR\rf^\ddag_*L
\end{align}
in $\rD(\rM_N^\no,L)$.

\begin{lem}\label{le:marginal}
The maps \eqref{eq:marginal_1} and \eqref{eq:marginal_2} together give a short exact sequence
\[
0 \to \IC(\rM_N^\no,\rR\rf^\ddag_*L\res_{\rM_N^{(4)}}) \to \rR\rf^\ddag_*L \to L_{\rM_N^{(1)}}(1-r)[2-2r] \to 0
\]
in the abelian category $\Perv(\rM_N^\no,L)$. Moreover, the above short exact sequence has a unique splitting.
\end{lem}

\begin{proof}
By Proposition \ref{pr:siegel_parahoric}(3) and Lemma \ref{re:eo}(2), the reduced subscheme of $(\rf^\ddag)^{-1}\rM_N^{(1)}$ is smooth over $\rM_N^{(1)}$ of relative dimension $r-1$ with geometrically connected fibers. The lemma then follows from Proposition \ref{pr:siegel_parahoric}(5), Theorem \ref{th:siegel_parahoric}(1) and the Borho--MacPherson theorem.
\end{proof}

\fi

\begin{proposition}\label{pr:marginal}
The natural maps
\[
\gr^\rF_1\rR\Psi(\Omega_{N,L}^\eta)\to L_{\ol\rM_N^{(1)}}(-r)[1-2r],\quad
\gr^\rF_{-1}\rR\Psi(\Omega_{N,L}^\eta)\to L_{\ol\rM_N^{(1)}}(1-r)[1-2r]
\]
obtained from \eqref{eq:marginal_1} are both isomorphisms in $\Perv(\ol\rM_N,L)$.
\end{proposition}

We start from a weaker statement.

\begin{proposition}\label{pr:marginal_1}
The intersection $\gr^\rF_{-1}\rR\Psi(\Omega_{N,L}^\eta)\cap\IC(\ol\rM_N^\no,\rR\rf^\ddag_*L\res_{\rM_N^{(4)}})[-1]$ is zero.
\end{proposition}

The case for $N=2$ is trivial as $\rM^{(4)}$ is empty. Now assume $N\geq 4$. It suffices to show that $\gr^\rF_{-1}\rR\Psi(\Omega_{N,L}^\eta)\res_{\ol\rM_N^{(4)}}=0$.

\begin{lem}\label{le:marginal_1}
The projection map $\rR\rf^{(4)}_*L=\rR\rf^{(4|0)}_*L\oplus\rR\rf^{(4|1)}_*L\oplus\rR\rf^{(4|2)}_*L\to\rR\rf^{(4|0)}_*L=\rf^{((4))}_*L$ (Lemma \ref{le:central_1}(2)) induces an isomorphism
\begin{align*}
\rF_0\rR\Psi(\Omega_{N,L}^\eta)\res_{\ol\rM_N^{(4)}}
\xrightarrow\sim\Omega_{N,L}^{((4))}\res_{\ol\rM_N^{(4)}}
\end{align*}
on $\ol\rM_N^{(4)}$. In particular, $\rF_0\rR\Psi(\Omega_{N,L}^\eta)\res_{\ol\rM_N^{(4)}}$ is an $L$-linear local system.
\end{lem}

\begin{proof}
By Lemma \ref{le:central_2}, we have $\rF_0\rR\Psi(\Omega_{N,L}^\eta)\res_{\rM_N^{(4)}}=\Omega_{N,L}^{(4)}$. Thus, we need to study the map $\tQ_N^{(4)}$. We use Remark \ref{re:triangle} so that
\[
\tQ_N^{(4)}=
\begin{pmatrix}
\tQ_N^{(4|0,0)} & \tQ_N^{(4|1,0)} & \tQ_N^{(4|2,0)} \\
\tQ_N^{(4|0,1)} & \tQ_N^{(4|1,1)} & \tQ_N^{(4|2,1)} \\
\tQ_N^{(4|0,2)} & \tQ_N^{(4|1,2)} & \tQ_N^{(4|2,2)}
\end{pmatrix}
\]
under the decomposition $\rR\rf^{(4)}_*L=\rR\rf^{(4|0)}_*L\oplus\rR\rf^{(4|1)}_*L\oplus\rR\rf^{(4|2)}_*L$. The above matrix of operators are induced from the diagram
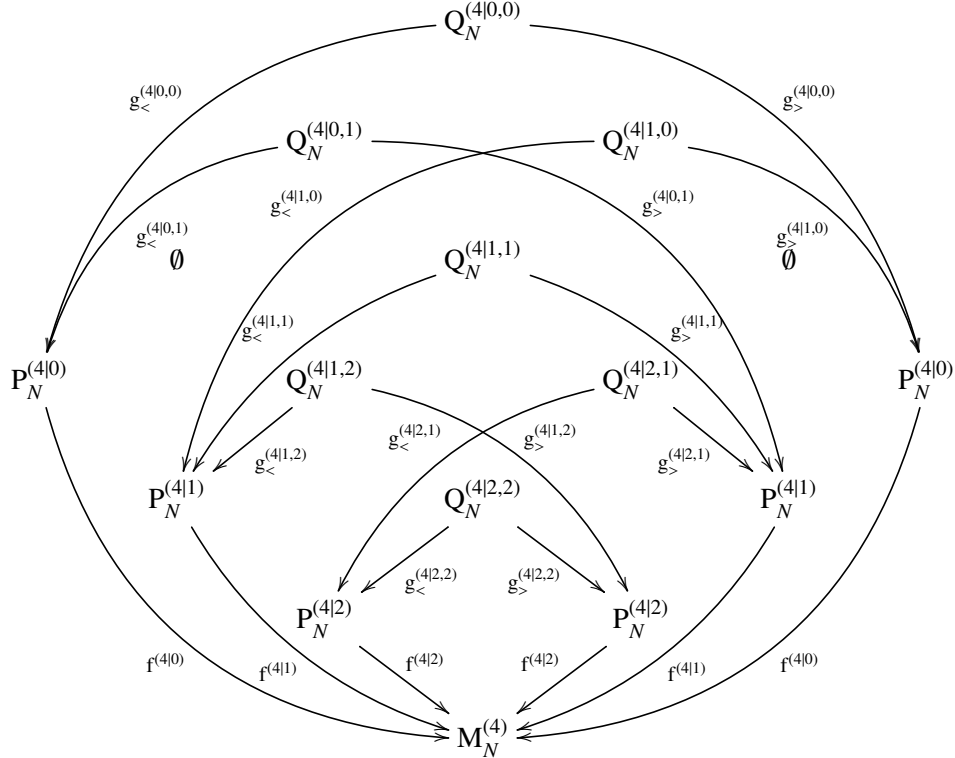
\begin{align*}
\xymatrix{
&&& \rQ_N^{(4|0,0)} \ar@/_3.5pc/[dddlll]_-{\rg_<^{(4|0,0)}}
\ar@/^3.5pc/[dddrrr]^-{\rg_>^{(4|0,0)}} \\
&& \rQ_N^{(4|0,1)}\ar@/_2.2pc/[ddll]^-{\rg_<^{(4|0,1)}}\ar@/^4.0pc/[dddrrr]^-{\rg_>^{(4|0,1)}} &&
\rQ_N^{(4|1,0)}\ar@/^2.2pc/[ddrr]_-{\rg_>^{(4|1,0)}}\ar@/_4.0pc/[dddlll]_-{\rg_<^{(4|1,0)}} \\
& \emptyset && \rQ_N^{(4|1,1)}\ar@/_1.4pc/[ddll]_-{\rg_<^{(4|1,1)}}\ar@/^1.4pc/[ddrr]^-{\rg_>^{(4|1,1)}} && \emptyset \\
\rP_N^{(4|0)} \ar@/_3.5pc/[dddrrr]_-{\rf^{(4|0)}} && \rQ_N^{(4|1,2)}\ar[dl]^-{\rg_<^{(4|1,2)}}\ar@/^1.8pc/[ddrr]^-{\rg_>^{(4|1,2)}} && \rQ_N^{(4|2,1)}\ar[dr]_-{\rg_>^{(4|2,1)}}\ar@/_1.8pc/[ddll]_-{\rg_<^{(4|2,1)}} && \rP_N^{(4|0)} \ar@/^3.5pc/[dddlll]^-{\rf^{(4|0)}} \\
& \rP_N^{(4|1)} \ar@/_1.4pc/[ddrr]_-{\rf^{(4|1)}} && \rQ_N^{(4|2,2)}\ar[dl]^-{\rg_<^{(4|2,2)}}
\ar[dr]_-{\rg_>^{(4|2,2)}} && \rP_N^{(4|1)}\ar@/^1.4pc/[ddll]^-{\rf^{(4|1)}} \\
&& \rP_N^{(4|2)}\ar[dr]^-{\rf^{(4|2)}} && \rP_N^{(4|2)}\ar[dl]_-{\rf^{(4|2)}} \\
&&& \rM_N^{(4)}
}
\end{align*}
in which all morphisms are finite flat.

By Proposition \ref{pr:degree}, if we identify $\rR\rf^{(2|0)}_*L$ and $\rR\rf^{(2|1)}_*L$ with $\rf^{((2))}_*L$ as in Lemma \ref{le:central_1}(1), then
\[
\begin{pmatrix}
\tQ_N^{(4|0,0)} & \tQ_N^{(4|1,0)} & \tQ_N^{(4|2,0)} \\
\tQ_N^{(4|0,1)} & \tQ_N^{(4|1,1)} & \tQ_N^{(4|2,1)} \\
\tQ_N^{(4|0,2)} & \tQ_N^{(4|1,2)} & \tQ_N^{(4|2,2)}
\end{pmatrix}
=
\begin{pmatrix}
\pres\natural\tQ_N^{((4))}+\tfrac{p^N-1}{p^2-1} & p^{2r-3}(p^2+1) & 0 \\
1 & p^2\pres\natural\tQ_N^{((4))}+(p-1)+\tfrac{p^N-1}{p^2-1} & p^{2r-1} \\
0 & p^2+1 & p^4\pres\natural\tQ_N^{((4))}+(p-1)(p^2+1)+\tfrac{p^N-1}{p^2-1}
\end{pmatrix}.
\]
By a brutal computation, we find that
\begin{align*}
&\det
\begin{pmatrix}
\pres\natural\tQ_N^{((4))}-\tl_{N,p}^1 & p^{2r-3}(p^2+1) & 0 \\
1 & p^2\pres\natural\tQ_N^{((4))}+(p-1)-\tl_{N,p}^1 & p^{2r-1} \\
0 & p^2+1 & p^4\pres\natural\tQ_N^{((4))}+(p-1)(p^2+1)-\tl_{N,p}^1
\end{pmatrix} \\
&=p^6\(\pres\natural\tQ_N^{((4))}-(\tl_{N-4,p}^1+1)\)\(\pres\natural\tQ_N^{((4))}-(\tl_{N-4,p}^{-1}+1)\)
\(\pres\natural\tQ_N^{((4))}-(\tl_{N-4,p}^{-3}+1)\).
\end{align*}
Since for $0\leq j\leq r-2$, both $\tl_{N-4,p}^{-1}-\tl_{N-4,p}^j$ and $\tl_{N-4,p}^{-3}-\tl_{N-4,p}^j$ are invertible in $L$, Lemma \ref{le:galois} implies that
\[
\(\pres\natural\tQ_N^{((4))}-(\tl_{N-4,p}^{-1}+1)\)\(\pres\natural\tQ_N^{((4))}-(\tl_{N-4,p}^{-3}+1)\)
\]
is an automorphism of $\rf^{((4))}_* L$. The lemma then follows.
\end{proof}

\begin{lem}\label{le:marginal_2}
We have $\gr^\rF_0\rR\Psi(\Omega_{N,L}^\eta)\res_{\ol\rM_N^{(4)}}\simeq\Omega_{N,L}^{((4))}\res_{\ol\rM_N^{(4)}}$.
\end{lem}

\begin{proof}
First, note that $\gr^\rF_0\rR\Psi(\Omega_{N,L}^\eta)\res_{\ol\rM_N^{(4)}}$ is a direct summand of $(\rf^\triangle_*L\oplus\rf^\blacktriangle_*L)\res_{\ol\rM_N^{(4)}}$, which is an $L$-linear local system by Lemma \ref{le:central_1}. Thus, it suffices to show the isomorphism in the lemma over an open dense subscheme of $\ol\rM_N^{(4)}$. By Proposition \ref{pr:central}(2), we have $\gr^\rF_0\rR\Psi(\Omega_{N,L}^\eta)\simeq\IC(\ol\rM_N,\Omega_{N,L}^{((2))})$. It suffices to show that for every generic point $y$ of $\rM_N^{(4)}$, we have $\IC(\rM_N,\Omega_{N,L}^{((2))})\res_y\simeq\Omega_{N,L}^{((4))}\res_y$.

Let $x$ be the generic point of the trait $T\coloneqq\Spec\cO_{\rM_N,y}$, with $\rG_x$ its Galois group. Let $\rD_y\subseteq\rG_x$ be a decomposition subgroup at $y$ (unique up to conjugation), $\rI_y\subseteq\rD_y$ its inertia, so that $\rG_y\coloneqq\rD_y/\rI_y$ is the Galois group of $y$. Put $\widetilde{x}\coloneqq x\times_{\rM_N^{(2)}}\widetilde{\rM_N^{(2)}}$  (Construction \ref{co:filtration}(6)). If we denote by $\rU$ the subgroup of $\rU_{N-2}$ that fixes every point of $\widetilde{x}$, then $\rU$ is a quotient of $\rG_x$. In particular, we may regard $\Omega_{N-2,L}^1$ from Lemma \ref{le:siegel_decomposition} as an $L[\rG_x]$-module, which gives the local system $\Omega_{N,L}^{((2))}\res_x$ on $x$. As $\cO_{\rM_N,y}$ is a discrete valuation ring, $\IC(\rM_N,\Omega_{N,L}^{((2))})\res_y$ is the local system given by the $L[\rG_y]$-module $(\Omega_{N-2,L}^1)^{\rI_y}=(\Omega_{N-2,L}^1)^\rI$, where $\rI$ denotes the image of $\rI_y$ in $\rU$. Now we compute $(\Omega_{N-2,L}^1)^\rI$ in several steps.

\textbf{Step 1.} Denote by $\widetilde{T}$ the normalization of $T$ in $\widetilde{x}$. Write $(A_0,\lambda_0,\eta_0^p;A,\lambda,\eta^p)$ for the object representing the morphism $\widetilde{T}\to\rM_N$, with the level structure
\[
(\eta_\fp^{1/2})_{\widetilde{x}}\colon \rX_{N-2}\xrightarrow\sim
\Hom^{\lambda_0[\fp],\lambda[\fp]^{1/2}}_{\dF_{p^2}}(A_0[\fp]_{\widetilde{x}},A[\fp]^{1/2}_{\widetilde{x}})
\]
over the generic fiber $\widetilde{x}$. Denote by $\fY$ the set of closed points of $\widetilde{T}$ and $\fL$ the set of isotropic $\dF_{p^2}$-lines in $\rX_{N-2}$.

\textbf{Step 2.} We now construct a natural map $\rL_\obj\colon\fY\to\fL$ through Step 4.

Take a point $y'\in\fY$; put $W\coloneqq\cO_{\widetilde{T},y'}$, which is a discrete valuation ring. Let $x'\in\widetilde{x}$ be the generic point of $\Spec W$, $\kappa(\ol{x'})$ the algebraic closure of the field defining $x'$ (that is, the fraction field of $W$), and $\ol{W}$ the integral closure of $W$ in $\kappa(\ol{x'})$ with $\kappa(\ol{y'})$ its residue field (which is an algebraic closure of the field defining $y'$). Put $\cD\coloneqq\rH^\dr_1(A/\ol{W})_{\tau_\infty}\oplus\rH^\dr_1(A/\ol{W})_{\tau_\infty^\tc}$ (Notation \ref{re:frobenius_verschiebung}(1)), which is naturally a neutral $\tD_{\ol{W}}$-module of signature $(N-1,1)$ (\S\ref{ss:eo}). The principal polarization $\lambda$ induces a perfect alternating pairing $\langle\;,\;\rangle\colon\cD\times\cD\to\ol{W}$ that satisfies $\langle\tF x,y\rangle=\langle x,\tV y\rangle^\sigma$ and such that both $\cD^\leftarrow$ and $\cD^\rightarrow$ are totally isotropic with respect to $\langle\;,\;\rangle$. Note that $\nu_{A/\ol{W},\tau_\infty}$ (Notation \ref{no:frobenius_verschiebung}(4)) is a saturated isotropic $\ol{W}$-submodule of $\cD^\leftarrow$ of rank $1$. Since on $\cD^\leftarrow\otimes_{\ol{W}}\kappa(\ol{y'})$, $\Ker\tF$ and $\Ker\tV$ has zero intersection, we know that $\tF\nu_{A/\ol{W},\tau_\infty}$ is a saturated isotropic $\ol{W}$-submodule of $\cD^\rightarrow$ of rank $1$. Put $\cD^{>1/2}\coloneqq\nu_{A/\ol{W},\tau_\infty}\oplus\tF\nu_{A/\ol{W},\tau_\infty}$, which is a saturated neutral $\tD_{\ol{W}}$-submodule of $\cD$ of rank $2$; it is isotropic since $m_{x'}=2$. Denote by $\cD^{\geq 1/2}$ the orthogonal complement of $\cD^{>1/2}$ in $\cD$; and finally put $\cD^{1/2}\coloneqq\cD^{\geq 1/2}/\cD^{>1/2}$. It is straightforward to check that $\cD^{1/2}$ is a neutral $\tD_{\ol{W}}$-module, with a perfect alternating pairing $\langle\;,\;\rangle^{1/2}$ induced from $\langle\;,\;\rangle$. Put $\cD^?_{x'}\coloneqq\cD^?\otimes_{\ol{W}}\kappa(\ol{x'})$ and $\cD^?_{y'}\coloneqq\cD^?\otimes_{\ol{W}}\kappa(\ol{y'})$.

Since $m_{x'}=2$, the maps $\tF,\tV\colon\cD^{1/2\leftarrow}_{x'}\to\cD^{1/2\rightarrow}_{x'}$ are $(\kappa(\ol{x'}),\sigma)$-linear and $(\kappa(\ol{x'}),\sigma^{-1})$-linear isomorphisms, respectively. Put
\[
\cX^{1/2}_{x'}\coloneqq(\cD^{1/2\leftarrow}_{x'})^{\tF=\tV},
\]
which is an $\dF_{p^2}$-vector space of dimension $N-2$ and $\kappa(\ol{x'})$-linearly spans $\cD^{1/2\leftarrow}_{x'}$. We equip $\cX^{1/2}_{x'}$ with a pairing $(\;,\;)\colon\cX^{1/2}_{x'}\times\cX^{1/2}_{x'}\to\dF_{p^2}$ given by the formula $(e,f)=\langle e,\tF f\rangle$, under which $\cX^{1/2}_{x'}$ becomes a hermitian space over $\dF_{p^2}$ of rank $N-2$.

Since $m_{y'}=4$, the maps $\tF,\tV\colon\cD^{1/2\leftarrow}_{y'}\to\cD^{1/2\rightarrow}_{y'}$ have the following properties:
\begin{enumerate}
  \item both $\Ker\tF$ and $\Ker\tV$ have dimension $1$ and they have zero intersection;

  \item $\cX^{1/2}_{y'}\coloneqq(\cD^{1/2\leftarrow}_{y'})^{\tF=\tV}$ is an $\dF_{p^2}$-vector space of dimension $N-4$, equipped with a similar hermitian pairing.
\end{enumerate}

\textbf{Step 3.} Put
\[
\cL_{y'}\coloneqq\Ker\(\cX^{1/2}_{x'}\cap\cD^{1/2\leftarrow}\to\cD^{1/2\leftarrow}_{y'}\),
\]
which is an $\dF_{p^2}$-linear subspace of $\cX^{1/2}_{x'}$. We claim that $\cL_{y'}$ is an isotropic line in $\cX^{1/2}_{x'}$ and that the induced map
\begin{align}\label{eq:marginal_3}
(\cX^{1/2}_{x'}\cap\cD^{1/2\leftarrow})/\cL_{y'}\to\cX^{1/2}_{y'}
\end{align}
is an isometry of hermitian spaces over $\dF_{p^2}$.

For every nonzero element $e\in\cX^{1/2}_{x'}$, we may choose an element $t\in\kappa(\ol{x'})$, unique up to a scalar in $\ol{W}^\times$, such that $te\in\cD^{1/2\leftarrow}$ and maps nontrivially to $\cD^{1/2\leftarrow}_{y'}$. Then we have
\[
\tF(t e)=t^p\tF(e)=t^p\tV(e)=t^{p-1/p}t^{1/p}\tV(e)=t^{p-1/p}\tV(t e)\in t^{p-1/p}\cD^{1/2\rightarrow}.
\]
In particular, if $t\in\ol{W}\setminus\ol{W}^\times$, then the image of $te$ in $\cD^{1/2\leftarrow}_{y'}$ belongs to $\Ker\tF$, which has dimension $1$ by (1). It follows that $\cX^{1/2}_{x'}\cap\cD^{1/2\leftarrow}$ has dimension at least $N-3$. However, we cannot have $\cX^{1/2}_{x'}\subseteq\cD^{1/2\leftarrow}$ since otherwise we may find a nonzero element $f\in\cX^{1/2}_{x'}$ such that $\tF f=0$ in $\cD^{1/2\rightarrow}_{y'}$ hence $(e,f)=\langle e,\tF f\rangle\in \dF_{p^2}\setminus\ol{W}^\times=\{0\}$ for every $e\in\cX^{1/2}_{x'}$, contradicting with the nondegeneracy of $(\;,\;)$. Thus, $\cX^{1/2}_{x'}\cap\cD^{1/2\leftarrow}$ has dimension $N-3$. For every $e\in\cL_{y'}$, we have $(e,e)=\langle e,\tF e\rangle\in\dF_{p^2}\setminus\ol{W}^\times=\{0\}$, which implies that $\cL_{y'}$ is totally isotropic. Now since the image of the map $\cX^{1/2}_{x'}\cap\cD^{1/2\leftarrow}\to\cD^{1/2\leftarrow}_{y'}$ is contained in $\cX^{1/2}_{y'}$, $\cL_{y'}$ has dimension at least $1$ by (2). The dimension cannot be greater than $1$ since otherwise we may find a nonzero element $e\in\cL_{y'}$ such that $(e,f)=0$ for every $f\in\cX^{1/2}_{x'}$, again a contradiction. Thus, $\cL_{y'}$ is an isotropic line and it is straightforward to check that \eqref{eq:marginal_3} is an isometry.

\textbf{Step 4.} We continue the construction of the map $\rL_\obj$ claimed in Step 2. For the abelian scheme $A_0$ in the moduli interpretation, we define similarly $\cD_0\coloneqq\rH^\dr_1(A_0/\ol{W})_{\tau_\infty}\oplus\rH^\dr_1(A_0/\ol{W})_{\tau_\infty^\tc}$. Now both $\tF,\tV\colon\cD_0^\leftarrow\to\cD_0^\rightarrow$ are isomorphisms, so that $\cX_0\coloneqq(\cD_0^\leftarrow)^{\tF=\tV}$ is an $\dF_{p^2}$-vector space of dimension $1$ contained in $\cD_0^\leftarrow$, equipped with a similar hermitian pairing. It is straightforward to check that the natural map
\begin{align}\label{eq:marginal_5}
\Hom^{\lambda_0[\fp],\lambda[\fp]^{1/2}}_{\dF_{p^2}}(A_0[\fp]_{x'},A[\fp]^{1/2}_{x'})
\to\Hom_{\dF_{p^2}}(\cX_0,\cX^{1/2}_{x'})
\end{align}
is an isometry of hermitian spaces over $\dF_{p^2}$. Thus, there is a unique element $\rL_{y'}\in\fL$ such that the image spanned by maps in $(\eta_\fp^{1/2})_{x'}(\rL_{y'})\subseteq\Hom_{\dF_{p^2}}(\cX_0,\cX^{1/2}_{x'})$ is exactly $\cL_{y'}$. The map $\rL_\obj\colon\fY\to\fL$ has been constructed.

\textbf{Step 5.} Now we fix an element $\rL\in\fL$ contained in $\rY$. Let $\rD_\rL\subseteq\rU_{N-2}$ be the subgroup stabilizing $\rL$. Put $\rX_{N-4}\coloneqq \rL^\perp/\rL$, which is a hermitian space over $\dF_{p^2}$ of rank $N-4$, and put $\rU_{N-4}\coloneqq\rU(\rX_{N-4})(\dF_p)$. Then we have a natural homomorphism $\rD_\rL\to\rU_{N-4}$, whose kernel we denote by $\rI_\rL$.

Fix an element $y'\in\fY$ such that $\rL_{y'}=\rL$, with $x'\in\widetilde{x}$ the generic point specializing to $y'$. Then we have $\rU=\Gal(x'/x)$. Let $\rD\subseteq\rU$ be the decomposition subgroup at $y'$, which is contained in $\rD_\rL$ by the existence of the map $\rL_\obj$. Now by the isometries \eqref{eq:marginal_3}, \eqref{eq:marginal_5} and a similar isometry
\[
\Hom^{\lambda_0[\fp],\lambda[\fp]^{1/2}}_{\dF_{p^2}}(A_0[\fp]_{y'},A[\fp]^{1/2}_{y'})
\to\Hom_{\dF_{p^2}}(\cX_0,\cX^{1/2}_{y'}),
\]
the level structure $(\eta_\fp^{1/2})_{\widetilde{x}}$ induces a level structure
\[
(\eta_\fp^{1/2})_{y'}\colon \rX_{N-4}\xrightarrow\sim
\Hom^{\lambda_0[\fp],\lambda[\fp]^{1/2}}_{\dF_{p^2}}(A_0[\fp]_{y'},A[\fp]^{1/2}_{y'}).
\]
Thus, we obtain a morphism $y'\to\widetilde{y}\coloneqq y\times_{\rM_N^{(4)}}\widetilde{\rM_N^{(4)}}$ (Construction \ref{co:filtration}(6)). In particular, the inertia subgroup $\rI\subseteq\rD$ at $y'$ is contained in $\rI_\rL$. If $(\rU,\rD,\rI)$ is a specializing triple (Definition \ref{de:specialization}),\footnote{In fact, we believe that $(\rU,\rD,\rI)=(\rU_{N-2},\rD_\rL,\rI_\rL)$.} then by Lemma \ref{le:specialization_1} and Lemma \ref{le:specialization}, we have $(\Omega_{N-2,L}^1)^\rI\simeq\Omega_{N-4,L}^1$ as $L[\rG_y]$-modules (via the natural homomorphism $\rG_y\simeq\rD/\rI\to\rD_\rL/\rI_\rL=\rU_{N-4}$). The lemma then follows by Lemma \ref{le:galois}.

\textbf{Step 6.} Finally, we show that $(\rU,\rD,\rI)$ is a specializing triple by verifying the five conditions in Definition \ref{de:specialization}. Conditions (1) and (2) have been verified. Condition (3) follows from Lemma \ref{le:cover} for $m=2$. Conditions (4) and (5) follow from Lemma \ref{le:cover} (which in particular implies that every generic point of $\rP_N^{(2)}$ specializes to two generic points of $\rP_N^{(4)}$).
\end{proof}

\begin{remark}
We expect that $\gr^\rF_0\rR\Psi(\Omega_{N,L}^\eta)\res_{\ol\rM_N^{(m)}}\simeq\Omega_{N,L}^{((m))}\res_{\ol\rM_N^{(m)}}$ for every even integer $0<m\leq N$.
\end{remark}

\begin{proof}[Proof of Proposition \ref{pr:marginal_1}]
Recall that it suffices to show that $\gr^\rF_{-1}\rR\Psi(\Omega_{N,L}^\eta)\res_{\ol\rM_N^{(4)}}=0$.

Note that $\gr^\rF_{-1}\rR\Psi(\Omega_{N,L}^\eta)\res_{\ol\rM_N^{(4)}}[1]$ is a direct summand of $\rR\rf^\ddag_*L\res_{\ol\rM_N^{(4)}}$, hence is an $L$-linear local system on $\ol\rM_N^{(4)}$. By Lemma \ref{le:marginal_2}, $\gr^\rF_0\rR\Psi(\Omega_{N,L}^\eta)\res_{\ol\rM_N^{(4)}}$ is an $L$-linear local system of the same rank as $\Omega_{N,L}^{((4))}$. Combining with Lemma \ref{le:marginal_1}, we have a short exact sequence
\[
0 \to \rF_0\rR\Psi(\Omega_{N,L}^\eta)\res_{\ol\rM_N^{(4)}}\to\gr^\rF_0\rR\Psi(\Omega_{N,L}^\eta)\res_{\ol\rM_N^{(4)}} \to
\gr^\rF_{-1}\rR\Psi(\Omega_{N,L}^\eta)\res_{\ol\rM_N^{(4)}}[1] \to 0
\]
of $L$-linear local systems (placed in degree $0$) on $\ol\rM_N^{(4)}$, in which the first term is also of the same rank as $\Omega_{N,L}^{((4))}$. It follows that $\gr^\rF_{-1}\rR\Psi(\Omega_{N,L}^\eta)\res_{\ol\rM_N^{(4)}}=0$. The proposition is proved.
\end{proof}

Now we focus on the proof of Proposition \ref{pr:marginal}.

\begin{lem}\label{le:marginal_5}
The morphism $\rj\colon\rP_N^\blacktriangle\to\rP_N^\triangle$ from Lemma \ref{le:marginal_3} preserves $\rP_N^{\langle1\rangle}$; and the restriction $\rj\res_{\rP_N^{\langle1\rangle}}\colon\rP_N^{\langle1\rangle}\to\rP_N^{\langle1\rangle}$ is a flat universal homeomorphism of degree $p^{2r-2}$.
\end{lem}

\begin{proof}
By Lemma \ref{le:involution}(2) and Lemma \ref{le:marginal_3}(3), we have the commutative diagram
\[
\xymatrix{
\pres\varpi\rB_N \ar[r]^-{\ri_\rB}_-\simeq \ar@/_1pc/[rdd]^-{\iota} & \rP_N^{\langle1\rangle} \ar@{=}[r]\ar[d]^-{\ri\circ\rp^{\langle1\rangle\blacktriangle}} & \rP_N^{\langle1\rangle} \ar[r]^-{\rj\res_{\rP_N^{\langle1\rangle}}}\ar[d]^-{\rp^{\langle1\rangle\blacktriangle}} & \rP_N^{\langle1\rangle} \ar[d]_-{\rp^{\langle1\rangle\triangle}}\ar@{=}[r] & \rP_N^{\langle1\rangle} \ar[d]_-{\ri\circ\rp^{\langle1\rangle\triangle}} & \pres\varpi\rB_N \ar[l]_-{\ri_\rB}^-\simeq  \ar@/^1pc/[ldd]_-{\iota}\\
&\pres\varpi\rP_N^\triangle \ar[r]^-{\pres\varpi\ri}\ar[d]^-{\pres\varpi\rf^\triangle} & \rP_N^\blacktriangle \ar[r]^-{\rj}
& \rP_N^\triangle \ar[r]^-{\ri} & \pres\varpi\rP_N^\blacktriangle \ar[d]_-{\pres\varpi\rf^\blacktriangle} \\
& \pres\varpi\rM_N \ar[rrr]^-{\phi_\rM} &&& \pres\varpi\rM_N
}
\]
Since $\iota$ is generically an immersion, the composition of the top row must be $\phi_\rB$ in view of \eqref{eq:frobenius}.  Thus, the lemma follows from Lemma \ref{le:frobenius}.
\end{proof}

\begin{proof}[Proof of Proposition \ref{pr:marginal}]
By Remark \ref{re:local_system} and the duality, it suffices to show the statement for $\gr^\rF_1\rR\Psi(\Omega_{N,L}^\eta)$. By (the dual statement of) Proposition \ref{pr:marginal_1} and the fact that \eqref{eq:decomposition} splits, the map
\[
\gr^\rF_1\rR\Psi(\Omega_{N,L}^\eta)\to L_{\ol\rM_N^{(1)}}(-r)[1-2r]
\]
induced from \eqref{eq:decomposition} is a direct summand. Since $L_{\ol\rM_N^{(1)}}(-r)[1-2r]$ is indecomposable (in the abelian category $\Perv(\ol\rM_N,L)$), it suffices to show that $\gr^\rF_1\rR\Psi(\Omega_{N,L}^\eta)$ is nonzero.

Assume the converse that $\gr^\rF_1\rR\Psi(\Omega_{N,L}^\eta)$ vanishes. Then
\[
\rR\ol\rf^\ddag_*L(-1)[-2]=\bigoplus_{\substack{j=0 \\ j\neq 1}}^r\gr^\rF_1\rR\Psi(\Omega_{N,L}^{\eta,j}).
\]
Consider the map $\gr_1^\rF\rR\Psi L[-1]\to\gr_0^\rF\rR\Psi L$, which is identified with
\[
\ol\rp^\ddag_*L(-1)[-2] \to \ol\rp^\triangle_*L \oplus \ol\rp^\blacktriangle_*L,
\]
which is the Gysin map to the first factor and \emph{negative} the Gysin map to the second factor. Applying $\rR\ol\rf_*$, we obtain the map
\[
\rR\ol\rf^\ddag_*L(-1)[-2]\to\rR\ol\rf^\triangle_*L\oplus\rR\ol\rf^\blacktriangle_*L,
\]
so that the image of the direct summand $L_{\ol\rM_N^{(1)}}(-r)[1-2r]$ is contained in
\[
\bigoplus_{\substack{j=0 \\ j\neq 1}}^r\gr^\rF_0\rR\Psi(\Omega_{N,L}^{\eta,j})\subseteq
\rR\ol\rf^\triangle_*L\oplus\rR\ol\rf^\blacktriangle_*L.
\]
Further applying $\ol\rm^{(1)!}$, we obtain the map
\begin{align}\label{eq:marginal_9}
\ol\rm^{(1)!}\rR\ol\rf^\ddag_*L(-1)[-2]\to\ol\rm^{(1)!}\rR\ol\rf^\triangle_*L\oplus\ol\rm^{(1)!}\rR\ol\rf^\blacktriangle_*L,
\end{align}
which is identified with
\begin{align}\label{eq:marginal_10}
\rR\ol\rf^{(1)}_*(\ol\rp^{\langle1\rangle\ddag})^!L(-1)[-2]\to
\rR\ol\rf^{(1)}_*(\ol\rp^{\langle1\rangle\triangle})^!L\oplus\rR\ol\rf^{(1)}_*(\ol\rp^{\langle1\rangle\blacktriangle})^!L.
\end{align}
By the purity theorem, all of the three terms in \eqref{eq:marginal_10} can be canonically identified with $\rR\ol\rf^{\langle1\rangle}_*L(-r)[-2r]$. By Lemma \ref{le:marginal_5} and Lemma \ref{le:marginal_3}(2), the map \eqref{eq:marginal_10} coincides with the graph of $-p^{-1}\cdot(\rj\res_{\ol\rP_N^{\langle1\rangle}})^*$ under such identification. Now the direct summand $\ol\rm^{(1)!}L_{\ol\rM_N^{(1)}}(-r)[-2r]$ of $\ol\rm^{(1)!}\rR\ol\rf^\ddag_*L(-1)[-2]$ coincides with $\rR^0\ol\rf^{\langle1\rangle}_*L(-r)[-2r]$, on which $(\rj\res_{\ol\rP_N^{\langle1\rangle}})^*$ acts trivially. Thus, the image of $\ol\rm^{(1)!}L_{\ol\rM_N^{(1)}}(-r)[-2r]$ under \eqref{eq:marginal_9}, which is nonzero, has zero intersection with
\[
\bigoplus_{\substack{j=0 \\ j\neq 1}}^r\ol\rm^{(1)!}\gr_0^\rF\rR\Psi(\Omega_{N,L}^{\eta,j})
\]
by Proposition \ref{pr:central}. This is a contradiction. Thus, the proposition is proved.
\end{proof}

\subsection{Relation with Abel--Jacobi map}
\label{ss:boosting}

In this subsection, we recall the Abel--Jacobi map for the basic locus and reveal a connection between such map and the boosting map \eqref{eq:boosting_2}.

For every $\varsigma\in\fS$, let $\pres\varsigma\rm^\rb\colon\pres\varsigma\rM_N^\rb\to\pres\varsigma\rM_N$ be the closed immersion, and denote by $\rH^0_\fT(\ol{\pres\varsigma\rB}_N,L)^\diamond$ the kernel of the cycle class map $\rH^0_\fT(\ol{\pres\varsigma\rB}_N,L)\to\rH^{2r}_\fT(\ol{\pres\varsigma\rM}_N,L(r))$.

\begin{definition}\label{co:abel}
We define the \emph{essential Abel--Jacobi map}
\[
\pres\varsigma\alpha\colon\rH^0_\fT(\ol{\pres\varsigma\rB}_N,L)^\diamond\to\rH^{2r-1}_\fT(\ol{\pres\varsigma\rM}_N,L(r))_{p^{-2r}\phi_\rM^*},
\]
where the target means the maximal quotient of $\rH^{2r-1}_\fT(\ol{\pres\varsigma\rM}_N,L(r))$ on which $p^{-2r}\phi_\rM^*$ acts trivially, as follows: For a class $c\in\rH^0_\fT(\ol{\pres\varsigma\rB}_N,L)^\diamond$, write $c'\in\rH^{2r}_\fT(\ol{\pres\varsigma\rM}_N^\rb,(\pres\varsigma\rm^\rb)^!L(r))$ for its refined cycle class, whose image in $\rH^{2r}_\fT(\ol{\pres\varsigma\rM}_N,L(r))$ vanishes. By the Gysin exact sequence, we may find an element $\tilde{c}\in\rH^{2r-1}_\fT(\pres\varsigma\rM_N\setminus\ol{\pres\varsigma\rM}_N^\rb,L(r))$ that lifts $c'$. Then we define $\pres\varsigma\alpha(c)$ to be $(1-p^{-2r}\phi_\rM^*)\tilde{c}$.\footnote{Indeed, by a similar argument of \cite{LTXZZ}*{Proposition~4.4.4}, we know that $\phi_\rS^*$ acts trivially on $\rH^0_\fT(\ol{\pres\varsigma\rS}_N,L)$, so that $\phi_\rB^*$ acts trivially on $\rH^0_\fT(\ol{\pres\varsigma\rB}_N,L)$. By Lemma \ref{le:frobenius}, this further implies that $\phi_\rM^*$ acts on $\rH^{2r}_\fT(\ol{\pres\varsigma\rM}_N^\rb,(\pres\varsigma\rm^\rb)^!L(r))$ by the scalar $p^{2r}$. Thus, $(1-p^{-2r}\phi_\rM^*)\tilde{c}$ is a well-defined element in $\rH^{2r-1}_\fT(\ol{\pres\varsigma\rM}_N,L(r))_{p^{-2r}\phi_\rM^*}$.}
\end{definition}

\begin{remark}\label{re:abel}
When $\dF_p^\Phi=\dF_{p^2}$, $\pres\varsigma\alpha$ is nothing but the usual Abel--Jacobi map induced by the Hoschchild--Serre spectral sequence by Lemma \ref{le:frobenius}.
\end{remark}

By Proposition \ref{pr:central}(2) and Proposition \ref{pr:marginal}, the first page of $\pres\Omega\rE^{p,q}_s$ is as follows:
\begin{align*}
\boxed{
\xymatrix{
q\geq 2r+1 & 0 \ar[r] & \rH^q_\fT(\ol\rM_N,\IC(\rM_N,\Omega_{N,L}^{(2)})(r)) \ar[r] & 0 \\
q=2r & \rH^0_\fT(\ol\rM_N^{(1)},L) \ar[r]^-{\pres\Omega\rd^{-1,2r}_1} & \rH^{2r}_\fT(\ol\rM_N,\IC(\rM_N,\Omega_{N,L}^{(2)})(r)) \ar[r] & 0 \\
q=2r-1 & 0 \ar[r] & \rH^{2r-1}_\fT(\ol\rM_N,\IC(\rM_N,\Omega_{N,L}^{(2)})(r)) \ar[r]  & 0 \\
q=2r-2 & 0 \ar[r] & \rH^{2r-2}_\fT(\ol\rM_N,\IC(\rM_N,\Omega_{N,L}^{(2)})(r)) \ar[r]^-{\pres\Omega\rd^{0,2r-2}_1} &
\rH^0_\fT(\ol\rM_N^{(1)},L(1))\\
q\leq 2r-3 & 0 \ar[r] & \rH^q_\fT(\ol\rM_N,\IC(\rM_N,\Omega_{N,L}^{(2)})(r)) \ar[r] & 0 \\
\pres\Omega\rE^{p,q}_1 & p=-1 & p=0 & p=1
}
}
\end{align*}

\begin{lem}\label{le:boosting_1}
The quotient $L$-module $\rH^{2r-1}_\fT(\ol{\pres\varpi\rM}_N,L(r))_{p^{-2r}\phi_\rM^*}$ is a quotient of $\coker\theta$ \eqref{eq:boosting_1}.
\end{lem}

\begin{proof}
It suffices to show that the map
\[
\pres\varpi\rf_!\circ\ri_!\colon\rH^{2r-1}_\fT(\ol\rP_N,\gr_0^\rF\rR\Psi L)\to\rH^{2r-1}_\fT(\ol{\pres\varpi\rM}_N,L)
\]
sends $\rH^{2r-1}_\fT(\ol\rM_N,\gr_0^\rF\rR\Psi(\Omega_{N,L}^\eta))$ into $(1-p^{-2r}\cdot\phi_\rM^*)\rH^{2r-1}_\fT(\ol{\pres\varpi\rM}_N,L)$. By Proposition \ref{pr:central}(2), every element in $\rH^{2r-1}_\fT(\ol\rM_N,\gr_0^\rF\rR\Psi(\Omega_{N,L}^\eta))$ can be written as $\tj_{-p^{-1}}(c)$ for some element $c\in\rH^{2r-1}_\fT(\ol\rP_N^\triangle,L)$. By Lemma \ref{le:boosting_2} below, we have $\pres\varpi\rf_!\ri_!\tj_{-p^{-1}}(c)=(1-p^{-2r}\phi_\rM^*)\pres\varpi\rf^\blacktriangle_!\ri_!c$ in $\rH^{2r-1}_\fT(\ol{\pres\varpi\rM}_N^{]2r-1[},L)$. However, since $\pres\varpi\rM_N^{[2r-1]}$ is of codimension $r$ in $\pres\varpi\rM_N$ by Proposition \ref{pr:eo}(3), the restriction map $\rH^{2r-1}_\fT(\ol{\pres\varpi\rM}_N,L)\to\rH^{2r-1}_\fT(\ol{\pres\varpi\rM}_N^{]2r-1[},L)$ is injective, which implies that $\pres\varpi\rf_!\ri_!\tj_{-p^{-1}}(c)=(1-p^{-2r}\phi_\rM^*)\pres\varpi\rf^\blacktriangle_!\ri_!c$ holds in $\rH^{2r-1}_\fT(\ol{\pres\varpi\rM}_N,L)$ as well. The lemma follows.
\end{proof}

\begin{lem}\label{le:boosting_2}
For every integer $i$ and every $l\in L$, the map
\[
\pres\varpi\rf_!\circ\ri_!\colon\rH^i_\fT(\ol\rP_N^{]2r-1[},\gr_0^\rF\rR\Psi L)\to\rH^i_\fT(\ol{\pres\varpi\rM}_N^{]2r-1[},L)
\]
sends $\tj_l(c)$ (Notation \ref{no:jmap}) to $(1+p^{1-2r}l\cdot\phi_\rM^*)\pres\varpi\rf^\blacktriangle_!\ri_!c$ for every $c\in\rH^i_\fT(\ol\rM_N^{]2r-1[},\rR\rf^\triangle_*L)=\rH^i_\fT(\ol\rP_N^{]2r-1[}\cap\ol\rP_N^\triangle,L)$.
\end{lem}

Note that $\ri$ sends the open subscheme $\rP_N^{]2r-1[}$ to $\pres\varpi\rP_N^{]2r-1[}$.

\begin{proof}
By definition, we have $\pres\varpi\rf_!\ri_!\tj_l(c)=\pres\varpi\rf^\blacktriangle_!\ri_!c+l\cdot\pres\varpi\rf^\triangle_!\ri_!\rj^*c$. Let $c'\in\rH^i_\fT(\ol{\pres\varpi\rP}_N^{]2r-1[}\cap\ol{\pres\varpi\rP}_N^\triangle,L)$ be the unique element such that $c=\ri^*(\pres\varpi\rj)^*c'$, where $\pres\varpi\rj\colon\pres\varpi\rP_N^\blacktriangle\to\pres\varpi\rP_N^\triangle$ is the morphism in Lemma \ref{le:marginal_3} but for $\pres\varpi\rP_N$. Then $\pres\varpi\rf^\blacktriangle_!\ri_!c=p^{2r-1}\cdot\pres\varpi\rf^\triangle_!c'$ and $\pres\varpi\rf^\triangle_!\ri_!\rj^*c=\pres\varpi\rf^\triangle_!(\pres\varpi\rj\circ\ri\circ\rj\circ\pres\varpi\ri)^*c'$ by Lemma \ref{le:marginal_3}(1,2). By Lemma \ref{le:marginal_3}(3), the diagram
\[
\xymatrix{
\pres\varpi\rP_N^{]2r-1[}\cap\pres\varpi\rP_N^\triangle \ar[rr]^-{\pres\varpi\rj\circ\ri\circ\rj\circ\pres\varpi\ri}\ar[d]_-{\pres\varpi\rf^\triangle} && \pres\varpi\rP_N^{]2r-1[}\cap\pres\varpi\rP_N^\triangle \ar[d]^-{\pres\varpi\rf^\triangle}\\
\pres\varpi\rM_N^{]2r-1[} \ar[rr]^-{\phi_\rM} && \pres\varpi\rM_N^{]2r-1[}
}
\]
commutes. By Proposition \ref{pr:siegel_parahoric}(2), the two vertical morphisms (which is the same one) are finite flat. It follows that the diagram is Cartesian up to a universal homeomorphism of degree one; hence we have $\pres\varpi\rf^\triangle_!(\pres\varpi\rj\circ\ri\circ\rj\circ\pres\varpi\ri)^*c'=\phi_\rM^*\pres\varpi\rf^\triangle_! c'$. As $\pres\varpi\rf^\triangle_!c'=p^{1-2r}\cdot\pres\varpi\rf^\blacktriangle_!\ri_!c$, the lemma follows.
\end{proof}

\begin{remark}
We are not sure whether $\rH^{2r-1}_\fT(\ol{\pres\varpi\rM}_N,L(r))_{p^{-2r}\phi_\rM^*}$ coincides with $\coker\theta$ when $N\geq 4$.
\end{remark}

Put
\begin{align*}
\rH^0_\fT(\ol\rP_N^{\langle1\rangle},L)^\diamond\coloneqq
\Ker\((\rp^{\langle1\rangle\triangle}_!,-\rp^{\langle1\rangle\blacktriangle}_!)
\colon\rH^0_\fT(\ol\rP_N^{\langle1\rangle},L)
\to\rH^{2r}_\fT(\ol\rP_N^\triangle,L(r))\oplus\rH^{2r}_\fT(\ol\rP_N^\blacktriangle,L(r))\).
\end{align*}
Then the morphism $\rf^{\langle1\rangle}$ (Notation \ref{no:involution}) induces an isomorphism
\[
\rf^{\langle1\rangle*}\colon\pres\Omega\rE^{-1,2r}_2\xrightarrow\sim\rH^0_\fT(\ol\rP_N^{\langle1\rangle},L)^\diamond.
\]
Thus, the boosting map \eqref{eq:boosting_2} can be written as
\begin{align}\label{eq:boosting_6}
\beta\colon\rH^0_\fT(\ol\rP_N^{\langle1\rangle},L)^\diamond\to\coker\theta.
\end{align}

Now we reveal a hidden relation between \eqref{eq:boosting_6} and the essential Abel--Jacobi map for $\pres\varpi\rM_N$.

\begin{theorem}\label{th:boosting}
The morphism $\ri_\rB$ in Lemma \ref{le:involution}(1) induces a map $\ri_\rB^*\colon\rH^0_\fT(\ol\rP_N^{\langle1\rangle},L)^\diamond\to\rH^0_\fT(\ol{\pres\varpi\rB}_N,L)^\diamond$, such that the diagram
\[
\xymatrix{
\rH^0_\fT(\ol\rP_N^{\langle1\rangle},L)^\diamond \ar[r]^-{\beta}\ar[d]_-{\ri_\rB^*} & \coker\theta \ar@{->>}[d]^-{\text{Lemma \ref{le:boosting_1}}} \\
\rH^0_\fT(\ol{\pres\varpi\rB}_N,L)^\diamond \ar[r]^-{\pres\varpi\alpha} & \rH^{2r-1}_\fT(\ol{\pres\varpi\rM}_N,L(r))_{p^{-2r}\phi_\rM^*}
}
\]
commutes.
\end{theorem}

\begin{proof}
Lemma \ref{le:involution}(1) gives us an isomorphism $\ri_\rB\colon\rB_N\xrightarrow\sim\rP_N^{\langle1\rangle}$, so that $\ri_\rB^*=(\ri_\rB^{-1})_!$. It follows that the image of $\rH^0_\fT(\ol\rP_N^{\langle1\rangle},L)^\diamond$ under $\ri_\rB^*$ is contained in $\rH^0_\fT(\ol\rB_N,L)^\diamond$. Regarding the target of the map $\beta$ as $\rH^{2r-1}_\fT(\ol{\pres\varpi\rM}_N,L(r))_{p^{-2r}\phi_\rM^*}$, we need to show that $\beta=\pres\varpi\alpha\circ(\ri_\rB^{-1})_!$.

By Proposition \ref{pr:marginal}, the natural map $\gr_0^\rF\rR\Psi(\Omega_{N,L}^\eta)\res_{\ol\rM_N^{]2r-1[}}\to\rF_{\geq0}\rR\Psi(\Omega_{N,L}^\eta)\res_{\ol\rM_N^{]2r-1[}}$ is an isomorphism. Consider the following commutative diagram
\[
\xymatrix{
\rH^{2r-1}_\fT(\ol\rM_N,\gr_0^\rF\rR\Psi(\Omega_{N,L}^\eta)(r)) \ar@{=}[r]\ar[d] &
\rH^{2r-1}_\fT(\ol\rM_N,\gr_0^\rF\rR\Psi(\Omega_{N,L}^\eta)(r)) \ar[d]\ar[r]^-{\sigma_1} &
\rH^{2r-1}_\fT(\ol{\pres\varpi\rM}_N,L(r)) \ar[d] \\
\rH^{2r-1}_\fT(\ol\rM_N,\rF_{\geq0}\rR\Psi(\Omega_{N,L}^\eta)(r)) \ar[r]^-{\rho_2} \ar[d] &
\rH^{2r-1}_\fT(\ol\rM_N^{]2r-1[},\gr_0^\rF\rR\Psi(\Omega_{N,L}^\eta)(r)) \ar[d]\ar[r]^-{\sigma_2} &
\rH^{2r-1}_\fT(\ol{\pres\varpi\rM}_N^{]2r-1[},L(r)) \ar[d]^-{\delta} \\
\rH^0_\fT(\ol\rP_N^{\langle1\rangle},L) \ar[r]^-{\rho_3} &
\rH^{2r}_\fT(\ol\rM_N^{[2r-1]},\ol\rm^{\rb!}\gr_0^\rF\rR\Psi(\Omega_{N,L}^\eta)(r))\ar[r]^-{\sigma_3} &
\rH^{2r}_\fT(\ol{\pres\varpi\rM}_N^{[2r-1]},\ol{\pres\varpi\rm}^{\rb!}L(r))
}
\]
in the category $\Mod(L[\Gal(\ol\dF_p/\dF_p^\Phi)])$. Here, the map $\sigma_1$ is the composition
\begin{align*}
\rH^{2r-1}_\fT(\ol\rM_N,\gr_0^\rF\rR\Psi(\Omega_{N,L}^\eta)(r))
&\hookrightarrow\rH^{2r-1}_\fT(\ol\rM_N,\gr_0^\rF\rR\Psi(\bff^\eta_*L)(r))
=\rH^{2r-1}_\fT(\ol\rP^\triangle_N,L(r))\oplus\rH^{2r-1}_\fT(\ol\rP^\blacktriangle_N,L(r)) \\
&\to\rH^{2r-1}_\fT(\ol\rP^\triangle_N,L(r))\xrightarrow{\ri_!}\rH^{2r-1}_\fT(\ol{\pres\varpi\rP}^\blacktriangle_N,L(r))
\xrightarrow{\pres\varpi\rf^\blacktriangle_!}\rH^{2r-1}_\fT(\ol{\pres\varpi\rM}_N,L(r));
\end{align*}
and similarly for $\sigma_2$ and $\sigma_3$.

Take an element $b\in\rH^0_\fT(\ol\rP_N^{\langle1\rangle},L)^\diamond$, with a lift $\tilde{b}\in\rH^{2r-1}_\fT(\ol\rM_N,\rF_{\geq0}\rR\Psi(\Omega_{N,L}^\eta)(r))$. By Proposition \ref{pr:central}(2), $\rho_2(\tilde{b})=\tj_{-p^{-1}}(c)$ for some element
\[
c\in\rH^{2r-1}_\fT(\ol\rM_N^{]2r-1[},\IC(\rM_N,\Omega_{N,L}^{((2))})(r))
\subseteq\rH^{2r-1}_\fT(\ol\rP_N^{]2r-1[}\cap\ol\rP_N^\triangle,L(r)).
\]
By definition, $\beta(b)$ coincides with the image of $\pres\varpi\rf_!\ri_!\tj_{-p^{-1}}(c)$, which belongs to $\rH^{2r-1}_\fT(\ol{\pres\varpi\rM}_N^{]2r-1[},L(r))$, in the quotient $\rH^{2r-1}_\fT(\ol{\pres\varpi\rM}_N^{]2r-1[},L(r))_{p^{-2r}\phi_\rM^*}$. By Lemma \ref{le:boosting_2}, we have $\pres\varpi\rf_!\ri_!\tj_{-p^{-1}}(c)=(1-p^{-2r}\cdot\phi_\rM^*)\pres\varpi\rf^\blacktriangle_!\ri_!c$. On the other hand, $\pres\varpi\rf^\blacktriangle_!\ri_!c$ is nothing but $\sigma_2(\rho_2(\tilde{b}))$. It remains to show that $\delta(\sigma_2(\rho_2(\tilde{b})))$ coincides with the refined cycle class of $(\ri_\rB^{-1})_!b$ under $\iota$ in $\rH^{2r}_\fT(\ol{\pres\varpi\rM}_N^{[2r-1]},\ol{\pres\varpi\rm}^{\rb!}L(r))$. By the commutativity of the diagram, $\delta(\sigma_2(\rho_2(\tilde{b})))=\sigma_3(\rho_3(b))$, which is the refined cycle class of $b$ under $\pres\varpi\rf\circ\ri$, which is what we want by Lemma \ref{le:involution}(2).

The theorem is proved.
\end{proof}

\begin{remark}\label{re:significance}
We explain the significance of Theorem \ref{th:boosting}. Take $\rK^p\in\fK^p$ and suppose that we have a ring $\dT$ of certain Hecke operators away from $p$ acting on the Shimura variety $\Sh(\rV,\rK^p\rK_p)$ and other related moduli schemes. For a maximal ideal $\fm$ of $\dT$, we care about the surjectivity of the localized map $(\pres\varpi\alpha)_\fm$. By Theorem \ref{th:boosting}, the surjectivity of $(\pres\varpi\alpha)_\fm$ is implied by the surjectivity of $\beta_\fm$, which in turn is implied by the surjectivity of $\vartheta_\fm$ in \eqref{eq:boosting} by the definition of the boosting map \eqref{eq:boosting_2}, which is further implied by the surjectivity of $\vartheta^\eta_\fm$ from Remark \ref{re:vartheta}. Now, the surjectivity of $\vartheta^\eta_\fm$ is a problem only on the generic fibers, for which one can use information from places away from the place for the original problem (namely, $\fp$).
\end{remark}

\subsection{Appendix: Complements on Tate--Thompson representations}
\label{ss:complement}

In this subsection, we prove several auxiliary results concerning Tate--Thompson representations. Fix a quadratic extension of local fields $\kappa/\kappa^+$, with $q$ the cardinality of $\kappa^+$. Take an even nonnegative integer $N=2r$.

Consider a hermitian space $\rX_N$ over $\kappa$ of rank $N$. Put $\rU_N\coloneqq\rU(\rX_N)(\kappa^+)$. Fix a maximal isotropic $\kappa$-linear subspace $\rY$ of $\rX_N$ and let $\rP_\rY\subseteq\rU_N$ be the subgroup stabilizing $\rY$. For every integer $0\leq i\leq r$, denote by $\rQ_{\rY,i}$ the subset of $\rU_N$ consisting of $g$ satisfying $\dim_\kappa g\rY/(\rY\cap g\rY)=i$. Then $\rQ_{\rY,i}$ is a double $\rP_\rY$-coset; and we have $\rU_N=\coprod_{i=0}^r\rQ_{\rY,i}$.

Take a ring $L$. The $L$-algebra $L[\rP_\rY\backslash\rU_N/\rP_\rY]$ is a free $L$-module generated by $\tQ_{\rY,i}\coloneqq\CF_{\rQ_{\rY,i}}$ for $0\leq i\leq r$. In what follows, we write $\tQ_\rY\coloneqq\tQ_{\rY,1}$ for short. For every left $L[\rP_\rY\backslash\rU_N/\rP_\rY]$-module $M$ and every $l\in L$, we put $M_l\coloneqq\{m\in M\res \tQ_\rY m=l m\}$, which is an $L[\rP_\rY\backslash\rU_N/\rP_\rY]$-submodule. Recall the constants from Notation \ref{no:numerical1}.

\begin{lem}\label{le:siegel_hecke}
Suppose that $\tb_{N,q}$ is invertible in $L$. Then the natural map $L[T]\to L[\rP_\rY\backslash\rU_N/\rP_\rY]$ sending $T$ to $\tQ_\rY$ induces an isomorphism
\[
L[T]\left/\prod_{j=0}^r(T-\tl_{N,q}^j)\right.\to L[\rP_\rY\backslash\rU_N/\rP_\rY]
\]
of finite \'{e}tale $L$-rings.
\end{lem}

\begin{proof}
We first show that
\begin{align}\label{eq:siegel_hecke}
\tQ_\rY\tQ_{\rY,i}=
\frac{q^{2i+2}-1}{q^2-1}\tQ_{\rY,i+1}+\frac{(q^{2i}-1)(q-1)}{q^2-1}\tQ_{\rY,i}+\frac{q^{2i-1}(q^{2(r-i+1)}-1)}{q^2-1}\tQ_{\rY,i-1}
\end{align}
holds for every $0\leq i\leq r$ (where $\tQ_{\rY,r+1}=\tQ_{\rY,-1}=0$).

As the case for $i=0$ is trivial, we assume $i\geq 1$. Denote by $\fY$ the set of maximal isotropic subspaces of $\rX_N$. Then $L[\rP_\rY\backslash\rU_N/\rP_\rY]$ acts on $L[\fY]$, the $L$-module freely generated by $\fY$. It suffices to show \eqref{eq:siegel_hecke} with respect to the action on the element $[\rY]\in L[\fY]$. By definition,
\[
\tQ_{\rY,i}[\rY]=\sum_{\dim_\kappa\rY'/(\rY\cap\rY')=i}[\rY']
\]
for every $1\leq i\leq r$. In particular,
\[
\tQ_\rY\tQ_{\rY,i}[\rY]=\sum_{\dim_\kappa\rY'/(\rY\cap\rY')=i}\sum_{\dim_\kappa\rY''/(\rY'\cap\rY'')=1}[\rY''].
\]
Consider the diagram
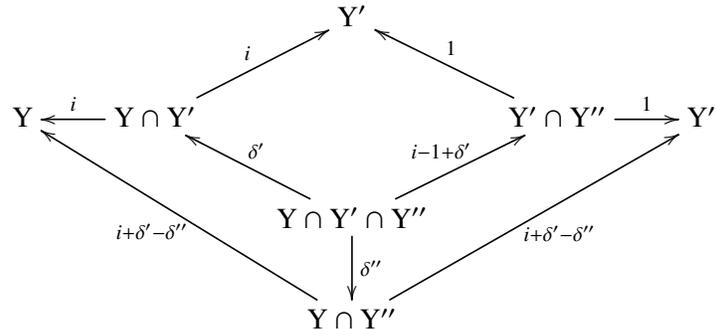
\[
\xymatrix{
& & \rY' \\
\rY & \rY\cap\rY'\ar[l]_-{i}\ar[ur]^{i} & & \rY'\cap\rY'' \ar[lu]_-{1}\ar[r]^-{1} & \rY'' \\
& & \rY\cap\rY'\cap\rY'' \ar[ul]_-{\delta'} \ar[ur]^-{i-1+\delta'} \ar[d]^-{\delta''}\\
& & \rY\cap\rY'' \ar[uull]^-{i+\delta'-\delta''} \ar[uurr]_-{i+\delta'-\delta''}
}
\]
of inclusions of $\kappa$-vector spaces in which the number by the arrow indicates the corresponding codimension. Fix an element $\rY''$ in the above diagram and we count the number of $\rY'$ that fits. Since $\delta',\delta''\in\{0,1\}$, there are four cases.
\begin{itemize}[label={\ding{118}}]
  \item Suppose that $(\delta',\delta'')=(0,0)$, which implies $\dim_\kappa\rY''/(\rY\cap\rY'')=i$. In this case, the number of choices of $\rY'\cap\rY''$ is same as the cardinality of $\dP^{i-1}(\kappa)$, which equals $\frac{q^{2i}-1}{q^2-1}$. The subset of elements $\rY'\in\fY$ with fixed intersection with $\rY''$ has cardinality $q+1$, of which two do not fit the above diagram, namely, $\rY''$ itself and $(\rY'\cap\rY'')+(\rY'\cap\rY'')^\perp\cap\rY$ (which is different from $\rY''$ since $i\geq 1$). Thus, the number of $\rY'$ that contributes to the sum equals $\frac{(q^{2i}-1)(q-1)}{q^2-1}$.

  \item Suppose that $(\delta',\delta'')=(1,1)$, which implies $\dim_\kappa\rY''/(\rY\cap\rY'')=i$. In this case, we have $\rY'=(\rY\cap\rY')+(\rY'\cap\rY'')$ and $\rY''=(\rY\cap\rY'')+(\rY'\cap\rY'')$. Since $\rY'$ and $\rY''$ are different maximal isotropic subspaces, there exists $x'\in\rY\cap\rY'$ and $x''\in\rY\cap\rY''$ such that $(x',x'')\neq 0$, which contradicts with the fact that $\rY$ is also isotropic. Thus, this case cannot happen.

  \item Suppose that $(\delta',\delta'')=(0,1)$, which implies $\dim_\kappa\rY''/(\rY\cap\rY'')=i-1$. In this case, the number of choices of $\rY\cap\rY'\cap\rY''=\rY\cap\rY'$ is same as the cardinality of $\dP^{r-i}(\kappa)$, which equals $\frac{q^{2(r-i+1)}-1}{q^2-1}$. The number of choices of $\rY'\cap\rY''$ with fixed intersection with $\rY\cap\rY''$ is same as the cardinality of $(\dP^{i-1}\setminus\dP^{i-2})(\kappa)$, which equals $q^{2i-2}$. Finally, the subset of elements $\rY'\in\fY$ with fixed intersection with $\rY''$ has cardinality $q+1$, of which ones does not fit the above diagram, namely, $\rY''$ itself. Thus, the number of $\rY'$ that contributes to the sum equals $\frac{q^{2i-1}(q^{2(r-i+1)}-1)}{q^2-1}$.

  \item Suppose that $(\delta',\delta'')=(1,0)$, which implies $\dim_\kappa\rY''/(\rY\cap\rY'')=i+1$ and $i<r$. In this case, the number of choices of $\rY'\cap\rY''$ is same as the cardinality of $\dP^i(\kappa)$, which equals $\frac{q^{2i+2}-1}{q^2-1}$. Once $\rY'\cap\rY''$ is fixed, there is exactly one $\rY'$ that fits the diagram, namely, $(\rY'\cap\rY'')+(\rY'\cap\rY'')^\perp\cap\rY$. Thus, the number of $\rY'$ that contributes to the sum equals $\frac{q^{2i+2}-1}{q^2-1}$.
\end{itemize}
Combining the four cases, \eqref{eq:siegel_hecke} follows.

Since $q^{2i}-1$ is invertible in $L$ for $1\leq i\leq r$, \eqref{eq:siegel_hecke} implies that $L[\rP_\rY\backslash\rU_N/\rP_\rY]$ has the form $L[\tQ_\rY]/f(\tQ_\rY)$ for some monic polynomial $f(T)\in L[T]$ of degree $r+1$. Note that since $\tb_{N,q}$ is invertible in $L$, the difference $\tl_{N,q}^j-\tl_{N,q}^{j'}$ is invertible in $L$ for $0\leq j<j'\leq r$. It remains to show that $\tl_{N,q}^j$ is a root of $f(T)$ for every $0\leq j\leq r$. Indeed, we can show, by a straightforward but tedious computation, that \eqref{eq:siegel_hecke} is satisfied with
\[
\tQ_{\rY,i}=(-1)^i\sum_{k=0}^i(-q)^{(i-k)(i+j-k)+k(k-1)/2}\qbinom{r-j}{i-k}_{q^2}\qbinom{j}{k}_{-q},\quad 0\leq i\leq r
\]
\cite{LTXZZ}*{Notation~1.3.1} for every $0\leq j\leq r$. As the right-hand side for $\tQ_{\rY,1}$ is nothing but $\tl_{N,q}^j$, the lemma follows.
\end{proof}

Now we study $L[\rP_\rY\backslash\rU_N]$ as an $L[\rP_\rY\backslash\rU_N/\rP_\rY]$-module via left convolution. For every integer $0\leq j\leq N+1$, put
\[
\Omega_{N,L}^j\coloneqq L[\rP_\rY\backslash\rU_N]_{\tl_{N,q}^j}
\]
In particular, $\Omega_{N,L}^j=\Omega_{N,L}^{N+1-j}$.

\begin{lem}\label{le:siegel_decomposition}
Suppose that $\tb_{N,q}$ is invertible in $L$. The natural map
\[
\bigoplus_{j=0}^r\Omega_{N,L}^j\to L[\rP_\rY\backslash\rU_N]
\]
is an isomorphism of $L[\rP_\rY\backslash\rU_N/\rP_\rY]$-modules. Moreover,
\begin{enumerate}
  \item $\Omega_{N,L}^j$ is stable under the right translation by $\rU_N$ for $0\leq j\leq r$;

  \item $\Omega_{N,L}^0$ is a free $L$-module of rank $1$;

  \item $\Omega_{N,\dC}^1$ is the Tate--Thompson representation from \cite{LTXZZ}*{\S C.2} (when $N=0$, we regard the Tate--Thompson representation as the trivial representation).
\end{enumerate}
\end{lem}

\begin{proof}
The claim of the decomposition follows from Lemma \ref{le:siegel_hecke}.

Part (1) follows from the fact that the action of $L[\rP_\rY\backslash\rU_N/\rP_\rY]$ on $L[\rP_\rY\backslash\rU_N]$ (via left convolution) commutes with the right translation.

For (2), it is clear that $\Omega_{N,L}^0$ is a free $L$-module generated by the characteristic function of $\rU_N$.

For (3), we may assume $N>0$. Let $\Omega_N$ be the Tate--Thompson representation. Then $\Omega_N^{\rP_\rY}$, which is a $\dC[\rP_\rY\backslash\rU_N/\rP_\rY]$-module via right convolution, has eigenvalue $\tl_{N,q}^1$ for the action of $\tQ_\rY$ by \cite{LTXZZ}*{Proposition~C.2.1} and \cite{Liu4}*{Lemma~2.2~\&~Remark~2.4}. Since $\Omega_N$ is (isomorphic to) a direct summand of $\dC[\rP_\rY\backslash\rU_N]$ as a $\dC[\rU_N]$-module and the fact that the left and right actions of $\dC[\rP_\rY\backslash\rU_N/\rP_\rY]$ on $\dC[\rP_\rY\backslash\rU_N]^{\rP_\rY}$ coincide, we have $\Omega_N\simeq\Omega_{N,\dC}^1$. Part (3) follows.
\end{proof}

Now we suppose that $N>0$. Fix a $\kappa$-line $\rL$ contained in $\rY$ and let $\rD_\rL\subseteq\rU_N$ be the subgroup stabilizing $\rL$. Put $\rX_{N-2}\coloneqq \rL^\perp/\rL$, which is a hermitian space over $\kappa$ of rank $N-2$, and put $\rU_{N-2}\coloneqq\rU(\rX_{N-2})(\kappa^+)$. Then we have a natural homomorphism $\rD_\rL\to\rU_{N-2}$, whose kernel we denote by $\rI_\rL$. For every $L[\rU_N]$-module $M$, $M^{\rI_L}$ is naturally an $L[\rU_{N-2}]$-module.

\begin{lem}\label{le:specialization}
Suppose that $\tb_{N,q}$ is invertible in $L$. Then the $L[\rU_{N-2}]$-module $(\Omega_{N,L}^1)^{\rI_\rL}$ is isomorphic to $\Omega_{N-2,L}^1$.
\end{lem}

\begin{proof}
The case for $N=2$ is trivial. We assume $N\geq 4$. By Lemma \ref{le:siegel_decomposition}, we have $\Omega_{N,L}^1\otimes_LL'=\Omega_{N,L'}^1$ for every $L$-ring $L'$. Thus, it suffices to show the lemma for every field $L$ whose characteristic does not divide $\tb_{N,q}$.

Choose a $\kappa$-line $\rL'\subseteq\rX$ that is not contained in the orthogonal complement of $\rL$. Then we have an orthogonal decomposition $\rX_N=(\rL\oplus\rL')\oplus\rX_{N-2}$, from which we regard $\rU_{N-2}$ as a subgroup of $\rU_N$ so that $\rD_\rL=\rI_\rL\rU_{N-2}$ and $\rD_{\rL'}=\rI_{\rL'}\rU_{N-2}$. Choose an element $w\in\rU_N$ that switches $\rL$ and $\rL'$ and induces the identity map on $\rX_{N-2}$. Let $\rZ$ be the orthogonal complement of $\rL'$ in $\rY$, which is a maximal isotropic subspace of $\rX_{N-2}$. We have isomorphisms
\[
\rho\colon L[\rP_\rY\backslash\rU_N]^{\rI_\rL}=L[\rP_\rY\backslash\rP_\rY\rD_\rL]^{\rI_\rL}\oplus L[\rP_\rY\backslash\rP_\rY w\rD_\rL]^{\rI_\rL}
\xrightarrow\sim L[\rP_\rY\backslash\rP_\rY\rD_\rL]^{\rI_\rL}\oplus L[\rP_\rY\backslash\rP_\rY\rD_{\rL'}]^{\rI_{\rL'}}
\xrightarrow\sim L[\rP_\rZ\backslash\rU_{N-2}]^{\oplus 2}
\]
of $L[\rU_{N-2}]$-modules. For every subgroup $\rH$ contained in $\rU_{N-2}$, $\rho$ induces an isomorphism
\[
L[\rP_\rY\backslash\rU_N]^{\rI_\rL\rH}\xrightarrow\sim\(L[\rP_\rZ\backslash\rU_{N-2}]^\rH\)^{\oplus 2}.
\]
In particular, $\dim_L(\rho\Omega_{N,L}^1)^{\rP_\rZ}=\dim_L(\Omega_{N,L}^1)^{\rP_\rY}$, which equals $1$ by \cite{LTXZZ}*{Proposition~C.2.1(2)}.\footnote{Though \cite{LTXZZ}*{Proposition~C.2.1} only stated for $L=\dC$, its proof works for every field $L$ whose characteristic does not divide $|\rU_N|$, that is, does not divide $\tb_{N,q}$.} In particular, $\rho\Omega_{N,L}^1$ is isomorphic to an irreducible summand of $L[\rP_\rZ\backslash\rU_{N-2}]$. Now taking $\rH$ to be a Borel subgroup of $\rU_{N-2}$ contained in $\rP_\rZ$, we have $\dim_L(\rho\Omega_{N,L}^1)^\rH=\dim_L(\Omega_{N,L}^1)^{\rI_\rL\rH}$, which equals $1$ by \cite{LTXZZ}*{Proposition~C.2.1(2)}. Thus, by \cite{LTXZZ}*{Proposition~C.2.1(2)}, $\rho\Omega_{N,L}^1$ is isomorphic to either $\Omega_{N-2,L}^1$ or $\Omega_{N-2,L}^0$. If we are in the latter case, then $(\Omega_{N,L}^1)^{\rD_\rL}=(\Omega_{N,L}^1)^{\rI_\rL\rU_{N-2}}=(\rho\Omega_{N,L}^1)^{\rU_{N-2}}\neq 0$. Since $\rD_\rL\cap\rP_\rY$ contains a Borel subgroup of $\rU_N$, we have $(\Omega_{N,L}^1)^{\rP_\rY\rD_\rL}\neq 0$ again by \cite{LTXZZ}*{Proposition~C.2.1(2)}, which is impossible as $\rP_\rY\rD_\rL=\rU_N$. The lemma is proved.
\end{proof}

\begin{definition}\label{de:specialization}
A \emph{specializing triple} for $\rU_N$ is a triple $(\rU,\rD,\rI)$ of (abstract) subgroups of $\rU_N$ satisfying
\begin{enumerate}
  \item $\rD\subseteq\rD_\rL\cap\rU$;

  \item $\rI\subseteq\rI_\rL\cap\rU$;

  \item the natural map $\rP\backslash\rU\to\rP_\rY\backslash\rU_N$ is a bijection, where $\rP\coloneqq\rP_\rY\cap\rU$;

  \item the natural map $\rP\backslash\rU/\rD\to\rP_\rY\backslash\rU_N/\rD_\rL$ is a bijection;

  \item for every $u\in\rU$, if we denote by $\rP_\rL^u$ (resp.\ $\rP^u$) the image of $\rP_\rY\cap u\rD_\rL u^{-1}$ (resp.\ $\rP\cap u\rD u^{-1}$) in $u\rD_\rL u^{-1}/u\rI_\rL u^{-1}$ (resp.\ $u\rD u^{-1}/u\rI u^{-1}$), then the natural map $\rP^u\backslash u\rD u^{-1}/u\rI u^{-1}\to\rP_\rL^u\backslash u\rD_\rL u^{-1}/u\rI_\rL u^{-1}$ is a bijection.
\end{enumerate}
\end{definition}

\begin{lem}\label{le:specialization_1}
If $(\rU,\rD,\rI)$ is a specializing triple for $\rU_N$, then the natural map
\[
L[\rP_\rY\backslash\rU_N]^{\rI_\rL}\to L[\rP_\rY\backslash\rU_N]^{\rI}
\]
is an isomorphism.
\end{lem}

\begin{proof}
Since $\rP_\rY\backslash\rU_N/\rD_\rL$ is a set of two elements, by Definition \ref{de:specialization}(4), we may choose an element $u\in\rU$ such that $\rU=\rP\rD\coprod\rP u\rD$ and $\rU_N=\rP_\rY\rD_\rL\coprod\rP_\rY u\rD_\rL$. We have a decomposition
\[
L[\rP_\rY\backslash\rU_N]^{\rI_\rL}=L[\rP_\rY\backslash\rP_\rY\rD_\rL]^{\rI_\rL}\oplus L[\rP_\rY\backslash\rP_\rY u\rD_\rL]^{\rI_\rL}.
\]
For the two factors, we have canonical isomorphisms
\[
L[\rP_\rY\backslash\rP_\rY\rD_\rL]^{\rI_\rL}\simeq L[\rP_\rL^1\backslash\rD_\rL/\rI_\rL],\quad
L[\rP_\rY\backslash\rP_\rY u\rD_\rL]^{\rI_\rL}\simeq L[\rP_\rY\backslash\rP_\rY u\rD_\rL u^{-1}]^{u\rI_\rL u^{-1}}
\simeq L[\rP_\rL^u\backslash u\rD_\rL u^{-1}/u\rI_\rL u^{-1}]
\]
of $L$-modules given by restrictions. We have the similar properties for $L[\rP\backslash\rU]^{\rI}$. They fit into the following commutative diagram
\[
\xymatrix{
L[\rP_\rY\backslash\rU_N]^{\rI_\rL} \ar[d]_-\simeq \ar[rr] && L[\rP\backslash\rU]^{\rI} \ar[d]^-\simeq \\
L[\rP_\rL^1\backslash\rD_\rL/\rI_\rL]\oplus L[\rP_\rL^u\backslash u\rD_\rL u^{-1}/u\rI_\rL u^{-1}] \ar[rr] &&
L[\rP^1\backslash\rD/\rI]\oplus L[\rP^u\backslash u\rD u^{-1}/u\rI u^{-1}]
}
\]
induced by restrictions. By Definition \ref{de:specialization}(5), the lower horizonal map is an isomorphism, which implies that the restriction map $L[\rP_\rY\backslash\rU_N]^{\rI_\rL}\to L[\rP\backslash\rU]^{\rI}$ is an isomorphism. However, such restriction map factors through $L[\rP_\rY\backslash\rU_N]^{\rI}\to L[\rP\backslash\rU]^{\rI}$, which is an isomorphism by Definition \ref{de:specialization}(3). The lemma follows.
\end{proof}

\section{Explicit reciprocity laws for Rankin--Selberg motives}
\label{ss:4}

In this section, we upgrade the two explicit reciprocity laws from \cite{LTXZZ}*{\S7} to the level of Iwasawa algebra.

\subsection{Setup for automorphic Rankin--Selberg motives}
\label{ss:setup}

Let $n\geq 1$ be an integer. We denote by $n_0$ and $n_1$ the unique even and odd numbers in $\{n,n+1\}$, respectively. Write $n_0=2r_0$ and $n_1=2r_1+1$. In particular, we have $n=r_0+r_1$.

In this and the next sections, we consider
\begin{itemize}[label={\ding{118}}]
  \item for $\alpha=0,1$, a relevant representation $\Pi_\alpha$ of $\GL_{n_\alpha}(\dA_F)$ (Definition \ref{de:relevant}),

  \item a strong coefficient field $E\subseteq\dC$ of both $\Pi_0$ and $\Pi_1$ (Remark \ref{re:galois}(3)).
\end{itemize}

Put $\Sigma^+_\mnm\coloneqq\Sigma^+_{\Pi_0}\cup\Sigma^+_{\Pi_1}$ (Notation \ref{no:satake}) and denote by $\Sigma_\mnm$ the set of places of $F$ above $\Sigma^+_\mnm$. We then have the Satake homomorphism
\[
\phi_{\Pi_\alpha}\colon\dT^{\Sigma^+_\mnm}_{n_\alpha}\to O_E
\]
for $\alpha=0,1$.

For every prime $\lambda$ of $E$, we have a continuous homomorphism
\[
\rho_{\Pi_\alpha,\lambda}\colon\Gamma_F\to\GL_{n_\alpha}(E_\lambda)
\]
associated with $\Pi_\alpha$ (Remark \ref{re:galois}(3)) for $\alpha=0,1$, satisfying that $\rho_{\Pi_\alpha,\lambda}^\tc$ and $\rho_{\Pi_\alpha,\lambda}^\vee(1-n_\alpha)$ are conjugate. We denote by $O_\lambda$ the ring of integers of $E_\lambda$. The following three assumptions are \cite{LTXZZ}*{Assumptions~7.1.1,~7.2.2~7.2.4}, respectively.

\begin{assumption}\label{as:first_irreducible}
For $\alpha=0,1$, the Galois representation $\rho_{\Pi_\alpha,\lambda}$ is residually absolutely irreducible.
\end{assumption}

\begin{assumption}\label{as:first_minimal}
Under Assumption \ref{as:first_irreducible}, $\bar\rho_{\Pi_0,\lambda,+}$ \cite{LTXZZ}*{Remark~6.1.7} is rigid for $(\Sigma^+_\mnm,\Sigma^+_{\lr,\rI})$ \cite{LTXZZ}*{Definition~6.3.3}; and $\bar\rho_{\Pi_0,\lambda}\res_{\Gal(\ol{F}/F(\zeta_\ell))}$ is absolutely irreducible.
\end{assumption}

\begin{assumption}\label{as:first_generic}
For $\alpha=0,1$, the composite homomorphism $\dT^{\Sigma^+_\mnm}_{n_\alpha}\xrightarrow{\phi_{\Pi_\alpha}}O_E\to O_E/\lambda$ is cohomologically generic \cite{LTXZZ}*{Definition~D.1.1}.
\end{assumption}

By the global class field theory, we have a canonical isomorphism $\Gal(F^\ab/F)\simeq F^\times\backslash(\dA_F^\infty)^\times$. Put
\begin{align}\label{eq:torus}
\bG_{F/F^+}\coloneqq\Ker\(\Nm_{F/F^+}\colon\Res_{O_F/O_{F^+}}\bG_{m,O_F}\to\bG_{m,O_{F^+}}\)
\end{align}
as a commutative group scheme over $O_{F^+}$. Denote by
\[
\Nm_{F/F^+}^-\colon \Res_{O_F/O_{F^+}}\bG_{m,O_F}\to\bG_{F/F^+}
\]
the homomorphism sending $a$ to $a/a^\tc$. Denote by $F^{\r{acyc}}$ the fixed field of the kernel of the composite map
\[
\Gal(F^\ab/F)=\widehat{F^\times\backslash(\dA_F^\infty)^\times}\xrightarrow{\Nm_{F/F^+}^-}\bG_{F/F^+}(F^+)\backslash\bG_{F/F^+}(\dA_{F^+}^\infty)
\to\bG_{F/F^+}(F^+)\backslash\bG_{F/F^+}(\dA_{F^+}^\infty)/\bG_{F/F^+}(O_{F^+_v}),
\]
which is then the maximal anticyclotomic extension of $F$.

\begin{lem}\label{le:split}
Let $v$ be a prime of $F^+$ inert in $F$. Then the unique place of $F$ above $v$ splits completely in $F^{\r{acyc}}$.
\end{lem}

\begin{proof}
The proof is borrowed from the one of \cite{BD05}*{Lemma~2.5}. It is clear that $v$ is unramified in $F^{\r{acyc}}$. Since $v$ is inert in $F$ and $F^{\r{acyc}}/F^+$ is an extension of dihedral type, the Frobenius element of $v$ has order two in $\Gal(F^{\r{acyc}}/F^+)$. The lemma follows.
\end{proof}

\subsection{First explicit reciprocity law}
\label{ss:first_reciprocity}

In this subsection, we refine the first explicit reciprocity law in \cite{LTXZZ}*{\S7.2} so that it can be applied to study the Iwasawa theory.

We start by choosing
\begin{itemize}[label={\ding{118}}]
  \item a prime $\lambda$ of $E$ whose underlying rational prime $\ell$ satisfies $\Sigma^+_\mnm\cap\Sigma^+_\ell=\emptyset$, $\ell>2(n_0+1)$, and that $\ell$ is unramified in $F$,

  \item a positive integer $m$,

  \item a (possibly empty) finite set $\Sigma^+_{\lr,\rI}$ of nonarchimedean places of $F^+$ that are inert in $F$, whose underlying rational primes are disjoint from those of $\Sigma^+_\mnm$, satisfying $\ell\nmid\|v\|(\|v\|^2-1)$ for $v\in\Sigma^+_{\lr,\rI}$,

  \item a finite set $\Sigma^+_\rI$ of nonarchimedean places of $F^+$ containing $\Sigma^+_\mnm\cup\Sigma^+_{\lr,\rI}$,

  \item a datum $\cV^\circ=(\rV^\circ_n,\rV^\circ_{n+1};\Lambda^\circ_n,\Lambda^\circ_{n+1};\rK^\circ_n,\rK^\circ_{n+1})$ in which
    \begin{itemize}
      \item $\rV^\circ_n$ is a standard \emph{definite} hermitian space over $F$ (Definition \ref{de:standard_hermitian_space}) of rank $n$ and $\rV^\circ_{n+1}=(\rV^\circ_n)_\sharp$;

      \item $\Lambda^\circ_n$ is a $\prod_{v\not\in\Sigma^+_\infty\cup\Sigma^+_\mnm}O_{F_v}$-lattice in $\rV^\circ_n\otimes_F\dA_F^{\Sigma_\infty\cup\Sigma_\mnm}$ satisfying $\Lambda^\circ_n\subseteq(\Lambda^\circ_n)^\vee$ and that $(\Lambda^\circ_{n,v})^\vee/\Lambda^\circ_{n,v}$ has length one (resp.\ zero) when $v\in\Sigma^+_{\lr,\rI}$ (resp.\ $v\not\in\Sigma^+_{\lr,\rI}$) and $\Lambda^\circ_{n+1}=(\Lambda^\circ_n)_\sharp$;

      \item $(\rK^\circ_n,\rK^\circ_{n+1})$ is an object in $\fK(\rV^\circ_n)_\sp$ of the form
         \begin{align*}
         \rK^\circ_N=\prod_{v\in\Sigma^+_\mnm}(\rK^\circ_N)_v\times
         \prod_{v\not\in\Sigma^+_\infty\cup\Sigma^+_\mnm}\rU(\Lambda^\circ_N)(O_{F^+_v})
         \end{align*}
         for $N\in\{n,n+1\}$,
    \end{itemize}

  \item a special inert prime (Definition \ref{de:special_inert}) $\fp$ of $F^+$ (with the underlying rational prime $p$) satisfying\footnote{Note that comparing to \cite{LTXZZ}*{\S7.2}, we have incorporated (PI7) into (PI4).}
      \begin{description}
        \item[(PI1)] $\Sigma^+_\rI$ does not contain $p$-adic places;

        \item[(PI2)] $\ell$ does not divide $p(p^2-1)$;

        \item[(PI3)] there exists a CM type $\Phi$ containing $\tau_\infty$ as in \cite{LTXZZ}*{\S5.1} satisfying $\dQ_p^\Phi=\dQ_{p^2}$ (and we choose remaining data in \cite{LTXZZ}*{\S5.1} with $\dQ_p^\Phi=\dQ_{p^2}$);

        \item[(PI4)] $P_{\balpha(\Pi_{0,\fp})}\modulo\lambda^m$ is level-raising special at $\fp$; $P_{\balpha(\Pi_{1,\fp})}\modulo\lambda$ is Tate generic at $\fp$; $P_{\balpha(\Pi_{0,\fp})\otimes\balpha(\Pi_{1,\fp})}\modulo\lambda^m$ is level-raising special at $\fp$ (all in Definition \ref{de:satake_condition});

        \item[(PI5)] $P_{\balpha(\Pi_{\alpha,\fp})}\modulo\lambda$ is intertwining generic at $\fp$ (Definition \ref{de:satake_condition}) for $\alpha=0,1$;

        \item[(PI6)] the natural map
            \[
            \frac{(O_E/\lambda^m)[\Sh(\rV^\circ_{n_0},\rK^\circ_{n_0})]}
            {\dT^{\Sigma^+_\rI\cup\Sigma^+_p}_{n_0}\cap\Ker\phi_{\Pi_0}}
            \otimes_{O_\lambda}\frac{(O_E/\lambda^m)[\Sh(\rV^\circ_{n_1},\rK^\circ_{n_1})]}
            {\dT^{\Sigma^+_\rI\cup\Sigma^+_p}_{n_1}\cap\Ker\phi_{\Pi_1}}
            \to\frac{(O_E/\lambda^m)[\Sh(\rV^\circ_{n_0},\rK^\circ_{n_0})]}
            {\dT^{\Sigma^+_\rI}_{n_0}\cap\Ker\phi_{\Pi_0}}
            \otimes_{O_\lambda}\frac{(O_E/\lambda^m)[\Sh(\rV^\circ_{n_1},\rK^\circ_{n_1})]}
            {\dT^{\Sigma^+_\rI}_{n_1}\cap\Ker\phi_{\Pi_1}}
            \]
            is an isomorphism of \emph{nontrivial} $O_\lambda$-modules.
      \end{description}

  \item a datum $\cV'=(\rV'_n,\rV'_{n+1};\Lambda'_n,\Lambda'_{n+1};\rK'_n,\rK'_{n+1})$ in which
    \begin{itemize}
      \item $\rV'_n$ is a standard \emph{indefinite} hermitian space over $F$ of rank $n$ and $\rV'_{n+1}=(\rV'_n)_\sharp$;

      \item $\Lambda'_n$ is a $\prod_{v\not\in\Sigma^+_\infty\cup\Sigma^+_\mnm}O_{F_v}$-lattice in $\rV'_n\otimes_F\dA_F^{\Sigma_\infty\cup\Sigma_\mnm}$ satisfying $\Lambda'_n\subseteq(\Lambda'_n)^\vee$ and that $(\Lambda'_{n,v})^\vee/\Lambda'_{n,v}$ has length one (resp.\ zero) when $v\in\Sigma^+_{\lr,\rI}\cup\{\fp\}$ (resp.\ $v\not\in\Sigma^+_{\lr,\rI}\cup\{\fp\}$) and $\Lambda'_{n+1}=(\Lambda'_n)_\sharp$;

      \item $(\rK'_n,\rK'_{n+1})$ is an object in $\fK(\rV'_n)_\sp$ of the form
         \begin{align*}
         \rK'_N=\prod_{v\in\Sigma^+_\mnm}(\rK'_N)_v\times\prod_{v\not\in\Sigma^+_\infty\cup\Sigma^+_\mnm}\rU(\Lambda'_N)(O_{F^+_v})
         \end{align*}
         for $N\in\{n,n+1\}$,
    \end{itemize}
    together with a datum $\tj=(\tj_n,\tj_{n+1})$ in which $\tj_N\colon\rV^\circ_N\otimes_\dQ\dA^{\infty,p}\to\rV'_N\otimes_\dQ\dA^{\infty,p}$ are isometries sending $(\Lambda_N)^p$ to $(\Lambda'_N)^p$ and $(\rK_N)^p$ to $(\rK'_N)^p$ for $N\in\{n,n+1\}$, satisfying $\tj_{n+1}=(\tj_n)_\sharp$.\footnote{Note that the pair $(\cV',\tj)$ is nothing but an indefinite uniformization datum for $\cV^\circ$ in \cite{LTXZZ}*{Notation~5.10.1}.}
\end{itemize}

For $\alpha=0,1$, define the ideal
\[
\fn_\alpha\coloneqq
\dT^{\Sigma^+_\rI\cup\Sigma^+_p}_{n_\alpha}\cap\Ker\(\dT^{\Sigma^+_\mnm}_{n_\alpha}\xrightarrow{\phi_{\Pi_\alpha}}O_E\to O_E/\lambda^m\)
\]
of $\dT^{\Sigma^+_\rI\cup\Sigma^+_p}_{n_\alpha}$.

\begin{definition}[First reciprocity map]\label{de:first_reciprocity}
Assume Assumptions \ref{as:first_irreducible}, \ref{as:first_minimal}, \ref{as:first_generic}, and Hypothesis \ref{hy:unitary_cohomology} for both $n$ and $n+1$. We define a map
\begin{align*}
\varrho_\sing&\colon
\rH^1_\sing(F_\fp,
\rH^{2n-1}_\et((\Sh(\rV'_{n_0},\rK'_{n_0})\times_F\Sh(\rV'_{n_1},\rK'_{n_1}))_{\ol{F}},O_\lambda(n))/(\fn_0,\fn_1)) \\
&\to
O_\lambda[\Sh(\rV^\circ_{n_0},\rK^\circ_{n_0})\times\Sh(\rV^\circ_{n_1},\rK^\circ_{n_1})]/(\fn_0,\fn_1)
\end{align*}
of $O_\wp$-modules, called \emph{first reciprocity map}, which turns out to be an isomorphism, to be the composition of the following four maps (we will freely use the notation from \cite{LTXZZ}*{\S7.2}.\footnote{Careful readers may notice that in \cite{LTXZZ}*{\S7.2}, it is required that $(\rK^\circ_{n_0})_v$ is a \emph{transferable} open compact subgroup of $\rU(\rV^\circ_{n_0})(F^+_v)$ for $v\in\Sigma^+_\mnm$. However, as explained in \cite{LTXZZ1}*{Remark~8.2}, this requirement is unnecessary.}):
\begin{itemize}[label={\ding{118}}]
  \item the canonical isomorphism
      \begin{align*}
      \rH^1_\sing(F_\fp,\rH^{2n-1}_\et((\Sh(\rV'_{n_0},\rK'_{n_0})\times_F\Sh(\rV'_{n_1},\rK'_{n_1}))_{\ol{F}},O_\lambda(n))/(\fn_0,\fn_1))
      \xrightarrow\sim
      \rH^1_\sing(\dQ_{p^2},\rH^{2n-1}_\fT(\ol\rQ,\rR\Psi O_\lambda(n))/(\fn_0,\fn_1))
      \end{align*}
      induced by \cite{LTXZZ}*{(5.2)},

  \item the canonical $O_\lambda$-linear isomorphism
      \begin{align*}
      \rH^1_\sing(\dQ_{p^2},\rH^{2n-1}_\fT(\ol\rQ,\rR\Psi O_\lambda(n))/(\fn_0,\fn_1))
      \xrightarrow\sim\coker\Delta^n/(\fn_0,\fn_1).
      \end{align*}
      by \cite{LTXZZ}*{Proposition~7.2.7},

  \item the canonical $O_\lambda$-linear isomorphism
      \begin{align*}
      \nabla_{/(\fn_0,\fn_1)}\colon \coker\Delta^n/(\fn_0,\fn_1)\xrightarrow{\sim}
      O_\lambda[\Sh(\rV^\circ_{n_0},\rK^\circ_{n_0})\times\Sh(\rV^\circ_{n_1},\rK^\circ_{n_1})]/(\fn_0,\fn_1).
      \end{align*}
      by \cite{LTXZZ}*{Theorem~7.2.8(2)}, and

  \item the \emph{inverse} of the automorphism
      \[
      (p+1)\cdot\tI_{n_0,\fp}\otimes\tT_{n_1,\fp}\colon O_\lambda[\Sh(\rV^\circ_{n_0},\rK^\circ_{n_0})\times\Sh(\rV^\circ_{n_1},\rK^\circ_{n_1})]/(\fn_0,\fn_1)\xrightarrow\sim
      O_\lambda[\Sh(\rV^\circ_{n_0},\rK^\circ_{n_0})\times\Sh(\rV^\circ_{n_1},\rK^\circ_{n_1})]/(\fn_0,\fn_1),
      \]
      in which $(p+1)$ is invertible by (PI2); $\tI_{n_0,\fp}\otimes\tT_{n_1,\fp}$ is invertible by (PI4--6) and \cite{LTXZZ}*{Propositions B.3.5(1)~\&~B.4.3(2)}.
\end{itemize}
\end{definition}

In order to apply the above construction in the context of the Iwasawa theory, we need to consider slightly more general open compact groups away from $\Sigma^+_p$. Consider the following data:
\begin{itemize}[label={\ding{118}}]
  \item another object $(\tilde\rK^\circ_n,\tilde\rK^\circ_{n+1})\in\fK(\rV^\circ_n)_\sp$ of the form
     \begin{align*}
     \tilde\rK^\circ_N=\prod_{v\in\Sigma^+_p}\rK^\circ_{N,v}\times\prod_{v\not\in\Sigma^+_\infty\cup\Sigma^+_p}\tilde\rK^\circ_{N,v}
     \end{align*}
     (and put $\tilde\rK'_N\coloneqq\tj_N(\tilde\rK^\circ_N)^p\times(\rK'_N)_p$) for $N\in\{n,n+1\}$,

  \item a Hecke operator
     \[
     \tP\in O_\lambda[(\tilde\rK^\circ_n)^p\times(\tilde\rK^\circ_{n+1})^p
     \backslash\rU(\rV^\circ_n)(\dA_F^{\infty,p})\times\rU(\rV^\circ_{n+1})(\dA_F^{\infty,p})/(\rK^\circ_n)^p\times(\rK^\circ_{n+1})^p],
     \]

  \item and a (finite anticyclotomic) extension $\tilde{F}/F$ contained in $F^\ab$ that is fixed by the kernel of the composite map
     \begin{align*}
     \Gal(F^\ab/F)=F^\times\backslash(\dA_F^\infty)^\times&\xrightarrow{\Nm_{F/F^+}^-}
     \bG_{F/F^+}(F^+)\backslash\bG_{F/F^+}(\dA_{F^+}^\infty) \\
     &\to\bG_{F/F^+}(F^+)\backslash\bG_{F/F^+}(\dA_{F^+}^\infty)/\det\tilde\rK^\circ_n
     =\bG_{F/F^+}(F^+)\backslash\bG_{F/F^+}(\dA_{F^+}^\infty)/\det\tilde\rK'_n,
     \end{align*}
     so that $\Gal(\tilde{F}/F)$ is a quotient of $\bG_{F/F^+}(F^+)\backslash\bG_{F/F^+}(\dA_{F^+}^\infty)/\bG_{F/F^+}(O_{F^+}\otimes\dZ_p)$.
\end{itemize}

\begin{notation}
We introduce two notations based on the previous data.
\begin{enumerate}
  \item We have canonical maps
      \[
      \Sh(\rV_n,\tilde\rK^\circ_n)\xrightarrow{\det}\bG_{F/F^+}(F^+)\backslash\bG_{F/F^+}(\dA_{F^+}^\infty)/\det\tilde\rK^\circ_n
      \to\Gal(\tilde{F}/F)
      \]
      of sets. For every element $\varsigma\in\Gal(\tilde{F}/F)$, we denote by $\Sh(\rV^\circ_n,\tilde\rK^\circ_n)_\varsigma$ the fiber of $\varsigma$ under the above composite map, and by
      \[
      \CF_{\graph\Sh(\rV^\circ_n,\tilde\rK^\circ_n)_\varsigma}
      \in\dZ[\Sh(\rV^\circ_n,\tilde\rK^\circ_n)\times\Sh(\rV^\circ_{n+1},\tilde\rK^\circ_{n+1})]
      \]
      the characteristic function of the image of $\Sh(\rV^\circ_n,\tilde\rK^\circ_n)_\varsigma$ under the diagonal map
      \[
      \graph\colon\Sh(\rV^\circ_n,\tilde\rK^\circ_n)\to\Sh(\rV^\circ_n,\tilde\rK^\circ_n)\times\Sh(\rV^\circ_{n+1},\tilde\rK^\circ_{n+1})
      \]
      of Shimura sets.

  \item We have canonical maps
      \[
      \pi_0\(\Sh(\rV'_n,\tilde\rK'_n)_{F^\ab}\)\to\bG_{F/F^+}(F^+)\backslash\bG_{F/F^+}(\dA_{F^+}^\infty)/\det\tilde\rK'_n
      \to\Gal(\tilde{F}/F)
      \]
      of sets. For every element $\varsigma\in\Gal(\tilde{F}/F)$, we denote by $\Sh(\rV'_n,\tilde\rK'_n)_\varsigma$ the union of components parameterized by the fiber of $\varsigma$ under the above composite map, which indeed is defined over $\tilde{F}$, and by
      \[
      [\graph\Sh(\rV'_n,\tilde\rK'_n)_\varsigma]\in
      \rZ^n((\Sh(\rV'_n,\tilde\rK'_n)\times_F\Sh(\rV'_{n+1},\tilde\rK'_{n+1}))_{\tilde{F}})
      \]
      the cycle given by the image of $\Sh(\rV'_n,\tilde\rK'_n)_\varsigma$ under the diagonal morphism
      \[
      \graph\colon\Sh(\rV'_n,\tilde\rK'_n)\to\Sh(\rV'_n,\tilde\rK'_n)\times_F\Sh(\rV'_{n+1},\tilde\rK'_{n+1})
      \]
      of Shimura varieties.
\end{enumerate}
\end{notation}

The Hecke operator $\tP$ induces maps
\begin{align*}
\tP_*&\colon O_\lambda[\Sh(\rV^\circ_n,\tilde\rK^\circ_n)\times\Sh(\rV^\circ_{n+1},\tilde\rK^\circ_{n+1})] \to
O_\lambda[\Sh(\rV^\circ_n,\rK^\circ_n)\times\Sh(\rV^\circ_{n+1},\rK^\circ_{n+1})], \\
\tP_*&\colon\rZ^n((\Sh(\rV'_n,\tilde\rK'_n)\times_F\Sh(\rV'_{n+1},\tilde\rK'_{n+1}))_{\tilde{F}})_{O_\lambda}
\to \rZ^n((\Sh(\rV'_n,\rK'_n)\times_F\Sh(\rV'_{n+1},\rK'_{n+1}))_{\tilde{F}})_{O_\lambda}.
\end{align*}

Under Assumption \ref{as:first_generic}, we have the Abel--Jacobi map
\[
\AJ\colon\rZ^n((\Sh(\rV'_{n_0},\rK'_{n_0})\times_F\Sh(\rV'_{n_1},\rK'_{n_1}))_{\tilde{F}})\to
\rH^1(\tilde{F},\rH^{2n-1}_\et((\Sh(\rV'_{n_0},\rK'_{n_0})\times_F\Sh(\rV'_{n_1},\rK'_{n_1}))_{\ol{F}},O_\lambda(n))/(\fn_0,\fn_1)).
\]
The fixed algebraic closure $\ol{F}^+_\fp$ of $F^+_\fp$ containing $\ol{F}$ from \S\ref{ss:notation} induces a place $\fp'$ of $\tilde{F}$ above $\fp$. Combining with Lemma \ref{le:inert}, we have the localization map
\begin{align*}
\loc_\fp&\colon
\rH^1(\tilde{F},\rH^{2n-1}_\et((\Sh(\rV'_{n_0},\rK'_{n_0})\times_F\Sh(\rV'_{n_1},\rK'_{n_1}))_{\ol{F}},O_\lambda(n))/(\fn_0,\fn_1)) \\
&\to\rH^1(\tilde{F}_{\fp'},
\rH^{2n-1}_\et((\Sh(\rV'_{n_0},\rK'_{n_0})\times_F\Sh(\rV'_{n_1},\rK'_{n_1}))_{\ol{F}},O_\lambda(n))/(\fn_0,\fn_1)) \\
&=\rH^1(F_\fp,
\rH^{2n-1}_\et((\Sh(\rV'_{n_0},\rK'_{n_0})\times_F\Sh(\rV'_{n_1},\rK'_{n_1}))_{\ol{F}},O_\lambda(n))/(\fn_0,\fn_1))
\end{align*}
and the singular quotient map
\begin{align*}
\partial_\fp&\colon
\rH^1(F_\fp,\rH^{2n-1}_\et((\Sh(\rV'_{n_0},\rK'_{n_0})\times_F\Sh(\rV'_{n_1},\rK'_{n_1}))_{\ol{F}},O_\lambda(n))/(\fn_0,\fn_1)) \\
&\to\rH^1_\sing(F_\fp,
\rH^{2n-1}_\et((\Sh(\rV'_{n_0},\rK'_{n_0})\times_F\Sh(\rV'_{n_1},\rK'_{n_1}))_{\ol{F}},O_\lambda(n))/(\fn_0,\fn_1)).
\end{align*}

The following result slightly generalizes \cite{LTXZZ}*{Theorem~7.2.8}.

\begin{theorem}[First explicit reciprocity law]\label{th:first}
Assume Assumptions \ref{as:first_irreducible}, \ref{as:first_minimal}, \ref{as:first_generic}, and Hypothesis \ref{hy:unitary_cohomology} for both $n$ and $n+1$. Then for every collection of data $(\tilde\rK^\circ_n,\tilde\rK^\circ_{n+1}),\tP,\tilde{F}$ as above,
\begin{align*}
\varrho_\sing\(\partial_\fp\loc_\fp\AJ\(\tP_*[\graph\Sh(\rV'_n,\tilde\rK'_n)]_\varsigma\)\)=
\tP_*\CF_{\graph\Sh(\rV^\circ_n,\tilde\rK^\circ_n)_{[\tj_n]\cdot\varsigma}}
\end{align*}
holds for every $\varsigma\in\Gal(\tilde{F}/F)$. Here, $[\tj_n]\in\bG_{F/F^+}(F^+)\backslash\bG_{F/F^+}(\dA_{F^+}^\infty)/\bG_{F/F^+}(O_{F^+}\otimes\dZ_p)$ is the element from Lemma \ref{le:determinant_first} applied to the isometry $\tj_n\colon\rV^\circ_n\otimes_\dQ\dA^{\infty,p}\xrightarrow\sim\rV'_n\otimes_\dQ\dA^{\infty,p}$.
\end{theorem}

\begin{proof}
In the proof, we will freely use the notation from \cite{LTXZZ}*{\S7.2}. By enlarging $\Sigma^+_\rI$, we may assume, without loss of generality, that $(\tilde\rK^\circ_N)_v=(\rK^\circ_N)_v$ for $v\not\in\Sigma^+_\infty\cup\Sigma^+_\rI$ and that $\tP$ is supported at places in $\Sigma^+_\rI$. Consider an arbitrary pair $(\hat\rK^\circ_n,\hat\rK^\circ_{n+1})\in\fK(\rV^\circ_n)\times\fK(\rV^\circ_{n+1})$ of the form
\begin{align*}
\hat\rK^\circ_N=\prod_{v\not\in\Sigma^+_\infty\cup\Sigma^+_\rI}\rK^\circ_{N,v}\times\prod_{v\in\Sigma^+_\rI}\hat\rK^\circ_{N,v}
\end{align*}
(and put $\hat\rK'_N\coloneqq\tj_N(\hat\rK^\circ_N)^p\times(\rK'_N)_p$) for $N\in\{n,n+1\}$.

Denote by $\hat\bQ$ the scheme parallel to $\bQ$ but with respect to the object $(\hat\rK^\circ_n,\hat\rK^\circ_{n+1})$. We have maps
\begin{align*}
\hat\Delta^n&\colon C_n(\hat\rQ,O_\lambda)\to C^n(\hat\rQ,O_\lambda), \\
\hat\nabla&\colon C^n(\hat\rQ,O_\lambda)\to O_\lambda[\Sh(\rV^\circ_{n_0},\hat\rK^\circ_{n_0})\times\Sh(\rV^\circ_{n_1},\hat\rK^\circ_{n_1})]
\end{align*}
from \cite{LTXZZ}*{\S5.11}. By (the proof of) \cite{LTXZZ}*{Theorem~7.2.8(1)}, $\hat\nabla$ induces a map
\[
\hat\nabla_{/(\fn_0,\fn_1)}\colon\coker\hat\Delta^n_{/(\fn_0,\fn_1)}
\to O_\lambda[\Sh(\rV^\circ_{n_0},\hat\rK^\circ_{n_0})\times\Sh(\rV^\circ_{n_1},\hat\rK^\circ_{n_1})]/
(\fn_0,\fn_1,(p+1)\tR^\circ_{n_0,\fp}-\tI^\circ_{n_0,\fp}),
\]
which is compatible when we change $(\hat\rK^\circ_n,\hat\rK^\circ_{n+1})$ (in the above form). As a consequence, we have the commutative diagram
\begin{align}\label{eq:first6}
\xymatrix{
\coker\tilde\Delta^n_{/(\fn_0,\fn_1)} \ar[rr]^-{\tilde\nabla_{/(\fn_0,\fn_1)}}\ar[d]_-{\tP_*} && O_\lambda[\Sh(\rV^\circ_{n_0},\tilde\rK^\circ_{n_0})\times\Sh(\rV^\circ_{n_1},\tilde\rK^\circ_{n_1})]/
(\fn_0,\fn_1,(p+1)\tR^\circ_{n_0,\fp}-\tI^\circ_{n_0,\fp}) \ar[d]^-{\tP_*} \\
\coker\Delta^n_{/(\fn_0,\fn_1)} \ar[rr]^-{\nabla_{/(\fn_0,\fn_1)}}&&
O_\lambda[\Sh(\rV^\circ_{n_0},\rK^\circ_{n_0})\times\Sh(\rV^\circ_{n_1},\rK^\circ_{n_1})]/
(\fn_0,\fn_1,(p+1)\tR^\circ_{n_0,\fp}-\tI^\circ_{n_0,\fp})
}
\end{align}
in which $\tilde\Delta^n$ and $\tilde\nabla$ are the corresponding maps for $(\tilde\rK^\circ_n,\tilde\rK^\circ_{n+1})$. Note that
\[
O_\lambda[\Sh(\rV^\circ_{n_0},\rK^\circ_{n_0})\times\Sh(\rV^\circ_{n_1},\rK^\circ_{n_1})]/
(\fn_0,\fn_1,(p+1)\tR^\circ_{n_0,\fp}-\tI^\circ_{n_0,\fp})
=O_\lambda[\Sh(\rV^\circ_{n_0},\rK^\circ_{n_0})\times\Sh(\rV^\circ_{n_1},\rK^\circ_{n_1})]/(\fn_0,\fn_1)
\]
as pointed out in the proof of \cite{LTXZZ}*{Theorem~7.2.8(1)}.

We continue to put $\tilde{\phantom{a}}$ on top of notation for objects with respect to $(\tilde\rK^\circ_n,\tilde\rK^\circ_{n+1})$. For example, we have schemes $\tilde\bM_N$ over $\bT$ (over $O_{F_\fp}=\dZ_{p^2}$) for $N\in\{n,n+1\}$, and the blow-up morphism $\tilde\sigma\colon\tilde\bQ\to\tilde{\b{P}}\coloneqq\tilde\bM_n\times_{\bT}\tilde\bM_{n+1}$. The ``modular interpretation'' isomorphism \cite{LTXZZ}*{(5.2)} gives a finite \'{e}tale map $\tilde\bM_n^\eta\to\Sh(\rV'_n,\tilde\rK'_n)_{F_\fp}$. For every $\varsigma\in\Gal(\tilde{F}/F)$, denote by $(\tilde\bM_n^\eta)_\varsigma$ the preimage of $(\Sh(\rV'_n,\tilde\rK'_n)_\varsigma)_{F_\fp}$, $(\tilde\bM_n)_\varsigma$ the Zariski closure of $(\tilde\bM_n^\eta)_\varsigma$ in $\tilde\bM_n$, $(\tilde{\b{P}}_\sp)_\varsigma$ the graph of the diagonal embedding $(\tilde\bM_n)_\varsigma\to\tilde\bM_n\times_{\bT}\tilde\bM_{n+1}=\tilde{\b{P}}$, and finally $(\tilde{\b{Q}}_\sp)_\varsigma$ the strict transform of $(\tilde{\b{P}}_\sp)_\varsigma$ along $\tilde\sigma$.

By the commutative diagram \eqref{eq:first6} and \cite{LTXZZ}*{Proposition~7.2.7}, it remains to show that
\begin{align}\label{eq:first5}
\nabla_{/(\fn_0,\fn_1)}\cl\((\tilde\rQ_\sp)_\varsigma\)=(p+1)\cdot\tI_{n_0,\fp}\otimes\tT_{n_1,\fp}
\(\CF_{\graph\Sh(\rV^\circ_n,\tilde\rK^\circ_n)_{[\tj_n]\cdot\varsigma}}\)
\end{align}
holds for every $\varsigma\in\Gal(\tilde{F}/F)$. When $\tilde{F}=F$, it has been verified in (the proof of) \cite{LTXZZ}*{Theorem~7.2.8(3)}. We now explain how to modify the argument for general $\tilde{F}$.

The uniformization map in \cite{LTXZZ}*{Construction~5.3.6} gives a finite map $\pi_0(\tilde\rS^\circ_n)\to\Sh(\rV^\circ_n,\tilde\rK^\circ_n)$. For every $\varsigma\in\Gal(\tilde{F}/F)$, we denote by $(\tilde\rS^\circ_n)_\varsigma$ the union of components parameterized by the preimage of $\Sh(\rV^\circ_n,\tilde\rK^\circ_n)_\varsigma$ under the above map. Examining the proof of \cite{LTXZZ}*{Theorem~7.2.8(3)}, we find that \eqref{eq:first5} follows from the following two claims:
\begin{enumerate}
  \item The scheme $\tilde\bM_n$ is the disjoint union of $(\tilde\bM_n)_\varsigma$ for all $\varsigma\in\Gal(\tilde{F}/F)$.

  \item The basic correspondence for the balloon stratum $\tilde\rS^\circ_n\leftarrow\tilde\rB^\circ_n\to\tilde\rM^\circ_n\subseteq\tilde\rM_n$ \cite{LTXZZ}*{(5.3)} restricts to a correspondence between $(\tilde\rS^\circ_n)_{[\tj_n]\cdot\varsigma}$ and $(\tilde\rM_n)_\varsigma$ for every $\varsigma\in\Gal(\tilde{F}/F)$.
\end{enumerate}
Claim (1) follows from the fact that $\tilde\bM_n$ is regular \cite{LTXZZ}*{Theorem~5.2.5} and that $\tilde\bM_n^\eta$ is the disjoint union of $(\tilde\bM_n^\eta)_\varsigma$ for all $\varsigma\in\Gal(\tilde{F}/F)$. Claim (2) follows from Lemma \ref{le:determinant_first}.

The theorem is proved.
\end{proof}

\if false

\begin{lem}\label{le:first}
Assume Assumptions \ref{as:first_irreducible}, \ref{as:first_generic}, and Hypothesis \ref{hy:unitary_cohomology} for both $n$ and $n+1$.
\begin{enumerate}
  \item The natural map
     \[
     \frac{(O_E/\lambda)[\Sh(\rV^\circ_{n_\alpha},\rK^\circ_{n_\alpha})]}
     {\dT^{\Sigma^+_\rI\cup\Sigma^+_p}_{n_\alpha}\cap\Ker\phi_{\Pi_\alpha}}
     \to\frac{(O_E/\lambda)[\Sh(\rV^\circ_{n_\alpha},\rK^\circ_{n_\alpha})]}
     {\dT^{\Sigma^+_\rI}_{n_\alpha}\cap\Ker\phi_{\Pi_\alpha}}
     \]
     is an isomorphism for $\alpha=0,1$.

  \item If we further assume Assumption \ref{as:first_minimal} and that $O_E[\Sh(\rV^\circ_{n_0},\rK^\circ_{n_0})]/(\dT^{\Sigma^+_\rI}_{n_0}\cap\Ker\phi_{\Pi_0})$ is nontrivial, then the natural map
     \[
     \frac{(O_E/\lambda^m)[\Sh(\rV^\circ_{n_0},\rK^\circ_{n_0})]}
     {\dT^{\Sigma^+_\rI\cup\Sigma^+_p}_{n_0}\cap\Ker\phi_{\Pi_0}}
     \to\frac{(O_E/\lambda^m)[\Sh(\rV^\circ_{n_0},\rK^\circ_{n_0})]}
     {\dT^{\Sigma^+_\rI}_{n_0}\cap\Ker\phi_{\Pi_0}}
     \]
     is an isomorphism.
\end{enumerate}
\end{lem}

\begin{proof}
In (1) and (2), we will take $\alpha=0,1$ and $\alpha=0$, respectively. Denote by $\sfT$ the $O_\lambda$-subring of $\End_{O_\lambda}\(O_\lambda[\Sh(\rV^\circ_{n_\alpha},\rK^\circ_{n_\alpha})]\)$ generated by $\dT^{\Sigma^+_\rI}_{n_\alpha}$, and $\fm$ the ideal of $\sfT$ generated by
\[
\dT^{\Sigma^+_\rI}_{n_\alpha}\cap\Ker\(\dT^{\Sigma^+}_{n_\alpha}\xrightarrow{\phi_{\Pi_\alpha}}O_E\to O_E/\lambda\).
\]
We have the similar pair $(\sfT',\fm')$ when we replace $\Sigma^+_\rI$ by $\Sigma^+_\rI\cup\Sigma^+_p$. We then have the natural homomorphism $\sfT'_{\fm'}\to\sfT_\fm$ of local rings, which is injective.

For (1), it amounts to showing that the natural map $\sfT'/\fm'\to\sfT/\fm$ is an isomorphism. If $\sfT/\fm$ is zero, then so is $\sfT'_{\fm'}$. Thus, we may assume that $\sfT/\fm\neq 0$. Then $\sfT/\fm$ is a field over $O_E/\lambda$ generated by symmetric polynomials of Frobenius eigenvalues for the (absolutely irreducible) residue representation of $\rho_{\Pi_\alpha,\lambda}$ away from $\Sigma^+_\rI$. By the Chebotarev density theorem, this is the same field if we only consider Frobenius away from $\Sigma^+_\rI\cup\Sigma^+_p$. Thus, $\sfT'/\fm'\to\sfT/\fm$ is an isomorphism and (2) follows.

For (2), the nonvanishing of $O_E[\Sh(\rV^\circ_{n_0},\rK^\circ_{n_0})]/(\dT^{\Sigma^+_\rI}_{n_0}\cap\Ker\phi_{\Pi_0})$ implies that $\sfT_\fm$ hence $\sfT'_{\fm'}$ are nonzero. By \cite{LTXZZ2}*{Theorem~3.38}, the natural homomorphism $\sfT'_{\fm'}\to\sfT_\fm$ is an isomorphism. Since the natural maps
\[
\frac{\sfT}{\dT^{\Sigma^+_\rI}_{n_0}\cap\Ker\phi_{\Pi_0}}\to\frac{\sfT_\fm}{\dT^{\Sigma^+_\rI}_{n_0}\cap\Ker\phi_{\Pi_0}},\quad
\frac{\sfT'}{\dT^{\Sigma^+_\rI\cup\Sigma^+_p}_{n_0}\cap\Ker\phi_{\Pi_0}}\to
\frac{\sfT'_{\fm'}}{\dT^{\Sigma^+_\rI\cup\Sigma^+_p}_{n_0}\cap\Ker\phi_{\Pi_0}}
\]
are both isomorphism, (2) follows.
\end{proof}

\fi

\subsection{Second explicit reciprocity law}
\label{ss:second_reciprocity}

In this subsection, we refine the second explicit reciprocity law in \cite{LTXZZ}*{\S7.3} so that it can be applied to study the Iwasawa theory.

We start by choosing
\begin{itemize}[label={\ding{118}}]
  \item a prime $\lambda$ of $E$, whose underlying rational prime $\ell$ satisfies $\Sigma^+_\mnm\cap\Sigma^+_\ell=\emptyset$,

  \item a positive integer $m$,

  \item a (possibly empty) finite set $\Sigma^+_{\lr,\r{II}}$ of nonarchimedean places of $F^+$ that are inert in $F$, whose underlying rational primes are disjoint from those of $\Sigma^+_\mnm$, satisfying $\ell\nmid\|v\|(\|v\|^2-1)$ for $v\in\Sigma^+_{\lr,\r{II}}$,

  \item a finite set $\Sigma^+_{\r{II}}$ of nonarchimedean places of $F^+$ containing $\Sigma^+_\mnm\cup\Sigma^+_{\lr,\r{II}}$,

  \item a datum $\cV=(\rV_n,\rV_{n+1};\Lambda_n,\Lambda_{n+1};\rK_n,\rK_{n+1})$ in which
    \begin{itemize}
      \item $\rV_n$ is a standard \emph{indefinite} hermitian space over $F$ (Definition \ref{de:standard_hermitian_space}) of rank $n$ and $\rV_{n+1}=(\rV_n)_\sharp$;

      \item $\Lambda_n$ is a $\prod_{v\not\in\Sigma^+_\infty\cup\Sigma^+_\mnm}O_{F_v}$-lattice in $\rV_n\otimes_F\dA_F^{\Sigma_\infty\cup\Sigma_\mnm}$ satisfying $\Lambda_n\subseteq\Lambda_n^\vee$ and that $\Lambda_{n,v}^\vee/\Lambda_{n,v}$ has length one (resp.\ zero) when $v\in\Sigma^+_{\lr,\r{II}}$ (resp.\ $v\not\in\Sigma^+_{\lr,\r{II}}$) and $\Lambda_{n+1}=(\Lambda_n)_\sharp$;

      \item $(\rK_n,\rK_{n+1})$ is an object in $\fK(\rV_n)_\sp$ of the form
         \begin{align*}
         \rK_N=\prod_{v\in\Sigma^+_\mnm}(\rK_N)_v\times\prod_{v\not\in\Sigma^+_\infty\cup\Sigma^+_\mnm}\rU(\Lambda_N)(O_{F^+_v})
         \end{align*}
         for $N\in\{n,n+1\}$,
    \end{itemize}

  \item a special inert prime (Definition \ref{de:special_inert}) $\fp$ of $F^+$ (with the underlying rational prime $p$) satisfying\footnote{Note that comparing to \cite{LTXZZ}*{\S7.3}, we have incorporated (PII7) into (PII4).}
      \begin{description}
        \item[(PII1)] $\Sigma^+_{\r{II}}$ does not contain $p$-adic places;

        \item[(PII2)] $\ell$ does not divide $p(p^2-1)$;

        \item[(PII3)] there exists a CM type $\Phi$ containing $\tau_\infty$ as in \S\ref{ss:qs_initial} satisfying $\dQ_p^\Phi=\dQ_{p^2}$ (and we choose remaining data in \S\ref{ss:qs_initial} with $\dQ_p^\Phi=\dQ_{p^2}$);

        \item[(PII4)] $P_{\balpha(\Pi_{0,\fp})}\modulo\lambda^m$ is level-raising special at $\fp$; $P_{\balpha(\Pi_{1,\fp})}\modulo\lambda$ is Tate generic at $\fp$; $P_{\balpha(\Pi_{0,\fp})\otimes\balpha(\Pi_{1,\fp})}\modulo\lambda^m$ is level-raising special at $\fp$ (all in Definition \ref{de:satake_condition});
      \end{description}

  \item a datum $\cV'=(\rV'_n,\rV'_{n+1};\Lambda'_n,\Lambda'_{n+1};\rK'_n,\rK'_{n+1})$ in which
    \begin{itemize}
      \item $\rV'_n$ is a standard \emph{definite} hermitian space over $F$ of rank $n$ and $\rV'_{n+1}=(\rV'_n)_\sharp$;

      \item $\Lambda'_n$ is a $\prod_{v\not\in\Sigma^+_\infty\cup\Sigma^+_\mnm}O_{F_v}$-lattice in $\rV'_n\otimes_F\dA_F^{\Sigma_\infty\cup\Sigma_\mnm}$ satisfying $\Lambda'_n\subseteq(\Lambda'_n)^\vee$ and that $(\Lambda'_{n,v})^\vee/\Lambda'_{n,v}$ has length one (resp.\ zero) when $v\in\Sigma^+_{\lr,\r{II}}\cup\{\fp\}$ (resp.\ $v\not\in\Sigma^+_{\lr,\r{II}}\cup\{\fp\}$) and $\Lambda'_{n+1}=(\Lambda'_n)_\sharp$;

      \item $(\rK'_n,\rK'_{n+1})$ is an object in $\fK(\rV'_n)_\sp$ of the form
         \begin{align*}
         \rK'_N=\prod_{v\in\Sigma^+_\mnm}(\rK'_N)_v\times\prod_{v\not\in\Sigma^+_\infty\cup\Sigma^+_\mnm}\rU(\Lambda'_N)(O_{F^+_v})
         \end{align*}
         for $N\in\{n,n+1\}$,
    \end{itemize}
    together with a datum $\tj=(\tj_n,\tj_{n+1})$ in which $\tj_N\colon\rV_N\otimes_\dQ\dA^{\infty,p}\to\rV'_N\otimes_\dQ\dA^{\infty,p}$ are isometries sending $(\Lambda_N)^p$ to $(\Lambda'_N)^p$ and $(\rK_N)^p$ to $(\rK'_N)^p$ for $N\in\{n,n+1\}$, satisfying $\tj_{n+1}=(\tj_n)_\sharp$.
\end{itemize}

Introduce the following ideals of $\dT^{\Sigma^+_{\r{II}}\cup\Sigma^+_p}_{n_\alpha}$, for $\alpha=0,1$
\begin{align*}
\begin{dcases}
\fm_\alpha\coloneqq\dT^{\Sigma^+_{\r{II}}\cup\Sigma^+_p}_{n_\alpha}\cap\Ker\(\dT^{\Sigma^+_\mnm}_{n_\alpha}
\xrightarrow{\phi_{\Pi_\alpha}}O_E\to O_E/\lambda\),\\
\fn_\alpha\coloneqq\dT^{\Sigma^+_{\r{II}}\cup\Sigma^+_p}_{n_\alpha}\cap\Ker\(\dT^{\Sigma^+_\mnm}_{n_\alpha}
\xrightarrow{\phi_{\Pi_\alpha}}O_E\to O_E/\lambda^m\).
\end{dcases}
\end{align*}

\begin{remark}\label{re:ribet}
Under Assumption \ref{as:first_generic}, the natural map $\rH^0_\fT(\ol\rB_{n_0},O_\lambda)^\diamond_{\fm_0}\to\rH^0_\fT(\ol\rB_{n_0},O_\lambda)_{\fm_0}$ is an isomorphism. In view of Definition \ref{co:abel} and Remark \ref{re:abel}, we have
the localized and quotient Abel--Jacobi maps
\begin{align*}
\mho_{0,\fm_0}&\colon\rH^0_\fT(\ol\rB_{n_0},O_\lambda)_{\fm_0}
\to\rH^1(\dF_{p^2},\rH^{n_0-1}_\fT(\ol\rM_{n_0},O_\lambda(r_0))_{\fm_0}), \\
\mho_0/\fn_0&\colon\rH^0_\fT(\ol\rB_{n_0},O_\lambda)/{\fn_0}
\to\rH^1(\dF_{p^2},\rH^{n_0-1}_\fT(\ol\rM_{n_0},O_\lambda(r_0))/{\fn_0}),
\end{align*}
where $\mho_0$ is the map in Definition \ref{co:abel} (and Remark \ref{re:abel}) for $\rV=\rV_{n_0}$ and $\varsigma=1$.
\end{remark}

Next, we would like to define the second reciprocity map $\varrho_\unr$, analogous to the first one in Definition \ref{de:first_reciprocity}. To do this, we will freely use the notation from \cite{LTXZZ}*{\S7.3}. For $\alpha=0,1$, put $\rV^\star_{n_\alpha}\coloneqq\rV'_{n_\alpha}$, $\ti_{n_\alpha}\coloneqq\tj_{n_\alpha}$ and choose hermitian lattices $\{\Lambda^\star_{n_\alpha,\fq}\}_{\fq\mid p}$ satisfying
\begin{itemize}[label={\ding{118}}]
  \item $\Lambda'_{n_\alpha,\fp}\subseteq\Lambda^\star_{n_\alpha,\fp}\subseteq p^{-1}\Lambda'_{n_\alpha,\fp}$,

  \item $p\Lambda^\star_{n_\alpha,\fp}\subseteq(\Lambda^\star_{n_\alpha,\fp})^\vee$ such that $(\Lambda^\star_{n_\alpha,\fp})^\vee/p\Lambda^\star_{n_\alpha,\fp}$ has length $1-\alpha$,

  \item $(\Lambda^\star_{n,\fp})_\sharp\subseteq\Lambda^\star_{n+1,\fp}\subseteq p^{-1}(\Lambda^\star_{n,\fp})_\sharp^\vee$,

  \item $\Lambda^\star_{n_\alpha,\fq}=\Lambda'_{n_\alpha,\fq}$ for $\alpha=0,1$ and $\fq\neq\fp$.
\end{itemize}
Then $(\rV^\star_n,\ti_n,\{\Lambda^\star_{n,\fq}\}_{\fq\mid p})$ and $(\rV^\star_{n+1},\ti_{n+1},\{\Lambda^\star_{n+1,\fq}\}_{\fq\mid p})$ form definite uniformization data for $\rV_n$ and $\rV_{n+1}$ \cite{LTXZZ}*{Definition~4.4.1}, respectively, satisfying the conditions in \cite{LTXZZ}*{Notation~4.5.7}. For $\alpha=0,1$, let $\rK^\star_{n_\alpha,\fq}$ be the stabilizer of $\Lambda^\star_{n_\alpha,\fq}$; denote by
\[
\tT^{\star\prime}_{n_\alpha,\fp}
\in\dZ[\rK^\star_{n_\alpha,\fp}\backslash\rU(\rV'_{n_\alpha})(F^+_\fp)/\rK'_{n_\alpha,\fp}],\quad
\tT^{\prime\star}_{n_\alpha,\fp}
\in\dZ[\rK^\prime_{n_\alpha,\fp}\backslash\rU(\rV'_{n_\alpha})(F^+_\fp)/\rK^\star_{n_\alpha,\fp}]
\]
the characteristic functions of $\rK^\star_{n_\alpha,\fp}\rK'_{n_\alpha,\fp}$ and $\rK'_{n_\alpha,\fp}\rK^\star_{n_\alpha,\fp}$, respectively; and put
\[
\tI'_{n_\alpha,\fp}\coloneqq\tT^{\prime\star}_{n_\alpha,\fp}\circ\tT^{\star\prime}_{n_\alpha,\fp}
\in\dZ[\rK'_{n_\alpha,\fp}\backslash\rU(\rV'_{n_\alpha})(F^+_\fp)/\rK'_{n_\alpha,\fp}].
\]
Finally, put $\rK^\star_{n_\alpha}\coloneqq\ti_{n_\alpha}\rK_{n_\alpha}^p\times\prod_{\fq\mid p}\rK^\star_{n_\alpha,\fq}$ and $\rK^\star_\sp\coloneqq\rK^\star_n\cap\rK^\star_{n+1}$.

\begin{definition}[Second reciprocity map]\label{de:second_reciprocity}
Assume Assumption \ref{as:first_irreducible}, Assumption \ref{as:first_generic}, and Hypothesis \ref{hy:unitary_cohomology} for both $n$ and $n+1$. We define a map
\begin{align*}
\varrho_\unr&\colon
O_\lambda[\Sh(\rV'_{n_0},\rK'_{n_0})\times\Sh(\rV'_{n_1},\rK'_{n_1})]/(\fn_0,\fn_1) \\
&\to
\rH^1_\unr(F_\fp,\rH^{2n-1}_\et((\Sh(\rV_{n_0},\rK_{n_0})\times_F\Sh(\rV_{n_1},\rK_{n_1}))_{\ol{F}},O_\lambda(n))/(\fn_0,\fn_1))
\end{align*}
of $O_\wp$-modules, called \emph{second reciprocity map}, as the composition of following four maps:
\begin{itemize}[label={\ding{118}}]
  \item the map
      \begin{align*}
      (\tT^{\prime\star}_{n_0,\fp}\otimes\tT^{\prime\star}_{n_1,\fp})^{-1}\circ(1\otimes\tI^\prime_{n_1,\fp})
      &\colon O_\lambda[\Sh(\rV'_{n_0},\rK'_{n_0})\times\Sh(\rV'_{n_1},\rK'_{n_1})]/(\fn_0,\fn_1) \\
      &\xrightarrow{\sim} O_\lambda[\Sh(\rV^\star_{n_0},\rK^\star_{n_0})\times\Sh(\rV^\star_{n_1},\rK^\star_{n_1})]/(\fn_0,\fn_1),
      \end{align*}
      in which the map
      \[
      \tT^{\prime\star}_{n_0,\fp}\otimes\tT^{\prime\star}_{n_1,\fp}\colon
      O_\lambda[\Sh(\rV^\star_{n_0},\rK^\star_{n_0})\times\Sh(\rV^\star_{n_1},\rK^\star_{n_1})]/(\fn_0,\fn_1)\to
      O_\lambda[\Sh(\rV'_{n_0},\rK'_{n_0})\times\Sh(\rV'_{n_1},\rK'_{n_1})]/(\fn_0,\fn_1)
      \]
      is invertible by Lemma \ref{le:second} below,

  \item the \emph{inverse} of the endomorphism
      \[
      1\otimes\tT^\star_{n_1,\fp}\colon
      O_\lambda[\Sh(\rV^\star_{n_0},\rK^\star_{n_0})\times\Sh(\rV^\star_{n_1},\rK^\star_{n_1})]/(\fn_0,\fn_1)\to
      O_\lambda[\Sh(\rV^\star_{n_0},\rK^\star_{n_0})\times\Sh(\rV^\star_{n_1},\rK^\star_{n_1})]/(\fn_0,\fn_1),
      \]
      which is invertible by the claim (1) in the proof of \cite{LTXZZ}*{Theorem~7.3.4},

  \item the map
      \begin{align*}
      \mho_0/\fn_0\otimes((\iota_{n_1})_!\circ\pi_{n_1}^*)/\fn_1
      &\colon O_\lambda[\Sh(\rV^\star_{n_0},\rK^\star_{n_0})\times\Sh(\rV^\star_{n_1},\rK^\star_{n_1})]/(\fn_0,\fn_1) \\
      &\to\rH^1(\dF_{p^2},\rH^{2r_0-1}_\fT(\ol\rM_{n_0},O_\lambda(r_0))/\fn_0)\otimes_{O_\lambda}
      \rH^{2r_1}_\fT(\ol\rM_{n_1},O_\lambda(r_1)/\fn_1)_{\Gal(\ol\dF_p/\dF_{p^2})},
      \end{align*}
      and

  \item the natural isomorphism
      \begin{align*}
      &\rH^1(\dF_{p^2},\rH^{2r_0-1}_\fT(\ol\rM_{n_0},O_\lambda(r_0))/\fn_0)\otimes_{O_\lambda}
      \rH^{2r_1}_\fT(\ol\rM_{n_1},O_\lambda(r_1)/\fn_1)_{\Gal(\ol\dF_p/\dF_{p^2})} \\
      &\xrightarrow\sim
      \rH^1_\unr(F_\fp,
      \rH^{2n-1}_\et((\Sh(\rV_{n_0},\rK_{n_0})\times_F\Sh(\rV_{n_1},\rK_{n_1}))_{\ol{F}},O_\lambda(n))/(\fn_0,\fn_1)).
      \end{align*}
\end{itemize}
\end{definition}

\begin{remark}\label{re:second}
We have
\begin{enumerate}
  \item By \cite{LTXZZ}*{Lemma~7.3.3}, the map
      \[
      (\iota_{n_1})_!\circ\pi_{n_1}^*\colon O_\lambda[\Sh(\rV^\star_{n_1},\rK^\star_{n_1})]\simeq
      \rH^0_\fT(\rS_{n_1},O_\lambda)
      \to\rH^{2r_1}_\fT(\ol\rM_{n_1},O_\lambda(r_1))_{\Gal(\ol\dF_p/\dF_{p^2})}
      \]
      is an isomorphism after localization at $\fm_1$. It follows that $\varrho_\unr$ is surjective/bijective if the map $\mho_0/\fn_0$ from Remark \ref{re:ribet} is surjective/bijective.

  \item One can check that the map $\varrho_\unr$ does not depend on the choice of the auxiliary data $\{\Lambda^\star_{n_\alpha,\fq}\}_{\fq\mid p}$ (thought we do not need this in the argument below).
\end{enumerate}
\end{remark}

\begin{lem}\label{le:second}
Let the notation be in Definition \ref{de:second_reciprocity}. For $\alpha=0,1$, both of the two maps
\begin{align*}
\tT^{\star\prime}_{n_\alpha,\fp}&\colon O_\lambda[\Sh(\rV^\prime_{n_\alpha},\rK^\prime_{n_\alpha})]\to O_\lambda[\Sh(\rV^\star_{n_\alpha},\rK^\star_{n_\alpha})] \\
\tT^{\prime\star}_{n_\alpha,\fp}&\colon O_\lambda[\Sh(\rV^\star_{n_\alpha},\rK^\star_{n_\alpha})]\to O_\lambda[\Sh(\rV^\prime_{n_\alpha},\rK^\prime_{n_\alpha})]
\end{align*}
become isomorphisms after localizing at $\fm_\alpha$.
\end{lem}

\begin{proof}
When $\alpha=1$, the proof of (1) reduces to \cite{LTXZZ}*{Lemma~6.1.9} for $N=n_1$: indeed, by putting $\rV^\circ_{n_1}\coloneqq\pres{p}\rV'_{n_1}(=\pres{p}\rV^\star_{n_1})$ (here again, $\pres{p}$ means that we multiply the hermitian form by $p$) under which we take $\Lambda^\circ_{n_1,\fp}=\Lambda^\star_{n_1}$ and $\Lambda^\bullet_{n_1,\fp}=p^{-1}\Lambda'_{n_1,\fp}$, we reduce $\tT^{\star\prime}_{n_1,\fp}$ to $\tT^{\circ\bullet}_{n_1,\fp}$ and $\tT^{\prime\star}_{n_1,\fp}$ to $\tT^{\bullet\circ}_{n_1,\fp}$. When $\alpha=0$, (1) follows from the same argument in the proof of \cite{LTXZZ}*{Lemma~6.1.9} by using Proposition \ref{pr:intertwining} below instead of \cite{LTXZZ}*{Proposition~B.3.5(1)}.
\end{proof}

In order to apply the above construction in the context of the Iwasawa theory, we need to consider slightly more general open compact groups away from $\Sigma^+_p$. Consider the following data:
\begin{itemize}[label={\ding{118}}]
  \item another object $(\tilde\rK_n,\tilde\rK_{n+1})\in\fK(\rV_n)_\sp$ of the form
     \begin{align*}
     \tilde\rK_N=\prod_{v\in\Sigma^+_p}\rK_{N,v}\times\prod_{v\not\in\Sigma^+_\infty\cup\Sigma^+_p}\tilde\rK_{N,v}
     \end{align*}
     (and put $\tilde\rK'_N\coloneqq\tj_N(\tilde\rK_N)^p\times(\rK'_N)_p$) for $N\in\{n,n+1\}$,

  \item a Hecke operator
     \[
     \tP\in O_\lambda[(\tilde\rK_n)^p\times(\tilde\rK_{n+1})^p
     \backslash\rU(\rV_n)(\dA_F^{\infty,p})\times\rU(\rV_{n+1})(\dA_F^{\infty,p})/(\rK_n)^p\times(\rK_{n+1})^p],
     \]

  \item and a (finite anticyclotomic) extension $\tilde{F}/F$ contained in $F^\ab$ that is fixed by the kernel of the composite map
     \begin{align*}
     \Gal(F^\ab/F)=F^\times\backslash(\dA_F^\infty)^\times&\xrightarrow{\Nm_{F/F^+}^-}
     \bG_{F/F^+}(F^+)\backslash\bG_{F/F^+}(\dA_{F^+}^\infty) \\
     &\to\bG_{F/F^+}(F^+)\backslash\bG_{F/F^+}(\dA_{F^+}^\infty)/\det\tilde\rK_n
     =\bG_{F/F^+}(F^+)\backslash\bG_{F/F^+}(\dA_{F^+}^\infty)/\det\tilde\rK'_n,
     \end{align*}
     so that $\Gal(\tilde{F}/F)$ is a quotient of $\bG_{F/F^+}(F^+)\backslash\bG_{F/F^+}(\dA_{F^+}^\infty)/\bG_{F/F^+}(O_{F^+}\otimes\dZ_p)$.
\end{itemize}

\begin{notation}
We introduce two notations based on the previous data.
\begin{enumerate}
  \item We have canonical maps
      \[
      \pi_0\(\Sh(\rV_n,\tilde\rK_n)_{F^\ab}\)\to\bG_{F/F^+}(F^+)\backslash\bG_{F/F^+}(\dA_{F^+}^\infty)/\det\tilde\rK_n
      \to\Gal(\tilde{F}/F)
      \]
      of sets. For every element $\varsigma\in\Gal(\tilde{F}/F)$, we denote by $\Sh(\rV_n,\tilde\rK_n)_\varsigma$ the union of components parameterized by the fiber of $\varsigma$ under the above composite map, which indeed is defined over $F$, and by
      \[
      [\graph\Sh(\rV_n,\tilde\rK_n)_\varsigma]\in
      \rZ^n((\Sh(\rV_n,\tilde\rK_n)\times_F\Sh(\rV_{n+1},\tilde\rK_{n+1}))_{\tilde{F}})
      \]
      the cycle given by the image of $\Sh(\rV_n,\tilde\rK_n)_\varsigma$ under the diagonal morphism
      \[
      \graph\colon\Sh(\rV_n,\tilde\rK_n)\to\Sh(\rV_n,\tilde\rK_n)\times_F\Sh(\rV_{n+1},\tilde\rK_{n+1})
      \]
      of Shimura varieties.

  \item We have canonical maps
      \[
      \Sh(\rV'_n,\tilde\rK'_n)\xrightarrow{\det}\bG_{F/F^+}(F^+)\backslash\bG_{F/F^+}(\dA_{F^+}^\infty)/\det\tilde\rK'_n
      \to\Gal(\tilde{F}/F)
      \]
      of sets. For every element $\varsigma\in\Gal(\tilde{F}/F)$, we denote by $\Sh(\rV'_n,\tilde\rK'_n)_\varsigma$ the fiber of $\varsigma$ under the above composite map, and by
      \[
      \CF_{\graph\Sh(\rV'_n,\tilde\rK'_n)_\varsigma}\in\dZ[\Sh(\rV'_n,\tilde\rK'_n)\times\Sh(\rV'_{n+1},\tilde\rK'_{n+1})]
      \]
      the characteristic function of the image of $\Sh(\rV'_n,\tilde\rK'_n)_\varsigma$ under the diagonal map
      \[
      \graph\colon\Sh(\rV'_n,\tilde\rK'_n)\to\Sh(\rV'_n,\tilde\rK'_n)\times\Sh(\rV'_{n+1},\tilde\rK'_{n+1})
      \]
      of Shimura sets.
\end{enumerate}
\end{notation}

The Hecke operator $\tP$ induces maps
\begin{align*}
\tP_*&\colon\rZ^n((\Sh(\rV_n,\tilde\rK_n)\times_F\Sh(\rV_{n+1},\tilde\rK_{n+1}))_{\tilde{F}})_{O_\lambda}
\to \rZ^n((\Sh(\rV_n,\rK_n)\times_F\Sh(\rV_{n+1},\rK_{n+1}))_{\tilde{F}})_{O_\lambda},\\
\tP_*&\colon O_\lambda[\Sh(\rV'_n,\tilde\rK'_n)\times\Sh(\rV'_{n+1},\tilde\rK'_{n+1})] \to
O_\lambda[\Sh(\rV'_n,\rK'_n)\times\Sh(\rV'_{n+1},\rK'_{n+1})].
\end{align*}

Under Assumption \ref{as:first_generic}, we have the Abel--Jacobi map
\[
\AJ\colon\rZ^n((\Sh(\rV_{n_0},\rK_{n_0})\times_F\Sh(\rV_{n_1},\rK_{n_1}))_{\tilde{F}})\to
\rH^1(\tilde{F},\rH^{2n-1}_\et((\Sh(\rV_{n_0},\rK_{n_0})\times_F\Sh(\rV_{n_1},\rK_{n_1}))_{\ol{F}},O_\lambda(n))/(\fn_0,\fn_1)).
\]
The fixed algebraic closure $\ol{F}^+_\fp$ of $F^+_\fp$ containing $\ol{F}$ from \S\ref{ss:notation} induces a place $\fp'$ of $\tilde{F}$ above $\fp$. Combining with Lemma \ref{le:inert}, we have the localization map
\begin{align*}
\loc_\fp&\colon
\rH^1(\tilde{F},\rH^{2n-1}_\et((\Sh(\rV_{n_0},\rK_{n_0})\times_F\Sh(\rV_{n_1},\rK_{n_1}))_{\ol{F}},O_\lambda(n))/(\fn_0,\fn_1)) \\
&\to
\rH^1(\tilde{F}_{\fp'},\rH^{2n-1}_\et((\Sh(\rV_{n_0},\rK_{n_0})\times_F\Sh(\rV_{n_1},\rK_{n_1}))_{\ol{F}},O_\lambda(n))/(\fn_0,\fn_1)) \\
&=
\rH^1(F_\fp,\rH^{2n-1}_\et((\Sh(\rV_{n_0},\rK_{n_0})\times_F\Sh(\rV_{n_1},\rK_{n_1}))_{\ol{F}},O_\lambda(n))/(\fn_0,\fn_1)).
\end{align*}
By \cite{LTXZZ}*{Remark~7.3.5}, the image of $\loc_\fp\circ\AJ$ is contained in the $O_\lambda$-submodule
\[
\rH^1_\unr(F_\fp,
\rH^{2n-1}_\et((\Sh(\rV_{n_0},\rK_{n_0})\times_F\Sh(\rV_{n_1},\rK_{n_1}))_{\ol{F}},O_\lambda(n))/(\fn_0,\fn_1)).
\]

The following result refines \cite{LTXZZ}*{Theorem~7.3.4}.

\begin{theorem}[Second explicit reciprocity law]\label{th:second}
Assume Assumption \ref{as:first_irreducible}, Assumption \ref{as:first_generic}, and Hypothesis \ref{hy:unitary_cohomology} for both $n$ and $n+1$. Then for every collection of data $(\tilde\rK_n,\tilde\rK_{n+1}),\tP,\tilde{F}$ as above,
\begin{align*}
\varrho_\unr\(\tP_*\CF_{\graph\Sh(\rV'_n,\tilde\rK'_n)_\varsigma}\)=
\loc_\fp\AJ\(\tP_*[\graph\Sh(\rV_n,\tilde\rK_n)_{[\tj_n]\cdot\varsigma}]\)
\end{align*}
holds for every $\varsigma\in\Gal(\tilde{F}/F)$. Here, $[\tj_n]\in\bG_{F/F^+}(F^+)\backslash\bG_{F/F^+}(\dA_{F^+}^\infty)/\bG_{F/F^+}(O_{F^+}\otimes\dZ_p)$ is the element from Lemma \ref{le:determinant_second} applied to the isometry $\tj_n\colon\rV_n\otimes_\dQ\dA^{\infty,p}\xrightarrow\sim\rV'_n\otimes_\dQ\dA^{\infty,p}$.
\end{theorem}

\begin{proof}
In the proof, we will freely use the notation from \cite{LTXZZ}*{\S7.3}. By enlarging $\Sigma^+_{\r{II}}$, we may assume, without loss of generality, that $(\tilde\rK_N)_v=(\rK_N)_v$ for $v\not\in\Sigma^+_\infty\cup\Sigma^+_{\r{II}}$ and that $\tP$ is supported at places in $\Sigma^+_{\r{II}}$. Consider an arbitrary pair $(\hat\rK_n,\hat\rK_{n+1})\in\fK(\rV_n)\times\fK(\rV_{n+1})$ of the form
\begin{align*}
\hat\rK_N=\prod_{v\not\in\Sigma^+_\infty\cup\Sigma^+_\r{II}}\rK_{N,v}\times\prod_{v\in\Sigma^+_\r{II}}\hat\rK_{N,v}
\end{align*}
(and put $\hat\rK^\star_N\coloneqq\tj_N(\hat\rK_N)^p\times(\rK^\star_N)_p$) for $N\in\{n,n+1\}$.

Put $\hat{\phantom{a}}$ on top of notation for objects from \cite{LTXZZ}*{\S7.3} and Definition \ref{de:second_reciprocity} with respect to $(\hat\rK_n,\hat\rK_{n+1})$. We then have the map
\begin{align*}
\hat\inc^{\star,\star}_{!/(\fn_0,\fn_1)}\colon
O_\lambda[\Sh(\rV^\star_{n_0},\hat\rK^\star_{n_0})\times\Sh(\rV^\star_{n_1},\hat\rK^\star_{n_1})]/(\fn_0,\fn_1)\to
\rH^{2n}_\fT(\hat\rM_{n_0}\times_{\rT}\hat\rM_{n_1},O_\lambda(n))/(\fn_0,\fn_1)
\end{align*}
defined in \cite{LTXZZ}*{Definition~4.6.1}, which is compatible when we change $(\hat\rK_n,\hat\rK_{n+1})$ (in the above form). As a consequence, we have the commutative diagram
\begin{align}\label{eq:second1}
\xymatrix{
O_\lambda[\Sh(\rV^\star_{n_0},\tilde\rK^\star_{n_0})\times\Sh(\rV^\star_{n_1},\tilde\rK^\star_{n_1})]/(\fn_0,\fn_1) \ar[rr]^-{\tilde\inc^{\star,\star}_{!/(\fn_0,\fn_1)}}\ar[d]_-{\tP_*} && \rH^{2n}_\fT(\tilde\rM_{n_0}\times_{\rT}\tilde\rM_{n_1},O_\lambda(n))/(\fn_0,\fn_1) \ar[d]^-{\tP_*} \\
O_\lambda[\Sh(\rV^\star_{n_0},\rK^\star_{n_0})\times\Sh(\rV^\star_{n_1},\rK^\star_{n_1})]/(\fn_0,\fn_1)  \ar[rr]^-{\inc^{\star,\star}_{!/(\fn_0,\fn_1)}}&&
\rH^{2n}_\fT(\rM_{n_0}\times_{\rT}\rM_{n_1},O_\lambda(n))/(\fn_0,\fn_1)
}
\end{align}
in which $\tilde\inc^{\star,\star}_!$ is the corresponding map for $(\tilde\rK_n,\tilde\rK_{n+1})$. Note that by the last displayed formula in the proof of \cite{LTXZZ}*{Theorem~7.3.4}, the target of $\inc^{\star,\star}_{!/(\fn_0,\fn_1)}$ is naturally isomorphic to
\[
\rH^1(\dF_{p^2},\rH^{2r_0-1}_\fT(\ol{\rM}_{n_0},O_\lambda(r_0))/\fn_0)\otimes_{O_\lambda}
(\rH^{2r_1}_\fT(\ol{\rM}_{n_1},O_\lambda(r_1))/\fn_1)_{\Gal(\ol\dF_p/\dF_{p^2})}
\]
under which $\inc^{\star,\star}_{!/(\fn_0,\fn_1)}$ is identified with $\mho_0/\fn_0\otimes((\iota_{n_1})_!\circ\pi_{n_1}^*)/\fn_1$.

We put $\tilde{\phantom{a}}$ on top of notation for objects with respect to $(\tilde\rK_n,\tilde\rK_{n+1})$. It is clear that $\det\tilde\rK'_n=\det\tilde\rK^\star_n=\det\tilde\rK^\star_\sp$, so that we can similarly define $\Sh(\rV^\star_n,\tilde\rK^\star_n)_\varsigma$, $\Sh(\rV^\star_n,\tilde\rK^\star_\sp)_\varsigma$, and $\CF_{\graph\Sh(\rV^\star_n,\tilde\rK^\star_\sp)_\varsigma}$. By the commutative diagram \eqref{eq:second1} and the claim (2) in the proof of \cite{LTXZZ}*{Theorem~7.3.4}, it remains to show that
\begin{align}\label{eq:second4}
(\id\times\tilde\pi_{n_1})_!(\id\times\tilde\iota_{n_1})^*\tilde\inc^{\star,\star}_!
\(\CF_{\graph\Sh(\rV^\star_n,\tilde\rK^\star_\sp)_\varsigma}\)
=\tT^\star_{n_1,\fp}(\id\times\tilde\pi_{n_1})_!(\id\times\tilde\iota_{n_1})^*\loc_\fp\([\graph\Sh(\rV_n,\tilde\rK_n)]_{[\tj_n]\cdot\varsigma}\)
\end{align}
and
\begin{align}\label{eq:second2}
(1\otimes\tI^\prime_{n_1,\fp})\CF_{\graph\Sh(\rV'_n,\tilde\rK'_n)_\varsigma}=
(\tT^{\prime\star}_{n,\fp}\otimes\tT^{\prime\star}_{n+1,\fp})\CF_{\graph\Sh(\rV^\star_n,\tilde\rK^\star_\sp)_\varsigma}
\end{align}
hold for every $\varsigma\in\Gal(\tilde{F}/F)$.

We first consider \eqref{eq:second4}. When $\tilde{F}=F$, it has been verified in \cite{LTXZZ}*{Theorem~4.6.2}. We now explain how to modify the argument for general $\tilde{F}$. The uniformization map in \cite{LTXZZ}*{Construction~4.4.2} gives a finite map $\pi_0(\tilde\rS_n)\to\Sh(\rV_n,\tilde\rK^\star_n)$. For every $\varsigma\in\Gal(\tilde{F}/F)$, we denote by $(\tilde\rS_n)_\varsigma$ the union of components parameterized by the preimage of $\Sh(\rV_n,\tilde\rK^\star_n)_\varsigma$ under the above map. Examining the proof of \cite{LTXZZ}*{Theorem~4.6.2}, we find that \eqref{eq:second4} follows from the following two claims:
\begin{enumerate}
  \item The scheme $\tilde\bM_n$ is the disjoint union of $(\tilde\bM_n)_\varsigma$ for all $\varsigma\in\Gal(\tilde{F}/F)$.

  \item The basic correspondence $\tilde\rS_n\leftarrow\tilde\rB_n\to\tilde\rM_n$ \cite{LTXZZ}*{(4.3)} restricts to a correspondence between $(\tilde\rS^\circ_n)_\varsigma$ and $(\tilde\rM_n)_{[\tj_n]\cdot\varsigma}$ for every $\varsigma\in\Gal(\tilde{F}/F)$.
\end{enumerate}
Claim (1) follows from the fact that $\tilde\bM_n$ is regular \cite{LTXZZ}*{Theorem~4.2.3} and that $\tilde\bM_n^\eta$ is the disjoint union of $(\tilde\bM_n^\eta)_\varsigma$ for all $\varsigma\in\Gal(\tilde{F}/F)$. Claim (2) follows from Lemma \ref{le:determinant_second}.

We then consider \eqref{eq:second2}. Put $\tilde\rK^\dag_N\coloneqq\tilde\rK'_N\cap\tilde\rK^\star_N$ for $N\in\{n,n+1\}$ and $\tilde\rK_\sp^\dag\coloneqq\tilde\rK'_n\cap\tilde\rK^\star_\sp$. For the notation of hermitian spaces, we will always use $\rV'_N$. We have the following commutative diagram of Shimura sets
\[
\xymatrix{
\Sh(\rV'_{n+1},\tilde\rK'_{n+1}) & \Sh(\rV'_{n+1},\tilde\rK^\dag_{n+1}) \ar[l]_-{\sh^{\dag\prime}_{n+1}}\ar[r]^-{\sh^{\dag\star}_{n+1}} & \Sh(\rV'_{n+1},\tilde\rK^\star_{n+1})  \\
& \Sh(\rV'_n,\tilde\rK^\dag_\sp) \ar[r]^-{\sh^{\dag\star}_\sp}\ar[u]_-{\sh^\dag_\uparrow}\ar[d]^-{\sh^\dag_\downarrow}
&  \Sh(\rV'_n,\tilde\rK^\star_\sp) \ar[u]_-{\sh^\star_\uparrow}\ar[d]^-{\sh^\star_\downarrow} \\
\Sh(\rV'_n,\tilde\rK'_n) \ar[uu]_-{\sh'_\uparrow} & \Sh(\rV'_n,\tilde\rK^\dag_n) \ar[l]_-{\sh^{\dag\prime}_n}\ar[r]^-{\sh^{\dag\star}_n} & \Sh(\rV'_n,\tilde\rK^\star_n)
}
\]
which is analogous to the right facet of the diagram in \cite{LTXZZ}*{Proposition~5.10.14}, satisfying the following properties (following \cite{LTXZZ}*{Remark~4.5.8}):
\begin{enumerate}[label=(\alph*)]
  \item The lower-right square is Cartesian, in which $\sh^\star_\downarrow$ has degree $1$ (resp.\ $p+1$) when $n$ is odd (resp.\ even).

  \item When $n$ is odd, the upper-right square is Cartesian.

  \item When $n$ is even, the square
     \[
     \xymatrix{
     \Sh(\rV'_{n+1},\tilde\rK'_{n+1}) & \Sh(\rV'_{n+1},\tilde\rK^\dag_{n+1}) \ar[l]_-{\sh^{\dag\prime}_{n+1}} \\
     \Sh(\rV'_n,\tilde\rK'_n) \ar[u]_-{\sh'_\uparrow} & \Sh(\rV'_n,\tilde\rK^\dag_\sp) \ar[u]_-{\sh^\dag_\uparrow}\ar[l]_-{\sh^{\dag\prime}_n\circ\sh^\dag_\downarrow}
     }
     \]
     is Cartesian.
\end{enumerate}
For \eqref{eq:second2}, it suffices to show that for every $f\in O_\lambda[\Sh(\rV'_{n_0},\tilde\rK'_{n_0})]\otimes_{O_\lambda}O_\lambda[\Sh(\rV'_{n_1},\tilde\rK^\star_{n_1})]$, we have
\begin{align*}
\sum_{s\in\Sh(\rV'_n,\tilde\rK^\star_\sp)_\varsigma}(\tT^{\star\prime}_{n_0,\fp}f)(\sh^\star_\downarrow(s),\sh^\star_\uparrow(s))
=\sum_{s\in\Sh(\rV'_n,\tilde\rK'_n)_\varsigma}(\tT^{\prime\star}_{n_1,\fp}f)(s,\sh'_\uparrow(s)),
\end{align*}
which follows from the exactly same argument in the proof of \cite{LTXZZ}*{Lemma~5.11.6} (but with the parity changed) by using properties (a,b,c) and the obvious fact that both the natural map $\Sh(\rV'_n,\tilde\rK^\star_\sp)\to\Sh(\rV'_n,\tilde\rK^\star_n)$ and the correspondence $\tT^{\star\prime}_{n_0,\fp}$ preserve subsets labeled by $\varsigma$.

The theorem is proved.
\end{proof}

\subsection{Appendix: Complements on basic correspondences}
\label{ss:uniformization}

In this subsection, we study compatibility of determinant maps for basic correspondences constructed in \cite{LTXZZ}*{\S4.3~\&~\S5.3} with respect to uniformization data.

We first consider the situation in \cite{LTXZZ}*{\S4.3}. Let $\rV$ be a standard indefinite hermitian space over $F$ of rank $N\geq 1$. Fix a definite uniformization data for $\rV$ at $\fp$ as in \cite{LTXZZ}*{Definition~4.4.1}; in particular, we have an isometry $\ti\colon\rV\otimes_\dQ\dA^{\infty,p}\to\rV^\star\otimes_\dQ\dA^{\infty,p}$.

Take an object $\rK^p\in\fK(\rV)^p$. Consider a pair $(y,x^\sharp)$ in which
\begin{itemize}[label={\ding{118}}]
  \item $y=(A_0,\lambda_0,\eta_0^p;A,\lambda,\eta^p;A^\star,\lambda^\star,\eta^{p\star};\alpha)\in\rB_\fp(\rV,\rK^p)(\ol\dF_p)$ with images
      \[
      x=(A_0,\lambda_0,\eta_0^p;A,\lambda,\eta^p)\in\rM_\fp(\rV,\rK^p)(\ol\dF_p),\quad s=(A_0,\lambda_0,\eta_0^p;A^\star,\lambda^\star,\eta^{p\star})\in\rS_\fp(\rV,\rK^p)(\ol\dF_p);
      \]

  \item $x^\sharp=(A_0,\lambda_0,\eta_0^p;A^\sharp,\lambda^\sharp,\eta^{p\sharp})\in\bM_\fp(\rV,\rK^p)(\ol\dZ_p)$ lifting $x$.
\end{itemize}
Put
\[
\rV_{x^\sharp}\coloneqq\Hom_F^{\lambda_0,\lambda^\sharp}(\rH_1(A_0(\dC),\dQ),\rH_1(A^\sharp(\dC),\dQ)),\quad
\rV_s\coloneqq\Hom_{O_F}^{\lambda_0,\lambda^\star}(A_0,A^\star)_\dQ,
\]
which are naturally hermitian spaces over $F$. Then $\rV_{x^\sharp}$ and $\rV_s$ are isometric to $\rV$ and $\rV^\star$, respectively. Choose isometries
\begin{itemize}[label={\ding{118}}]
  \item $\gamma_{x^\sharp}\colon\rV_{x^\sharp}\xrightarrow\sim\rV$ sending $\Hom_F^{\lambda_0,\lambda^\sharp}(\rH_1(A_0(\dC),\dZ_p),\rH_1(A^\sharp(\dC),\dZ_p))$ onto $\prod_{\fq\mid p}\Lambda_{n,\fq}$, and

  \item $\gamma_s\colon\rV_s\xrightarrow\sim\rV^\star$ sending $\Hom_{O_F}^{\lambda_0,\lambda^\star}(A_0,A^\star)_{\dZ_p}$ onto $\prod_{\fq\mid p}\Lambda^\star_{n,\fq}$.
\end{itemize}
We then have successive isometries
\begin{align*}
\rV\otimes_\dQ\dA^{\infty,p} &\xrightarrow{\eta^p} \Hom_{F\otimes_\dQ\dA^{\infty,p}}^{\lambda_0,\lambda}(\rH^\et_1(A_0,\dA^{\infty,p}),\rH^\et_1(A,\dA^{\infty,p})) \xrightarrow{\rho}\rV_{x^\sharp}\otimes_\dQ\dA^{\infty,p} \xrightarrow{\gamma_{x^\sharp}} \rV\otimes_\dQ\dA^{\infty,p}, \\
\rV\otimes_\dQ\dA^{\infty,p} &\xrightarrow{\eta^{p\star}}
\Hom_{F\otimes_\dQ\dA^{\infty,p}}^{\lambda_0,\lambda^\star}(\rH^\et_1(A_0,\dA^{\infty,p}),\rH^\et_1(A^\star,\dA^{\infty,p}))
\xrightarrow{\rho^\star} \rV_s\otimes_\dQ\dA^{\infty,p} \xrightarrow{\gamma_s} \rV^\star\otimes_\dQ\dA^{\infty,p},
\end{align*}
in which $\rho$ comes from the comparison theorem. We obtain two elements
\[
\gamma_{x^\sharp}\circ\rho\circ\eta^p,\quad \ti^{-1}\circ\gamma_s\circ\rho^\star\circ\eta^{p\star}
\]
in $\rU(\rV)(\dA_{F^+}^{\infty,p})/\rK^p$.

\begin{lem}\label{le:determinant_second}
There exists a unique element $[\ti]\in\bG_{F/F^+}(F^+)\backslash\bG_{F/F^+}(\dA_{F^+}^\infty)/\bG_{F/F^+}(O_{F^+}\otimes\dZ_p)$ such that for every choice of data $\rK^p,(y,x^\sharp),(\gamma_{x^\sharp},\gamma_s)$ as above, the quotient
\[
\frac{\det(\ti^{-1}\circ\gamma_s\circ\rho^\star\circ\eta^{p\star})}{\det(\gamma_{x^\sharp}\circ\rho\circ\eta^p)}
\in \bG_{F/F^+}(\dA_{F^+}^{\infty,p})/\det\rK^p
\]
coincides with $[\ti]$ in $\bG_{F/F^+}(F^+)\backslash\bG_{F/F^+}(\dA_{F^+}^\infty)/\bG_{F/F^+}(O_{F^+}\otimes\dZ_p)\det\rK^p$.
\end{lem}

\begin{proof}
The uniqueness is clear since
\[
\bG_{F/F^+}(F^+)\backslash\bG_{F/F^+}(\dA_{F^+}^\infty)/\bG_{F/F^+}(O_{F^+}\otimes\dZ_p)
=\varprojlim_{\rK^p}\bG_{F/F^+}(F^+)\backslash\bG_{F/F^+}(\dA_{F^+}^\infty)/\bG_{F/F^+}(O_{F^+}\otimes\dZ_p)\det\rK^p.
\]

Fix $\rK^p$. For every pair $(y,x^\sharp)$, put
\[
[\ti]_{(y,x^\sharp)}\coloneqq\frac{\det(\ti^{-1}\circ\gamma_s\circ\rho^\star\circ\eta^{p\star})}{\det(\gamma_{x^\sharp}\circ\rho\circ\eta^p)}
\in\bG_{F/F^+}(O_{F^+}\otimes\dZ_{(p)})\backslash\bG_{F/F^+}(\dA_{F^+}^{\infty,p})/\det\rK^p.
\]
It is clear that $[\ti]_{(y,x^\sharp)}$ does not depend on the pair $(\gamma_{x^\sharp},\gamma_s)$. It remains to show that for every other pair $(\tilde{y},\tilde{x}^\sharp)$, we have $[\ti]_{(y,x^\sharp)}=[\ti]_{(\tilde{y},\tilde{x}^\sharp)}$. Let $\tilde{x}$ and $\tilde{s}$ be the images of $\tilde{y}$ in $\rM_\fp(\rV,\rK^p)(\ol\dF_p)$ and $\rS_\fp(\rV,\rK^p)(\ol\dF_p)$, respectively.

We first assume $s=\tilde{s}$. By \cite{LTXZZ}*{Theorem~5.3.4} and the fact that $\rM_\fp(\rV,\rK^p)$ is regular \cite{LTXZZ}*{Theorem~4.2.3}, $x^\sharp$ and $\tilde{x}^\sharp$ are in the same geometric component of $\bM_\fp^\eta(\rV,\rK^p)$, which implies that $[\ti]_{(y,x^\sharp)}=[\ti]_{(\tilde{y},\tilde{x}^\sharp)}$. In general, we may pre-composite an isometry of $\rV\otimes_\dQ\dA^{\infty,p}$ to move $s$ to $\tilde{s}$. Then it follows from the previous case that $[\ti]_{(y,x^\sharp)}=[\ti]_{(\tilde{y},\tilde{x}^\sharp)}$.

The lemma is proved.
\end{proof}

We then consider the parallel situation in \cite{LTXZZ}*{\S5.3}. Let $\rV^\circ$ be a standard definite hermitian space over $F$ of rank $N\geq 1$. Fix an indefinite uniformization data for $\rV^\circ$ at $\fp$ as in \cite{LTXZZ}*{Definition~4.4.1}; in particular, we have an isometry $\tj\colon\rV^\circ\otimes_\dQ\dA^{\infty,p}\to\rV'\otimes_\dQ\dA^{\infty,p}$.

Take an object $\rK^{p\circ}\in\fK(\rV^\circ)^p$. Consider a pair $(y,x^\sharp)$ in which
\begin{itemize}[label={\ding{118}}]
  \item $y=(A_0,\lambda_0,\eta_0^p;A,\lambda,\eta^p;A^\circ,\lambda^\circ,\eta^{p\circ};\alpha)\in\rB^\circ_\fp(\rV^\circ,\rK^{p\circ})(\ol\dF_p)$ with images
      \[
      x=(A_0,\lambda_0,\eta_0^p;A,\lambda,\eta^p)\in\rM^\circ_\fp(\rV^\circ,\rK^{p\circ})(\ol\dF_p),\quad s=(A_0,\lambda_0,\eta_0^p;A^\circ,\lambda^\circ,\eta^{p\circ})\in\rS^\circ_\fp(\rV^\circ,\rK^{p\circ})(\ol\dF_p);
      \]

  \item $x^\sharp=(A_0,\lambda_0,\eta_0^p;A^\sharp,\lambda^\sharp,\eta^{p\sharp})\in\bM_\fp(\rV^\circ,\rK^{p\circ})(\ol\dZ_p)$ lifting $x$.
\end{itemize}
Put
\[
\rV_s\coloneqq\Hom_{O_F}^{\lambda_0,\lambda^\circ}(A_0,A^\circ)_\dQ,\quad
\rV_{x^\sharp}\coloneqq\Hom_F^{\lambda_0,\lambda^\sharp}(\rH_1(A_0(\dC),\dQ),\rH_1(A^\sharp(\dC),\dQ)),
\]
which are naturally hermitian spaces over $F$. Then $\rV_s$ and $\rV_{x^\sharp}$ are isometric to $\rV^\circ$ and $\rV'$, respectively. Choose isometries
\begin{itemize}[label={\ding{118}}]
  \item $\gamma_s\colon\rV_s\xrightarrow\sim\rV^\circ$ sending $\Hom_{O_F}^{\lambda_0,\lambda^\circ}(A_0,A^\circ)_{\dZ_p}$ onto $\prod_{\fq\mid p}\Lambda^\circ_{n,\fq}$, and

  \item $\gamma_{x^\sharp}\colon\rV_{x^\sharp}\xrightarrow\sim\rV'$ sending $\Hom_F^{\lambda_0,\lambda^\sharp}(\rH_1(A_0(\dC),\dZ_p),\rH_1(A^\sharp(\dC),\dZ_p))$ onto $\prod_{\fq\mid p}\Lambda'_{n,\fq}$.
\end{itemize}
We then have successive isometries
\begin{align*}
\rV^\circ\otimes_\dQ\dA^{\infty,p} &\xrightarrow{\eta^{p\circ}}
\Hom_{F\otimes_\dQ\dA^{\infty,p}}^{\lambda_0,\lambda^\circ}(\rH^\et_1(A_0,\dA^{\infty,p}),\rH^\et_1(A^\circ,\dA^{\infty,p}))
\xrightarrow{\rho^\circ} \rV_s\otimes_\dQ\dA^{\infty,p} \xrightarrow{\gamma_s} \rV^\circ\otimes_\dQ\dA^{\infty,p},\\
\rV^\circ\otimes_\dQ\dA^{\infty,p} &\xrightarrow{\eta^p} \Hom_{F\otimes_\dQ\dA^{\infty,p}}^{\lambda_0,\lambda}(\rH^\et_1(A_0,\dA^{\infty,p}),\rH^\et_1(A,\dA^{\infty,p})) \xrightarrow{\rho}\rV_{x^\sharp}\otimes_\dQ\dA^{\infty,p} \xrightarrow{\gamma_{x^\sharp}} \rV'\otimes_\dQ\dA^{\infty,p},
\end{align*}
in which $\rho$ comes from the comparison theorem. We obtain two elements
\[
\gamma_s\circ\rho^\circ\circ\eta^{p\circ},\quad \tj^{-1}\circ\gamma_{x^\sharp}\circ\rho\circ\eta^p
\]
in $\rU(\rV^\circ)(\dA_{F^+}^{\infty,p})/\rK^{p\circ}$.

\begin{lem}\label{le:determinant_first}
There exists a unique element $[\tj]\in\bG_{F/F^+}(F^+)\backslash\bG_{F/F^+}(\dA_{F^+}^\infty)/\bG_{F/F^+}(O_{F^+}\otimes\dZ_p)$ such that for every choice of data $\rK^{p\circ},(y,x^\sharp),(\gamma_s,\gamma_{x^\sharp})$ as above, the quotient
\[
\frac{\det(\tj^{-1}\circ\gamma_{x^\sharp}\circ\rho\circ\eta^p)}{\det(\gamma_s\circ\rho^\circ\circ\eta^{p\circ})}
\in \bG_{F/F^+}(\dA_{F^+}^{\infty,p})/\det\rK^{p\circ}
\]
coincides with $[\tj]$ in $\bG_{F/F^+}(F^+)\backslash\bG_{F/F^+}(\dA_{F^+}^\infty)/\bG_{F/F^+}(O_{F^+}\otimes\dZ_p)\det\rK^{p\circ}$.
\end{lem}

\begin{proof}
The proof is similar to that of Lemma \ref{le:determinant_first} and we omit.
\end{proof}

\begin{remark}
It is straightforward to see that the element $[\ti]$ (resp.\ $[\tj]$) from Lemma \ref{le:determinant_second} (resp.\ Lemma \ref{le:determinant_first}) is independent of the choice of the auxiliary Shimura data defining $\bT$.
\end{remark}

\subsection{Appendix: Intertwining Hecke operator in the nonsplit case}
\label{ss:intertwining}

In this subsection, we prove a result that is parallel to \cite{LTXZZ}*{Proposition~B.3.5(1)} in the case where the hermitian space is \emph{nonsplit}. We fix an unramified quadratic extension $F/F^+$ of nonarchimedean local fields and a uniformizer $\varpi$ of $F^+$. Let $q$ be the residue cardinality of $F^+$ and $\fp$ the maximal ideal of $O_F$. Let $N=2r$ be an even positive integer. For every integer $m\geq 0$, denote by $\rI_m$ and $\rJ_m$ the identity and anti-identity matrices of rank $m$, respectively.

Consider a hermitian space $\rV_N$ over $F$ (with respect to $F/F^+$) of rank $N$ together with a basis $\{e_1,\ldots,e_N\}$ under which the hermitian form is given by
\[
\begin{pmatrix}
&&& \rJ_{r-1} \\
& \varpi && \\
&& 1 & \\
\rJ_{r-1} &&&
\end{pmatrix}.
\]
Via this basis, we identify $\rU(\rV_N)$ as a closed subgroup of $\Res_{F/F^+}\GL_N$. Denote by $\rB_N$ the minimal parabolic subgroup of $\rU(\rV_N)$ stabilizing the chain $Fe_1\subseteq Fe_1\oplus Fe_2\subseteq\cdots\subseteq Fe_1\oplus\cdots\oplus Fe_{r-1}$. Then the Levi quotient of $\rB_N$ is canonically isomorphic to $\bG_{m,F}^{r-1}\times\rU(\rV_2)$ giving by the actions on $e_1,\ldots,e_{r-1},\rV_2$, where $\rV_2\coloneqq Fe_r\oplus Fe_{r+1}$. For a $\dZ[q^{-1}]$-ring $L$ and an $N$-tuple $\balpha=(\alpha_1,\ldots,\alpha_N)\in L^N$ satisfying $\alpha_i\alpha_{N+1-i}=1$ and $\alpha_r=q$, we have the normalized principal series $\pi_{\balpha}$ of $\rU(\rV_N)(F^+)$ (with coefficients in $L$) of the character $\boxtimes_{i=1}^{r-1}\chi_{\alpha_i}\boxtimes\CF$, where $\chi_{\alpha_i}$ is the unramified character of $F^\times$ sending $\varpi$ to $\alpha_i$.

Consider two lattices
\begin{align*}
\Lambda'_N=O_Fe_1\oplus\cdots\oplus O_F e_N,\quad
\Lambda^\star_N=\fp^{-1}e_1\oplus\cdots\oplus \fp^{-1}e_r\oplus O_F e_{r+1} \oplus\cdots\oplus O_F e_N
\end{align*}
of $\rV_N$. Then $\Lambda'_N\subseteq(\Lambda'_N)^\vee$ such that $(\Lambda'_N)^\vee/\Lambda'_N$ has $O_F$-length $1$; $\fp\Lambda^\star_N\subseteq(\Lambda^\star_N)^\vee$ such that $(\Lambda^\star_N)^\vee/\fp\Lambda^\star_N$ has $O_F$-length $1$. Let $\rK'_N$ and $\rK^\star_N$ be the stabilizers of $\Lambda'_N$ and $\Lambda^\star_N$, respectively, which are both special maximal subgroups of $\rU(\rV_N)(F^+)$. We have two \emph{commutative} Hecke algebras
\[
\dT'_N\coloneqq\dZ[\rK'_N\backslash\rU(\rV_N)(F^+)/\rK'_N],\quad
\dT^\star_N\coloneqq\dZ[\rK^\star_N\backslash\rU(\rV_N)(F^+)/\rK^\star_N],
\]
with units $\CF_{\rK'_N}$ and $\CF_{\rK^\star_N}$, respectively. Denote by $\tT^{\prime\star}_N\in\dZ[\rK'_N\backslash\rU(\rV_N)(F^+)/\rK^\star_N]$ and $\tT^{\star\prime}_N\in\dZ[\rK^\star_N\backslash\rU(\rV_N)(F^+)/\rK'_N]$ the characteristic function of the cosets $\rK'_N\rK^\star_N$ and $\rK^\star_N\rK'_N$, respectively. Put $\tI'_N\coloneqq\tT^{\prime\star}_N\circ\tT^{\star\prime}_N\in\dT'_N$.

\begin{proposition}\label{pr:intertwining}
Let $L$ be a $\dZ[q^{-1}]$-ring and $\balpha=(\alpha_1,\ldots,\alpha_N)\in L^N$ an $N$-tuple satisfying $\alpha_i\alpha_{N+1-i}=1$ and $\alpha_r=q$. Then $\tI'_N$ acts on $(\pi_{\balpha})^{\rK'_N}$ by the scalar
\[
q^{r^2-1}\prod_{i=1}^{r-1}\(\alpha_i+\frac{1}{\alpha_i}+2\).
\]
\end{proposition}

\begin{proof}
Denote by $f'$ and $f^\star$ the unique element in $\pi_{\balpha}$ satisfying $f'\res_{\rK'_N}\equiv 1$ and $f^\star\res_{\rK^\star_N}\equiv 1$, respectively. Then $(\pi_{\balpha})^{\rK'_N}$ and $(\pi_{\balpha})^{\rK^\star_N}$ are free $L$-modules generated by $f'$ and $f^\star$, respectively. The proposition follows from the following two identities:
\begin{align}\label{eq:intertwining1}
\tT^{\prime\star}_N f^\star = \prod_{i=1}^{r-1}(1+\alpha_i)\cdot f',
\end{align}
\begin{align}\label{eq:intertwining2}
\tT^{\star\prime}_N f' = q^{r^2-1}\prod_{i=1}^{r-1}(1+\alpha_i^{-1})\cdot f^\star.
\end{align}

Put $\rB'_N\coloneqq\rB_N(F)\cap\rK'_N$, $\rB^\star_N\coloneqq\rB_N(F)\cap\rK^\star_N$, and $\rK_N^\dag\coloneqq\rK'_N\cap\rK^\star_N$. For $i\in\{1,\ldots,r-1\}$, put
\[
s'_i\coloneqq
\begin{pmatrix}
\rI_{i-1} &&&& \\
&&& 1 & \\
&& \rI_{2r-2i} && \\
& 1 &&& \\
&&&& \rI_{i-1}
\end{pmatrix}
,\qquad
s^\star_i\coloneqq
\begin{pmatrix}
\rI_{i-1} &&&& \\
&&& \varpi^{-1} & \\
&& \rI_{2r-2i} && \\
& \varpi &&& \\
&&&& \rI_{i-1}
\end{pmatrix}
\]
which belong to $\rK'_N$ and $\rK^\star_N$, respectively. For every subset $I\subseteq\{1,\ldots,r-1\}$, put
\[
s'_I\coloneqq\prod_{i\in I}s'_i,\quad
s^\star_I\coloneqq\prod_{i\in I}s^\star_i.
\]

For \eqref{eq:intertwining1}, we have natural bijections
\[
\rK'_N\rK^\star_N/\rK^\star_N=\rK'_N/\rK^\dag_N=\prod_{I\subseteq\{1,\ldots,r-1\}}\rB'_N s'_I \rK^\dag_N/\rK^\dag_N.
\]
Then
\begin{align*}
\tT^{\prime\star}_N f^\star(1)&=\sum_{g\in\rK'_N\rK^\star_N/\rK^\star_N}f^\star(g) \\
&=\sum_{I\subseteq\{1,\ldots,r-1\}}\sum_{u\in\rB'_N/(\rB'_N\cap s'_I\rK^\dag_N s'_I)}f^\star(us'_I) \\
&=\sum_{I\subseteq\{1,\ldots,r-1\}}\sum_{u\in\rB'_N/(\rB'_N\cap s'_I\rK^\dag_N s'_I)}f^\star(us'_Is^\star_I) \\
&=\sum_{I\subseteq\{1,\ldots,r-1\}}\sum_{u\in\rB'_N/(\rB'_N\cap s'_I\rK^\dag_N s'_I)}\prod_{i\in I}\alpha_i q^{2i-2r-1} \\
&=\sum_{I\subseteq\{1,\ldots,r-1\}}\left|\rB'_N/(\rB'_N\cap s'_I\rK^\dag_N s'_I)\right|\cdot\prod_{i\in I}\alpha_i q^{2i-2r-1}.
\end{align*}
It is straightforward to check that the cardinality of $\rB'_N/(\rB'_N\cap s'_I\rK^\dag_N s'_I)$ is $\prod_{i\in I}q^{2r+1-2i}$, which implies \eqref{eq:intertwining1}.

For \eqref{eq:intertwining2}, we have natural bijections
\[
\rK^\star_N\rK'_N/\rK'_N=\rK^\star_N/\rK^\dag_N=\prod_{I\subseteq\{1,\ldots,r-1\}}\rB^\star_N s^\star_I \rK^\dag_N/\rK^\dag_N.
\]
Then
\begin{align*}
\tT^{\star\prime}_N f'(1)&=\sum_{g\in\rK^\star_N\rK'_N/\rK'_N}f'(g) \\
&=\sum_{I\subseteq\{1,\ldots,r-1\}}\sum_{u\in\rB^\star_N/(\rB^\star_N\cap s^\star_I\rK^\dag_N s^\star_I)}f'(us^\star_I) \\
&=\sum_{I\subseteq\{1,\ldots,r-1\}}\sum_{u\in\rB^\star_N/(\rB^\star_N\cap s^\star_I\rK^\dag_N s^\star_I)}f'(us^\star_Is'_I) \\
&=\sum_{I\subseteq\{1,\ldots,r-1\}}\sum_{u\in\rB^\star_N/(\rB^\star_N\cap s^\star_I\rK^\dag_N s^\star_I)}\prod_{i\in I}\alpha_i^{-1} q^{2r+1-2i} \\
&=\sum_{I\subseteq\{1,\ldots,r-1\}}\left|\rB^\star_N/(\rB^\star_N\cap s^\star_I\rK^\dag_N s^\star_I)\right|\cdot\prod_{i\in I}\alpha_i^{-1} q^{2r+1-2i}.
\end{align*}
It is straightforward to check that the cardinality of $\rB^\star_N/(\rB^\star_N\cap s^\star_I\rK^\dag_N s^\star_I)$ is $\prod_{i\in\{1,\ldots,r-1\}\setminus I}q^{2r+1-2i}$, which implies \eqref{eq:intertwining2}.

The proposition is proved.\footnote{A similar argument in the split case gives an alternative proof of \cite{LTXZZ}*{Proposition~B.3.5(1)}.}
\end{proof}

\section{Iwasawa's main conjecture}
\label{ss:5}

In this section, we give the precise statement and the proof of our main theorems toward the Iwasawa main conjecture formulated in \cite{Liu5}.

We keep the setup in \S\ref{ss:setup} and denote by $\epsilon(\Pi_0\times\Pi_1)\in\{\pm1\}$ the global Rankin--Selberg epsilon factor. To meet with the common practice in the Iwasawa theory, in this section, we will turn $\ell$-adic coefficients to $p$-adic ones for the Galois representation, and turn $p$-adic special inert primes to $\ell$-adic ones. Throughout this section, $F'$ will always stand for a \emph{finite} extension of $F$.

\subsection{Results toward the main conjecture}
\label{ss:main}

Take a prime $\wp$ of $E$ with the underlying rational prime $p$ coprime to $\Sigma^+_\mnm$, and denote by $O_\wp$ the ring of integers of $E_\wp$ (for which we fix a uniformizer also denoted by $\wp$).

Take a free \emph{$\wp$-ordinary} anticyclotomic extension $\cF/F$ (with respect to $\Pi_0,\Pi_1$), which, we recall from \S\ref{ss:main2}, means that
\begin{itemize}[label={\ding{118}}]
  \item $\Gal(\cF/F)$ is torsion free;

  \item every prime $w\in\Sigma_\cF$ divides $p$ (so that $\Sigma_\cF\subseteq\Sigma_p$), splits over $F^+$, and satisfies that both $\Pi_{0,w}$ and $\Pi_{1,w}$ are unramified and $\wp$-ordinary (Definition \ref{de:ordinary}).
\end{itemize}
Put $\Upsilon_\cF\coloneqq\Gal(\cF/F)$, $\Lambda_\cF\coloneqq O_\wp[[\Upsilon_\cF]]$ as a (compact) topological ring, and $\Lambda_{\cF,E_\wp}\coloneqq\Lambda_\cF\otimes_{O_\wp}E_\wp$.

We first define two Iwasawa limits of Selmer groups associated with $\rho_{\Pi_0,\wp}\otimes\rho_{\Pi_1,\wp}(n)$. Denote by $V$ the underlying continuous $E_\wp[\Gamma_F]$-module of the representation $\rho_{\Pi_0,\wp}\otimes\rho_{\Pi_1,\wp}(n)$. By \cite{Car14}*{Theorem~1.1}, we know that $V$ is pure of weight $-1$ at every prime of $F$. Choose a $\Gamma_F$-stable $O_\wp$-lattice $T$ of $V$ and put $W\coloneqq V/T$.

For every (finite) extension $F'/F$ and every place $w'$ of $F'$, define the Bloch--Kato Selmer condition as
\begin{align}\label{eq:selmer}
\rH^1_f(F'_{w'},V)\coloneqq
\begin{dcases}
\rH^1(F'_{w'},V)\:(=0),& \text{if $w'$ is not above $p$},\\
\Ker\(\rH^1(F'_{w'},V)\to\rH^1(F'_{w'},V\otimes_{\dQ_p}\dB_\cris)\),&\text{if $w'$ is above $p$}.
\end{dcases}
\end{align}
We then define $\rH^1_f(F'_{w'},W)$ (resp.\ $\rH^1_f(F'_{w'},T)$) to be the propagation of $\rH^1_f(F'_{w'},V)$, that is, the image (resp.\ preimage) of $\rH^1_f(F'_{w'},V)$ along the natural map $\rH^1(F'_{w'},V)\to\rH^1(F'_{w'},W)$ (resp.\ $\rH^1(F'_{w'},T)\to\rH^1(F'_{w'},V)$). The global Bloch--Kato Selmer groups are defined as
\begin{align*}
\rH^1_f(F',W)&\coloneqq\Ker\(\rH^1(F',W)\to\prod_{w'}\frac{\rH^1(F'_{w'},W)}{\rH^1_f(F'_{w'},W)}\), \\
\rH^1_f(F',T)&\coloneqq\Ker\(\rH^1(F',T)\to\prod_{w'}\frac{\rH^1(F'_{w'},T)}{\rH^1_f(F'_{w'},T)}\).
\end{align*}
Put
\begin{align*}
\sX(\cF,W)&\coloneqq\varinjlim_{F\subseteq F'\subseteq\cF}\rH^1_f(F',W),\\
\sS(\cF,T)&\coloneqq\varprojlim_{F\subseteq F'\subseteq\cF}\rH^1_f(F',T),
\end{align*}
as $\Lambda_\cF$-modules, where the transition maps are the restriction maps and the corestriction maps, respectively. Finally, define $\sX(\cF,T)$ to be the continuous $O_\wp$-linear Pontryagin dual of $\sX(\cF,T)$. It is a standard fact that the $\Lambda_\cF$-modules $\sS(\cF,T)$ and $\sX(\cF,T)$ are both finitely generated.

\begin{definition}\label{de:selmer}
We define
\begin{align*}
\sX(\cF,\rho_{\Pi_0,\wp}\otimes\rho_{\Pi_1,\wp}(n))&\coloneqq\sX(\cF,T)\otimes_{O_\wp}E_\wp,\\
\sS(\cF,\rho_{\Pi_0,\wp}\otimes\rho_{\Pi_1,\wp}(n))&\coloneqq\sS(\cF,T)\otimes_{O_\wp}E_\wp,
\end{align*}
which are finitely generated $\Lambda_{\cF,E_\wp}$-modules independent of the choice of $T$. Denote by $\sX_0(\cF,\rho_{\Pi_0,\wp}\otimes\rho_{\Pi_1,\wp}(n))$ the maximal $\Lambda_{\cF,E_\wp}$-torsion submodule of $\sX(\cF,\rho_{\Pi_0,\wp}\otimes\rho_{\Pi_1,\wp}(n))$.
\end{definition}

In \cite{Liu5}, the author defined
\begin{itemize}[label={\ding{118}}]
  \item when $\epsilon(\Pi_0\times\Pi_1)=1$, an anticyclotomic $\wp$-adic $L$-function $\sL_\cF(\Pi_0\times\Pi_1)\in\Lambda_{\cF,E_\wp}$ (which will be recalled in \S\ref{ss:lambda}; see Proposition \ref{pr:function} and Definition \ref{no:lambda}), and

  \item when $\epsilon(\Pi_0\times\Pi_1)=-1$, a $\Lambda_{\cF,E_\wp}$-submodule $\sK(\cF,\rho_{\Pi_0,\wp}\otimes\rho_{\Pi_1,\wp}(n))\subseteq\sS(\cF,\rho_{\Pi_0,\wp}\otimes\rho_{\Pi_1,\wp}(n))$ (which will be recalled in \S\ref{ss:kappa}; see Definition \ref{de:kappa}); and we put
      \[
      \sT(\cF,\rho_{\Pi_0,\wp}\otimes\rho_{\Pi_1,\wp}(n))\coloneqq
      \frac{\sS(\cF,\rho_{\Pi_0,\wp}\otimes\rho_{\Pi_1,\wp}(n))}{\sK(\cF,\rho_{\Pi_0,\wp}\otimes\rho_{\Pi_1,\wp}(n))}.
      \]
\end{itemize}

\begin{definition}\label{de:admissible}
We say that a prime $\wp$ of $E$ with the underlying rational prime $p$ is \emph{admissible} (with respect to $(\Pi_0,\Pi_1)$) if
\begin{description}
  \item[(A1)] $p>2(n_0+1)$;

  \item[(A2)] $\Sigma^+_\mnm$ does not contain $p$-adic places (in particular, $p$ is unramified in $F$);

  \item[(A3)] $\rho_{\Pi_0,\wp}\otimes\rho_{\Pi_1,\wp}$ is \emph{residually} absolutely irreducible;\footnote{Note that this is stronger that (L3) and (L4) in \cite{LTXZZ}*{Definition~8.1.1} combined.}

  \item[(A4)] for $\alpha=0,1$, there exists a prime $w_\alpha$ of $F$ not in $\Sigma_\mnm\cup\Sigma_p$ that splits over $F^+$, such that the elements of the Satake parameter $\balpha(\Pi_{\alpha,w_\alpha})$ are distinct for which $\|w\|$ is not a ratio;

  \item[(A5)] same as (L5) in \cite{LTXZZ}*{Definition~8.1.1} but with $\sP(T)=\prod_{i=1}^{n_0}(1-(-T)^i)$ in (L5-2) (so that the condition becomes stronger);

  \item[(A6)] same as (L6) in \cite{LTXZZ}*{Definition~8.1.1}.
\end{description}
\end{definition}

\begin{theorem}\label{th:iwasawa}
Assume Hypothesis \ref{hy:unitary_cohomology} for both $n$ and $n+1$. For every admissible prime $\wp$ of $E$ in the sense of Definition \ref{de:admissible} and every free $\wp$-ordinary anticyclotomic extension $\cF/F$, the following holds (while invoking Notation \ref{no:iwasawa}).
\begin{enumerate}
  \item Suppose that $\epsilon(\Pi_0\times\Pi_1)=1$. If $\sL_\cF(\Pi_0\times\Pi_1)\neq 0$, then
     \begin{enumerate}
       \item $\sX(\cF,\rho_{\Pi_0,\wp}\otimes\rho_{\Pi_1,\wp}(n))=\sX_0(\cF,\rho_{\Pi_0,\wp}\otimes\rho_{\Pi_1,\wp}(n))$;

       \item $\sS(\cF,\rho_{\Pi_0,\wp}\otimes\rho_{\Pi_1,\wp}(n))$ vanishes;

       \item $\sL_\cF(\Pi_0\times\Pi_1)$ belongs to $\Char_\cF\(\sX(\cF,\rho_{\Pi_0,\wp}\otimes\rho_{\Pi_1,\wp}(n))\)$.
     \end{enumerate}

  \item Suppose that $\epsilon(\Pi_0\times\Pi_1)=-1$. If $\sK(\cF,\rho_{\Pi_0,\wp}\otimes\rho_{\Pi_1,\wp}(n))\neq 0$, then
     \begin{enumerate}
       \item $\sX(\cF,\rho_{\Pi_0,\wp}\otimes\rho_{\Pi_1,\wp}(n))$ has $\Lambda_{\cF,E_\wp}$-rank one;

       \item $\sS(\cF,\rho_{\Pi_0,\wp}\otimes\rho_{\Pi_1,\wp}(n))$ is a torsion free $\Lambda_{\cF,E_\wp}$-module of rank one;

       \item $\Char_\cF\(\sT(\cF,\rho_{\Pi_0,\wp}\otimes\rho_{\Pi_1,\wp}(n))\)^2$ is contained in $\Char_\cF\(\sX_0(\cF,\rho_{\Pi_0,\wp}\otimes\rho_{\Pi_1,\wp}(n))\)$.
     \end{enumerate}
\end{enumerate}
\end{theorem}

This theorem implies Theorem \ref{th:main} in view of Remark \ref{pr:unitary_cohomology}. The proof of Theorem \ref{th:iwasawa} will be given in \S\ref{ss:euler}. Meanwhile, we prove Theorem \ref{th:admissible}.

\begin{proof}[Proof of Theorem \ref{th:admissible}]
Case (1) has been explained in Remark \ref{re:elliptic}.

For Case (2), (A1,A2) clearly exclude only finitely primes; (A5,A6) exclude only finitely primes by the same argument for (L5,L6) in \cite{LTXZZ}*{Lemma~8.1.4}; and (A4) excludes only finitely primes by the proof of \cite{LTXZZ}*{Corollary~D.1.4}. It remains to consider (A3). By \cite{LTXZZ}*{Proposition~4.2.3(1)} and the condition (2b), for all but finitely many $\wp$, both $\rho_{\Pi_0,\wp}$ and $\rho_{\Pi_1,\wp}$ are residually absolutely irreducible. Now by the condition (2a) and \cite{LTXZZ}*{Lemma~8.1.5}, (A3) holds for all but finitely many $\wp$.

For Case (3), (A1,A2) clearly exclude only finitely primes; (A5,A6) exclude only finitely primes by the same argument for (L5,L6) in \cite{LTXZZ}*{Lemma~8.1.4};\footnote{Indeed, for (A5), we need to use (3a) instead of (2a) now, which can be argued in the same manner; we leave the details to the readers.} and (A4) excludes only finitely primes by the proof of \cite{LTXZZ}*{Corollary~D.1.4}. It remains to consider (A3). By \cite{LTXZZ}*{Proposition~4.2.3(1)} and the condition (3c), for all but finitely many $\wp$, both $\rho_{\Pi_0,\wp}$ and $\rho_{\Pi_1,\wp}$ are residually absolutely irreducible. Now by the condition (3b) and \cite{LTXZZ}*{Lemma~8.1.5}, (A3) holds for all but finitely many $\wp$.

The theorem is proved.
\end{proof}

In the rest of this subsection, we make some early preparation.

\begin{lem}\label{le:generic}
Suppose that $\wp$ satisfies (A3,A4) in Definition \ref{de:admissible}. Then for $\alpha=0,1$, the composite homomorphism $\dT^{\Sigma^+_\mnm}_{n_\alpha}\xrightarrow{\phi_{\Pi_\alpha}}O_E\to O_E/\wp$ is cohomologically generic in the sense of \cite{LTXZZ}*{Definition~D.1.1}.
\end{lem}

\begin{proof}
When $F^+\neq\dQ$, this has been proved in \cite{LTXZZ}*{Proposition~D.1.3}. We assume $F^+=\dQ$. Take an arbitrary element $N\in\{n,n+1\}$, and write $\Pi$ for $\Pi_0$ or $\Pi_1$, depending on whether $N=n_0$ or $N=n_1$. We need to prove that for every standard indefinite hermitian space $\rV$ over $F$ of rank $N$ and every decomposable neat open compact subgroup $\rK\subseteq\rU(\rV)(\dA^\infty)$ that is hyperspecial away from a finite subset $\Box$ of rational primes containing $\Sigma^+_\mnm$,
\begin{align}\label{eq:generic}
\rH^i(\Sh(\rV,\rK)_{\ol{F}},O_\wp)_{\fm^\Box}=0,\quad i\neq N-1,
\end{align}
where $\fm^\Box$ denotes the kernel of the composite map $\dT^\Box_N\xrightarrow{\phi_\Pi}O_E\to O_E/\wp$. We would like to invoke the very recent result on the vanishing of torsion cohomology \cite{HL}. However, there is a small caveat that our Shimura variety $\Sh(\rV,\rK)$ is not of PEL type. To overcome this issue, we pass the problem to the Shimura variety for $\GU(\rV)$, which is of PEL type.

Denote by $\fK_\rK$ the set of decomposable neat open compact subgroups $\rK^\rG\subseteq\GU(\rV)(\dA^\infty)$ satisfying $\rK^\rG\cap\rU(\rV)(\dA^\infty)=\rK$. For every $\rK^\rG\in\fK_\rK$, we denote by $\Sh^\rG(\rV,\rK^\rG)$ the corresponding PEL type Shimura variety, defined over $F$ (or $\dQ$ when $N=2$); and we have a canonical morphism $\Sh(\rV,\rK)_{\ol{F}}\to\Sh^\rG(\rV,\rK^\rG)_{\ol{F}}$ (which does \emph{not} descend to $F$ in general). It is easy to see that one can find an element $\rK^\rG\in\fK_\rK$ such that $\Sh(\rV,\rK)_{\ol{F}}$ is identified with a connected component of $\Sh^\rG(\rV,\rK^\rG)_{\ol{F}}$. By the Chebotarev density theorem, we may choose a prime $w$ as in (A4) whose underlying rational prime $\ul{w}\not\in\Box$. Let $\dT_{N,\ul{w}}^\rG$ be the spherical Hecke algebra of $\GU(\rV_{\ul{w}})$, which is canonically a polynomial ring over $\dT_{N,\ul{w}}$ of one variable (coming from the similitude factor). Let $\fm_{\ul{w}}$ be the kernel of the composite map $\dT_{N,\ul{w}}\xrightarrow{\phi_\Pi}O_E\to O_E/\wp$ and $\fm_{\ul{w}}^\rG$ the maximal ideal of $\dT_{N,\ul{w}}^\rG$ that is the inflation of $\fm_{\ul{w}}$. Now we show \eqref{eq:generic}. By \cite{LTXZZ}*{Lemma~6.1.11} (which is applicable by (A3)) and the Poincar\'{e} duality, it suffices to show the vanishing of
$\rH^i(\Sh(\rV,\rK)_{\ol{F}},O_\wp)_{\fm^\Box}$ (for both $\Pi$ and $\Pi^\vee$), or rather the vanishing of $\rH^i(\Sh(\rV,\rK)_{\ol{F}},O_\wp)_{\fm_{\ul{w}}}$, for $i<N$. However, the latter is a direct summand of $\rH^i(\Sh^\rG(\rV,\rK^\rG)_{\ol{F}},O_\wp)_{\fm_{\ul{w}}^\rG}$, which vanishes for $i<N$ by \cite{HL}*{Theorem~1.3} (which is applicable by (A4)).

The lemma is proved.
\end{proof}

\begin{definition}\label{de:conjugation}
Let $M$ be a continuous $R[\Gamma_F]$-module over a certain (topological) coefficient ring $R$. Let $F'/F$ be a finite extension contained in $\cF$.
\begin{enumerate}
  \item We define the \emph{conjugation isomorphism} to be the isomorphism
      \[
      \rH^1(F',M)\simeq\rH^1(F',M^\tc)
      \]
      obtained by sending an extension to its $\tc$-conjugate, which is $R$-linear but \emph{inverse} $\Gal(F'/F)$-equivariant.

  \item Similarly, for every place $w$ of $F$, we have the local \emph{conjugation isomorphism}
      \[
      \bigoplus_{w'\mid w}\rH^1(F'_{w'},M)\simeq\bigoplus_{w'\mid w^\tc}\rH^1(F'_{w'},M^\tc).
      \]
\end{enumerate}
\end{definition}

\begin{notation}[Hodge--Tate filtration]\label{no:ordinary}
Take a prime $w\in\Sigma_\cF$ and $?\in\{\tc,\;\}$. Since $V^?$ is ordinary crystalline at $w$, there exists a unique increasing filtration
$\Fil_w^\bullet T^?$ (indexed by $\dZ$) of $O_\wp[\Gamma_{F_w}]$-modules satisfying that
\begin{itemize}[label={\ding{118}}]
  \item for every $j\in\dZ$, $\Gr_w^jT^?\coloneqq\Fil_w^jT^?/\Fil_w^{j-1}T^?$ is a free $O_\wp$-module on which the inertia subgroup $\rI_{F_w}$ acts via $\chi_\cyc^{-j}$, where $\chi_\cyc\colon\Gamma_F\to\dZ_p^\times$ is the $p$-adic cyclotomic character,

  \item $\Fil_w^jT^?=0$ for $j<-n$ and $\Fil_w^jT^?=T^?$ for $j>n-1$.
\end{itemize}
We similarly define $(\Fil_w^\bullet V^?,\Gr_w^\bullet V^?)$ and $(\Fil_w^\bullet W^?,\Gr_w^\bullet W^?)$ in the obvious way.
\end{notation}

\begin{lem}\label{le:crystalline}
For every $w\in\Sigma_\cF$ and every finite extension $K/F_w$, the sequence
\begin{align*}
\rH^1(K,\Fil_w^{-1}V^\tc)\to\rH^1(K,V^\tc)\to\rH^1(K,V^\tc\otimes_{\dQ_p}\dB_\cris)
\end{align*}
is exact.
\end{lem}

\begin{proof}
Since the $E_\wp[\Gamma_K]$-module $V^\tc$ is pure of weight $-1$ and both $\Pi_0$ and $\Pi_1$ are $\wp$-ordinary at $w$, $V^\tc$ satisfies the Panchishkin condition in \cite{Nek93}*{6.7} with the filtration $\Fil_w^{-1}V^\tc\subseteq V^\tc$. Now the lemma follows from the first sentence in the proof of \cite{Nek93}*{Lemma~6.8}.
\end{proof}

\subsection{Iwasawa Selmer modules}
\label{ss:selmer}

In this section, we define and study Selmer groups of Galois modules over the Iwasawa algebra. Let $\wp$ and $\cF/F$ be as in \S\ref{ss:main}.

Let $\Upsilon$ be the maximal torsion free quotient of
\[
\bG_{F/F^+}(F^+)\backslash\bG_{F/F^+}(\dA_{F^+}^\infty)\left/
\prod_{v\not\in\Sigma^+_\infty\cup\Sigma_\cF^+}\bG_{F/F^+}(O_{F^+_v})\right..
\]
Put $\Lambda\coloneqq O_\wp[[\Upsilon]]$ as a (compact) topological ring, and $\Lambda_{E_\wp}\coloneqq\Lambda\otimes_{O_\wp}E_\wp$. Denote by $\dag$ the involution on $\Lambda$ and $\Lambda_{E_\wp}$ induced by the inverse on $\Upsilon$. We have the natural continuous surjective homomorphism $\Lambda\to\Lambda_\cF$ induced from the natural quotient $\Upsilon\to\Upsilon_\cF$. For every finite place $w$ of $F$, denote by $\Upsilon_{\cF,w}$ the image of $\Gamma_{F_w}$ in $\Upsilon_\cF$, that is, the decomposition subgroup of $w$ in $\Upsilon_\cF$.

\begin{notation}
Let $\Theta$ be a $\Lambda$-ring that is a finitely generated $\Lambda$-module, equipped with the induced topology.
\begin{enumerate}
  \item We denote by $\Theta^\dag$ the $\Lambda$-ring with the same underlying ring as $\Theta$ but the action of $\Lambda$ twisted by $\dag$.

  \item For every continuous $O_\wp[\Gamma_F]$-module $M$, put $M_\Theta\coloneqq M\otimes_{O_\wp}\Theta$, regarded as a $\Theta[\Gamma_F]$-module on which the action of $\Gamma_F$ is twisted by the character $\Gamma_F\to\Upsilon\to\Lambda^\times\to\Theta^\times$.

  \item In what follows, the notion $M^\tc_\Theta$ stands for $(M^\tc)_\Theta$, rather than $(M_\Theta)^\tc$ (which is isomorphic to $(M^\tc)_{\Theta^\dag}$ as a $\Theta[\Gamma_F]$-module).
\end{enumerate}
\end{notation}

The lemma below shows that $\rH^1(F_v,M_\Theta)$ behaves nicely for a place $v$ of $F^+$ that is inert in $F$.

\begin{lem}\label{le:inert}
Let $w$ be a nonarchimedean place of $F$ that is inert over $F^+$. For every continuous $O_\wp[\Gamma_F]$-module $M$ and every $\Lambda$-ring $\Theta$, the natural map $M\to M_\Theta$ induces an isomorphism
\[
\rH^1(F_w,M_\Theta)=\rH^1(F_w,M)\otimes_{O_\wp}\Theta
\]
of $\Theta$-modules.
\end{lem}

\begin{proof}
This is a consequence of Lemma \ref{le:split}, which implies that the decomposition group of $F$ at $w$ maps to the identity in $\Upsilon$.
\end{proof}

Recall that $V$ is the underlying continuous $E_\wp[\Gamma_F]$-module of the representation $\rho_{\Pi_0,\wp}\otimes\rho_{\Pi_1,\wp}(n)$. Fix a $\Gamma_F$-stable $O_\wp$-lattice $T$ of $V$ together with a perfect symmetric $\Gamma_F$-invariant pairing
\begin{align}\label{eq:pairing}
(\;,\;)\colon T \times T^\tc \to O_\wp(1).
\end{align}
Put $W\coloneqq V/T$.

Take a $\Lambda$-ring $\Theta$ that is a finitely generated $\Lambda$-module (for example, $\Lambda_\cF$), equipped with the induced topology. Denote by $W^\Theta$ and $V^\Theta$ the continuous Cartier and Tate duals of $T^\tc_\Theta$ and $V^\tc_\Theta$, that is,
\[
W^\Theta\coloneqq\Hom_{O_\wp}\(T^\tc_\Theta,E_\wp/O_\wp(1)\),\quad
V^\Theta\coloneqq\Hom_{O_\wp}\(V^\tc_\Theta,E_\wp(1)\)
\]
(continuous hom), respectively. It is clear from the pairing $(\;,\;)_\Theta$ that when $\Theta$ is the ring of integers of a finite extension of $E_\wp$, we have canonical $\Gamma_F$-equivariant isomorphisms $W^\Theta\simeq W_{\Theta^\dag}$ and $V^\Theta\simeq V_{\Theta^\dag}$.For every nonarchimedean place $w$ of $F$, we have the induced (continuous) local Tate pairing
\begin{align}\label{eq:tate}
\langle\;,\;\rangle_w &\colon \rH^1(F_w,W^\Theta)\times\rH^1(F_w,T^\tc_\Theta)\to\rH^2(F_w,E_\wp/O_\wp(1))=E_\wp/O_\wp(1),
\end{align}
which is perfect.

\begin{remark}\label{re:conjugation}
Shapiro's lemma gives canonical isomorphisms
\begin{align*}
\rH^1(F,T^\tc_{\Lambda_\cF^\dag})&=\(\varprojlim_{F\subseteq F'\subseteq\cF}\rH^1(F',T^\tc)\)\otimes_{\Lambda_\cF,\dag}\Lambda_\cF,\\
\rH^1(F,W^{\Lambda_\cF^\dag})&=\varinjlim_{F\subseteq F'\subseteq\cF}\rH^1(F',W),\\
\rH^1(F,V^\tc_{\Lambda_\cF^\dag})&=\rH^1(F,T^\tc_{\Lambda_\cF^\dag})\otimes_{O_\wp}E_\wp=
\(\varprojlim_{F\subseteq F'\subseteq\cF}\rH^1(F',T^\tc)\)\otimes_{\Lambda_\cF,\dag}\Lambda_{\cF,E_\wp}
\end{align*}
of $\Lambda_\cF$-modules (see \cite{Col98}*{Proposition~II.1.1} for more details). In particular, the conjugation isomorphisms (Definition \ref{de:conjugation}) induce canonical isomorphisms
\begin{align*}
\rH^1(F,T^\tc_{\Lambda_\cF^\dag})&\simeq\varprojlim_{F\subseteq F'\subseteq\cF}\rH^1(F',T),\\
\rH^1(F,V^\tc_{\Lambda_\cF^\dag})&=\rH^1(F,T^\tc_{\Lambda_\cF^\dag})\otimes_{O_\wp}E_\wp
\simeq\(\varprojlim_{F\subseteq F'\subseteq\cF}\rH^1(F',T)\)\otimes_{O_\wp}E_\wp
\end{align*}
of $\Lambda_\cF$-modules.\footnote{We warn the readers that $\rH^1(F,V^\tc_{\Lambda_\cF^\dag})$ is not isomorphic to $\varprojlim_{F\subseteq F'\subseteq\cF}\rH^1(F',V)$ in general.} Similar observations apply when we replace $F$ by $F_w$; and in what follows, we simply write $\rH^\bullet_?(F'_w,-)$ instead of $\bigoplus_{w'|w}\rH^\bullet_?(F'_{w'},-)$.
\end{remark}

In view of Remark \ref{re:conjugation}, the lemma below confirms the torsion freeness of $\sS(\cF,V)$, which is a $\Lambda_{\cF,E_\wp}$-submodule of $\rH^1(F,V^\tc_{\Lambda_\cF^\dag})$, in Theorem \ref{th:iwasawa} in a more general setting.

\begin{lem}\label{le:iwasawa1}
Suppose that $T/\wp T$ is an absolutely irreducible representation of $\Gamma_F$. For a $\Lambda$-ring $\Theta$ that is an integral domain, the $\Theta$-module $\rH^1(F,T^\tc_\Theta)$ is torsion free.
\end{lem}

\begin{proof}
For every nonzero element $f\in\Theta$, we have the short exact sequence
\[
0\to T^\tc\otimes_{O_\wp}\Theta\xrightarrow{f\cdot}T^\tc\otimes_{O_\wp}\Theta\to T^\tc\otimes_{O_\wp}\Theta/f\Theta\to 0
\]
as $\Theta$ is integral. It suffices to show that $\rH^0(F,T^\tc\otimes_{O_\wp}\Theta/f\Theta)$ vanishes. Indeed, since $T/\wp T$ is absolutely irreducible of dimension at least two, every $\Gamma_F$-stable subquotient of $T^\tc\otimes_{O_\wp}\Theta/f\Theta$ has trivial invariants. The lemma follows.
\end{proof}

We now give an alternative description for both $\sS(\cF,T)$ and $\sX(\cF,W)$.

\begin{definition}\label{de:selmer0}
Let $?$ be either empty or $\dag$. For every place $w$ of $F$, put
\begin{align*}
\rH^1_\ff(F_w,T^\tc_{\Lambda_\cF^?})\coloneqq
\begin{dcases}
\rH^1(F_w,T^\tc_{\Lambda_\cF^?}),& w\in\Sigma\setminus\Sigma_p,\\
\varprojlim_{F\subseteq F'\subseteq\cF}\rH^1_f(F'_w,T),& w\in\Sigma_p\setminus\Sigma_\cF,\\
\Ker\(\rH^1(F_w,T^\tc_{\Lambda_\cF^?})\to\rH^1(\rI_{F_w},(T^\tc/\Fil^{-1}_wT^\tc)_{\Lambda_\cF^?})\)
,& w\in\Sigma_\cF,\\
\end{dcases}
\end{align*}
and define $\rH^1_\ff(F_w,W^{\Lambda_\cF^?})$ to be the annihilator of $\rH^1_\ff(F_w,T^\tc_{\Lambda_\cF^?})$ under the pairing \eqref{eq:tate}. Globally, put
\begin{align*}
\rH^1_\ff(F,T^\tc_{\Lambda_\cF^?})&=\Ker\(\rH^1(F,T^\tc_{\Lambda_\cF^?})\to
\prod_{w\in\Sigma}\frac{\rH^1(F_w,T^\tc_{\Lambda_\cF^?})}{\rH^1_\ff(F_w,T^\tc_{\Lambda_\cF^?})}\),\\
\rH^1_\ff(F,W^{\Lambda_\cF^?})&=\Ker\(\rH^1(F,W^{\Lambda_\cF^?})\to
\prod_{w\in\Sigma}\frac{\rH^1(F_w,W^{\Lambda_\cF^?})}{\rH^1_\ff(F_w,W^{\Lambda_\cF^?})}\).
\end{align*}
\end{definition}

\begin{lem}\label{le:selmer0}
Assume $p>n$ when $\Sigma_\cF\neq\emptyset$. In view of Remark \ref{re:conjugation}, we have
\begin{align*}
\sS(\cF,T)&=\rH^1_\ff(F,T^\tc_{\Lambda_\cF^\dag}),\\
\sX(\cF,W)&=\rH^1_\ff(F,W^{\Lambda_\cF^\dag}).
\end{align*}
\end{lem}

\begin{proof}
For $\sS(\cF,T)$, it suffices to show that
\[
\rH^1_\ff(F_w,T^\tc_{\Lambda_\cF^\dag})=\varprojlim_{F\subseteq F'\subseteq\cF}\rH^1_f(F'_w,T).
\]
For $w\in\Sigma\setminus\Sigma_p$, this is obvious as $\rH^1_f(F'_w,T)=\rH^1(F'_w,T)$. For $w\in\Sigma_p\setminus\Sigma_\cF$, this is just the definition. Now we take a place $w\in\Sigma_\cF$. By Lemma \ref{le:crystalline}, we have
\[
\rH^1_f(F'_w,T)=\Ker\(\rH^1(F'_w,T)\to\rH^1(F'_w,V^\tc/\Fil^{-1}_wV^\tc)\).
\]
Thus, it suffices to show that the natural maps
\begin{align*}
\rH^1(K,V^\tc/\Fil^{-1}_wV^\tc)&\to\rH^1(\rI_K,V^\tc/\Fil^{-1}_wV^\tc), \\
\rH^1(\rI_K,T^\tc/\Fil^{-1}_wT^\tc)&\to\rH^1(\rI_K,V^\tc/\Fil^{-1}_wV^\tc)
\end{align*}
are both injective for every totally ramified finite extension $K/F_w$ of $p$-power degree. The first one is injective since $V$ is pure of weight $-1$ at $w$. For the second one, it suffices to show that the natural map $\rH^0(\rI_K,V^\tc/\Fil^{-1}_wV^\tc)\to\rH^0(\rI_K,W^\tc/\Fil^{-1}_wW^\tc)$ is surjective. Since $K/F_w$ has $p$-power degree and $p>n$, the previous map is identified with $\Gr_w^0V^\tc\to\Gr_w^0W^\tc$, which is indeed surjective.

For $\sX(\cF,W)$, it follows from the well-known fact that $\rH^1_f(F'_{w'},V)$ and $\rH^1_f(F'_{w'},V^\tc)$ are mutual annihilators under the local Tate pairing $\rH^1(F'_{w'},V)\times\rH^1(F'_{w'},V^\tc)\to E_\wp$ for every place $w'$ of $F'$.
\end{proof}

\begin{definition}\label{de:selmer1}
Let $\Theta$ be a $\Lambda$-ring that is a finite free $O_\wp$-module. For every place $w$ of $F$, put
\[
\rH^1_\ff(F_w,V^\tc_\Theta)\coloneqq
\begin{dcases}
0, & w\in\Sigma_\infty,\\
\Ker\(\rH^1(F_w,V^\tc_\Theta)\to\rH^1(\rI_{F_w},V^\tc_\Theta)\),& w\in\Sigma\setminus(\Sigma_\infty\cup\Sigma_p),\\
\Ker\(\rH^1(F_w,V^\tc_\Theta)\to\rH^1(F_w,V^\tc_\Theta\otimes_{\dQ_p}\dB_\cris)\),& w\in\Sigma_p\setminus\Sigma_\cF, \\
\Ker\(\rH^1(F_w,V^\tc_\Theta)\to\rH^1(\rI_{F_w},(V^\tc/\Fil_w^{-1}V^\tc)_\Theta)\),& w\in\Sigma_\cF,
\end{dcases}
\]
let $\rH^1_\ff(F_w,T^\tc_\Theta)$ be the preimage of $\rH^1_\ff(F_w,V^\tc_\Theta)$, and define $\rH^1_\ff(F_w,W^\Theta)$ to be the annihilator of $\rH^1_\ff(F_w,T^\tc_\Theta)$ under the pairing \eqref{eq:tate}. We then have the corresponding global Selmer groups $\rH^1_\ff(F,V^\tc_\Theta)$, $\rH^1_\ff(F,T^\tc_\Theta)$ and $\rH^1_\ff(F,W^\Theta)$ accordingly.
\end{definition}

\begin{remark}\label{re:selfdual}
For $w\not\in\Sigma_\cF$, $\rH^1_\ff(F_w,V^\tc_\Theta)$ is nothing but the Bloch--Kato local Selmer structure $\rH^1_f(F_w,V^\tc_\Theta)$ \eqref{eq:selmer}. However, this may not be the case for $w\in\Sigma_\cF$, unless when the induced character $\Upsilon\to\Theta^\times$ has finite order.
\end{remark}

\begin{lem}\label{le:selmer1}
Assume $p\geq 2n+1$ when $\Sigma_\cF\neq\Sigma_p$. Let $\Theta$ be a $\Lambda_\cF$ ring that is a finite free $O_\wp$-module. For $?$ that is either empty or a place $w$ of $F$,
\begin{enumerate}
  \item the image of $\rH^1_\ff(F_?,T^\tc_{\Lambda_\cF^\dag})$ under the natural map $\rH^1(F_?,T^\tc_{\Lambda_\cF^\dag})\to\rH^1(F_?,T^\tc_{\Theta^\dag})$ is contained in $\rH^1_\ff(F_?,T^\tc_{\Theta^\dag})$;

  \item the image of $\rH^1_\ff(F_?,W^{\Theta^\dag})$ under the natural map $\rH^1(F_?,W^{\Theta^\dag})\to\rH^1(F_?,W^{\Lambda_\cF^\dag})$ is contained in $\rH^1_\ff(F_?,W^{\Lambda_\cF^\dag})$.
\end{enumerate}
\end{lem}

\begin{proof}
Part (2) follows from (1). For (1), the global statement follows from the local one. For the local statement, it is trivial for $w\in\Sigma_\infty$; it follows from the vanishing of $\rH^1(F_w,V^\tc_{\Lambda_\cF^\dag})$ for $w\in\Sigma\setminus(\Sigma_\infty\cup\Sigma_p)$; it follows from Lemma \ref{le:bk}(1) (with $a=-n$ and $b=n-1$) for $w\in\Sigma_p\setminus\Sigma_\cF$; it follows from the definitions for $w\in\Sigma_\cF$.
\end{proof}

\begin{notation}\label{no:point}
For every closed point $x$ of $\Spec\Lambda_{\cF,E_\wp}$, we denote by
\begin{itemize}[label={\ding{118}}]
  \item $E_x$ its residue field, which is a finite extension of $E_\wp$,

  \item $O_x$ the ring of integers of $E_x$ and $\wp_x$ the maximal ideal of $O_x$,

  \item $\Theta_x$ the image of the natural homomorphism $\Lambda_\cF\to O_x$, which is a subring of $O_x$ of finite index,

  \item $\fP_x$ the kernel of $\Lambda_\cF\to\Theta_x$.
\end{itemize}
\end{notation}

\begin{proposition}\label{pr:specialize}
Assume $p\geq 2n+1$ and that $T/\wp T$ is an absolutely irreducible representation of $\Gamma_F$. For every point $z$ of $\Spec\Lambda_{\cF,E_\wp}$ of codimension at most one, there exists a Zariski open dense subset $\sU_z$ of the Zariski closure of $z$, such that the natural map
\[
\rH^1_\ff(F,V^\tc_{\Lambda_\cF^\dag})/\fP_x\rH^1_\ff(F,V^\tc_{\Lambda_\cF^\dag})\to\rH^1_\ff(F,V^\tc_{O_x^\dag}),
\]
which exists by Lemma \ref{le:selmer1}(1), is injective for every closed point $x$ of $\sU_z$.
\end{proposition}

\begin{proof}
Let $\sV$ be the minimal Zariski closed subset of $\Spec\Lambda_{\cF,E_\wp}$ away from which the (finitely generated) $\Lambda_{\cF,E_\wp}$-module $\rH^1(F,V^\tc_{\Lambda_\cF^\dag})/\rH^1_\ff(F,V^\tc_{\Lambda_\cF^\dag})$ is locally free. Then the natural map
\[
\rH^1_\ff(F,V^\tc_{\Lambda_\cF^\dag})/\fP_x\rH^1_\ff(F,V^\tc_{\Lambda_\cF^\dag})
\to\rH^1(F,V^\tc_{\Lambda_\cF^\dag})/\fP_x\rH^1(F,V^\tc_{\Lambda_\cF^\dag})
\]
is injective for closed points $x$ away from $\sV$. We claim that the codimension of $\sV$ is at least two.

Assuming the above claim, we prove the proposition. Denote $F^\ur\subseteq\ol{F}$ the maximal extension of $F$ that is unramified outside $\Sigma_\mnm\cup\Sigma_p$ and write $\rH^\bullet(F,-)$ for $\rH^\bullet(F^\ur/F,-)$ for short. Then it suffices to consider $\rH^1(F,-)$ instead of $\rH^1_\ff(F,-)$. Choose an element $f\in\Lambda_\cF$ that defines $z$. Then we have an inclusion
\[
\rH^1(F,V^\tc_{\Lambda_\cF^\dag})/(f)\hookrightarrow\rH^1(F,V^\tc_{(\Lambda_\cF/(f))^\dag})
\]
of finitely generated $\Lambda_{\cF,E_\wp}/(f)$-modules. Take a regular Zariski open dense subset $\sU_z$ of $\Spec\Lambda_{\cF,E_\wp}/(f)$ such that both $\rH^1(F,V^\tc_{(\Lambda_\cF/(f))^\dag})/\rH^1(F,V^\tc_{\Lambda_\cF^\dag})/(f)$ and $\rH^2(F,V^\tc_{(\Lambda_\cF/(f))^\dag})$ are locally free over $\sU_z$, which is possible. Consider a closed point $x$ of $\sU_z$. If we denote by $\fQ_x$ the image of $\fP_x$ in $\Lambda_\cF/(f)$, then the natural map $\rH^1(F,V^\tc_{\Lambda_\cF^\dag})/\fP_x\to\rH^1(F,V^\tc_{(\Lambda_\cF/(f))^\dag})/\fQ_x$ is injective. To show that $\sU_z$ meets the requirement of the lemma, it suffices to show the following statement: For every $\Lambda_\cF$-ring $\Theta$ such that $\Theta^\eta\coloneqq\Theta\otimes_{O_\wp}E_\wp$ is a regular noetherian local ring with the maximal ideal $\fQ$ such that $\rH^2(F,V^\tc_\Theta)$ is a finite free $\Theta^\eta$-module, the natural map $\rH^1(F,V^\tc_\Theta)/\fQ\to\rH^1(F,V^\tc_{\Theta/\fQ})$ is an isomorphism.

Indeed, choose a sequence $\{f_1,\ldots,f_d\}$ in $\Theta$ that is regular in $\Theta^\eta$ and generates $\fQ$. For every $0\leq i\leq d$, we have the natural map
\[
q_i\colon\rH^1(F,V^\tc_\Theta)/(f_1,\ldots,f_i)\to\rH^1(F,V^\tc_{\Theta/(f_1,\ldots,f_i)}).
\]
Our goal is to show that $q_d$ is an isomorphism. We prove by induction on $i$ that $q_i$ is an isomorphism and that $\rH^2(F,V^\tc_{\Theta/(f_1,\ldots,f_i)})$ is a finite free $\Theta^\eta/(f_1,\ldots,f_i)$-module. When $i=0$, the first statement is trivial and the second one follows from the assumption. Suppose that the two statements hold for $i$. Consider the short exact sequence
\[
0 \to V^\tc_{\Theta/(f_1,\ldots,f_i)} \xrightarrow{f_{i+1}\cdot}
V^\tc_{\Theta/(f_1,\ldots,f_i)}  \to V^\tc_{\Theta/(f_1,\ldots,f_{i+1})} \to 0
\]
of $\Theta^\eta[\Gamma_F]$-modules, which induces the short exact sequence
\[
\xymatrix{
0 \ar[r]& \rH^1(F,V^\tc_{\Theta/(f_1,\ldots,f_i)})/(f_{i+1})
\ar[r]& \rH^1(F,V^\tc_{\Theta/(f_1,\ldots,f_{i+1})})
\ar[r]& \rH^2(F,V^\tc_{\Theta/(f_1,\ldots,f_i)})[f_{i+1}] \ar[r]& 0.
}
\]
By the induction hypothesis, we have $\rH^2(F,V^\tc_{\Theta/(f_1,\ldots,f_i)})[f_{i+1}]=0$ and that $q_i$ is an isomorphism. It follows that $q_{i+1}$ is an isomorphism as well. Since $\rH^3(F,V^\tc_{\Theta/(f_1,\ldots,f_i)})=0$, we have
\[
\rH^2(F,V^\tc_{\Theta/(f_1,\ldots,f_{i+1})})=\rH^2(F,V^\tc_{\Theta/(f_1,\ldots,f_i)})/(f_{i+1}).
\]
By the induction hypothesis, $\rH^2(F,V^\tc_{\Theta/(f_1,\ldots,f_i)})/(f_{i+1})$ is a finite free $\Theta^\eta/(f_1,\ldots,f_{i+1})$-module.

It remains to verify the claim that the codimension of $\sV$ is at least two. Since $\rH^1(F,V^\tc_{\Lambda_\cF^\dag})/\rH^1_\ff(F,V^\tc_{\Lambda_\cF^\dag})$ is a $\Lambda_{\cF,E_\wp}$-submodule of
\[
\bigoplus_{w\in\Sigma_p}\rH^1(F_w,V^\tc_{\Lambda_\cF^\dag})/\rH^1_\ff(F_w,V^\tc_{\Lambda_\cF^\dag}),
\]
it suffices to show that the for every $w\in\Sigma_p$, the maximal $\Lambda_{\cF,E_\wp}$-torsion submodule of $\rH^1(F_w,V^\tc_{\Lambda_\cF^\dag})/\rH^1_\ff(F_w,V^\tc_{\Lambda_\cF^\dag})$ is supported on a Zariski closed subset of codimension at least two. For $w\in\Sigma_p\setminus\Sigma_\cF$, we have
\[
\rH^1(F_w,V^\tc_{\Lambda_\cF^\dag})/\rH^1_\ff(F_w,V^\tc_{\Lambda_\cF^\dag})
=\rH^1(F_w,T^\tc_{\Lambda_\cF^\dag})/\rH^1_\ff(F_w,T^\tc_{\Lambda_\cF^\dag})\otimes_{O_\wp}E_\wp,
\]
in which $\rH^1(F_w,T^\tc_{\Lambda_\cF^\dag})/\rH^1_\ff(F_w,T^\tc_{\Lambda_\cF^\dag})$ is a torsion free $\Lambda_\cF$-module by Lemma \ref{le:bk}(2) (with $a=-n$ and $b=n-1$). For $w\in\Sigma_\cF$, we have an inclusion
\[
\rH^1(F_w,V^\tc_{\Lambda_\cF^\dag})/\rH^1_\ff(F_w,V^\tc_{\Lambda_\cF^\dag})
\hookrightarrow\rH^1(F_w,(V^\tc/\Fil^{-1}_wV^\tc)_{\Lambda_\cF^\dag})
\]
of $\Lambda_{\cF,E_\wp}$-modules. Since $\cF/F$ is free and $p>n$, it follows easily that away from the Zariski closed subset $\sW_\cF$ introduced in Definition \ref{de:support} below, $\rH^1(F_w,(V^\tc/\Fil^{-1}_wV^\tc)_{\Lambda_\cF^\dag})$ is (locally) torsion free. Now, by Lemma \ref{le:support} below, the maximal $\Lambda_{\cF,E_\wp}$-torsion submodules of $\rH^1(F_w,(V^\tc/\Fil^{-1}_wV^\tc)_{\Lambda_\cF^\dag})$ and hence $\rH^1(F_w,V^\tc_{\Lambda_\cF^\dag})/\rH^1_\ff(F_w,V^\tc_{\Lambda_\cF^\dag})$ are supported on a Zariski closed subset of codimension at least two.

The proposition is finally proved.
\end{proof}

\begin{definition}\label{de:support}
We define $\sW_\cF$ to be the support of the $\Lambda_{\cF,E_\wp}$-module
\[
\bigoplus_{w\in\Sigma_\cF}(\Gr^0_w V^\tc_{\Lambda_\cF^\dag})_{\Gamma_{F_w}},
\]
which is a Zariski closed subset of $\Spec\Lambda_{\cF,E_\wp}$.
\end{definition}

\begin{lem}\label{le:support}
The codimension of $\sW_\cF$ in $\Spec\Lambda_{\cF,E_\wp}$ is at least two.
\end{lem}

\begin{proof}
For every $w\in\Sigma_\cF$, let $\sW_w$ be the support of $(\Gr^0_w V^\tc_{\Lambda_\cF^\dag})_{\Gamma_{F_w}}$. It suffices to consider the codimension of $\sW_w$. First, suppose that $\Gamma_{\cF,w}$ has rank at least two, then $\sW_w$ is contained in the zero locus of at least two relatively prime elements in $\Lambda_{\cF,E_\wp}$, which implies that its codimension is at least two. Second, suppose that $\Gamma_{\cF,w}$ has rank one. In this case, the residue extension degree of $\cF/F$ at $w$ is one. In other words, there exists a Frobenius element $\Phi\in\Gamma_{F_w}$ whose image in $\Upsilon_\cF$ is trivial. Then $(\Gr^0_w V^\tc_{\Lambda_\cF^\dag})_\Phi=(\Gr^0_wV^\tc)_\Phi\otimes_{O_\wp}\Lambda_\cF^\dag$, which is trivial since $V$ is pure of weight $-1$ at $w$.
\end{proof}

\begin{proposition}\label{pr:control}
Assume $p\geq 2n+1$ and that $T/\wp T$ is an absolutely irreducible representation of $\Gamma_F$. For every closed point $x$ of $\Spec\Lambda_{\cF,E_\wp}$, consider the cospecialization map
\[
\cosp_x\colon\rH^1_\ff(F,W^{O_x^\dag})\to\rH^1_\ff(F,W^{\Lambda_\cF^\dag})[\fP_x]
\]
from Lemma \ref{le:selmer1}(2).
\begin{enumerate}
  \item The kernel of $\cosp_x$ is finite with order bounded by a constant depending only on $|O_x/\Theta_x|$.

  \item Suppose that $\cF/F$ is free. For every affinoid subdomain $\sU$ of $\Spec\Lambda_{\cF,E_\wp}$\footnote{Here, the $E_\wp$-ring $\Lambda_{\cF,E_\wp}$ naturally defines a rigid analytic space over $E_\wp$. By an affinoid subdomain of $\Spec\Lambda_{\cF,E_\wp}$, we mean an affinoid subdomain of that rigid analytic space.} disjoint from $\sW_\cF$, the cokernel of $\cosp_x$ is finite with order bounded by a constant depending only on $\sU$ and $[E_x:E_\wp]$.
\end{enumerate}
\end{proposition}

\begin{proof}
Denote $F^\ur\subseteq\ol{F}$ the maximal extension of $F$ that is unramified outside $\Sigma_\mnm\cup\Sigma_p$ and write $\rH^\bullet(F,-)$ for $\rH^\bullet(F^\ur/F,-)$ for short. We first note that, by (the same proof of) \cite{MR04}*{Lemma~3.5.3}, the absolute irreducibility of $T/\wp T$ implies that the natural map
\[
\rH^1(F,W^{\Theta_x^\dag})=\rH^1(F,W^{\Lambda_\cF^\dag}[\fP_x])\to\rH^1(F,W^{\Lambda_\cF^\dag})[\fP_x]
\]
is an isomorphism for every closed point $x$ of $\Spec\Lambda_{\cF,E_\wp}$.

For (1), $\Ker\cosp_x$ is then contained in the kernel of the map $\rH^1(F,W^{O_x^\dag})\to\rH^1(F,W^{\Theta_x^\dag})$, which is dual to an $O_\wp$-submodule of $\rH^1(F,T\otimes_{O_\wp}(O_x/\Theta_x))$. Part (1) follows as $\rH^1(F,T\otimes_{O_\wp}(O_x/\Theta_x))$ is finite with order bounded by a constant depending only on $|O_x/\Theta_x|$.

For (2), $\coker\cosp_x$ is then contained in
\[
\bigoplus_{w\in\Sigma_\mnm\cup\Sigma_p}
\Ker\(\rH^1_\ff(F_w,W^{\Lambda_\cF^\dag}[\fP_x])\to\rH^1_\ff(F_w,W^{\Lambda_\cF^\dag})[\fP_x]\)
\]
by the snake lemma. By local Tate duality, the order of the above group is the same as that of
\[
\bigoplus_{w\in\Sigma_\mnm\cup\Sigma_p}
\coker\(\rH^1_\ff(F_w,T^\tc_{\Lambda_\cF^\dag})\to\rH^1_\ff(F_w,T^\tc_{\Theta_x^\dag})\).
\]
There are three cases.

For $w\in\Sigma_p\setminus\Sigma_\cF$, the map $\rH^1_\ff(F_w,T^\tc_{\Lambda_\cF^\dag})\to\rH^1_\ff(F_w,T^\tc_{\Theta_x^\dag})$ is surjective by Lemma \ref{le:bk}(1).

For $w\in\Sigma\setminus\Sigma_p$, the map $\rH^1_\ff(F_w,T^\tc_{\Lambda_\cF^\dag})\to\rH^1_\ff(F_w,T^\tc_{\Theta_x^\dag})$ is identified with the composition
\[
\rH^1_\ff(F_w,T^\tc_{\Lambda_\cF^\dag})=\rH^1_\unr(F_w,T^\tc_{\Lambda_\cF^\dag})=
\rH^1(\kappa_w,(T^\tc_{\Lambda_\cF^\dag})^{\rI_{F_w}})\to\rH^1(\kappa_w,(T^\tc_{\Theta_x^\dag})^{\rI_{F_w}})\to\rH^1_\ff(F_w,T^\tc_{\Theta_x^\dag}),
\]
in which the second one is an inclusion. As the map $(T^\tc_{\Lambda_\cF^\dag})^{\rI_{F_w}}\to(T^\tc_{\Theta_x^\dag})^{\rI_{F_w}}$, which is nothing but $(T^\tc)^{\rI_{F_w}}\otimes_{\Lambda_\cF}\Lambda_\cF^\dag\to(T^\tc)^{\rI_{F_w}}\otimes_{\Lambda_\cF}\Theta_x^\dag$, is surjective, the first map is surjective as $\rH^2(\kappa_w,-)$ vanishes. On the other hand, we have
\[
\frac{\rH^1_\ff(F_w,T^\tc_{\Theta_x^\dag})}{\rH^1(\kappa_w,(T^\tc_{\Theta_x^\dag})^{\rI_{F_w}})}
=\rH^1(\rI_{F_w},T^\tc_{\Theta_x^\dag})^{\Gamma_{\kappa_w}}[p^\infty]=\(\rH^1(\rI_{F_w},T^\tc)[p^\infty]\otimes_{O_\wp}\Theta_x^\dag\)^{\Gamma_{\kappa_w}}
\]
whose cardinality is bounded by $\#(\rH^1(\rI_{F_w},T^\tc)[p^\infty])^{[E_x:E_\wp]}$.

For $w\in\Sigma_\cF$, by definition, the image of $\rH^1(F_w,(\Fil_w^{-1}T^\tc)_{\Lambda_\cF^\dag})$ in $\rH^1(F_w,T^\tc_{\Lambda_\cF^\dag})$ is contained in $\rH^1_\ff(F_w,T^\tc_{\Lambda_\cF^\dag})$. Put
\[
K\coloneqq\coker\(\rH^1(F_w,(\Fil_w^{-1}T^\tc)_{\Lambda_\cF^\dag})\to\rH^1_\ff(F_w,T^\tc_{\Theta_x^\dag})\),
\]
which as an $O_\wp$-module can be generated by at most $n(n+1)([F_w:\dQ_p]+2)[E_x:E_\wp]$ elements. Thus, it suffices to show that $K$ is annihilated by a power of $\wp$ depending only on $\sU$. The $O_\wp$-module $K$ fits into the exact sequence
\[
\resizebox{\hsize}{!}{
\xymatrix{
\coker\(\rH^1(F_w,(\Fil_w^{-1}T^\tc)_{\Lambda_\cF^\dag})\to\rH^1(F_w,(\Fil_w^{-1}T^\tc)_{\Theta_x^\dag})\)
\to K \to\ker\(\rH^1(F_w,(T^\tc/\Fil_w^{-1}T^\tc)_{\Theta_x^\dag})\to\rH^1(\rI_{F_w},(V^\tc/\Fil_w^{-1}V^\tc)_{\Theta_x^\dag})\).
}
}
\]
Since $x$ does not belong to $\sW_\cF$, the map $\rH^1(F_w,(V^\tc/\Fil_w^{-1}V^\tc)_{\Theta_x^\dag})\to\rH^1(\rI_{F_w},(V^\tc/\Fil_w^{-1}V^\tc)_{\Theta_x^\dag})$ is injective. Thus, it to show that both
\begin{align}\label{eq:specialize1}
\coker\(\rH^1(F_w,(\Fil_w^{-1}T^\tc)_{\Lambda_\cF^\dag})\to\rH^1(F_w,(\Fil_w^{-1}T^\tc)_{O_x^\dag})\)
\end{align}
and
\begin{align}\label{eq:specialize2}
\Ker\(\rH^1(F_w,(T^\tc/\Fil_w^{-1}T^\tc)_{\Theta_x^\dag})\to\rH^1(F_w,(V^\tc/\Fil_w^{-1}V^\tc)_{\Theta_x^\dag})\)
\end{align}
are annihilated by a power of $\wp$ depending only on $\sU$. Note that \eqref{eq:specialize2} is isomorphic to
\[
\coker\(\rH^0(F_w,(V^\tc/\Fil_w^{-1}V^\tc)_{\Theta_x^\dag})\to\rH^0(F_w,(W^\tc/\Fil_w^{-1}W^\tc)_{\Theta_x^\dag})\),
\]
which coincides with
\[
\coker\(\rH^0(F_w,(\Gr^0_wV^\tc)_{\Theta_x^\dag})\to\rH^0(F_w,(\Gr^0_wW^\tc)_{\Theta_x^\dag})\)
\]
since $\Upsilon_{\cF,w}$ is torsion free and $p>n$. As $x$ does not belong to $\sW_\cF$, the above cokernel is just $\rH^0(F_w,(\Gr^0_wW^\tc)_{\Theta_x^\dag})$. It is clear that $\rH^0(F_w,(\Gr^0_wW^\tc)_{\Theta_x^\dag})$ and hence \eqref{eq:specialize2} are annihilated by a power of $\wp$ depending only on $\sU$.

Finally, we consider \eqref{eq:specialize1}. By local Tate duality, its Cartier dual is isomorphic to
\begin{align*}
\ker\(\rH^1(F_w,(W/\Fil_w^{-1}W)^{\Theta_x^\dag})\to\rH^1(F_w,(W/\Fil_w^{-1}W)^{\Lambda_\cF^\dag})\),
\end{align*}
which is isomorphic to
\[
\coker\(\rH^0(F_w,(W/\Fil_w^{-1}W)^{\Lambda_\cF^\dag})\to\rH^0(F_w,(W/\Fil_w^{-1}W)^{\fP_x^\dag})\),
\]
where $(W/\Fil_w^{-1}W)^{\fP_x^\dag}\coloneqq\coker\((W/\Fil_w^{-1}W)^{O_x^\dag}\to(W/\Fil_w^{-1}W)^{\Lambda_\cF^\dag}\)$. As $\Upsilon_{\cF,w}$ is torsion free and $p>n$, the above cokernel coincides with
\begin{align}\label{eq:specialize3}
\coker\(\rH^0(F_w,(\Gr^0_wW)^{\Lambda_\cF^\dag})\to\rH^0(F_w,(\Gr^0_wW)^{\fP_x^\dag})\).
\end{align}
Since $V$ is the tensor product of two representations with regular Hodge--Tate weights (up to a Tate twist), the action of $\Gamma_{F_w}$ on $(\Gr^0_wW)^{\Lambda_\cF^\dag}$ faithfully factors through an abelian quotient $\Gamma_{F_w}\to\Upsilon'_{F,w}$ that is isomorphic to the product of $\dZ_p^d$ with a finite group of order coprime to $p$. It follows that the Pontryagin dual of \eqref{eq:specialize3} coincides with
\begin{align}\label{eq:specialize4}
\Ker\(\((\Gr^0_wT^\tc)_{\fP_x^\dag}\)_{\Upsilon'_{\cF,w}}\to\((\Gr^0_wT^\tc)_{\Lambda_\cF^\dag}\)_{\Upsilon'_{\cF,w}}\),
\end{align}
which is isomorphic to
\[
\Ker\(\rH^d(\Upsilon'_{\cF,w},(\Gr^0_wT^\tc)_{\fP_x^\dag})
\to\rH^d(\Upsilon'_{\cF,w},(\Gr^0_wT^\tc)_{\Lambda_\cF^\dag})\).
\]
Now the above cokernel is a quotient of $\rH^{d-1}(\Upsilon'_{\cF,w},(\Gr^0_wT^\tc)_{\Theta_x^\dag})$ on which $\Upsilon'_{\cF,w}$ acts trivially. Since $\sU$ is disjoint from $\sW_\cF$, there exists an integer $k(\sU)\geq 0$ depending only on $\sU$ such that the ideal generated by the image of $\{\gamma-1\res\gamma\in\Upsilon'_{\cF,w}\}$ in $\End_{O_\wp}\((\Gr^0_wT^\tc)_{\Theta_x^\dag}\)$ contains $\wp^{k(\sU)}(\Gr^0_wT^\tc)_{\Theta_x^\dag}$ for every closed point $x$ of $\sU$. It follows that $\rH^{d-1}(\Upsilon'_{\cF,w},(\Gr^0_wT^\tc)_{\Theta_x^\dag})$ and hence \eqref{eq:specialize4} are annihilated by $\wp^{k(\sU)}$.

The proposition is proved.
\end{proof}

\begin{lem}\label{le:selfdual}
Assume $p>n$. Let $x$ be a closed point of $\Spec\Lambda_{\cF,E_\wp}$ away from $\sW_\cF$. For every place $w$ of $F$, the subspaces $\rH^1_\ff(F_w,V^\tc_{O_x^\dag})$ and $\rH^1_\ff(F_{w^\tc},V^\tc_{O_x^\dag})$ are mutual annihilators under the pairing
\[
\rH^1(F_w,V^\tc_{O_x^\dag})\times\rH^1(F_{w^\tc},V^\tc_{O_x^\dag})
=\rH^1(F_w,(V_{O_x^\dag})^\tc)\times\rH^1(F_{w^\tc},(V_{O_x^\dag})^\tc)
=\rH^1(F_w,(V_{O_x^\dag})^\tc)\times\rH^1(F_w,V_{O_x^\dag})
\to E_\wp
\]
in which the second identity is the conjugation isomorphism and the last pairing is the local Tate pairing induced from \eqref{eq:pairing}.
\end{lem}

\begin{proof}
For $w\not\in\Sigma_\cF$, $\rH^1_\ff$ coincides with the Bloch--Kato local Selmer structure, so that the statement is well-known. For $w\in\Sigma_\cF$, as the restriction of \eqref{eq:pairing} to $(\Fil^{-1}V_{O_x^\dag})^\tc\times\Fil^{-1}V_{O_x^\dag}$ vanishes, $\rH^1_\ff(F_w,V^\tc_{O_x^\dag})$ and $\rH^1_\ff(F_{w^\tc},V^\tc_{O_x^\dag})$ annihilate each other. Now since $x$ does not belong to $\sW_\cF$, it follows easily that
\[
\dim_{E_\wp}\rH^1_\ff(F_w,V^\tc_{O_x^\dag})=\frac{1}{2}\dim_{E_\wp}\rH^1(F_w,(V_{O_x^\dag})^\tc)=
\frac{1}{2}\dim_{E_\wp}\rH^1(F_{w^\tc},V^\tc_{O_x^\dag})=\dim_{E_\wp}\rH^1_\ff(F_{w^\tc},V^\tc_{O_x^\dag}).
\]
Thus, the statement holds for $w$ as well.
\end{proof}

\subsection{Congruence modules}
\label{ss:congruence}

In this subsection, we define the notion of congruence modules, which are algebraic quantities carried over through the process of arithmetic level raising.

From this point to the end of \S\ref{ss:euler}, we will
\begin{itemize}[label={\ding{118}}]
  \item assume Hypothesis \ref{hy:unitary_cohomology} for both $n$ and $n+1$,

  \item assume that $\wp$ is an admissible prime of $E$ in the sense of Definition \ref{de:admissible}; In particular, $T/\wp T$ is an absolutely irreducible representation of $\Gamma_F$ (of dimension $n(n+1)$) by (A3).
\end{itemize}

\begin{notation}
We introduce the following notation.
\begin{enumerate}
  \item Denote by $\fL_0$ the set of primes of $F^+$ not in $\Sigma^+_\mnm\cup\Sigma^+_p$ that are \emph{inert} in $F$. We sometimes also regard $\fL_0$ as a subset of $\Sigma$ according to the context.

  \item Denote by $\fN_0$ the set of (possibly empty) finite sets consisting of elements in $\fL_0$ with \emph{distinct} underlying rational primes. For $\fn\in\fN_0$, put $\fn\fl\coloneqq\fn\cup\{\fl\}$ for $\fl\in\fL_0$ whose underlying rational prime is coprime to $\fn$, and $\fn/\fl\coloneqq\fn\setminus\{\fl\}$ for $\fl\in\fn$.
  \item Following \cite{How06}, we define a \emph{directed} graph $\fX_0$ with vertices $v(\fn)$ indexed by elements in $\fN_0$, and arrows $a(\fn,\fn\fl)$ from $v(\fn)$ to $v(\fn\fl)$ for $\fl\in\fL_0\setminus\fn$.

  \item For every $\fn\in\fN_0$, put $\epsilon(\fn)\coloneqq(-1)^{|\fn|}$, and denote by $\underline\fn$ the set of underlying rational primes of $\fn$.

  \item Finally, put
      \begin{align*}
      \fN_0^\defin\coloneqq\{\fn\in\fN_0\res\epsilon(\fn)=\epsilon(\Pi_0\times\Pi_1)\},\quad
      \fN_0^\indef\coloneqq\{\fn\in\fN_0\res\epsilon(\fn)\neq\epsilon(\Pi_0\times\Pi_1)\}.
      \end{align*}
\end{enumerate}
\end{notation}

We now attach to the graph $\fX_0$ some data.

\begin{notation}\label{no:initial}
We start from the \emph{initial datum} $\cV=(\rV_n,\rV_{n+1};\Lambda_n,\Lambda_{n+1};\rK_n,\rK_{n+1})$ in which
\begin{itemize}[label={\ding{118}}]
  \item $\rV_n$ is a standard definite/indefinite hermitian space (Definition \ref{de:standard_hermitian_space} over $F$ of rank $n$ when $\epsilon(\Pi_0\times\Pi_1)=+1/{-1}$ and $\rV_{n+1}=(\rV_n)_\sharp$, satisfying the following: there exists an irreducible admissible representation $\pi_\alpha$ of $\rU(\rV_{n_\alpha})(\dA_{F^+}^\infty)$ whose base change is $\Pi_\alpha$ for $\alpha=0,1$, such that
      \[
      \Hom_{\rU(\rV_n)(\dA_{F^+}^\infty)}\(\pi_0\boxtimes\pi_1,\dC\)\neq 0,
      \]
      where $\rU(\rV_n)$ is regarded as the subgroup of $\rU(\rV_{n_0})\times\rU(\rV_{n_1})$ that is the graph of the natural homomorphism $\rU(\rV_n)\to\rU(\rV_{n+1})$;

  \item $\Lambda_n$ is a self-dual $\prod_{v\not\in\Sigma^+_\infty\cup\Sigma^+_\mnm}O_{F_v}$-lattice in $\rV_n\otimes_F\dA_F^{\Sigma_\infty\cup\Sigma_\mnm}$ and $\Lambda_{n+1}=(\Lambda_n)_\sharp$;

  \item $(\rK_n,\rK_{n+1})$ is an object in $\fK(\rV_n)_\sp$ (Definition \ref{de:neat_category}) of the form
     \begin{align*}
     \rK_N=\prod_{v\in\Sigma^+_\mnm}(\rK_N)_v\times\prod_{v\not\in\Sigma^+_\infty\cup\Sigma^+_\mnm}\rU(\Lambda_N)(O_{F^+_v})
     \end{align*}
     for $N\in\{n,n+1\}$ such that $(\pi_\alpha)^{\rK_{n_\alpha}}\neq 0$ for $\alpha=0,1$.
\end{itemize}

\begin{enumerate}
  \item For every vertex $v(\fn)$ of $\fX_0$, we choose a datum $\cV^\fn=(\rV^\fn_n,\rV^\fn_{n+1};\Lambda^\fn_n,\Lambda^\fn_{n+1};\rK^\fn_n,\rK^\fn_{n+1})$ in which
    \begin{itemize}[label={\ding{118}}]
      \item $\rV^\fn_n$ is a standard definite/indefinite hermitian space over $F$ of rank $n$ when $\fn\in\fN_1^\defin/\fN_1^\indef$ and $\rV^\fn_{n+1}=(\rV^\fn_n)_\sharp$;

      \item $\Lambda^\fn_n$ is a $\prod_{v\not\in\Sigma^+_\infty\cup\Sigma^+_\mnm}O_{F_v}$-lattice in $\rV^\fn_n\otimes_F\dA_F^{\Sigma_\infty\cup\Sigma_\mnm}$ satisfying $\Lambda^\fn_n\subseteq(\Lambda^\fn_n)^\vee$ and that $(\Lambda^\fn_{n,v})^\vee/\Lambda^\fn_{n,v}$ has length one (resp.\ zero) when $v\in\fn$ (resp.\ $v\not\in\fn$), and $\Lambda^\fn_{n+1}=(\Lambda^\fn_n)_\sharp$;

      \item $(\rK^\fn_n,\rK^\fn_{n+1})$ is an object in $\fK(\rV^\fn_n)_\sp$ of the form
         \begin{align*}
         \rK^\fn_N=\prod_{v\in\Sigma^+_\mnm}(\rK^\fn_N)_v\times\prod_{v\not\in\Sigma^+_\infty\cup\Sigma^+_\mnm}\rU(\Lambda^\fn_N)(O_{F^+_v})
         \end{align*}
         for $N\in\{n,n+1\}$,
    \end{itemize}
    together with a datum $\tj^\fn=(\tj^\fn_n,\tj^\fn_{n+1})$ in which $\tj^\fn_N\colon\rV_N\otimes_\dQ\dA^{\infty,\underline\fn}\to\rV^\fn_N\otimes_\dQ\dA^{\infty,\underline\fn}$ are isometries sending $(\Lambda_N)^{\underline\fn}$ to $(\Lambda^\fn_N)^{\underline\fn}$ and $(\rK_N)^{\underline\fn}$ to $(\rK^\fn_N)^{\underline\fn}$ for $N\in\{n,n+1\}$, satisfying $\tj^\fn_{n+1}=(\tj^\fn_n)_\sharp$.

    When $\fn=\emptyset$, we just take $\cV^\emptyset$ to be $\cV$ and $\tj^\emptyset$ to be the identity.

  \item For every arrow $a=a(\fn,\fn\fl)$ of $\fX_0$, we choose a datum $\tj^a=(\tj^a_n,\tj^a_{n+1})$ in which
      \[
      \tj^a_N\colon\rV^\fn_N\otimes_\dQ\dA^{\infty,\ell}\to\rV^{\fn\fl}_N\otimes_\dQ\dA^{\infty,\ell}
      \]
      (with $\ell$ the underlying rational prime of $\fl$) are isometries sending $(\Lambda^\fn_N)^\ell$ to $(\Lambda^{\fn\fl}_N)^\ell$ and $(\rK^\fn_N)^\ell$ to $(\rK^{\fn\fl}_N)^\ell$ and such that $\tj^{\fn\fl}_N=(\tj^a_N)^{\underline{\fn\fl}}\circ(\tj^\fn_N)^{\underline{\fn\fl}}$ for $N\in\{n,n+1\}$, satisfying $\tj^a_{n+1}=(\tj^a_n)_\sharp$.

  \item For every $\fn\in\fN_0$, put
      \[
      \cM^\fn\coloneqq
      \begin{dcases}
      O_\wp[\Sh(\rV^\fn_{n_0},\rK^\fn_{n_0})\times\Sh(\rV^\fn_{n_1},\rK^\fn_{n_1})], &\fn\in\fN_0^\defin, \\
      \rH^{2n-1}_\et((\Sh(\rV^\fn_{n_0},\rK^\fn_{n_0})\times_F\Sh(\rV^\fn_{n_1},\rK^\fn_{n_1}))_{\ol{F}},O_\wp(n)), &\fn\in\fN_0^\indef,
      \end{dcases}
      \]
      which are modules over $O_\wp$ and $O_\wp[\Gamma_F]$, respectively, both finitely generated over $O_\wp$.

  \item For every $\fm\in\fN_0$ and $\alpha=0,1$, we
      \begin{itemize}[label={\ding{118}}]
        \item put
           \[
           \dT^\fm_\alpha\coloneqq\dT^{\Sigma^+_\mnm\cup\Sigma^+_{\underline\fm}}_{n_\alpha}
           \]
           which acts on $\cM^\fn$ for every $\fn\in\fN_0$ satisfying $\fn\subseteq\fm$;

        \item for every integer $k\geq 1$, put
           \[
           \sfn^{\fm,k}_\alpha\coloneqq
           \dT^\fm_\alpha\cap\Ker\(\dT^{\Sigma^+_\mnm}_{n_\alpha}\xrightarrow{\phi_{\Pi_\alpha}}O_E\to O_E/\wp^k\);
           \]

        \item write $\sfm^\fm_\alpha$ for $\sfn^{\fm,1}_\alpha$, which is a maximal ideal of $\dT^\fm_\alpha$.
      \end{itemize}
\end{enumerate}
\end{notation}

We fix a finite subset $\Box_\cV\subseteq\Sigma^+\setminus(\Sigma^+_\infty\cup\Sigma^+_\mnm\cup\Sigma^+_p)$ such that for $\alpha=0,1$, the $O_\wp$-subrings of $\End_{O_\wp}\cM^\emptyset$ generated by $\bigotimes_{v\in\Box_\cV}\dT_{n_\alpha,v}$ and $\dT^\emptyset_\alpha$ are the same. The subset $\Box_\cV$ depends on the choice of $(\rK_n,\rK_{n+1})$ in the initial datum $\cV$.

\begin{definition}\label{de:congruence}
Let $k\geq 1$ be a positive integer.
\begin{enumerate}
  \item We denote by $\fL_k$ the set of primes $\fl$ of $F^+$ (with the underlying rational prime $\ell$) satisfying
      \begin{description}
        \item[(C1)] $\fl\in\fL_0$ and $\Sigma^+_\ell\cap\Box_\cV=\emptyset$;

        \item[(C2)] $p$ does not divide $\ell\prod_{i=1}^{n_0}(1-(-\ell)^i)$;

        \item[(C3)] $\fl$ is a very special inert prime (Definition \ref{de:special_inert});

        \item[(C4)] $P_{\balpha(\Pi_{0,\fl})}\modulo\wp^k$ is level-raising special at $\fl$; $P_{\balpha(\Pi_{1,\fl})}\modulo\wp$ is Tate generic at $\fl$; $P_{\balpha(\Pi_{0,\fl})\otimes\balpha(\Pi_{1,\fl})}\modulo\wp^k$ is level-raising special at $\fl$ (all in Definition \ref{de:satake_condition});

        \item[(C5)] $P_{\balpha(\Pi_{\alpha,\fl})}\modulo\wp$ is intertwining generic at $\fl$ (Definition \ref{de:satake_condition}) for $\alpha=0,1$.
      \end{description}

  \item We denote by $\fN_k$ the set of (possibly empty) finite sets consisting of elements in $\fL_k$ with \emph{distinct} underlying rational primes.

  \item We define the directed graph $\fX_k$ to be the full subgraph of $\fX_0$ spanned by $\fN_k$.

  \item We say that an element $\fl$ of $\fL_k$ is \emph{effective} if $\cM^{\{\fl\}}/(\sfm^{\{\fl\}}_0,\sfm^{\{\fl\}}_1)$ is nontrivial, and that an element $\fn$ of $\fN_k$ is \emph{effective} if it either empty or contains an effective element of $\fL_k$. We denote by $\fN_k^\eff$ the subset of $\fN_k$ of effective elements. See Remark \ref{re:ihara} for the reason of introducing this notion.
\end{enumerate}

\end{definition}

\begin{remark}\label{re:location}
We have the following remarks concerning Definition \ref{de:congruence}.
\begin{enumerate}
  \item It is clear from the definition that $\fL_0\supseteq\fL_1\supseteq\fL_2\supseteq\cdots$.

  \item Since both $\Pi_0$ and $\Pi_1$ are tempered, we have $\bigcap_k\fL_k=\emptyset$.

  \item For every $\fl\in\fL_k$ with the underlying rational prime $\ell$, there is a unique decomposition
      \[
      T/\wp^kT=(T/\wp^kT)^{\phi_\fl=1}\oplus(T/\wp^kT)^{\phi_\fl=\ell^2}\oplus S
      \]
      of $O_\wp[\Gamma_{F,\fl}]$-modules such that $\rH^i(F_\fl,S)=0$ for every $i\in\dZ$. In particular, both $\rH^1_\unr(F_\fl,T^\tc/\wp^kT^\tc)$ and $\rH^1_\sing(F_\fl,T^\tc/\wp^kT^\tc)$ are free $O_\wp/\wp^k$-modules of rank one.

  \item For every $\fl\in\fL_1$, the set $\{k\geq 1\res \fl\in\fL_k\}$ has a maximum, which we denote by $k(\fl)$.

  \item By (A5) in Definition \ref{de:admissible} and the Chebotarev density theorem, $\fL_k$ has infinitely many elements for every $k\geq 1$.

  \item By (A3,A4,A6) in Definition \ref{de:admissible} and Lemma \ref{le:generic}, for every arrow $a=a(\fn,\fn\fl)\in\fX_k$ with $\fn\in\fN_k^\defin$, Assumptions \ref{as:first_irreducible}, \ref{as:first_generic}, \ref{as:first_generic} (which are required in Definition \ref{de:first_reciprocity} and Theorem \ref{th:first}) are satisfied for the data $\lambda=\wp$, $m=k$, $\Sigma^+_{\lr,\rI}=\fn$, $\cV^\circ=\cV^\fn$, $\fp=\fl$, $\cV'=\cV^{\fn\fl}$, and $\tj=\tj^a$.

  \item By (A3,A4,A6) in Definition \ref{de:admissible} and Lemma \ref{le:generic}, for every arrow $a=a(\fn,\fn\fl)\in\fX_k$ with $\fn\in\fN_k^\indef$, Assumption \ref{as:first_irreducible} and Assumption \ref{as:first_generic} (which are required in Definition \ref{de:second_reciprocity} and Theorem \ref{th:second}) are satisfied the data $\lambda=\wp$, $m=k$, $\Sigma^+_{\lr,\r{II}}=\fn$, $\cV=\cV^\fn$, $\fp=\fl$, $\cV'=\cV^{\fn\fl}$, and $\tj=\tj^a$.

  \item Similar to Remark \ref{re:ribet}, for every arrow $a=a(\fn,\fn\fl)$ of $\fX_k$ with $\fn\in\fN_k^\indef$ and every $\fm\in\fN_k$ satisfying $\fn\fl\subseteq\fm$, we have the map $\mho^a_0/\sfn^{\fm,k}_0$ for the reduction of the Shimura variety $\Sh(\rV^\fn_{n_0},\rK^\fn_{n_0})$ at the place $\fl$.
\end{enumerate}
\end{remark}

Put $\fN^\defin_k\coloneqq\fN_k\cap\fN^\defin_0$ and $\fN^\indef_k\coloneqq\fN_k\cap\fN^\indef_0$. For every $\fl\in\fL_1$, we fix isomorphisms
\begin{align}\label{eq:rigidify}
\rH^1_\unr(F_\fl,T^\tc/\wp^{k(\fl)}T^\tc)\simeq O_\wp/\wp^{k(\fl)},\quad
\rH^1_\sing(F_\fl,T^\tc/\wp^{k(\fl)}T^\tc)\simeq O_\wp/\wp^{k(\fl)}
\end{align}
once and for all, which is possible by Remark \ref{re:location}(3,4).

We now introduce congruence modules.

\begin{definition}\label{de:congruence1}
Let $k\geq 1$ be a positive integer.
\begin{enumerate}
  \item For $\fn\in\fN_k$, we define the \emph{congruence module} ($\modulo\wp^k$) at $\fn$ to be
     \[
     \cC^{\fn,k}\coloneqq
     \begin{dcases}
     \Hom_{O_\wp}\(\cM^\fn/(\sfn^{\fn,k}_0,\sfn^{\fn,k}_1),O_\wp/\wp^k\), &\fn\in\fN_k^\defin, \\
     \Hom_{O_\wp[\Gamma_F]}\(\cM^\fn/(\sfn^{\fn,k}_0,\sfn^{\fn,k}_1),T^\tc/\wp^kT^\tc\), &\fn\in\fN_k^\indef.
     \end{dcases}
     \]

  \item For every arrow $a=a(\fn,\fn\fl)$ of $\fX_k$ with $\fn\in\fN_k^\defin$, we define the \emph{reciprocity map}
     \[
     \varrho^{a,k}\colon\cC^{\fn,k}\to\cC^{\fn\fl,k}
     \]
     of $O_\wp$-modules, which is an isomorphism for $\fn\in\fN_k^\defin\cap\fN_k^\eff$, as follows. When $\fn\not\in\fN_k^\eff$, we define $\varrho^{a,k}$ to be the zero map. Now suppose that $\fn\in\fN_k^\eff$. By Proposition \ref{pr:congruence}(2) below, we may apply Definition \ref{de:first_reciprocity} to the data $\lambda=\wp$, $m=k$, $\Sigma^+_{\lr,\rI}=\fn$, $\Sigma^+_\rI=\Sigma^+_\mnm\cup\Sigma^+_{\underline\fn}$, $\cV^\circ=\cV^\fn$, $\fp=\fl$, $\cV'=\cV^{\fn\fl}$, and $\tj=\tj^{a}$ to obtain an isomorphism
     \[
     \varrho_\sing\colon\rH^1_\sing(F_\fl,\cM^{\fn\fl}/(\sfn^{\fn\fl,k}_0,\sfn^{\fn\fl,k}_1))
     \xrightarrow{\sim}\cM^\fn/(\sfn^{\fn\fl,k}_0,\sfn^{\fn\fl,k}_1).
     \]
     We then define $\varrho^{a,k}$ to be the composition of the following four maps:
     \begin{itemize}[label={\ding{118}}]
       \item the natural map
           \[
           \cC^{\fn,k}=\Hom_{O_\wp}\(\cM^\fn/(\sfn^{\fn,k}_0,\sfn^{\fn,k}_1),O_\wp/\wp^k\)
           \xrightarrow\sim\Hom_{O_\wp}\(\cM^\fn/(\sfn^{\fn\fl,k}_0,\sfn^{\fn\fl,k}_1),O_\wp/\wp^k\),
           \]
           which is an isomorphism by By Proposition \ref{pr:congruence}(1) below,

       \item the isomorphism
           \[
           \Hom_{O_\wp}\(\cM^\fn/(\sfn^{\fn\fl,k}_0,\sfn^{\fn\fl,k}_1),O_\wp/\wp^k\)
           \xrightarrow\sim\Hom_{O_\wp}\(\rH^1_\sing(F_\fl,\cM^{\fn\fl}/(\sfn^{\fn\fl,k}_0,\sfn^{\fn\fl,k}_1)),O_\wp/\wp^k\)
           \]
           that is the dual of $\varrho_\sing$,

       \item the isomorphism
           \[
           \Hom_{O_\wp}\(\rH^1_\sing(F_\fl,\cM^{\fn\fl}/(\sfn^{\fn\fl,k}_0,\sfn^{\fn\fl,k}_1)),O_\wp/\wp^k\)
           \xrightarrow{\sim}\Hom_{O_\wp}\(\rH^1_\sing(F_\fl,\cM^{\fn\fl}/(\sfn^{\fn\fl,k}_0,\sfn^{\fn\fl,k}_1)),\rH^1_\sing(F_\fl,T^\tc/\wp^k T^\tc)\)
           \]
           induced by \eqref{eq:rigidify}, and

       \item the isomorphism
           \[
           \Hom_{O_\wp}\(\rH^1_\sing(F_\fl,\cM^{\fn\fl}/(\sfn^{\fn\fl,k}_0,\sfn^{\fn\fl,k}_1)),\rH^1_\sing(F_\fl,T^\tc/\wp^k T^\tc)\)\xrightarrow\sim\cC^{\fn\fl,k}
           \]
           that is the inverse of the natural isomorphism from Lemma \ref{le:congruence1} below.
     \end{itemize}

  \item For every arrow $a=a(\fn,\fn\fl)$ of $\fX_k$ with $\fn\in\fN_k^\indef$, we define the \emph{reciprocity map}
     \[
     \varrho^{a,k}\colon\cC^{\fn,k}\to\cC^{\fn\fl,k}
     \]
     of $O_\wp$-modules, which is an isomorphism for $\fn\in(\fN_k^\indef\cap\fN_k^\eff)\setminus\{\emptyset\}$, as follows. When $\fn\not\in\fN_k^\eff$, we define $\varrho^{a,k}$ to be the zero map. Now suppose that $\fn\in\fN_k^\eff$. By Proposition \ref{pr:congruence}(2) below, we may apply Definition \ref{de:second_reciprocity} to the data $\lambda=\wp$, $m=k$, $\Sigma^+_{\lr,\r{II}}=\fn$, $\Sigma^+_{\r{II}}=\Sigma^+_\mnm\cup\Sigma^+_{\underline\fn}$, $\cV=\cV^\fn$, $\fp=\fl$, $\cV'=\cV^{\fn\fl}$, and $\tj=\tj^{a}$ to obtain a map
     \[
     \varrho_\unr\colon\cM^{\fn\fl}/(\sfn^{\fn\fl,k}_0,\sfn^{\fn\fl,k}_1)
     \to\rH^1_\unr(F_\fl,\cM^\fn/(\sfn^{\fn\fl,k}_0,\sfn^{\fn\fl,k}_1)),
     \]
     which is an isomorphism when $\fn\neq\emptyset$ by Remark \ref{re:second}(1) and Proposition \ref{pr:congruence}(3) below. We then define $\varrho^{a,k}$ to be the composition of the following four maps:
     \begin{itemize}[label={\ding{118}}]
       \item the map
           \[
           \cC^{\fn,k}\xrightarrow\sim\Hom_{O_\wp[\Gamma_F]}\(\cM^\fn/(\sfn^{\fn\fl,k}_0,\sfn^{\fn\fl,k}_1),T^\tc/\wp^k T^\tc\)
           \]
           which is an isomorphism by By Proposition \ref{pr:congruence}(1) below,

       \item the natural isomorphism
           \[
           \Hom_{O_\wp[\Gamma_F]}\(\cM^\fn/(\sfn^{\fn\fl,k}_0,\sfn^{\fn\fl,k}_1),T^\tc/\wp^k T^\tc\)
           \xrightarrow\sim
           \Hom_{O_\wp}\(\rH^1_\unr(F_\fl,\cM^\fn/(\sfn^{\fn\fl,k}_0,\sfn^{\fn\fl,k}_1)),\rH^1_\unr(F_\fl,T^\tc/\wp^k T^\tc)\)
           \]
           from Lemma \ref{le:congruence1} below,

       \item the isomorphism
           \[
           \Hom_{O_\wp}\(\rH^1_\unr(F_\fl,\cM^\fn/(\sfn^{\fn\fl,k}_0,\sfn^{\fn\fl,k}_1)),\rH^1_\unr(F_\fl,T^\tc/\wp^k T^\tc)\)
           \xrightarrow\sim\Hom_{O_\wp}\(\rH^1_\unr(F_\fl,\cM^\fn/(\sfn^{\fn\fl,k}_0,\sfn^{\fn\fl,k}_1)),O_\wp/\wp^k\)
           \]
           induced by \eqref{eq:rigidify}, and

       \item the map
           \[
           \Hom_{O_\wp}\(\rH^1_\unr(F_\fl,\cM^\fn/(\sfn^{\fn\fl,k}_0,\sfn^{\fn\fl,k}_1)),O_\wp/\wp^k\)
           \to\cC^{\fn\fl,k}=\Hom_{O_\wp}\(\cM^{\fn\fl}/(\sfn^{\fn\fl,k}_0,\sfn^{\fn\fl,k}_1),O_\wp/\wp^k\)
           \]
           that is the dual of $\varrho_\unr$.
     \end{itemize}
\end{enumerate}
\end{definition}

\begin{remark}\label{re:congruence}
We have the following remarks concerning the above construction.
\begin{enumerate}
  \item By our choice of $(\rK_n,\rK_{n+1})$ in the initial data in Notation \ref{no:initial}, $\cC^{\emptyset,k}$ is nontrivial for every $k\geq 1$.

  \item The fact that $\varrho^{a,k}$ is an isomorphism for $\fn\in\fN_k^\defin\cap\fN_k^\eff$ together with (1) imply that when $\epsilon(\Pi_0\times\Pi_1)=1$, every element of $\fL_k$ is effective.

  \item Note that $\fX_k$ is a subgraph of $\fX_{k-1}$, and we have a natural reduction map $\cC^{\fn,k}\to\cC^{\fn,k-1}$ for every $\fn\in\fN_k$ with $k\geq 2$. It follows from the construction that the diagram
      \[
      \xymatrix{
      \cC^{\fn,k} \ar[r]^-{\varrho^{a,k}}\ar[d] & \cC^{\fn\fl,k} \ar[d] \\
      \cC^{\fn,k-1} \ar[r]^-{\varrho^{a,k-1}} & \cC^{\fn\fl,k-1}
      }
      \]
      commutes for every arrow $a=a(\fn,\fn\fl)$ of $\fX_k$ with $k\geq 2$.
\end{enumerate}
\end{remark}

\begin{proposition}\label{pr:congruence}
Let $k$ be a positive integer. The following statements hold for every $\fn\in\fN_k^\eff$.
\begin{enumerate}
  \item For every $\fm\in\fN_k$ containing $\fn$, the natural map $\cM^\fn/(\sfn^{\fm,k}_0,\sfn^{\fm,k}_1)\to\cM^\fn/(\sfn^{\fn,k}_0,\sfn^{\fn,k}_1)$ is bijective.

  \item The cardinality of $\cC^{\fn,k}$ is same as that of $\cC^{\emptyset,k}$; in particular, $\cC^{\fn,k}$ is nontrivial by Remark \ref{re:congruence}(1).

  \item If $\fn\in\fN_k^\indef\setminus\{\emptyset\}$, then for every arrow $a=a(\fn,\fn\fl)$ of $\fX_k$ and every $\fm\in\fN_k$ containing $\fn\fl$, the map $\mho^a_0/\sfn^{\fm,k}_0$ from Remark \ref{re:location}(8) is bijective.
\end{enumerate}
\end{proposition}

\begin{proof}
We prove the three statements simultaneously by the induction on $|\fn|$. When $|\fn|=0$, that is, $\fn=\emptyset$, (1) follows from (C1) in Definition \ref{de:congruence}(1); (2) is obvious; and (3) is vacuous. Now we consider an element $\fn\in\fN_k^\eff$ with $|\fn|>0$ and assume that (1--3) are known for every $\fn'\in\fN_k^\eff$ with $|\fn'|<|\fn|$.

We first consider the case where $\fn\in\fN_k^\defin$ and $|\fn|>1$. Choose an element $\fl'$ of $\fn$ so that $\fn/\fl'$ remains effective; and write $a'=a(\fn/\fl',\fn)$. Then $\fn/\fl'\in(\fN_k^\indef\cap\fN_k^\eff)\setminus\{\emptyset\}$ with $|\fn/\fl'|=|\fn|-1$, for which (3) applies. Apply Definition \ref{de:second_reciprocity} to the data $\lambda=\wp$, $m=k$, $\Sigma^+_{\lr,\r{II}}=\fn/\fl'$, $\Sigma^+_{\r{II}}=\Sigma^+_\mnm\cup\Sigma^+_{\underline{\fm/\fl'}}$, $\cV=\cV^{\fn/\fl'}$, $\fp=\fl$, $\cV'=\cV^\fn$, and $\tj=\tj^{a'}$. We obtain a map
\[
\varrho_\unr\colon\cM^\fn/(\sfn^{\fm,k}_0,\sfn^{\fm,k}_1)
\xrightarrow{\sim}\rH^1_\unr(F_\fl,\cM^{\fn/\fl'}/(\sfn^{\fm,k}_0,\sfn^{\fm,k}_1)),
\]
which is an isomorphism by Remark \ref{re:second}(1). Thus, (1) for $\fn$ follows from (1) for $\fn/\fl'$. Taking dual of $\varrho_\unr$ and by \eqref{eq:rigidify}, we obtain an isomorphism
\[
\Hom_{O_\wp}\(\rH^1_\unr(F_\fl,\cM^{\fn/\fl'}/(\sfn^{\fn,k}_0,\sfn^{\fn,k}_1)),\rH^1_\unr(F_\fl,T^\tc/\wp^k T^\tc)\)
\simeq\cC^{\fn,k}.
\]
By Lemma \ref{le:congruence1} below and (1) for $\fn/\fl'$, the left-hand side is canonically isomorphic to $\cC^{\fn/\fl',k}$. Thus, (2) for $\fn$ follows. Part (3) is vacuous for $\fn$.

We then consider the case where $\fn\in\fN_k^\defin$ and $|\fn|=1$, so that $\fn=\{\fl\}$ for some effective element $\fl\in\fL_k$. We have
\[
\cM^\fn/(\sfn^{\fm,k}_0,\sfn^{\fm,k}_1)=
O_\wp[\Sh(\rV^\fn_{n_0},\rK^\fn_{n_0})]/\sfn^{\fm,k}_0\otimes_{O_\wp}O_\wp[\Sh(\rV^\fn_{n_1},\rK^\fn_{n_1})]/\sfn^{\fm,k}_1.
\]
On the other hand, the K\"{u}nneth formula induces an isomorphism
\begin{align*}
&\rH^1_\unr(F_\fl,\cM^\emptyset/(\sfn^{\fm,k}_0,\sfn^{\fm,k}_1)) \\
&\simeq
\rH^1_\unr(F_\fl,\rH^{2r_0-1}_{\et}(\Sh(\rV^\emptyset_{n_0},\rK^\emptyset_{n_0})_{\ol{F}},O_\wp(r_0))/\sfn^{\fm,k}_0)
\otimes_{O_\wp}
\(\rH^{2r_1}_{\et}(\Sh(\rV^\emptyset_{n_1},\rK^\emptyset_{n_1})_{\ol{F}},O_\wp(r_1))/\sfn^{\fm,k}_1\)_{\Gamma_{F_\fl}}.
\end{align*}
Thus, by Lemma \ref{le:congruence1} below and (1,2) for $\emptyset\in\fN_k$, in order to show (1,2) for $\fn$, it suffices to show that
\begin{enumerate}[label=(\alph*)]
  \item $O_\wp[\Sh(\rV^\fn_{n_0},\rK^\fn_{n_0})]/\sfn^{\fm,k}_0$ and $\rH^1_\unr(F_\fl,\rH^{2r_0-1}_{\et}(\Sh(\rV^\emptyset_{n_0},\rK^\emptyset_{n_0})_{\ol{F}},O_\wp(r_0))/\sfn^{\fm,k}_0)$ have the same cardinality;

  \item $O_\wp[\Sh(\rV^\fn_{n_1},\rK^\fn_{n_1})]/\sfn^{\fm,k}_1$ and $\(\rH^{2r_1}_{\et}(\Sh(\rV^\emptyset_{n_1},\rK^\emptyset_{n_1})_{\ol{F}},O_\wp(r_1))/\sfn^{\fm,k}_1\)_{\Gamma_{F_\fl}}$ have the same cardinality.
\end{enumerate}
Part (a) follows from Lemma \ref{le:congruence2} below, which is applicable since $\cC^{\emptyset,k}$ is nontrivial and $\cC^{\fn,k}$ is also nontrivial by Definition \ref{de:congruence}(4). Part (b) follows from Lemma \ref{le:second}(1) (for $\alpha=1$) and \cite{LTXZZ}*{Lemma~7.3.3(2)}. Part (3) is vacuous for $\fn$.

Finally we consider the case where $\fn\in\fN_k^\indef$. Choose an element $\fl'$ of $\fn$ so that $\fn/\fl'$ remains effective; and write $a'=a(\fn/\fl',\fn)$. Then $\fn/\fl'\in\fN_k^\defin$ with $|\fn/\fl'|=|\fn|-1$, for which (2) applies. Apply Definition \ref{de:first_reciprocity} to the data $\lambda=\wp$, $m=k$, $\Sigma^+_{\lr,\rI}=\fn/\fl'$, $\Sigma^+_\rI=\Sigma^+_\mnm\cup\Sigma^+_{\underline{\fm/\fl}}$, $\cV^\circ=\cV^{\fn/\fl'}$, $\fp=\fl$, $\cV'=\cV^\fn$, and $\tj=\tj^{a'}$. We obtain an isomorphism
\[
\varrho_\sing\colon\rH^1_\sing(F_\fl,\cM^\fn/(\sfn^{\fm,k}_0,\sfn^{\fm,k}_1))
\xrightarrow{\sim}\cM^{\fn/\fl'}/(\sfn^{\fm,k}_0,\sfn^{\fm,k}_1).
\]
Thus, (1) for $\fn$ follows from (1) for $\fn/\fl'$, using Lemma \ref{le:congruence1} below. Taking dual of $\varrho_\sing$ and by \eqref{eq:rigidify}, we obtain an isomorphism
\[
\cC^{\fn/\fl',k}\simeq
\Hom_{O_\wp}\(\rH^1_\sing(F_\fl,\cM^\fn/(\sfn^{\fn,k}_0,\sfn^{\fn,k}_1)),\rH^1_\sing(F_\fl,T^\tc/\wp^k T^\tc)\).
\]
By Lemma \ref{le:congruence1} below, the right-hand side is canonically isomorphic to $\cC^{\fn,k}$. Thus, (2) for $\fn$ follows.

It remains to show (3) for $\fn$. By Lemma \ref{le:congruence2} below and (2) for $\fn$, it suffices to show that $\mho^a_0/\sfn^{\fm,k}_0$ is surjective. By Nakayama's lemma, it suffices to show that $\mho^a_0/\sfm^\fm_0$ is surjective. As explained in Remark \ref{re:significance}, by Theorem \ref{th:boosting} (applied to $\pres\varpi\rV_{n_0}$, where $\varpi\in F^+$ is an element chosen as in \S\ref{ss:setup} for $\fp=\fl$ with the additional requirement that $\val_{\fl''}(\varpi)=0$ for every prime $\fl''$ of $F^+$ that has the same residue characteristic with $\fl'$), the definition of the boosting map \eqref{eq:boosting_2} and Remark \ref{re:vartheta}, it suffices to show that the $\Gamma_F$-equivariant map
\[
\vartheta^\eta\colon\rH^{2r_0-1}_{\et}(\Sh(\rV^\fn_{n_0},\pres\varpi\rK^\fn_{n_0})_{\ol{F}},\Omega_{n_0,O_\wp}^\eta(r_0))
\to\rH^{2r_0-1}_{\et}(\Sh(\rV^\fn_{n_0},\rK^\fn_{n_0})_{\ol{F}},O_\wp(r_0))
\]
(from Remark \ref{re:vartheta} with $\rV=\pres\varpi\rV_{n_0}$ and $\fp=\fl$) is surjective after taking quotient by $\sfm^\fm_0$. By Lemma \ref{le:congruence1} below, it suffices to show that the induced map $\rH^1_\sing(F_{\fl'},\vartheta^\eta/\sfm^\fm_0)$ is surjective. Denote by $\Omega_{n_0,O_\wp}[\Sh(\rV^{\fn/\fl'}_{n_0},\pres\varpi\rK^{\fn/\fl'}_{n_0})]$ the $O_\wp$-module of $\Omega_{n_0,O_\wp}$-equivariant $O_\wp$-valued functions on (the Shimura set) $\Sh(\rV^{\fn/\fl'}_{n_0},\pres\varpi\rK^{\fn/\fl'}_{n_0})$ (once again, the local system $\Omega_{n_0,O_\wp}$ is defined with respect to the place $\fl$). Then we have a commutative diagram
\[
\xymatrix{
\rH^1_\sing(F_{\fl'},\rH^{2r_0-1}_{\et}(\Sh(\rV^\fn_{n_0},\pres\varpi\rK^\fn_{n_0})_{\ol{F}},\Omega_{n_0,O_\wp}^\eta(r_0))/\sfm^\fm_0) \ar[r]^-{\text{\ding{192}}} &
\rH^1_\sing(F_{\fl'},\rH^{2r_0-1}_{\et}(\Sh(\rV^\fn_{n_0},\rK^\fn_{n_0})_{\ol{F}},O_\wp(r_0))/\sfm^\fm_0) \\
\rF_{-1}\rH^1(\rI_{F_{\fl'}},\rH^{2r_0-1}_{\et}(\Sh(\rV^\fn_{n_0},\pres\varpi\rK^\fn_{n_0})_{\ol{F}},\Omega_{n_0,O_\wp}^\eta(r_0))/\sfm^\fm_0) \ar[r]^-{\text{\ding{193}}} \ar[u]^-{\text{\ding{195}}} \ar@{->>}[d]_-{\text{\ding{197}}} &
\rF_{-1}\rH^1(\rI_{F_{\fl'}},\rH^{2r_0-1}_{\et}(\Sh(\rV^\fn_{n_0},\rK^\fn_{n_0})_{\ol{F}},O_\wp(r_0))/\sfm^\fm_0) \ar[u]_-{\text{\ding{196}}}^\simeq \ar[d]^-{\text{\ding{198}}}_\simeq \\
\Omega_{n_0,O_\wp}[\Sh(\rV^{\fn/\fl'}_{n_0},\pres\varpi\rK^{\fn/\fl'}_{n_0})]/\sfm^{\fm/\fl'}_0 \ar[r]^-{\text{\ding{194}}} & O_\wp[\Sh(\rV^{\fn/\fl'}_{n_0},\rK^{\fn/\fl'}_{n_0})]/\sfm^{\fm/\fl'}_0
}
\]
in which the map \ding{192} is $\rH^1_\sing(F_{\fl'},\vartheta^\eta/\sfm^\fm_0)$; the map \ding{193} is $\rF_{-1}\rH^1(\rI_{F_{\fl'}},\vartheta^\eta/\sfm^\fm_0)$; the map \ding{194} is dual to the map $\bi$ defined above \cite{LTXZZ}*{Corollary~6.3.5} (for $\rV_N=\pres\varpi\rV^{\fn/\fl'}_{n_0}$ and $\fp=\fl$); the maps \ding{195} and \ding{196} are the maps from \cite{LTXZZ}*{Lemma~5.9.3(6)}; the maps \ding{197} and \ding{198} are the surjective maps from \cite{LTXZZ}*{Proposition~6.3.1(4)}. Since $O_\wp[\Sh(\rV^{\fn/\fl'}_{n_0},\rK^{\fn/\fl'}_{n_0})]/\sfm^{\fm/\fl'}_0$ is nontrivial by (2), we may apply \cite{LTXZZ}*{Theorem~6.3.4(4)} to conclude that both \ding{196} and \ding{198} are isomorphisms.\footnote{We warn the readers that \cite{LTXZZ}*{Theorem~6.3.4} does not apply to \ding{195} and \ding{197} because the rigidity assumption in \cite{LTXZZ}*{Theorem~6.3.4(b)} fails for $(\Sigma^+_\mnm\cup\{\fl\},\fn)$.} By (the dual statement of) \cite{LTXZZ}*{Corollary~6.3.5}, \text{\ding{194}} is surjective hence \ding{192} is surjective as well.

The proposition is all proved.
\end{proof}

\begin{lem}\label{le:congruence1}
For every integer $k\geq 1$ and every elements $\fn\in\fN_k^\indef$ and $\fm\in\fN_k$ satisfying $\fn\subseteq\fm$, the $O_\wp[\Gamma_F]$-module $\cM^\fn/(\sfn^{\fm,k}_0,\sfn^{\fm,k}_1)$ is isomorphic to the direct sum of finitely many modules of the form $T^\tc/\wp^{k'}T^\tc$ with $k'\leq k$. In particular,
\begin{enumerate}
  \item for every element $\fl\in\fn$, the natural map
     \[
     \Hom_{O_\wp[\Gamma_F]}\(\cM^\fn/(\sfn^{\fm,k}_0,\sfn^{\fm,k}_1),T^\tc/\wp^k T^\tc\)
     \to\Hom_{O_\wp}\(\rH^1_\sing(F_\fl,\cM^\fn/(\sfn^{\fm,k}_0,\sfn^{\fm,k}_1)),\rH^1_\sing(F_\fl,T^\tc/\wp^k T^\tc)\)
     \]
     is an isomorphism;

  \item for every element $\fl\in\fm\setminus\fn$, the natural map
     \[
     \Hom_{O_\wp[\Gamma_F]}\(\cM^\fn/(\sfn^{\fm,k}_0,\sfn^{\fm,k}_1),T^\tc/\wp^k T^\tc\)
     \to\Hom_{O_\wp}\(\rH^1_\unr(F_\fl,\cM^\fn/(\sfn^{\fm,k}_0,\sfn^{\fm,k}_1)),\rH^1_\unr(F_\fl,T^\tc/\wp^k T^\tc)\)
     \]
     is an isomorphism.
\end{enumerate}
\end{lem}

\begin{proof}
The second statement follows from the first, Remark \ref{re:location}(3), and the fact that $T/\wp T$ is an absolutely irreducible representation. Now we show the first statement.

For $\alpha=0,1$, denote by $\sfT^\fm_\alpha$ the $O_\wp$-subring of $\End_{O_\wp}\cM^\fn$ generated by $\dT^\fm_\alpha$. If $\sfm^\fm_\alpha$ maps surjectively to $\sfT^\fm_\alpha$ for some $\alpha\in\{0,1\}$, then $\cM^\fn/(\sfn^{\fm,k}_0,\sfn^{\fm,k}_1)=0$. Otherwise, $\sfm^\fm_\alpha$ generates a maximal ideal of $\sfT^\fm_\alpha$ for $\alpha=0,1$. For $\alpha=0,1$, let
\[
\rho_{\sfm^\fm_\alpha}\colon\Gamma_F\to\GL_{n_\alpha}((\sfT^\fm_\alpha)_{\sfm^\fm_\alpha})
\]
be the homomorphism constructed similarly as in the proof of \cite{LTXZZ2}*{Theorem~3.38}, lifting the (absolutely irreducible) residue representation of $\rho_{\Pi_\alpha,\wp}^\tc(r_\alpha)$. By the K\"{u}nneth formula, we have
\[
\cM^\fn_{(\sfm^\fm_0,\sfm^\fm_1)}=
\rH^{2r_0-1}_\et(\Sh(\rV^\fn_{n_0},\rK^\fn_{n_0})_{\ol{F}},O_\wp(r_0))_{\sfm^\fm_0}
\otimes_{O_\wp}\rH^{2r_1}_\et(\Sh(\rV^\fn_{n_1},\rK^\fn_{n_1})_{\ol{F}},O_\wp(r_1))_{\sfm^\fm_1}.
\]
By the same argument above \cite{LTXZZ}*{Proposition~6.4.1}, we have isomorphisms
\begin{align*}
\rH^{2r_0-1}_\et(\Sh(\rV^\fn_{n_0},\rK^\fn_{n_0})_{\ol{F}},O_\wp(r_0))_{\sfm^\fm_0}&\simeq
\(\sfH_0\otimes_{(\sfT^\fm_0)_{\sfm^\fm_0}}(\sfT^\fm_0)_{\sfm^\fm_0}^{\oplus n_0}\),\\
\rH^{2r_1}_\et(\Sh(\rV^\fn_{n_1},\rK^\fn_{n_1})_{\ol{F}},O_\wp(r_1))_{\sfm^\fm_1}&\simeq
\(\sfH_1\otimes_{(\sfT^\fm_1)_{\sfm^\fm_1}}(\sfT^\fm_1)_{\sfm^\fm_1}^{\oplus n_1}\),
\end{align*}
of $O_\lambda[\Gamma_F]$-modules for some $(\sfT^\fm_\alpha)_{\sfm^\fm_\alpha}$-module $\sfH_\alpha$ (with trivial $\Gamma_F$-action) for $\alpha=0,1$. Since for $\alpha=0,1$, the natural map $O_\wp\to(\sfT^\fm_\alpha)_{\sfm^\fm_\alpha}$ induces an isomorphism $O_\wp/\wp^k\xrightarrow\sim\sfT^\fm_\alpha/\sfn^{\fm,k}_\alpha$, it remains to show that $(\rho_{\sfm^\fm_0}\modulo\sfn^{\fm,k}_0)\otimes(\rho_{\sfm^\fm_1}\modulo\sfn^{\fm,k}_1)$ and $T^\tc/\wp^k T^\tc$ are isomorphic liftings of $T^\tc/\wp T^\tc$ -- this follows from the Chebotarev density theorem and \cite{Car94}*{Th\'{e}or\`{e}me~1}.
\end{proof}

\begin{lem}\label{le:congruence2}
Let $k$ be a positive integer. For every arrow $a(\fn,\fn\fl)$ of $\fX_k$ with $\fn\in\fN_k^\indef$ and every $\fm\in\fN_k$ containing $\fn\fl$, if both of the following two $O_\wp$-modules
\[
O_\wp[\Sh(\rV^{\fn\fl}_{n_0},\rK^{\fn\fl}_{n_0})]/\sfn^{\fm,k}_0,\quad
\rH^1_\unr(F_\fl,\rH^{2r_0-1}_{\et}(\Sh(\rV^\fn_{n_0},\rK^\fn_{n_0})_{\ol{F}},O_\wp(r_0))/\sfn^{\fm,k}_0)
\]
are nontrivial, then they are isomorphic.
\end{lem}

\begin{proof}
We adopt the same strategy in the proof of the counterpart of this lemma in \cite{LTXZZ}*{Theorem~6.3.4}. We follow the setup in \cite{LTXZZ}*{\S6.4} and make necessary changes. Take $\lambda=\wp$, $\Sigma^+_\lr=\fn$, $\Sigma^+=\Sigma^+_\mnm\cup\Sigma^+_{\underline{\fm/\fl}}$, and $\fp=\fl$. We have global deformation rings\footnote{However, there is a typo on \cite{LTXZZ}*{Page~281}, namely, the formula $r_\mix^\natural(t)\sfv'=\sfx\sfv+\sfv'$ should be $r_\mix^\natural(t)\sfv=\sfx\sfv'+\sfv$.}
\[
\sfR^\mix,\quad
\sfR^\unr=\sfR^\mix/(\sfx),\quad
\sfR^\ram=\sfR^\mix/(\sfs-\ell^{-2r}),\quad
\sfR^{\r{cong}}=\sfR^\mix/(\sfs-\ell^{-2r},\sfx),
\]
where $\ell$ is the underlying rational prime of $\fl$ (rather than $p$ in the reference). Let $\sfT^\unr$ and $\sfT^\ram$ be the $O_\wp$-subrings of
\[
\End_{O_\wp}\(\rH^{2r_0-1}_{\et}(\Sh(\rV^\fn_{n_0},\rK^\fn_{n_0})_{\ol{F}},O_\wp(r_0))\),\quad
\End_{O_\wp}\(O_\wp[\Sh(\rV^{\fn\fl}_{n_0},\rK^{\fn\fl}_{n_0})]\)
\]
generated by $\dT^\fm_0$.\footnote{Note that the roles  $\sfT^\unr$ and $\sfT^\ram$ are switched from those in \cite{LTXZZ}*{\S6.4} since now the Shimura variety has the hyperspecial level at $\fl$ and the Shimura set does not.} The assumption in the statement implies that both $(\sfT^\unr)_{\sfm^\fm_0}$ and $(\sfT^\ram)_{\sfm^\fm_0}$ are nontrivial. By \cite{LTXZZ2}*{Theorem~3.38}, we have
\begin{itemize}[label={\ding{118}}]
  \item the natural homomorphisms $\sfR^\unr\to(\sfT^\unr)_{\sfm^\fm_0}$ and $\sfR^\ram\to(\sfT^\ram)_{\sfm^\fm_0}$ are both isomorphisms;

  \item $\rH^{2r_0-1}_{\et}(\Sh(\rV^\fn_{n_0},\rK^\fn_{n_0})_{\ol{F}},O_\wp(r_0))_{\sfm^\fm_0}
      \simeq\sfH\otimes_{\sfR^\unr}(\sfR^\unr)^{\oplus n_0}$ for a free $\sfR^\unr$-module $\sfH$ of rank, say, $d_\unr$;

  \item $O_\wp[\Sh(\rV^{\fn\fl}_{n_0},\rK^{\fn\fl}_{n_0})]_{\sfm^\fm_0}$ is a free $\sfR^\ram$-module of rank, say, $d_\ram$.
\end{itemize}

By the same argument for \cite{LTXZZ}*{Proposition~6.4.1}, we have $d_\unr=d_\ram$. It is straightforward to show that $\rH^1_\unr(F_\fl,\rH^{2r_0-1}_{\et}(\Sh(\rV^\fn_{n_0},\rK^\fn_{n_0})_{\ol{F}},O_\wp(r_0))_{\sfm^\fm_0})$ is isomorphic to $\sfH\otimes_{\sfR^\unr}\sfR^{\r{cong}}$. Thus, for the lemma, it remains to show that the quotient map $O_\wp[\Sh(\rV^{\fn\fl}_{n_0},\rK^{\fn\fl}_{n_0})]\to O_\wp[\Sh(\rV^{\fn\fl}_{n_0},\rK^{\fn\fl}_{n_0})]\otimes_{\sfR^\ram}\sfR^{\r{cong}}$ becomes an isomorphism after taking quotient by $\sfn^{\fm,k}_0$. In other words, it remains to show that the element $\sfx$, which measures the monodromy operator at $\fl$, acts by zero on the quotient $O_\wp[\Sh(\rV^{\fn\fl}_{n_0},\rK^{\fn\fl}_{n_0})]/\sfn^{\fm,k}_0$. The heuristic reason for the previous claim is that the Galois representation ``associated with'' $O_\wp[\Sh(\rV^{\fn\fl}_{n_0},\rK^{\fn\fl}_{n_0})]/\sfn^{\fm,k}_0$ should be isomorphic to the direct sum of finitely many copies of $T_0^\tc/\wp^k T_0^\tc$, where $T_0$ denotes a $\Gamma_F$-stable $O_\wp$-lattice of $\rho_{\Pi_0,\wp}(r_0)$, which is unramified at $\fl$.

To rigorously show this, we need to use the level-raising isomorphism in \cite{LTXZZ}*{Theorem~6.3.4} at an auxiliary element $\fl'\in\fL_k\setminus\fm$ satisfying that the natural map $O_\wp[\Sh(\rV^{\fn\fl}_{n_0},\rK^{\fn\fl}_{n_0})]/\sfn^{\fm',k}_0\to O_\wp[\Sh(\rV^{\fn\fl}_{n_0},\rK^{\fn\fl}_{n_0})]/\sfn^{\fm,k}_0$ is an isomorphism (such element exists by Remark \ref{re:location}(5)). Put $\fm'\coloneqq\fm\fl'\in\fN_k$.
By \cite{LTXZZ}*{Theorem~6.3.4(4)}, we have a natural isomorphism
\begin{align}\label{eq:spot}
O_\wp[\Sh(\rV^{\fn\fl}_{n_0},\rK^{\fn\fl}_{n_0})]/\sfn^{\fm,k}_0
\simeq\rH^1_\sing(F_{\fl'},\rH^{2r_0-1}_{\et}(\Sh(\rV^{\fn\fl\fl'}_{n_0},\rK^{\fn\fl\fl'}_{n_0})_{\ol{F}},O_\wp(r_0))/\sfn^{\fm',k}_0)
\end{align}
that is equivariant under the action of the local deformation ring at $\fl$, in particular, the element $\sfx$. By \cite{LTXZZ}*{Theorem~6.3.4(5)}, $\rH^{2r_0-1}_{\et}(\Sh(\rV^{\fn\fl\fl'}_{n_0},\rK^{\fn\fl\fl'}_{n_0})_{\ol{F}},O_\wp(r_0))/\sfn^{\fm',k}_0$ is isomorphic to the direct sum of finitely many copies of $T_0^\tc/\wp^k T_0^\tc$, which is unramified at $\fl$. It follows that the action of $\sfx$ on the right-hand side of \eqref{eq:spot} is zero, hence also on the left-hand side. The lemma then follows.
\end{proof}

\begin{remark}\label{re:ihara}
We conjecture that when $\epsilon(\Pi_0\times\Pi_1)=-1$, the map $\mho^a_0/\sfm^{\{\fl\}}_0$ is also surjective for $a=a(\emptyset,\{\fl\})$, which can be regarded as an analogue of the Ihara lemma. However, we are unable to prove it -- this is the reason for introducing effective elements.
\end{remark}

\subsection{The lambda maps}
\label{ss:lambda}

This subsection has two goals. First, we construct for every $\fn\in\fN_k^\defin$ with $k\geq 1$ a map
\[
\blambda^{\fn,k}_\cF\colon\cC^{\fn,k}\to\Lambda_\cF/\wp^k\Lambda_\cF
\]
compatible with respect to $k$. Second, we recall from \cite{Liu5} the definition of the $\wp$-adic $L$-function $\sL_\cF(\Pi_0\times\Pi_1)\in\Lambda_{\cF,E_\wp}$ and its relation with $\blambda^{\emptyset,k}_\cF$ when $\epsilon(\Pi_0\times\Pi_1)=1$.

\begin{definition}\label{de:frame}
A \emph{frame} (at $\Sigma_\cF^+$) is a collection of isomorphisms $\rU(\Lambda_{N,v})\simeq\GL_{N,O_{F^+_v}}$ for $N\in\{n,n+1\}$ and $v\in\Sigma_\cF^+$, such that for every $v\in\Sigma_\cF^+$, the natural embedding $\rU(\Lambda_{n,v})\hookrightarrow\rU(\Lambda_{n+1,v})$ corresponds to the assignment $h\mapsto\diag(h,1)$.
\end{definition}

We choose a frame, through which we regard $\pi_{\alpha,v}^\vee$ as an (unramified) irreducible representation of $\GL_{n_\alpha}(F^+_v)$ for $\alpha=0,1$ and every $v\in\Sigma_\cF^+$. Since it is ordinary, we may write its Satake parameter as $\{\omega_{\alpha,v,1},\ldots,\omega_{\alpha,v,n_\alpha}\}$ such that $\omega_{\alpha,v,i}^\natural\coloneqq\sqrt{\|v\|}^{n_\alpha+1-2i}\omega_{\alpha,v,i}\in O_\wp^\times$. Put
\[
\omega(\pi_{0,v}^\vee,\pi_{1,v}^\vee)\coloneqq
\prod_{\alpha=0}^1\prod_{j=1}^n\prod_{i=1}^j\omega_{\alpha,v,i}^\natural\in O_\wp^\times
\]
for later use.

Denote by $\rI_{N,v}\subseteq\GL_N(O_{F^+_v})=\rK_{N,v}$ the upper-triangular Iwahori subgroup for $N\in\{n,n+1\}$ and $v\in\Sigma_\cF^+$. For $v\in\Sigma_\cF^+$ and $\alpha=0,1$, we define the \emph{ordinary projection operator} $\tP_{\alpha,v}\in O_\wp[\rI_{n_\alpha,v}\backslash\GL_{n_\alpha}(F^+_v)/\GL_{n_\alpha}(O_{F^+_v})]$ to be
\begin{align*}
\tP_{\alpha,v}&\coloneqq\prod_{i=1}^{n_\alpha-1}
\(\prod_{j=1}^{i-1}\(\|v\|^{i-j}\omega_{\alpha,v,i}^\natural\cdot\tV_{n_\alpha,v,j-1}-\tV_{n_\alpha,v,j}\)
\prod_{j=i+1}^{n_\alpha}\(\omega_{\alpha,v,i}^\natural\cdot\tV_{n_\alpha,v,j-1}-\|v\|^{j-i}\tV_{n_\alpha,v,j}\)\) \\
&\in O_\wp[\rI_{n_\alpha,v}\backslash\GL_{n_\alpha}(F^+_v)/\GL_{n_\alpha}(O_{F^+_v})],
\end{align*}
following \cite{KMS00}*{Proposition~4.2}. Then $\tP_{\alpha,v}$ induces an isomorphism from $(\pi_{\alpha,v}^\vee)^{\GL_{n_\alpha}(O_{F^+_v})}$ to the ordinary line in $(\pi_{\alpha,v}^\vee)^{\rI_{n_\alpha,v}}$. We then put
\[
\tP_v\coloneqq\tP_{0,v}\otimes\tP_{1,v}\in
O_\wp[\rI_{n,v}\backslash\GL_n(F^+_v)/\GL_n(O_{F^+_v})]\otimes_{O_\wp}
O_\wp[\rI_{n+1,v}\backslash\GL_{n+1}(F^+_v)/\GL_{n+1}(O_{F^+_v})].
\]
Put $\rI_{N,\Sigma_\cF^+}\coloneqq\prod_{v\in\Sigma_\cF^+}\rI_{N,v}$ and $\rK_{N,\Sigma_\cF^+}\coloneqq\prod_{v\in\Sigma_\cF^+}\GL_N(O_{F^+_v})$ for $N\in\{n,n+1\}$. Put
\[
\tP\coloneqq\prod_{v\in\Sigma_\cF^+}\tP_v\in
O_\wp[\rI_{n,\Sigma_\cF^+}\times\rI_{n+1,\Sigma_\cF^+}\backslash\GL_n(F^+_{\Sigma_\cF^+})\times\GL_{n+1}(F^+_{\Sigma_\cF^+})/\rK_{n,\Sigma_\cF^+}\times\rK_{n+1,\Sigma_\cF^+}].
\]

\begin{notation}
For every $v\in\Sigma_\cF^+$ and every $\fc_v\in\dZ_{>0}$, we put
\[
\fU_{\fc_v}\coloneqq\left\{(x,x^{-1})\res x\in 1+p^{\fc_v}O_{F^+_v}\right\}\subseteq F_v^\times.
\]
For a tuple $\fc=(\fc_v)_v\in(\dZ_{>0})^{\Sigma_\cF^+}$, we
\begin{itemize}[label={\ding{118}}]
  \item put $p^\fc\coloneqq(p^{\fc_v})_{v\in\Sigma_\cF^+}\in\prod_{v\in\Sigma_\cF^+} O_{F^+_v}\setminus\{0\}$;

  \item put $\fU_\fc\coloneqq\prod_{v\in\Sigma_\cF^+}\fU_{\fc_v}$, which maps naturally to $\Upsilon$ with a finite kernel;

  \item denote by $F_\fc\subseteq \pres{\wp}\cF$ the fixed subfield of (the image of) $\fU_\fc$, which is a finite extension of $F$;

  \item put
     \begin{align*}
     \pres{\fc}g_n&\coloneqq
     \begin{pmatrix}
     (p^\fc)^n & & & \\
     & (p^\fc)^{n-1} & & \\
     & & \ddots & \\
     & & & p^\fc
     \end{pmatrix}\in\GL_n(F^+_{\Sigma_\cF^+}), \\
     \pres{\fc}g_{n+1}&\coloneqq
     \begin{pmatrix}
         &  & 1 & 1 \\
         & \iddots &  & \vdots \\
        1 &  &  & 1 \\
        0 & \cdots & 0 & 1 \\
      \end{pmatrix}
     \cdot
     \begin{pmatrix}
     (p^\fc)^n & & & \\
     & (p^\fc)^{n-1} & & \\
     & & \ddots & \\
     & & & 1
     \end{pmatrix}\in\GL_{n+1}(F^+_{\Sigma_\cF^+}),
     \end{align*}

  \item put $\pres{\fc}\rI_{n+1,\Sigma_\cF^+}\coloneqq\pres{\fc}g_{n+1}\cdot\rI_{n+1,\Sigma_\cF^+}\cdot(\pres{\fc}g_{n+1})^{-1}$ and
     $\pres{\fc}\rI_{n,\Sigma_\cF^+}\coloneqq\pres{\fc}g_n\cdot\rI_{n,\Sigma_\cF^+}\cdot(\pres{\fc}g_n)^{-1}\cap\pres{\fc}\rI_{n+1,\Sigma_\cF^+}$,

  \item define
     \[
     \pres{\fc}\tP\in O_\wp[\pres{\fc}\rI_{n,\Sigma_\cF^+}\times
     \pres{\fc}\rI_{n+1,\Sigma_\cF^+}\backslash\GL_n(F^+_{\Sigma_\cF^+})\times\GL_{n+1}(F^+_{\Sigma_\cF^+})/\rK_{n,\Sigma_\cF^+}\times\rK_{n+1,\Sigma_\cF^+}]
     \]
     to be the convolution of the characteristic function of $(\pres{\fc}\rI_{n,\Sigma_\cF^+}\times\pres{\fc}\rI_{n+1,\Sigma_\cF^+})(\pres{\fc}g_n,\pres{\fc}g_{n+1})(\rI_{n,\Sigma_\cF^+}\times\rI_{n+1,\Sigma_\cF^+})$ and $\tP$.
\end{itemize}
\end{notation}

\begin{notation}
Let $\fn\in\fN_0$ be an element. For $N\in\{n,n+1\}$ and $v\in\Sigma_\cF^+$, we identify $\GL_{N,O_{F^+_v}}$ with $\rU(\Lambda^\fn_{N,v})$ via the frame and the isometry $\tj^\fn_{N,v}$. Put $\pres{\fc}\rK^\fn_N\coloneqq(\rK^\fn_N)^{\Sigma_\cF^+}\times\pres{\fc}\rI_{N,\Sigma_\cF^+}$ for $N\in\{n+1\}$, so that $(\pres{\fc}\rK^\fn_n,\pres{\fc}\rK^\fn_{n+1})\in\fK(\rV^\fn_n)_\sp$.
\end{notation}

We are ready to define the map $\blambda^{\fn,k}$. When $\fn\not\in\fN_k^\eff$, we define $\blambda^{\fn,k}$ to be the zero map. Suppose that $\fn\in\fN_k^\eff$. Take an element $\phi\in\cC^{\fn,k}$. Consider a tuple $\fc=(\fc_v)_v\in(\dZ_{>0})^{\Sigma_\cF^+}$. By \cite{Liu5}*{Lemma~4.11}, $\Gamma(F_\fc/F)$ is naturally a quotient of the target of the determinant map
\[
\det\colon\Sh(\rV^\fn_n,\pres{\fc}\rK^\fn_n)\to \bG_{F/F^+}(F^+)\backslash\bG_{F/F^+}(\dA_{F^+}^\infty)/\det\pres{\fc}\rK^\fn_n,
\]
so that for every $\varsigma\in\Gal(F_\fc/F)$, we have the subset $\Sh(\rV^\fn_n,\pres{\fc}\rK^\fn_n)_\varsigma$ of $\Sh(\rV^\fn_n,\pres{\fc}\rK^\fn_n)$ that is the fiber of $\varsigma$ under the determinant map. Put
\[
\blambda^{\fn,k}(\phi)_\fc\coloneqq\prod_{v\in\Sigma_\cF^+}\omega(\pi^\vee_{0,v},\pi^\vee_{1,v})^{-\fc_v}
\sum_{\varsigma\in\Gal(F_\fc/F)}\phi\(\pres{\fc}\tP_*\CF_{\graph\Sh(\rV^\fn_n,\pres{\fc}\rK^\fn_n)_\varsigma}\)[\varsigma]\in
(O_\wp/\wp^k)[\Gal(F_\fc/F)],
\]
in which $\pres{\fc}\tP_*\CF_{\graph\Sh(\rV^\fn_n,\pres{\fc}\rK^\fn_n)_\varsigma}$ is regarded as an element in $\cM^\fn/(\sfn^{\fn,k}_0,\sfn^{\fn,k}_1)$, the domain of $\phi$. By \cite{Liu5}*{Proposition~5.1}, there exists a (unique) element
\[
\blambda^{\fn,k}(\phi)\in\varprojlim_{\fc}(O_\wp/\wp^k)[\Gal(F_\fc/F)]=\Lambda/\wp^k\Lambda
\]
whose image in $(O_\wp/\wp^k)[\Gal(F_\fc/F)]$ is $\blambda^{\fn,k}(\phi)_\fc$. It is clear that the construction is compatible with respect to $k$, that is, the diagram
\[
\xymatrix{
\cC^{\fn,k} \ar[rr]^-{\blambda^{\fn,k}}\ar[d] && \Lambda/\wp^k\Lambda \ar[d] \\
\cC^{\fn,k-1} \ar[rr]^-{\blambda^{\fn,k-1}} && \Lambda/\wp^{k-1}\Lambda
}
\]
commutes for $k\geq 2$. In particular, when $\epsilon(\Pi_0\times\Pi_1)=1$, $\emptyset\in\fN_k^\defin$ for every $k\geq 1$ hence we obtain a map
\[
\blambda^{\emptyset,\infty}\colon\cC^{\emptyset,\infty}\coloneqq\varprojlim_{k}\cC^{\emptyset,k}\to\Lambda=\varprojlim_k\Lambda/\wp^k\Lambda
\]
as the limit of $\blambda^{\emptyset,k}$.

\begin{proposition}\label{pr:function}
Suppose that $\epsilon(\Pi_0\times\Pi_1)=1$.
\begin{enumerate}
  \item There exists an element $\sL(\Pi_0\times\Pi_1)\in\Lambda_{E_\wp}$, unique up to a scalar in $E_\wp^\times$, satisfying the following property: There exists a constant $C\in\dC^\times$ such that for every finite order character $\chi\colon\Upsilon\to(\ol{E_\wp})^\times$ of conductor $\fU_\fc$ for some $\fc=(\fc_v)_v\in(\dZ_{>0})^{\Sigma_\cF^+}$ and every $E$-linear embedding $\iota\colon\ol{E_\wp}\to\dC$,
      \[
      \iota\sL(\Pi_0\times\Pi_1)(\chi)=C\cdot
      \iota\prod_{v\in\Sigma_\cF^+}
      \(\frac{\|v\|^{\frac{n(n+1)(2n+1)}{6}}}{\omega(\pi_{0,v},\pi_{1,v})\omega(\pi^\vee_{0,v},\pi^\vee_{1,v})}\)^{\fc_v}
      \cdot L(\tfrac{1}{2},(\Pi_0\times\Pi_1)\otimes\iota(\chi\circ\Nm_{F/F^+}^-)).
      \]

  \item Possibly after shrinking $(\rK_{n,v},\rK_{n+1,v})$ for $v\in\Sigma^+_\mnm$, there exists an element $\phi\in\cC^{\emptyset,\infty}$ such that $(\blambda^{\emptyset,\infty}(\phi)^\dag)^2$ and $\sL(\Pi_0\times\Pi_1)$ generate the same ideal of $\Lambda_{E_\wp}$.
\end{enumerate}
\end{proposition}

\begin{proof}
For (1), the uniqueness is clear; and the existence follows from \cite{Liu5}*{Theorem~5.2}.

For (2), we have a map
\[
\phi_{\obj}\colon
O_\wp[\Sh(\rV_{n_0},\rK_{n_0})][\Ker\phi_{\Pi_0^\vee}]\otimes_{O_\wp}O_\wp[\Sh(\rV_{n_1},\rK_{n_1})][\Ker\phi_{\Pi_1^\vee}]
\to\cC^{\emptyset,\infty}
\]
induced by the Petersson inner product with respect the counting measure. It follows from the construction that for $\varphi^\vee$ in the source,
\[
\blambda^{\emptyset,\infty}(\phi_{\varphi^\vee})=\sP_{\tP\varphi^\vee},
\]
where $\sP$ is defined in \cite{Liu5}*{Proposition~5.1}. Since $\tP_{\alpha,v}$ induces an isomorphism from $(\pi_{\alpha,v}^\vee)^{\GL_{n_\alpha}(O_{F^+_v})}$ to the ordinary line in $(\pi_{\alpha,v}^\vee)^{\rI_{n_\alpha,v}}$ for $\alpha=0,1$ and $v\in\Sigma_\cF^+$, (2) follows from \cite{Liu5}*{Proposition~5.8~\&~Corollary~5.3}.
\end{proof}

\begin{definition}\label{no:lambda}
Denote by $\blambda^{\fn,k}_\cF$ the composition of $\blambda^{\fn,k}$ with the natural map $\Lambda/\wp^k\Lambda\to\Lambda_\cF/\wp^k\Lambda_\cF$. Moreover,
\begin{itemize}[label={\ding{118}}]
  \item when $\epsilon(\Pi_0\times\Pi_1)=1$, denote by $\blambda^{\emptyset,\infty}_\cF$ the composition of $\blambda^{\emptyset,\infty}$ with the natural map $\Lambda\to\Lambda_\cF$;

  \item when $\epsilon(\Pi_0\times\Pi_1)=1$, denote by $\sL_\cF(\Pi_0\times\Pi_1)$ the image of $\sL(\Pi_0\times\Pi_1)$ in $\Lambda_{\cF,E_\wp}$.
\end{itemize}
\end{definition}

\begin{corollary}\label{co:function}
Possibly after shrinking $(\rK_{n,v},\rK_{n+1,v})$ for $v\in\Sigma^+_\mnm$, there exists an element $\phi\in\cC^{\emptyset,\infty}$ such that $(\blambda^{\emptyset,\infty}_\cF(\phi)^\dag)^2$ and $\sL_\cF(\Pi_0\times\Pi_1)$ generate the same ideal of $\Lambda_{\cF,E_\wp}$.
\end{corollary}

\begin{proof}
This follows immediately from Proposition \ref{pr:function}(2).
\end{proof}

\begin{remark}\label{re:lambda_special}
In the upcoming article \cite{LS}, the authors are able to evaluate $\sL(\Pi_0\times\Pi_1)$ at all finite order characters. In particular, they show that there is no exceptional zero for $\sL(\Pi_0\times\Pi_1)$ at classical points. As a consequence, the nonvanishing of $\sL(\Pi_0\times\Pi_1)$ at the trivial character is equivalent to the nonvanishing of $L(\frac{1}{2},\Pi_0\times\Pi_1)$.
\end{remark}

\subsection{The kappa maps}
\label{ss:kappa}

This subsection has two goals. First, we construct for every $\fn\in\fN_k^\indef$ with $k\geq 1$ a map
\[
\bkappa^{\fn,k}_\cF\colon\cC^{\fn,k}\to\rH^1_{(\fn)}(F,T^\tc\otimes_{O_\wp}\Lambda_\cF/\wp^k\Lambda_\cF)
\]
(a submodule of $\rH^1(F,T^\tc\otimes_{O_\wp}\Lambda_\cF/\wp^k\Lambda_\cF)$ to be defined below) compatible with respect to $k$. Second, we recall from \cite{Liu5} the definition of the $\Lambda$-submodule $\sK(\cF,\rho_{\Pi_0,\wp}\otimes\rho_{\Pi_1,\wp}(n))$ and its relation with $\kappa^{\emptyset,k}_\cF$ when $\epsilon(\Pi_0\times\Pi_1)=-1$.

We choose a frame (Definition \ref{de:frame}) and keep the relevant notation from the previous subsection. When $\fn\not\in\fN_k^\eff$, we define $\bkappa^{\fn,k}$ to be the zero map. Suppose that $\fn\in\fN_k^\eff$. Take an element $\phi\in\cC^{\fn,k}$. Consider a tuple $\fc=(\fc_v)_v\in(\dZ_{>0})^{\Sigma_\cF^+}$. By \cite{Liu5}*{Lemma~4.11}, $\Gamma(F_\fc/F)$ is naturally a quotient of
\[
\pi_0\(\Sh(\rV^\fn_n,\pres{\fc}\rK^\fn_n)_{\ol{F}}\)=
\bG_{F/F^+}(F^+)\backslash\bG_{F/F^+}(\dA_{F^+}^\infty)/\det\pres{\fc}\rK^\fn_n.
\]
For every $\varsigma\in\Gal(F_\fc/F)$, we denote by $\Sh(\rV^\fn_n,\pres{\fc}\rK^\fn_n)_\varsigma$ the union of geometric components indexed over $\varsigma$, which is defined over $F_\fc$. Then we have
\[
[\graph\Sh(\rV^\fn_n,\pres{\fc}\rK^\fn_n)_\varsigma]
\in\rZ^n((\Sh(\rV^\fn_{n_0},\pres{\fc}\rK^\fn_{n_0})\times_F\Sh(\rV^\fn_{n_1},\pres{\fc}\rK^\fn_{n_1}))_{F_\fc}),
\]
hence
\[
\pres{\fc}\tP_*[\graph\Sh(\rV^\fn_n,\pres{\fc}\rK^\fn_n)_\varsigma]
\in\rZ^n((\Sh(\rV^\fn_{n_0},\rK^\fn_{n_0})\times_F\Sh(\rV^\fn_{n_1},\rK^\fn_{n_1}))_{F_\fc})_{O_\wp}
\]
and
\[
\AJ\(\pres{\fc}\tP_*[\graph\Sh(\rV^\fn_n,\pres{\fc}\rK^\fn_n)_\varsigma]\)
\in\rH^1(F_\fc,\cM^\fn/(\sfn^{\fn,k}_0,\sfn^{\fn,k}_1)).
\]
We put
\[
\bkappa^{\fn,k}(\phi)_\fc\coloneqq\prod_{v\in\Sigma_\cF^+}\omega(\pi^\vee_{0,v},\pi^\vee_{1,v})^{-\fc_v}\cdot
\rH^1(F_\fc,\phi)\(\AJ\(\pres{\fc}\tP_*[\graph\Sh(\rV^\fn_n,\pres{\fc}\rK^\fn_n)_1]\)\)\in\rH^1(F_\fc,T^\tc/\wp^k T^\tc).
\]
By (the same proof of) \cite{Liu5}*{Proposition~7.2}, there exists a (unique) element
\[
\bkappa^{\fn,k}(\phi)\in\varprojlim_{\fc}\rH^1(F_\fc,T^\tc/\wp^k T^\tc)=\rH^1(F,T^\tc\otimes_{O_\wp}\Lambda/\wp^k\Lambda)
\]
whose image in $\rH^1(F_\fc,T^\tc/\wp^k T^\tc)$ is $\bkappa^{\fn,k}(\phi)_\fc$. It is clear that the construction is compatible with respect to $k$, that is, the diagram
\[
\xymatrix{
\cC^{\fn,k} \ar[rr]^-{\bkappa^{\fn,k}}\ar[d] && \rH^1(F,T^\tc\otimes_{O_\wp}\Lambda/\wp^k\Lambda) \ar[d] \\
\cC^{\fn,k-1} \ar[rr]^-{\bkappa^{\fn,k-1}} && \rH^1(F,T^\tc\otimes_{O_\wp}\Lambda/\wp^{k-1}\Lambda)
}
\]
commutes for $k\geq 2$. In particular, when $\epsilon(\Pi_0\times\Pi_1)=-1$, $\emptyset\in\fN_k^\indef$ for every $k\geq 1$ hence we obtain a map
\[
\bkappa^{\emptyset,\infty}\colon\cC^{\emptyset,\infty}\coloneqq\varprojlim_{k}\cC^{\emptyset,k}\to\rH^1(F,T^\tc_\Lambda)
=\varprojlim_k\rH^1(F,T^\tc\otimes_{O_\wp}\Lambda/\wp^k\Lambda)
\]
as the limit of $\bkappa^{\emptyset,k}$.

\begin{notation}\label{no:kappa}
Recall the natural homomorphism $\Lambda\to\Lambda_\cF$.
\begin{enumerate}
  \item Denote by $\bkappa^{\fn,k}_\cF$ (resp.\ $\bkappa^{\emptyset,\infty}_\cF$) the composition of $\bkappa^{\fn,k}$ (resp.\ $\bkappa^{\emptyset,\infty}$) with the natural map
      \[
      \rH^1(F,T^\tc\otimes_{O_\wp}\Lambda/\wp^k\Lambda)\to\rH^1(F,T^\tc\otimes_{O_\wp}\Lambda_\cF/\wp^k\Lambda_\cF)\quad
      \text{resp.\ }
      \rH^1(F,T^\tc\otimes_{O_\wp}\Lambda)\to\rH^1(F,T^\tc\otimes_{O_\wp}\Lambda_\cF).
      \]

  \item Denote by $\pres\dag\bkappa^{\fn,k}_\cF$ the composition of $\bkappa^{\fn,k}_\cF$ with the tautological $\dag$-linear isomorphism
      \[
      \rH^1(F,T^\tc\otimes_{O_\wp}\Lambda_\cF/\wp^k\Lambda_\cF)\to\rH^1(F,T^\tc\otimes_{O_\wp}\Lambda^\dag_\cF/\wp^k\Lambda^\dag_\cF).
      \]
      Similarly, when $\epsilon(\Pi_0\times\Pi_1)=-1$, we have
      \[
      \pres\dag\bkappa^{\emptyset,\infty}_\cF\colon\cC^{\emptyset,\infty}\to
      \rH^1(F,T^\tc\otimes_{O_\wp}\Lambda^\dag_\cF)=\varprojlim_{F\subseteq F'\subseteq\cF}\rH^1(F',T),
      \]
      whose image actually belongs to $\sS(\cF,T)$.
\end{enumerate}
\end{notation}

\begin{definition}\label{de:kappa}
Suppose that $\epsilon(\Pi_0\times\Pi_1)=-1$. For every intermediate field $F\subseteq\cF\subseteq \pres{\wp}\cF$, we define
\[
\sK(\cF,\rho_{\Pi_0,\wp}\otimes\rho_{\Pi_1,\wp}(n))\subseteq\sS(\cF,\rho_{\Pi_0,\wp}\otimes\rho_{\Pi_1,\wp}(n))
\]
to be the $\Lambda_{\cF,E_\wp}$-submodule generated by the image of $\pres\dag\bkappa^{\emptyset,\infty}_\cF(\phi)$ in $\sS(\cF,V)$, for every $\phi\in\cC^{\emptyset,\infty}$ (with respect to all possible pairs $(\rK_{n,v},\rK_{n+1,v})$ for $v\in\Sigma^+_\mnm$ in the initial datum $\cV$).\footnote{It is straightforward to check that this definition coincides with \cite{Liu5}*{Notation~7.6(1)}.}
\end{definition}

\begin{remark}\label{re:kappa_special}
It is straightforward to check, by the multiplicity one property of local Bessel models, that the nonvanishing of $\sK(F,\rho_{\Pi_0,\wp}\otimes\rho_{\Pi_1,\wp}(n))\subseteq\rH^1_f(F,\rho_{\Pi_0,\wp}\otimes\rho_{\Pi_1,\wp}(n))$ is equivalent to the nonvanishing of the element $\AJ^{\Pi_0,\Pi_1}_\wp(\graph\Sh(\rV_n,\rK_n))$ in \cite{LTXZZ}*{Theorem~8.3.2}.
\end{remark}

In the rest of this subsection, we study local properties of the map $\bkappa^{\fn,k}_\cF$.

\begin{definition}\label{de:selmer2}
Let $\Theta$ be either $\Lambda_\cF$, $\Lambda_\cF^\dag$, or a $\Lambda$-ring that is a finite free $O_\wp$-module, and $k\geq 1$ an integer.
\begin{enumerate}
  \item For every place $w$ of $F$, we define $\rH^1_\ff(F_w,T^\tc_{\Theta/\wp^k\Theta})$ to be the propagation of the Selmer structure $\rH^1_\ff(F_w,T^\tc_\Theta)$ from Definitions \ref{de:selmer0} or \ref{de:selmer1} along the quotient map $T^\tc_\Theta\to T^\tc_{\Theta/\wp^k\Theta}$.

  \item For $w\in\fL_k$, we denote by $\rH^1_\ordi(F_w,T^\tc_{\Theta/\wp^k\Theta})$ the image of the natural map
      \[
      \rH^1(F_w,(T^\tc/\wp^k T^\tc)^{\phi_w=\ell^2}\otimes_{O_\wp}\Theta)\to\rH^1(F_w,(T^\tc/\wp^k T^\tc)\otimes_{O_\wp}\Theta)=\rH^1(F_w,T^\tc_{\Theta/\wp^k\Theta}).
      \]

  \item For $\fn\in\fN_k$, define\footnote{We warn the readers that the submodule $\rH^1_{\ff(\fn)}(F,T^\tc_{\Theta/\wp^k\Theta})$ defined below does depend on the lattice $T$ and also $\Theta$, not only on the $O_\wp[\Gamma_F]$-module $T^\tc_{\Theta/\wp^k\Theta}$ in its notation.}
      \begin{align*}
      \rH^1_{(\fn)}(F,T^\tc_{\Theta/\wp^k\Theta})&\coloneqq
      \Ker\(\rH^1(F,T^\tc_{\Theta/\wp^k\Theta})\to
      \prod_{w\not\in\Sigma_\mnm\cup\Sigma_\cF\cup\fn}
      \frac{\rH^1(F_w,T^\tc_{\Theta/\wp^k\Theta})}{\rH^1_\ff(F_w,T^\tc_{\Theta/\wp^k\Theta})}\times
      \prod_{w\in\fn}\frac{\rH^1(F_w,T^\tc_{\Theta/\wp^k\Theta})}{\rH^1_\ordi(F_w,T^\tc_{\Theta/\wp^k\Theta})}\), \\
      \rH^1_{\ff(\fn)}(F,T^\tc_{\Theta/\wp^k\Theta})&\coloneqq
      \Ker\(\rH^1(F,T^\tc_{\Theta/\wp^k\Theta})\to
      \prod_{w\not\in\fn}\frac{\rH^1(F_w,T^\tc_{\Theta/\wp^k\Theta})}{\rH^1_\ff(F_w,T^\tc_{\Theta/\wp^k\Theta})}\times
      \prod_{w\in\fn}\frac{\rH^1(F_w,T^\tc_{\Theta/\wp^k\Theta})}{\rH^1_\ordi(F_w,T^\tc_{\Theta/\wp^k\Theta})}\)
      \subseteq\rH^1_{(\fn)}(F,T^\tc_{\Theta/\wp^k\Theta}).
      \end{align*}
      We sometimes also write $\rH^1_{(\fn)}(F,T^\tc\otimes_{O_\wp}\Theta/\wp^k\Theta)$ and $\rH^1_{\ff(\fn)}(F,T^\tc\otimes_{O_\wp}\Theta/\wp^k\Theta)$ for $\rH^1_{(\fn)}(F,T^\tc_{\Theta/\wp^k\Theta})$ and $\rH^1_{\ff(\fn)}(F,T^\tc_{\Theta/\wp^k\Theta})$, respectively, to be visually more clear.
\end{enumerate}
\end{definition}

\begin{remark}\label{re:selmer2}
For $w\in\fL_k$, it is easy to show that $\rH^1_\ff(F_w,T^\tc_{\Theta/\wp^k\Theta})=\rH^1_\unr(F_w,T^\tc_{\Theta/\wp^k\Theta})$. On the other hand, by Remark \ref{re:location}(3) and Lemma \ref{le:inert}, the natural maps
\[
\rH^1_\ordi(F_w,T^\tc_{\Theta/\wp^k\Theta})\to\rH^1_\sing(F_w,T^\tc_{\Theta/\wp^k\Theta})
\leftarrow\rH^1_\sing(F_w,T^\tc/\wp^k T^\tc)\otimes_{O_\wp}\Theta
\]
are both isomorphisms. In particular, \eqref{eq:rigidify} provides us with two isomorphisms
\[
\rH^1_\ff(F_w,T^\tc_{\Theta/\wp^k\Theta})\simeq\Theta/\wp^k\Theta,\quad
\rH^1_\ordi(F_w,T^\tc_{\Theta/\wp^k\Theta})\simeq\Theta/\wp^k\Theta;
\]
hence for every $w\in\fL_k$, we have the localization map
\[
\loc_w\colon\rH^1_{(\fn)}(F,T^\tc\otimes_{O_\wp}\Theta/\wp^k\Theta)\to
\begin{dcases}
\rH^1_\ordi(F_w,T^\tc_{\Theta/\wp^k\Theta})\simeq\Theta/\wp^k\Theta, & w\in\fn, \\
\rH^1_\ff(F_w,T^\tc_{\Theta/\wp^k\Theta})\simeq\Theta/\wp^k\Theta, & w\not\in\fn.
\end{dcases}
\]
\end{remark}

\begin{proposition}\label{pr:kappa}
Let $w$ denote a place of $F$.
\begin{enumerate}
  \item Suppose that $w\not\in\fn\cup\Sigma_\mnm\cup\Sigma_\cF$. For every $k\geq 1$ and every $\fn\in\fN_k^\indef$, the image of $\loc_w\circ\bkappa^{\fn,k}_\cF$ is contained in $\rH^1_\ff(F_w,T^\tc_{\Lambda_\cF/\wp^k\Lambda_\cF})$.

  \item Suppose that $w\in\fn$. For every $k\geq 1$ and every $\fn\in\fN_k^\indef$, the image of $\loc_w\circ\bkappa^{\fn,k}_\cF$ is contained in $\rH^1_\ordi(F_w,T^\tc_{\Lambda_\cF/\wp^k\Lambda_\cF})$.

  \item Suppose that $w\in\Sigma_\mnm$. There exists an integer constant $k(w)\geq 0$ such that for every $k\geq 1$ and every $\fn\in\fN_k^\indef$, the image of $\loc_w\circ(\wp^{k(w)}\cdot\bkappa^{\fn,k}_\cF)$ is contained in $\rH^1_\ff(F_w,T^\tc_{\Lambda_\cF/\wp^k\Lambda_\cF})$.

  \item Suppose that $w\in\Sigma_\cF$ and that $\cF/F$ is free. For every $k\geq 1$ and every $\fn\in\fN_k^\indef$, the image of $\loc_w\circ\bkappa^{\fn,k}_\cF$ is contained in
      \[
      \Ker\(\rH^1(F_w,T^\tc_{\Lambda_\cF/\wp^k\Lambda_\cF})\to
      \rH^1(\rI_{F_w},(T^\tc/\Fil^{-1}_wT^\tc)_{\Lambda_\cF/\wp^k\Lambda_\cF})\).
      \]
\end{enumerate}
As a consequence of (1) and (2), the image of $\bkappa^{\fn,k}_\cF$ is contained in $\rH^1_{(\fn)}(F,T^\tc\otimes_{O_\wp}\Lambda_\cF/\wp^k\Lambda_\cF)$.
\end{proposition}

\begin{proof}
Take an element $\fn\in\fN_k^\indef$ with $k\geq 1$. We may assume $\fn\in\fN_k^\eff$ since otherwise $\bkappa^{\fn,k}_\cF=0$. We recall some constructions from the proof of Lemma \ref{le:congruence1}. For $\alpha=0,1$, denote by $\sfT^\fn_\alpha$ the $O_\wp$-subring of $\End_{O_\wp}\cM^\fn$ generated by $\dT^\fn_\alpha$. Then $\sfm^\fn_\alpha$ generates a maximal ideal of $\sfT^\fn_\alpha$ for $\alpha=0,1$ by Proposition \ref{pr:congruence}(2). For $\alpha=0,1$, let
\[
\rho_{\sfm^\fn_\alpha}\colon\Gamma_F\to\GL_{n_\alpha}((\sfT^\fn_\alpha)_{\sfm^\fn_\alpha})
\]
be the corresponding lifting the (absolutely irreducible) residue representation of $\rho_{\Pi_\alpha,\wp}^\tc(r_\alpha)$. There is a decomposition of $O_\wp[\Gamma_F]$-modules
\[
\bbT^\tc\coloneqq\cM^\fn_{(\sfm^\fn_0,\sfm^\fn_1)}=\bbT_0^\tc\otimes_{O_\wp}\bbT_1^\tc,
\]
in which
\begin{align*}
\bbT_0^\tc&\coloneqq\rH^{2r_0-1}_\et(\Sh(\rV^\fn_{n_0},\rK^\fn_{n_0})_{\ol{F}},O_\wp(r_0))_{\sfm^\fm_0}\simeq
\(\sfH_0\otimes_{(\sfT^\fm_0)_{\sfm^\fm_0}}(\sfT^\fm_0)_{\sfm^\fm_0}^{\oplus n_0}\),\\
\bbT_1^\tc&\coloneqq\rH^{2r_1}_\et(\Sh(\rV^\fn_{n_1},\rK^\fn_{n_1})_{\ol{F}},O_\wp(r_1))_{\sfm^\fm_1}\simeq
\(\sfH_1\otimes_{(\sfT^\fm_1)_{\sfm^\fm_1}}(\sfT^\fm_1)_{\sfm^\fm_1}^{\oplus n_1}\),
\end{align*}
for some $(\sfT^\fn_\alpha)_{\sfm^\fn_\alpha}$-module $\sfH_\alpha$ (with trivial $\Gamma_F$-action) for $\alpha=0,1$. Moreover, as $O_\wp[\Gamma_F]$-modules, $(\rho_{\sfm^\fn_0}\modulo\sfn^{\fn,k}_0)\otimes(\rho_{\sfm^\fn_1}\modulo\sfn^{\fn,k}_1)\simeq T^\tc/\wp^k T^\tc$. Put $\sfT\coloneqq(\sfT^\fm_0)_{\sfm^\fm_0}\otimes_{O_\wp}(\sfT^\fm_1)_{\sfm^\fm_1}$ and $\sfH\coloneqq\sfH_0\otimes_{O_\wp}\sfH_1$ as a $\sfT$-module. Below, we put $F'_\fc\coloneqq F_\fc\cap\cF$ for every tuple $\fc\in(\dZ_{>0})^{\Sigma_\cF^+}$.

For (1), there are two cases. Suppose that $w$ is not above $p$. Then $V^\tc$ is unramified at $w$ and $\rH^1_\ff(F_w,T^\tc_{\Lambda_\cF/\wp^k\Lambda_\cF})=\rH^1_\unr(F_w,T^\tc_{\Lambda_\cF/\wp^k\Lambda_\cF})$. Now since $\rho_{\sfm^\fn_\alpha}$ is also unramified at $w$ for $\alpha=0,1$, it follows from \cite{Rub00}*{Lemma~1.3.5 \& Lemma~1.3.8} that the image of $\loc_w\circ\bkappa^{\fn,k}$ is contained in $\rH^1_\unr(F_w,T^\tc_{\Lambda_\cF/\wp^k\Lambda_\cF})$. Suppose that $w$ is above $p$. Then $V^\tc$ is crystalline at $w$ and $\rH^1_\ff(F_w,T^\tc_{\Lambda_\cF/\wp^k\Lambda_\cF})=\rH^1_\ns(F_w,T^\tc_{\Lambda_\cF/\wp^k\Lambda_\cF})$ by Lemma \ref{le:bk}(1). Now since $\rho_{\sfm^\fn_\alpha}$ is also crystalline at $w$ for $\alpha=0,1$, the image of $\loc_w\circ\bkappa^{\fn,k}$ is contained in $\rH^1_\ns(F_w,T^\tc_{\Lambda_\cF/\wp^k\Lambda_\cF})$ as $\AJ\(\pres{\fc}\tP_*[\graph\Sh(\rV^\fn_n,\pres{\fc}\rK^\fn_n)_\varsigma]\)
\in\rH^1_\ns(F_\fc,\cM^\fn/(\sfn^{\fn,k}_0,\sfn^{\fn,k}_1))$. Part (1) is proved.

For (2), there exists a unique decomposition $\bbT^\tc=\bbM_0\oplus\bbM_1$ of $\sfT[\Gamma_{F_w}]$-modules satisfying:
\begin{enumerate}[label=(\alph*)]
  \item $\bbM_1$ is unramified at $w$ and $\rH^i(F_w,\bbM_1)=0$ for $i\in\dZ$;

  \item $\bbM_0$ admits a unique $\sfT[\Gamma_{F_w}]$-linear filtration $0\to\bbL'\to\bbM_0\to\bbL\to 0$ in which $\bbL'$ (resp.\ $\bbL$) is isomorphic to $\sfH$ on which $\Gamma_w$ acts via the cyclotomic character (resp.\ trivially), and the monodromy map induces an \emph{injective} map from $\bbL$ to $\bbL'$.
\end{enumerate}
The $O_\wp[\Gamma_{F_w}]$-linear map $T^\tc/\wp^k T^\tc\to(T^\tc/\wp^k T^\tc)_{\Gamma_{F_w}}$ induces a map $\rH^1(F_w,T^\tc/\wp^k T^\tc)\to\rH^1(F_w,(T^\tc/\wp^k T^\tc)_{\Gamma_{F_w}})$, whose kernel coincides with $\rH^1_\ordi(F_w,T^\tc/\wp^k T^\tc)$. Thus, in view of Lemma \ref{le:split}, it suffices to show that the natural map
\[
\rH^1(F_w,\bbT^\tc)\to\rH^1(F_w,(\bbT^\tc)_{\Gamma_{F_w}})
\]
vanishes. It is clear from (a) and (b) that $(\bbT^\tc)_{\Gamma_{F_w}}=\bbL$. Consider the exact sequence
\[
\rH^1(F_w,\bbT^\tc)\to\rH^1(F_w,\bbL)\to\rH^2(F_w,\bbL'\oplus\bbM_1)\to\rH^2(F_w,\bbT^\tc).
\]
It follows from (a) and (b) that $\rH^1(F_w,\bbL)\simeq\sfH$ and $\rH^2(F_w,\bbL'\oplus\bbM_1)=\rH^2(F_w,\bbL')\simeq\sfH$, which are finite free $O_\wp$-modules of the same rank. Now since the monodromy map is an injective map from $\bbL$ to $\bbL'$, $\rH^2(F_w,\bbT^\tc)$ is $O_\wp$-torsion. It follows that the map $\rH^1(F_w,\bbT^\tc)\to\rH^1(F_w,\bbL)$ vanishes, which implies (2).

For (3), it amounts to showing that there exists an integer $k(w)\geq 0$ such that $\wp^{k(w)}$ annihilates the cokernel of the map $\rH^1(F_w,T^\tc_{\Lambda_\cF})\to\rH^1(F_w,T^\tc_{\Lambda_\cF/\wp^k\Lambda_\cF})$, which is nothing but $\rH^2(F_w,T^\tc_{\Lambda_\cF})[\wp^k]$. Since $\rH^2(F_w,T^\tc_{\Lambda_\cF})$ is a finitely generated $\Lambda_\cF$-module, the existence of $k(w)$ follows.

For (4), as $\bbV^\tc$ is also ordinary crystalline at $w$, we have the Hodge--Tate filtration as in Notation \ref{no:ordinary} for $\bbV^\tc$ as well. For every tuple $\fc\in(\dZ_{>0})^{\Sigma_\cF^+}$, the image of the map
\[
\loc_w\circ\AJ\colon\rZ^n((\Sh(\rV^\fn_{n_0},\rK^\fn_{n_0})\times_F\Sh(\rV^\fn_{n_1},\rK^\fn_{n_1}))_{F_\fc})_{O_\wp}
\to\rH^1(F_w,\bbT^\tc_{\Lambda_{F_\fc}})=\bigoplus_{w'\mid w}\rH^1(F_{\fc,w'},\bbT^\tc)
\]
is contained in $\bigoplus_{w'\mid w}\rH^1_\ff(F_{\fc,w'},\bbT^\tc)$, where $\rH^1_\ff(F_{\fc,w'},\bbT^\tc)$ is the propagation of $\rH^1_\ff(F_{\fc,w'},\bbV^\tc)$. By Lemma \ref{le:selmer1}(1) (for $\bbV$), the image of $\loc_{w'}\circ\AJ$ is contained in the kernel of $\rH^1(F_{\fc,w'},\bbT^\tc)\to\rH^1(\rI_{F_{\fc,w'}},\bbV^\tc/\Fil_w^{-1}\bbV^\tc)$, which is same as the kernel of $\rH^1(F_{\fc,w'},\bbT^\tc)\to\rH^1(\rI_{F_{\fc,w'}},\bbT^\tc/\Fil_w^{-1}\bbT^\tc)$. Since $p\geq 2n+1$, the map $\phi\colon\bbT^\tc\to T^\tc/\wp^k T^\tc$ must preserve Hodge--Tate filtrations for every $\phi\in\cC^{\fn,k}$, which implies (4).

The proposition has been proved.
\end{proof}

\subsection{The Iwasawa Rankin--Selberg bipartite Euler system}
\label{ss:euler}

In this subsection, we prove Theorem \ref{th:iwasawa} via the so-called Iwasawa Rankin--Selberg bipartite Euler system.

For every integer $k\geq 0$, denote by $\fX_k^{\r{arrow}}$ the set of arrows of $\fX_k$. For every integer $k\geq 1$, we call the data
\begin{align}\label{eq:iwasawa0}
\left(\{\cV^\fn,\tj^\fn,\cC^{\fn,k}\res\fn\in\fN_k\},\{\tj^a,\varrho^{a,k}\res a\in\fX_k^{\r{arrow}}\},
\{\blambda^{\fn,k}_\cF\res\fn\in\fN_k^\defin\},\{\bkappa^{\fn,k}_\cF\res\fn\in\fN_k^\indef\}\right)
\end{align}
defined from the previous three subsections the \emph{$\cF$-Iwasawa Rankin--Selberg bipartite Euler system} for $(\Pi_0,\Pi_1)$ modulo $\wp^k$ (with respect to the initial data $\cV$). It enjoys the following remarkable relations.

\begin{theorem}\label{th:euler}
Let $a=a(\fn,\fn\fl)$ be an arrow of $\fX_k$ for some integer $k\geq 1$.
\begin{enumerate}
  \item When $\fn\in\fN_k^\defin$, the diagram
      \[
      \xymatrix{
      \cC^{\fn,k} \ar[rr]^-{\varrho^{a,k}} \ar[d]_-{\blambda^{\fn,k}_\cF} && \cC^{\fn\fl,k} \ar[d]^-{\bkappa^{\fn\fl,k}_\cF} \\
      \Lambda_\cF/\wp^k\Lambda_\cF && \rH^1_{(\fn\fl)}(F,T^\tc\otimes_{O_\wp}\Lambda_\cF/\wp^k\Lambda_\cF) \ar[ll]_-{[\tj^a_n]^{-1}\cdot\loc_\fl}
      }
      \]
      commutes, where $[\tj^a_n]$ is the element from Theorem \ref{th:first} (for $\rV^\circ_n=\rV^\fn_n$, $\rV'_n=\rV^{\fn\fl}_n$, $\fp=\fl$) regarded as an (invertible) element in $\Lambda$.

  \item When $\fn\in\fN_k^\indef$, the diagram
      \[
      \xymatrix{
      \cC^{\fn,k} \ar[rr]^-{\varrho^{a,k}} \ar[d]_-{\bkappa^{\fn,k}_\cF} && \cC^{\fn\fl,k} \ar[d]^-{\blambda^{\fn\fl,k}_\cF} \\
      \rH^1_{(\fn)}(F,T^\tc\otimes_{O_\wp}\Lambda_\cF/\wp^k\Lambda_\cF) \ar[rr]^-{[\tj^a_n]\cdot\loc_\fl}  && \Lambda_\cF/\wp^k\Lambda_\cF
      }
      \]
      commutes, where $[\tj^a_n]$ is the element from Theorem \ref{th:second} (for $\rV_n=\rV^\fn_n$, $\rV'_n=\rV^{\fn\fl}_n$, $\fp=\fl$) regarded as an (invertible) element in $\Lambda$.
\end{enumerate}
\end{theorem}

\begin{proof}
Parts (1) and (2) follow from Theorem \ref{th:first} and Theorem \ref{th:second} (for $\tilde{F}=F_\fc\cap\cF$ with all $\fc\in(\dZ_{>0})^{\Sigma_\cF^+}$), respectively, along with the constructions of the relevant maps.
\end{proof}

Now we can focus on the proof of Theorem \ref{th:iwasawa}. We start with two lemmas while recalling Notation \ref{no:point}.

\begin{lem}\label{le:iwasawa2}
For every closed point $x$ of $\Spec\Lambda_{\cF,E_\wp}$ and every integer $k\geq 1$,
\begin{enumerate}
  \item the natural map $\rH^1_\ff(F,T^\tc_{O_x^\dag})/\wp_x^k\rH^1_\ff(F,T^\tc_{O_x^\dag})\to\rH^1_\ff(F,T^\tc_{(O_x^\dag/\wp_x^k O_x)^\dag})$ is injective;

  \item the natural map $\rH^1_\ff(F,W_{O_x/\wp^k O_x})\to\rH^1_\ff(F,W_{O_x})[\wp^k]$ is an isomorphism.
\end{enumerate}
\end{lem}

\begin{proof}
Since the quotient $O_x$-module $\rH^1(F_w,T^\tc_{O_x^\dag})/\rH^1_\ff(F_w,T^\tc_{O_x^\dag})$ is torsion free for every place $w$ of $F$ by Definition \ref{de:selmer1}, (1) follows from \cite{MR04}*{Lemma~3.7.1}. Since $T/\wp T$ is absolutely irreducible of dimension at least two, (2) follows by \cite{MR04}*{Lemma~3.5.4}.
\end{proof}

\begin{lem}\label{le:iwasawa3}
For every nonzero element $s\in\rH^1(F,T/\wp T)$ and every $k\geq 1$, there are infinitely many elements $\fl\in\fL_k$ such that $\loc_\fl(s)\neq 0$.
\end{lem}

\begin{proof}
Since $\rH^0(F,T/\wp T)=0$, we may identify $\rH^1(F,T/\wp T)$ with $\rH^1(F,T/\wp^k T)[\wp]$. Denote by $S$ the $O_\wp$-submodule of $\rH^1(F,T/\wp^k T)$ generated by $s$, which is free over $O_\wp/\wp$ of rank one. Choose an element $\fl\in\fL_k$, which is possible by Remark \ref{re:location}(5). Let $\gamma$ be the element in \cite{LTXZZ}*{Notation~2.6.2} given by a Frobenius element at $\fl$. By \cite{LTXZZ}*{Proposition~2.6.6} (with $m=k$ and $m_0=k-1$) and (A3) in Definition \ref{de:admissible}, $(S,\gamma)$-abundant $1$-tuple exists. Then by \cite{LTXZZ}*{Proposition~2.6.7} and the Chebotarev density theorem, we may find infinitely many elements $\fl\in\fL_k$ such that $\loc_\fl(s)\neq 0$.
\end{proof}

We study specializations of the $\cF$-Iwasawa Rankin--Selberg bipartite Euler system \eqref{eq:iwasawa0} at closed points. For every closed point $x$ of $\Spec\Lambda_{\cF,E_\wp}$, write $\sp_x$ for various specialization maps induced by the map $\Lambda_\cF\to O_x$.

\begin{proposition}\label{pr:specialization}
For every affinoid subdomain $\sU$ of $\Spec\Lambda_{\cF,E_\wp}$ disjoint from $\sW_\cF$ (Definition \ref{de:support}), there exists an integer $k(\sU)\geq 0$ such that for every closed point $x$ of $\sU$, the image of the map
\[
\sp_x\circ(\wp^{k(\sU)}\cdot\pres{\dag}\bkappa^{\fn,k}_\cF)
\]
is contained in $\rH^1_{\ff(\fn)}(F,T^\tc_{(O_x/\wp^k O_x)^\dag})$ (Definition \ref{de:selmer2}(3)) for every $k\geq 1$ and every $\fn\in\fN_k^\indef$. Here, $\pres{\dag}\bkappa^{\fn,k}_\cF$ is the map in Notation \ref{no:kappa}(2).
\end{proposition}

\begin{proof}
By Proposition \ref{pr:kappa}, it suffices to show that for every $w\in\Sigma_\cF$, there exists an integer $k(\sU,w)\geq 0$ such that for every closed point $x$ of $\sU$ and every integer $k\geq 1$,
\[
\wp^{k(\sU,w)}\cdot\Ker\(\rH^1(F_w,T^\tc_{(O_x/\wp^kO_x)^\dag})\to
\rH^1(\rI_{F_w},(T^\tc/\Fil^{-1}_wT^\tc)_{(O_x/\wp^kO_x)^\dag})\)
\]
is contained in $\rH^1_\ff(F_w,T^\tc_{(O_x/\wp^kO_x)^\dag})$.

Consider
\[
\Ker\(\rH^1(F_w,(T^\tc/\Fil^{-1}_wT^\tc)_{(O_x/\wp^kO_x)^\dag})
\to\rH^1(\rI_{F_w},(T^\tc/\Fil^{-1}_wT^\tc)_{(O_x/\wp^kO_x)^\dag})\)
\simeq\rH^1(\kappa_w,((T^\tc/\Fil^{-1}_wT^\tc)_{(O_x/\wp^kO_x)^\dag})^{\rI_{F_w}}),
\]
which coincides with $\rH^1(\kappa_w,((\Gr^0_wT^\tc)_{(O_x/\wp^kO_x)^\dag})^{\rI_{F_w}})$ as $\Upsilon_{\cF,w}$ is torsion free and $p>n$. Since $\sU$ is away from $\sW_\cF$, there exists an integer $k'(\sU,w)\geq 0$ depending only on $\sU$ such that the ideal generated by the image of $\{\gamma-1\res\gamma\in\Gamma_{F_w}\}$ in $\End_{O_x}\((\Gr^0_wT^\tc)_{O_x^\dag}\)$ contains $\wp^{k'(\sU,w)}(\Gr^0_wT^\tc)_{O_x^\dag}$ for every closed point $x$ of $\sU$. It follows that $\wp^{k'(\sU,w)}$ annihilates $\rH^1(\kappa_w,((\Gr^0_wT^\tc)_{(O_x/\wp^kO_x)^\dag})^{\rI_{F_w}})$ for every closed point $x$ of $\sU$ and every $k\geq 1$. In other words,
\[
\wp^{k'(\sU,w)}\cdot\Ker\(\rH^1(F_w,T^\tc_{(O_x/\wp^kO_x)^\dag})\to
\rH^1(\rI_{F_w},(T^\tc/\Fil^{-1}_wT^\tc)_{(O_x/\wp^kO_x)^\dag})\)
\]
is contained in the image of the map
\[
\rH^1(F_w,(\Fil^{-1}_wT^\tc)_{(O_x/\wp^kO_x)^\dag})\to\rH^1(F_w,T^\tc_{(O_x/\wp^kO_x)^\dag}).
\]
Now we consider
\[
\coker\(\rH^1(F_w,(\Fil^{-1}_wT^\tc)_{O_x^\dag})\to\rH^1(F_w,(\Fil^{-1}_wT^\tc)_{(O_x/\wp^kO_x)^\dag})\),
\]
which is contained in $\rH^2(F_w,(\Fil^{-1}_wT^\tc)_{O_x^\dag})$, whose Cartier dual is isomorphic $\rH^0(F_w,(W/\Fil^{-1}_wW)_{O_x})$. Again, we have $\rH^0(F_w,(W/\Fil^{-1}_wW)_{O_x})=\rH^0(F_w,(\Gr^0_wW)_{O_x})$. Since $\sU$ is away from $\sW_\cF$, there exists an integer $k''(\sU,w)$ depending only on $\sU$ such that $\wp^{k''(\sU,w)}$ annihilates $\rH^2(F_w,(\Fil^{-1}_wT^\tc)_{O_x^\dag})$ for every closed point $x$ of $\sU$. Thus, we may take $k(\sU,w)=k'(\sU,w)+k''(\sU,w)$.

The proposition is proved.
\end{proof}

Take an affinoid subdomain $\sU$ as in Proposition \ref{pr:specialization} and a closed point $x$ in $\sU$. Put
\begin{align*}
\lambda_{\sU,x}^{\fn,k}&\coloneqq \sp_x\circ(\wp^{k(\sU)}\cdot\pres{\dag}\blambda^{\fn,k}_\cF)
\colon\cC^{\fn,k}\to O_x/\wp^k O_x,\\
\kappa_{\sU,x}^{\fn,k}&\coloneqq \sp_x\circ(\wp^{k(\sU)}\cdot\pres{\dag}\bkappa^{\fn,k}_\cF)
\colon\cC^{\fn,k}\to \rH^1_{\ff(\fn)}(F,T^\tc_{(O_x/\wp^k O_x)^\dag})
\end{align*}
(including $k=\infty$) whenever applicable. Here, $\pres{\dag}\blambda^{\fn,k}_\cF$ denotes the composition of $\blambda^{\fn,k}_\cF$ (Definition \ref{no:lambda}) and the involution $\dag$ on $\Lambda_\cF$.

For every pair of integers $1\leq k\leq j$, put
\[
\delta_{\sU,x}(k,j)\coloneqq
\min\left\{\left.\min_{\phi\in\cC^{\fn,k}}\ind\(\lambda_{\sU,x}^{\fn,k}(\phi),O_x/\wp^k O_x\)\,\right| \fn\in\fN^\defin_j\right\},
\]
where $\ind\(\lambda_{\sU,x}^{\fn,k}(\phi),O_x/\wp^k O_x\)\coloneqq\sup\left\{i\in\dN\left|\lambda_{\sU,x}^{\fn,k}(\phi)\in\wp_x^iO_x/\wp^k O_x\right.\right\}$. As $\delta_{\sU,x}(k,j)\leq\delta_{\sU,x}(k,j+1)$, we may put
\[
\delta_{\sU,x}(k)\coloneqq\lim_{j\to\infty}\delta_{\sU,x}(k,j)\leq\infty.
\]
We omit $\sU$ in all the subscripts when it is irrelevant or clear from the context.

\begin{proposition}\label{pr:iwasawa}
Let $x$ be a closed point of $\Spec\Lambda_{\cF,E_\wp}$ disjoint from $\sW_\cF$ (Definition \ref{de:support}) and $k\geq 1$ an integer.
\begin{enumerate}
  \item Suppose that $\epsilon(\Pi_0\times\Pi_1)=1$. If $\lambda_x^{\emptyset,k}\neq 0$, then
      \begin{enumerate}
        \item $\rH^1_\ff(F,T^\tc_{O_x^\dag})$ vanishes;

        \item $\rH^1_\ff(F,W_{O_x})$ has $O_x$-corank zero;

        \item we have
            \[
            \length_{O_x}\(\Hom_{O_x}\(\rH^1_\ff(F,W_{O_x}),E_x/O_x\)\)+2\delta_x(k)
            =2\cdot\min_{\phi\in\cC^{\emptyset,\infty}}\length_{O_x}\(O_x/\lambda_x^{\emptyset,\infty}(\phi)\).
            \]
      \end{enumerate}

  \item Suppose that $\epsilon(\Pi_0\times\Pi_1)=-1$. If $\kappa_x^{\emptyset,k}\neq 0$, then
      \begin{enumerate}
        \item $\rH^1_\ff(F,T^\tc_{O_x^\dag})$ is free of rank one over $O_x$;

        \item $\rH^1_\ff(F,W_{O_x})$ has $O_x$-corank one;

        \item we have
            \[
            \length_{O_x}\(\Hom_{O_x}\(\rH^1_\ff(F,W_{O_x}),E_x/O_x\)[\wp^\infty]\)+2\delta_x(k)
            =2\cdot\min_{\phi\in\cC^{\emptyset,\infty}}
            \length_{O_x}\(\rH^1_\ff(F,T^\tc_{O_x^\dag})/O_x\kappa_x^{\emptyset,\infty}(\phi)\).
            \]
      \end{enumerate}
\end{enumerate}
\end{proposition}

\begin{proof}
Let $j$ be an integer at least $2k$. We apply the discussion in \S\ref{ss:bipartite} to the following situation:
\[
R=O_x/\wp^k O_x,\quad T=T_{O_x/\wp^k O_x},\quad \fL=\fL_j,\quad(\sF,\Sigma_\sF)=(\ff,\Sigma_\mnm\cup\Sigma_p),
\]
so $\fN=\fN_j$. By (A3) in Definition \ref{de:admissible} and Lemma \ref{le:iwasawa3}, $T$ satisfies Hypotheses \ref{hy:euler}. By Lemma \ref{le:selfdual} and \cite{How06}*{Lemma~3.3.4(d)}, the Selmer structure $(\sF,\Sigma_\sF)$ is self-dual and cartesian.

Assume $\lambda_x^{\emptyset,k}\neq 0$ and $\kappa_x^{\emptyset,k}\neq 0$ in (1) and (2), respectively. By Theorem \ref{th:euler} and Proposition \ref{pr:specialization}, we know that $\fN_j^\defin=\fN^\even$ and $\fN_j^\indef=\fN^\odd$ by the same argument for \cite{How06}*{Lemma~3.3.5}, hence the data $(\{\cC^{\fn,k}\},\{\varrho^{a,k}\},\{\lambda^{\fn,k}_x\},\{\kappa^{\fn,k}_x\})$ form a generalized bipartite Euler system for $(T,\sF,\fL)$ (Definition \ref{de:euler}). By the same argument for \cite{How06}*{Lemma~3.3.6} (using Lemma \ref{le:euler1}), we have $(\cC^{\fn,k})^\free=\cC^{\fn,k}$ (Notation \ref{no:free}) for every $\fn\in\fN_j^\indef=\fN^\odd$ containing an effective element in $\fL_j$ (Definition \ref{de:congruence}(4)).

We first consider (1). By Remark \ref{re:congruence}(2) and the construction of $\varrho^{a,k}$ in Definition \ref{de:congruence1}, we may take $v=v(\emptyset)$ as an effective vertex (Definition \ref{de:euler}) of the generalized bipartite Euler system $(\{\cC^{\fn,k}\},\{\varrho^{a,k}\},\{\lambda^{\fn,k}_x\},\{\kappa^{\fn,k}_x\})$. As $\emptyset\in\fN^\even$, we may write $\rH^1_\ff(F,T_{O_x/\wp^k O_x})=M\oplus M$ for an $O_x/\wp^kO_x$-module $M$. Applying Proposition \ref{pr:euler}(2) to $\fn=\emptyset$, we obtain the equality
\begin{align*}
\length_{O_x}M+\delta_x(k,j)=\min_{\phi\in\cC^{\emptyset,\infty}}\ind\(\lambda^{\emptyset,k}_x(\phi),O_x/\wp^kO_x\)
\end{align*}
of (finite) integers. This implies that $M$ is annihilated by an ideal of $O_x$ strictly containing $\wp^kO_x$. By Lemma \ref{le:iwasawa2}(2) and the relation $\rH^1_\ff(F,T^\tc_{O_x^\dag})=\varprojlim_j\rH^1_\ff(F,W_{O_x})[\wp^j]$, we obtain (1a) and (1b), together with an isomorphism
\[
\Hom_{O_x}\(\rH^1_\ff(F,W_{O_x}),E_x/O_x\)\simeq M\oplus M
\]
of $O_x$-modules. Finally, we have
\[
\min_{\phi\in\cC^{\emptyset,\infty}}\ind\(\lambda^{\emptyset,k}_x(\phi),O_x/\wp^kO_x\)
=\min_{\phi\in\cC^{\emptyset,\infty}}\ind\(\lambda^{\emptyset,\infty}_x(\phi),O_x\)
=\min_{\phi\in\cC^{\emptyset,\infty}}\length_{O_x}\(O_x/\lambda_x^{\emptyset,\infty}(\phi)\).
\]
Thus, (1c) follows by taking $j\to\infty$.

We then consider (2). Take an element $\phi\in\cC^{\emptyset,\infty}$ such that $\kappa_x^{\emptyset,k}(\phi)\neq 0$ so in particular $\kappa_x^{\emptyset,\infty}(\phi)\neq 0$. Write $\kappa_x^{\emptyset,\infty}(\phi)=c\cdot s_\phi$ for some elements $c\in O_x$ and $s_\phi\in\rH^1_\ff(F,T^\tc_{O_x^\dag})\setminus\wp_x\rH^1_\ff(F,T^\tc_{O_x^\dag})$. By Lemma \ref{le:iwasawa2}(1), the image of $s_\phi$ in $\rH^1_\ff(F,T^\tc_{(O_x/\wp_xO_x)^\dag})=\rH^1_\ff(F,T/\wp T)\otimes_{O_\wp/\wp}O_x/\wp_x$, say $\bar{s}_\phi$, is nonzero. By Lemma \ref{le:iwasawa3}, we may find an element $\fl\in\fL_j$ such that $\loc_\fl(\bar{s}_\phi)\neq 0$. By Theorem \ref{th:euler}(2), we have $\cC^{\{\fl\},k}\neq 0$, that is, $\fl$ is effective (Definition \ref{de:congruence}(4)). By the construction of $\varrho^{a,k}$ in Definition \ref{de:congruence1}, we may take $v=v(\{\fl\})$ as an effective vertex of the generalized bipartite Euler system $(\{\cC^{\fn,k}\},\{\varrho^{a,k}\},\{\lambda^{\fn,k}_x\},\{\kappa^{\fn,k}_x\})$. As $\emptyset\in\fN^\odd$, we may write $\rH^1_\ff(F,T_{O_x/\wp^k O_x})=O_x/\wp^kO_x\oplus M\oplus M$ for an $O_x/\wp^kO_x$-module $M$. Applying Proposition \ref{pr:euler}(2) to $\fn=\emptyset$, we obtain the equality
\begin{align*}
\length_{O_x}M+\delta_x(k,j)=
\min_{\phi\in\cC^{\emptyset,\infty}}\ind\(\kappa^{\emptyset,k}_x(\phi),\rH^1_\ff(F,T^\tc_{(O_x^\dag/\wp_x^k O_x)^\dag}\)
\end{align*}
of (finite) integers. This implies that $M$ is annihilated by an ideal of $O_x$ strictly containing $\wp^kO_x$. By Lemma \ref{le:iwasawa2}(2) and the relation $\rH^1_\ff(F,T^\tc_{O_x^\dag})=\varprojlim_j\rH^1_\ff(F,W_{O_x})[\wp^j]$, we obtain (2a) and (2b), together with an isomorphism
\[
\Hom_{O_x}\(\rH^1_\ff(F,W_{O_x}),E_x/O_x\)[\wp^\infty]\simeq M\oplus M
\]
of $O_x$-modules. Finally, Lemma \ref{le:iwasawa2}(1) implies that
\[
\ind\(\kappa^{\emptyset,k}_x(\phi),\rH^1_\ff(F,T^\tc_{(O_x^\dag/\wp_x^k O_x)^\dag}\)=
\length_{O_x}\(\rH^1_\ff(F,T^\tc_{O_x^\dag})/O_x\kappa_x^{\emptyset,\infty}(\phi)\).
\]
Thus, (2c) follows by taking $j\to\infty$.

The proposition is proved.
\end{proof}

We need one more lemma from geometry. Let $\sM$ be a finitely generated torsion $\Lambda_\cF$-module. Suppose that $\cF/F$ is infinite and consider a height one prime $\fq$ of $\Lambda_\cF$ such that $\fq\Lambda_{\cF,E_\wp}$ remains a height one prime of $\Lambda_{\cF,E_\wp}$. We say that a closed point $x$ of $\Lambda_{\cF,E_\wp}$ is a \emph{$\fq$-clean point for $\sM$} if $x$ locates in the vanishing locus of $\fq$ such that there exists an injective map
\[
\phi\colon\bigoplus_i\Lambda_\cF/\fq^{n_i}\to\sM
\]
of $\Lambda_\cF$-modules with $(\coker\phi)_x$ vanishing. By \cite{Bour}*{Chap.~VII~\S4.4,~Theorem~5}, $\fq$-clean points for $\sM$ are Zariski dense in $\Spec(\Lambda_\cF/\fq)$.

\begin{lem}\label{le:iwasawa4}
In the situation above, for every $\fq$-clean point $x$ for $\sM$ and every proper Zariski closed subset $\sZ$ of $\Spec\Lambda_{\cF,E_\wp}$, there exists a sequence $(x_m)_{m\geq 1}$ of closed points of $\Spec\Lambda_{\cF,E_\wp}\setminus\sZ$ such that
\begin{enumerate}
  \item $\Theta_{x_m}$ and $\Theta_x$ are isomorphic as $O_\wp$-rings for every $m\geq 1$;

  \item $\length_{O_\wp}\(\Theta_{x_m}\otimes_{\Lambda_\cF}\Theta_x\)$ tends to infinity when $m\to\infty$.\footnote{This means that the sequence $(x_m)_{m\geq 1}$ approaches to $x$ under the natural $p$-adic topology on the set of closed points of $\Spec\Lambda_{\cF,E_\wp}$.}
\end{enumerate}
Moreover, for every sequence $(x_m)_{m\geq 1}$ of closed points of $\Spec\Lambda_{\cF,E_\wp}$ satisfying (1) and (2), we have
\begin{align}\label{eq:iwasawa5}
\ord_\fq(\Char_\cF\sM)=\lim_{m\to\infty}l_m^{-1}\cdot\length_{O_\wp}\(\sM/\fP_{x_m}\sM\),
\end{align}
where $l_m\coloneqq\length_{O_\wp}\(\Theta_{x_m}\otimes_{\Lambda_\cF}\Lambda_\cF/\fq\)$ is a sequence of positive integers that tends to infinity by (2).
\end{lem}

\begin{proof}
Choose an injective map
\[
\phi\colon\bigoplus_i\Lambda_\cF/\fq^{n_i}\to\sM
\]
of $\Lambda_\cF$-modules with $(\coker\phi)_x$ vanishing, as in the definition.

As $\Upsilon_\cF$ is a finite free $\dZ_p$-module, $\Lambda_\cF$ is isomorphic to $O_\wp[[T_1,\ldots,T_d]]$ for a positive integer $d$ (as $\cF/F$ is infinite). Write $(t_1,\ldots,t_d)$ for the coordinates of $x$ in $\ol{E_\wp}$. Then $\Theta_x=O_\wp[t_1,\ldots,t_d]$. For every $m\geq 1$, there exists $(c_1,\ldots,c_d)\in(O_\wp^\times)^d$ such that $x_m\coloneqq(t_1+c_1\wp^m,\ldots,t_d+c_d\wp^m)$ does not locate in the union of $\sZ$ and the support of $\coker\phi$. It is clear that the sequence $(x_m)_{m\geq 1}$ satisfies (1) and (2).

It remains to check \eqref{eq:iwasawa5}. As $\ord_\fq(\Char_\cF\sM)=\sum_i n_i$, it suffices to show that the lengths of $O_\wp$-modules
\[
\Tor_1^{\Lambda_\cF}(\coker\phi,\Theta_{x_m})
\]
are uniformly bounded when $m\to\infty$. Choose a short exact sequence
\[
0\to \sM' \to \sM'' \to \coker\phi\to 0
\]
of $\Lambda_\cF$-modules in which $\sM''$ is a finitely generated flat $\Lambda_\cF$-module. Then $\Tor_1^{\Lambda_\cF}(\coker\phi,\Theta_{x_m})=(\sM'\otimes_{\Lambda_\cF}\Theta_{x_m})[\wp^\infty]$ for every $m\geq 1$. Write $\sM'\otimes_{\Lambda_\cF}\Theta_{x_m}=\Theta_{x_m}^{r_m}\oplus M_m$ for some integer $r_m\geq 0$ and some finitely generated torsion $\Theta_{x_m}$-module $M_m$. Write $\sM'\otimes_{\Lambda_\cF}\Theta_x=\Theta_x^r\oplus M$ in the similar way. Since $\sM'$ is also flat at $x$, we have $r_m=r$ for $m$ sufficiently large. Let $m_0$ be a positive integer such that $\wp^{m_0}$ annihilates $M$. Since $\sM'\otimes_{\Lambda_\cF}(\Theta_{x_m}/\wp^m\Theta_{x_m})=\sM'\otimes_{\Lambda_\cF}(\Theta_x/\wp^m\Theta_x)$, we have for $m$ sufficiently large that $r_m=r$ and $M_m$ is annihilated by $\wp^{m_0}$. Since $\sM'\otimes_{\Lambda_\cF}\Theta_{x_m}$ hence $M_m$ can be generated as an $O_\wp$-module by a finite set whose cardinality is independent of $m\geq 1$, $\length_{O_\wp}\((\sM'\otimes_{\Lambda_\cF}\Theta_{x_m})[\wp^\infty]\)$ are bounded independently of $m$.

The lemma is proved.
\end{proof}

\begin{proof}[Proof of Theorem \ref{th:iwasawa}]
Denote by $\sZ_\cF$ the union of $\sW_\cF$ (Definition \ref{de:support}) and the common vanishing locus of the image of the map $\pres{\dag}\blambda^{\emptyset,\infty}_\cF$ or $\pres{\dag}\bkappa^{\emptyset,\infty}_\cF$, depending on whether $\epsilon(\Pi_0\times\Pi_1)=1$ or $\epsilon(\Pi_0\times\Pi_1)=-1$, which a Zariski closed subset of $\Spec\Lambda_{\cF,E_\wp}$. By Corollary \ref{co:function} and Definition \ref{de:kappa}, the premises of the statements imply that $\sZ_\cF$ is a proper Zariski closed subset of $\Spec\Lambda_{\cF,E_\wp}$, possibly after shrinking $(\rK_{n,v},\rK_{n+1,v})$ for $v\in\Sigma^+_\mnm$ in the initial datum $\cV$. We also recall from Lemma \ref{le:selmer0} that $\sS(\cF,T)=\rH^1_\ff(F,T^\tc_{\Lambda_\cF^\dag})$ and
$\sX(\cF,W)=\rH^1_\ff(F,W^{\Lambda_\cF^\dag})$.

We first consider (1). For every closed point $x$ of $\Spec\Lambda_{\cF,E_\wp}\setminus\sZ_\cF$, there exists an integer $k=k(x)\geq 1$ such that $\lambda^{\emptyset,k}_x\neq 0$. By Proposition \ref{pr:iwasawa}(1a,1b), $\rH^1_\ff(F,T^\tc_{O_x^\dag})=0$ and $\rH^1_\ff(F,W_{O_x})$ has $O_x$-corank zero. Thus, (1a) follows by Proposition \ref{pr:control}; and (1b) follows by Lemma \ref{le:iwasawa1} and Proposition \ref{pr:specialize}.

For (1c), we may assume that $\cF/F$ is infinite since otherwise the statement follows from (1a). By Corollary \ref{co:function}, it suffices to show that for every height one prime $\fq$ of $\Lambda_\cF$ such that $\fq\Lambda_{\cF,E_\wp}$ is a height one prime of $\Lambda_{\cF,E_\wp}$,
\begin{align}\label{eq:iwasawa1}
\ord_\fq\(\Char_\cF\(\sX(\cF,T)\)\)\leq2\ord_\fq\(\Char_\cF\(\Lambda_\cF/\:\pres{\dag}\blambda^{\emptyset,\infty}_\cF(\phi)\)\)
\end{align}
holds for every $\phi\in\cC^{\emptyset,\infty}$. Take an element $\phi\in\cC^{\emptyset,\infty}$ such that $\pres{\dag}\blambda^{\emptyset,\infty}_\cF(\phi)\neq 0$ (otherwise, \eqref{eq:iwasawa1} is trivial) and choose a $\fq$-clean point $x$ for both $\sX(\cF,T)$ and $\Lambda_\cF/\:\pres{\dag}\blambda^{\emptyset,\infty}_\cF(\phi)$ away from $\sW_\cF$. Applying Lemma \ref{le:iwasawa4} to $x$ and $\sZ_\cF$, we obtain a sequence $(x_m)_{m\geq 1}$ (contained in an affinoid subdomain around $x$ disjoint from $\sW_\cF$) satisfying
\begin{align*}
\ord_\fq\(\Char_\cF\(\sX(\cF,T)\)\) &= \lim_{m\to\infty}l_m^{-1}\cdot\length_{O_\wp}\(\sX(\cF,T)/\fP_{x_m}\sX(\cF,T)\),\\
\ord_\fq\(\Char_\cF\(\Lambda_\cF/\:\pres{\dag}\blambda^{\emptyset,\infty}_\cF(\phi)\)\) &= \lim_{m\to\infty}l_m^{-1}\cdot
\length_{O_\wp}\(\Lambda_\cF/(\pres{\dag}\blambda^{\emptyset,\infty}_\cF(\phi),\fP_{x_m})\).
\end{align*}
Since $\Theta_{x_m}\simeq\Theta_x$ for every $m\geq 1$ and $l_m\to\infty$, we have by (the Pontryagin dual of) Proposition \ref{pr:control} that
\[
\lim_{m\to\infty}l_m^{-1}\cdot\length_{O_\wp}\(\sX(\cF,T)/\fP_{x_m}\sX(\cF,T)\)
=\lim_{m\to\infty}l_m^{-1}\cdot\length_{O_\wp}\(\Hom_{O_x}\(\rH^1_\ff(F,W_{O_{x_m}}),E_{x_m}/O_{x_m}\)\),
\]
and by trivial reasons that
\[
\lim_{m\to\infty}l_m^{-1}\cdot
\length_{O_\wp}\(\Lambda_\cF/(\pres{\dag}\blambda^{\emptyset,\infty}_\cF(\phi),\fP_{x_m})\)=
\lim_{m\to\infty}l_m^{-1}\cdot\length_{O_\wp}\(O_{x_m}/\lambda_{x_m}^{\emptyset,\infty}(\phi)\).
\]
Now \eqref{eq:iwasawa1} follows from Proposition \ref{pr:iwasawa}(1c) by taking, for every $m\geq 1$, an integer $k_m\geq 1$ satisfying $\lambda_{x_m}^{\emptyset,k_m}\neq 0$.

We then consider (2). By Proposition \ref{pr:specialize} and Lemma \ref{le:iwasawa2}(1), there exists a dense Zariski open subset $\sU$ of $\Spec\Lambda_{\cF,E_\wp}\setminus\sZ_\cF$ such that for every closed point $x$ of $\sU$, there exists an integer $k=k(x)\geq 1$ such that $\kappa^{\emptyset,k}_x\neq 0$. By Proposition \ref{pr:iwasawa}(2a,2b), $\rH^1_\ff(F,T^\tc_{O_x^\dag})$ is free of rank one over $O_x$ and $\rH^1_\ff(F,W_{O_x})$ has $O_x$-corank one. Thus, (2a) follows by Proposition \ref{pr:control}; and (2b) follows by Lemma \ref{le:iwasawa1} and Proposition \ref{pr:specialize}.

For (2c), we may assume that $\cF/F$ is infinite since otherwise the statement is trivial. By Definition \ref{de:kappa}, it suffices to show that for every height one prime $\fq$ of $\Lambda_\cF$ such that $\fq\Lambda_{\cF,E_\wp}$ remains a height one prime of $\Lambda_{\cF,E_\wp}$,
\begin{align}\label{eq:iwasawa2}
\ord_\fq\(\Char_\cF\(\sX_0(\cF,T)\)\)\leq2
\ord_\fq\(\Char_\cF\(\rH^1_\ff(F,T^\tc_{\Lambda_\cF^\dag})/\Lambda_\cF\pres{\dag}\bkappa^{\emptyset,\infty}_\cF(\phi)\)\)
\end{align}
holds for every $\phi\in\cC^{\emptyset,\infty}$, where $\sX_0(\cF,T)$ denotes the maximal $\Lambda_\cF$-torsion submodule of $\sX(\cF,T)$. Take an element $\phi\in\cC^{\emptyset,\infty}$ such that $\pres{\dag}\bkappa^{\emptyset,\infty}_\cF(\phi)\neq 0$ (otherwise, \eqref{eq:iwasawa2} is trivial). By Proposition \ref{pr:specialize}, we may choose a $\fq$-clean point $x$ for both $\sX_0(\cF,T)$ and $\rH^1_\ff(F,T^\tc_{\Lambda_\cF^\dag})/\Lambda_\cF\pres{\dag}\bkappa^{\emptyset,\infty}_\cF(\phi)$ away from $\sW_\cF$, such that $\sX(\cF,T)/\sX_0(\cF,T)$ is locally free at $x$ and that the natural map
\begin{align}\label{eq:iwasawa4}
\rH^1_\ff(F,V^\tc_{\Lambda_\cF^\dag})/\fP_x\rH^1_\ff(F,V^\tc_{\Lambda_\cF^\dag})\to\rH^1_\ff(F,V^\tc_{O_x^\dag})
\end{align}
is injective. Applying Lemma \ref{le:iwasawa4} to $x$ and $\sZ_\cF$, we obtain a sequence $(x_m)_{m\geq 1}$ (contained in an affinoid subdomain around $x$ disjoint from $\sW_\cF$) satisfying
\begin{align*}
\ord_\fq\(\Char_\cF\(\sX_0(\cF,T)\)\) &= \lim_{m\to\infty}l_m^{-1}\cdot\length_{O_\wp}\(\sX_0(\cF,T)/\fP_{x_m}\sX_0(\cF,T)\),\\
\ord_\fq\(\Char_\cF\(\rH^1_\ff(F,T^\tc_{\Lambda_\cF^\dag})/\Lambda_\cF\pres{\dag}\bkappa^{\emptyset,\infty}_\cF(\phi)\)\) &= \lim_{m\to\infty}l_m^{-1}\cdot\length_{O_\wp}
\(\rH^1_\ff(F,T^\tc_{\Lambda_\cF^\dag})/\Lambda_\cF\pres{\dag}\bkappa^{\emptyset,\infty}_\cF(\phi)+\fP_{x_m}\rH^1_\ff(F,T^\tc_{\Lambda_\cF^\dag})\).
\end{align*}
Since $\Theta_{x_m}\simeq\Theta_x$ for every $m\geq 1$ and $l_m\to\infty$, we have by (the Pontryagin dual of) Proposition \ref{pr:control} that
\[
\lim_{m\to\infty}l_m^{-1}\cdot\length_{O_\wp}\(\sX_0(\cF,T)/\fP_{x_m}\sX_0(\cF,T)\)
=\lim_{m\to\infty}l_m^{-1}\cdot\length_{O_\wp}\(\Hom_{O_{x_m}}\(\rH^1_\ff(F,W_{O_{x_m}}),E_{x_m}/O_{x_m}\)[\wp^\infty]\).
\]
We claim that
\begin{align}\label{eq:iwasawa3}
\lim_{m\to\infty}l_m^{-1}\cdot\length_{O_\wp}
\(\rH^1_\ff(F,T^\tc_{\Lambda_\cF^\dag})/\Lambda_\cF\pres{\dag}\bkappa^{\emptyset,\infty}_\cF(\phi)
+\fP_{x_m}\rH^1_\ff(F,T^\tc_{\Lambda_\cF^\dag})\)
\geq\lim_{m\to\infty}l_m^{-1}\cdot\length_{O_\wp}
\(\rH^1_\ff(F,T^\tc_{O_{x_m}^\dag})/O_{x_m}\kappa^{\emptyset,\infty}_{x_m}(\phi)\).
\end{align}
Assuming this, \eqref{eq:iwasawa2} follows from Proposition \ref{pr:iwasawa}(2c) by taking, for every $m\geq 1$, an integer $k_m\geq 1$ satisfying $\kappa_{x_m}^{\emptyset,k_m}\neq 0$.

To check \eqref{eq:iwasawa3}, it suffices to show that
\begin{align*}
\limsup_{m\to\infty}\length_{O_\wp}\(\coker\(\rH^1_\ff(F,T^\tc_{\Lambda_\cF^\dag})/\fP_{x_m}\rH^1_\ff(F,T^\tc_{\Lambda_\cF^\dag})
\to\rH^1_\ff(F,T^\tc_{O_{x_m}^\dag})\)\)
\end{align*}
is finite. Since the map \eqref{eq:iwasawa4} is injective and $\rH^1_\ff(F,V^\tc_{\Lambda_\cF^\dag})$ has rank one, there exists a positive integer $k_0$ such that the map
\[
\rH^1_\ff(F,T^\tc_{\Lambda_\cF^\dag})/(\fP_x,\wp^{k_0})\rH^1_\ff(F,T^\tc_{\Lambda_\cF^\dag})
\to\rH^1_\ff(F,T^\tc_{(O_x/\wp^{k_0}O_x)^\dag})
\]
is nonzero. By Lemma \ref{le:iwasawa4}(1,2), there exists a positive integer $m_0$ such that for every integer $m\geq m_0$, $(\fP_{x_m},\wp^{k_0})=(\fP_x,\wp^{k_0})$ and $O_{x_m}/\wp^{k_0}O_{x_m}=O_x/\wp^{k_0}O_x$. In particular, for $m\geq m_0$, the map
\[
\rH^1_\ff(F,T^\tc_{\Lambda_\cF^\dag})/(\fP_{x_m},\wp^{k_0})\rH^1_\ff(F,T^\tc_{\Lambda_\cF^\dag})
\to\rH^1_\ff(F,T^\tc_{(O_{x_m}/\wp^{k_0}O_{x_m})^\dag})
\]
is nonzero. Since this map factors through $\rH^1_\ff(F,T^\tc_{O_{x_m}^\dag})/\wp^{k_0}\rH^1_\ff(F,T^\tc_{O_{x_m}^\dag})$ and $\rH^1_\ff(F,T^\tc_{O_{x_m}^\dag})$ is a free $O_{x_m}$-module of rank one, we have
\begin{align*}
&\quad\length_{O_\wp}\(\coker\(\rH^1_\ff(F,T^\tc_{\Lambda_\cF^\dag})/\fP_{x_m}\rH^1_\ff(F,T^\tc_{\Lambda_\cF^\dag})
\to\rH^1_\ff(F,T^\tc_{O_{x_m}^\dag})\)\) \\
&<\length_{O_\wp}(O_{x_m}/\wp^{k_0}O_{x_m})+\length_{O_\wp}(O_{x_m}/\Theta_{x_m}) \\
&=\length_{O_\wp}(O_x/\wp^{k_0}O_x)+\length_{O_\wp}(O_x/\Theta_x)
\end{align*}
for every $m\geq m_0$. Thus, \eqref{eq:iwasawa3} hence (2c) are confirmed.

The theorem is all proved.
\end{proof}

\begin{proof}[Proof of Corollary \ref{co:main}]
Let $x$ be the origin of $\Spec\Lambda_{\cF,E_\wp}$, that is, the point defined by the augmentation ideal $\fA_F$ of $\Lambda_{\cF,E_\wp}$. As $\cF/F$ is free, $\cF/F$ is a $\dZ_p^d$-extension for some integer $d\geq 0$. We may assume that $\sL_\cF(\Pi_0\times\Pi_1)$ is nonzero since otherwise the corollary is trivial. Let $m$ be the order of $\sL_\cF(\Pi_0\times\Pi_1)$ at $x$, that is, the maximal integer such that $\sL_\cF(\Pi_0\times\Pi_1)\in\fA_F^m$. Note that by Remark \ref{re:selfdual}, $\rH^1_\ff(F,W^{\Lambda_\cF^\dag})=\rH^1_f(F,W^{\Lambda_\cF^\dag})$.

When $d=0$, the corollary follows from Theorem \ref{th:main}(1b). When $d=1$, by Theorem \ref{th:main}(1c), we know that $\rank_{O_\wp}\rH^1_f(F,W^{\Lambda_\cF^\dag})[\fP_x]\leq m$; hence the corollary follows from Proposition \ref{pr:control}. We now deduce the general case for $d>1$ to the case $d=1$. Let $\sL_m$ be the leading term of $\sL_\cF(\Pi_0\times\Pi_1)$ of degree $m$. Denote by $\fA_{\cF,\dQ_p}$ the augmentation ideal of $\Lambda_{\cF,\dQ_p}$ and by $\cT_\cF$ the projective space over $\dQ_p$ of dimension $d-1$ defined by the graded ring $\bigoplus_{i\geq 0}\fA_{\cF,\dQ_p}^i/\fA_{\cF,\dQ_p}^{i-1}$, so that $\sL_m$ defines a hypersurface of $\cT_\cF\otimes_{\dQ_p}E_\wp$ of dimension $m$. For every nonzero element $t\in\cT_\cF(\dQ_p)$, we have a unique intermediate field $F\subseteq\cF_t\subseteq\cF$ such that $\cF_t/F$ is a $\dZ_p$-extension and the tangent line of the subscheme $\Spec\Lambda_{\cF_t,\dQ_p}$ of $\Spec\Lambda_{\cF,\dQ_p}$ is given by $t$. Choose an element $t\in\cT_\cF(\dQ_p)$ that is not contained in the above hypersurface and such that $\Sigma_{\cF_t}=\Sigma_\cF$. Then $\sL_{\cF_t}(\Pi_0\times\Pi_1)$ is nonzero and its order at the origin $x$ is again $m$. In other words, we have reduced the case to $d=1$. The corollary has been proved.
\end{proof}

\subsection{Appendix: Integral $p$-adic representations}
\label{ss:integral}

In this subsection, we prove a lemma that replies on the Fontaine--Laffaille theory.

Let $p$ be an odd prime and $F/\dQ_p$ an unramified extension with the absolute Galois group $\Gamma_F$. For a topological $\dZ_p$-module $T$, we denote by $T^*$ the continuous Pontryagin dual of $T$, namely the topological $\dZ_p$-module of continuous $\dZ_p$-linear maps from $T$ to $\dQ_p/\dZ_p$, equipped with the strong topology.

We will use the Fontaine--Laffaille theory \cite{FL82} as summarized in \cite{BK90}*{\S4}. For every pair of integers $a\leq 0\leq b$ satisfying $b-a\leq p-2$, denote by $\r{MF}_F^{[a,b]}$ the (abelian) category of filter Dieudonn\'{e} module $D$ over $O_K$ satisfying $D^a=D$ and $D^{b+1}=0$, and $\dZ_p[\Gamma_F]$ the category of finitely generated $\dZ_p$-modules equipped with a continuous action by $\Gamma_F$. We have a functor $T\colon\r{MF}_F^{[a,b]}\to\dZ_p[\Gamma_F]$, which is exact and fully faithful. Denote by $\dZ_p[\Gamma_F]^{[a,b]}_\ns$ the essential image of $T$ and fix a functor $D\colon\dZ_p[\Gamma_F]^{[a,b]}_\ns\to\r{MF}_F^{[a,b]}$ inverse to $T$. For an object $T\in\dZ_p[\Gamma_F]^{[a,b]}_\ns$, we denote by $\rH^1_\ns(F,T)$ the $\dZ_p$-submodule of $\rH^1(F,T)$ classifying extensions in $\dZ_p[\Gamma_F]^{[a,b]}_\ns$. We have the formula
\begin{align}\label{eq:bk}
\rH^1_\ns(F,T)\simeq\coker\(1-f_0\colon D(T)^0\to D(T)\)
\end{align}
from \cite{BK90}*{Lemma~4.4} (so in particular, $\rH^1_\ns(F,T)=\varprojlim_m\rH^1_\ns(F,T\otimes\dZ/p^m)$). Put
\[
\rH^1_\ns(F,T^*(1))=\varinjlim_m\rH^1_\ns(F,T^*[p^m](1)),
\]
which is a $\dZ_p$-submodule of $\rH^1(F,T^*(1))$. Note that for every $\dZ_p$-ring $R$, the functors $T$ and $D$ preserve $R$-linear objects; and $\rH^1_\ns(F,T)\subseteq\rH^1(F,T)$ is a natural inclusion of $R$-modules if $T$ is $R$-linear.

If $T$ is a $p$-torsion free object in $\dZ_p[\Gamma_F]^{[a,b]}_\ns$, we have
\begin{align}\label{eq:bk1}
\rH^1_\ns(F,T)=\rH^1_f(F,T)\coloneqq\Ker\(\rH^1(F,T)\to\frac{\rH^1(F,T\otimes_{\dZ_p}\dQ_p)}{\rH^1_f(F,T\otimes_{\dZ_p}\dQ_p)}\)
\end{align}
and
\begin{align}\label{eq:bk2}
\rH^1_\ns(F,T^*(1))=\rH^1_f(F,T^*(1))\coloneqq\r{im}\(\rH^1_f(F,(T\otimes_{\dZ_p}\dQ_p)^\vee(1))\to\rH^1(F,T^*(1))\)
\end{align}
by \cite{Bre99}*{Proposition~6}.

Let $O$ be the ring of integers of a finite extension of $\dQ_p$. Consider a filtered poset $I$, a topological $O$-ring $\Lambda=\varprojlim_{i\in I}\Lambda_i$ in which $\Lambda_i$ is a local $O$-ring that is a finite free $O$-module and every transition homomorphism is surjective, and a finite free $\Lambda$-module $T$ equipped with a continuous action of $\Gamma_F$. For every index $i\in I$, put $T_i\coloneqq T\otimes_\Lambda\Lambda_i$. Then we have $\rH^1(F,T)=\varprojlim_{i\in I}\rH^1(F,T_i)$; and we define $\rH^1_f(F,T)\coloneqq\varprojlim_{i\in I}\rH^1_f(F,T_i)$ as a $\Lambda$-submodule of $\rH^1(F,T)$.

\begin{lem}\label{le:bk}
Suppose that for every $i\in I$, $T_i$ belongs to $\dZ_p[\Gamma_F]^{[a,b]}_\ns$ for integers $a\leq 0\leq b$ independent of $i$, satisfying $b-a\leq p-2$. Then
\begin{enumerate}
  \item For every surjective homomorphism $\Lambda$ to an $O$-ring $R$ that is a finite free $O$-module, the natural map $\rH^1(F,T)\to\rH^1(F,T\otimes_\Lambda R)$ sends $\rH^1_f(F,T)$ surjectively to $\rH^1_f(F,T\otimes_\Lambda R)$ if $T\otimes_\Lambda R$ also belongs to $\dZ_p[\Gamma_F]^{[a,b]}_\ns$.

  \item The $\Lambda$-module $\rH^1(F,T)/\rH^1_f(F,T)$ can be embedded into a finite free $\Lambda$-module.
\end{enumerate}
\end{lem}

\begin{proof}
We first consider (1). For every integer $m\geq 1$, put $R_m\coloneqq R\otimes\dZ/p^m$. We first show that the image of $\rH^1_f(F,T)$ is contained in $\rH^1_f(F,T\otimes_{\Lambda_G}R)$. By \eqref{eq:bk1}, it suffices to show that the image of $\rH^1_\ns(F,T)$ under the composite map
\[
\rH^1(F,T)\to\rH^1(F,T\otimes_\Lambda R)\to\rH^1(F,T\otimes_\Lambda R_m)
\]
is contained in $\rH^1_\ns(F,T\otimes_\Lambda R_m)$ for every $m\geq 1$. Since $R_m$ is a finite ring, the homomorphism $\Lambda\to R_m$ factors through $\Lambda_{i_m}$ for some index $i_m$. It follows that the image of $\rH^1_\ns(F,T)$ is contained in $\rH^1_\ns(F,T\otimes_\Lambda R_m)$. Next, we show that the induced map $\rH^1_f(F,T)\to\rH^1_f(F,T\otimes_\Lambda R)$ is surjective. By \eqref{eq:bk1}, the natural map $\rH^1_f(F,T\otimes_\Lambda R)\otimes\dF_p\to\rH^1(F,T\otimes_\Lambda R)\otimes\dF_p$ is injective, which implies that the natural map $\rH^1_f(F,T\otimes_{\Lambda_G}R)\otimes\dF_p\to\rH^1_\ns(F,T\otimes_{\Lambda_G}R_1)$ is injective. By Nakayama's lemma, it suffices to show that the map $\rH^1_f(F,T)\to\rH^1_\ns(F,T\otimes_\Lambda R_1)$ is surjective. By \eqref{eq:bk}, $\rH^1_\ns(F,-)$ preserves surjectivity. It follows that the maps $\rH^1_\ns(F,T_{i_1})\to\rH^1_\ns(F,T\otimes_\Lambda R_1)$ and $\rH^1_f(F,T)=\varprojlim_{i\in I}\rH^1_\ns(F,T_i)\to\rH^1_\ns(F,T_{i_1})$ are both surjective. Part (1) follows.

Now we consider (2). Put $\rH^1_?(F,T^*(1))\coloneqq\varinjlim_{i\in I}\rH^1_?(F,T_i^*(1))$ for $?\in\{\;,\ns\}$. By \eqref{eq:bk1} and \eqref{eq:bk2}, for every $i\in I$, $\rH^1_\ns(F,T_i)$ and $\rH^1_\ns(F,T_i^*(1))$ are mutual annihilators under the local Tate pairing
\[
\rH^1(F,T_i)\times\rH^1(F,T_i^*(1))\to\dQ_p/\dZ_p.
\]
It follows that $\rH^1_\ns(F,T)$ and $\rH^1_\ns(F,T^*(1))$ are also mutual annihilators under the local Tate pairing
\[
\rH^1(F,T)\times\rH^1(F,T^*(1))\to\dQ_p/\dZ_p.
\]
Thus, we obtain an isomorphism $\rH^1(F,T)/\rH^1_f(F,T)\simeq\rH^1_\ns(F,T^*(1))^*$ of $\Lambda$-modules. By \eqref{eq:bk} and passing to the colimit, we know that $\rH^1_\ns(F,T^*(1))$ is a quotient $\Lambda$-module of
\[
\varinjlim_{i\in I}\varinjlim_m D(T_i^*[p^m](1))=\varinjlim_{i\in I}D(T_i)^*
=\(\varprojlim_{i\in I}D(T_i)\)^*
\]
(equality of $\Lambda$-modules). It follows that $\rH^1(F,T)/\rH^1_f(F,T)$ is a $\Lambda$-submodule of $\varprojlim_{i\in I}D(T_i)$. Since $D(T_i)$ are free $\Lambda_i$-modules of the same (finite) rank and the transition maps are surjective, $\varprojlim_{i\in I}D(T_i)$ is a free $\Lambda$-module. Part (2) follows.
\end{proof}

\if false

Let $O$ be the ring of integers of a finite extension of $\dQ_p$ and $G$ a topological group that is isomorphic to a finite power of $\dZ_p$. For every subgroup $H<G$, we put $\Lambda_{G/H}\coloneqq O[[G/H]]$. Consider a finite free $\Lambda_G$-module $T$ equipped with a continuous action of $\Gamma_F$. For every subgroup $H<G$, put $T_H\coloneqq T\otimes_{\Lambda_G}\Lambda_{G/H}$. Then we have $\rH^1(F,T)=\varprojlim_H\rH^1(F,T_H)$ in which the limit is taken over all subgroups $H<G$ of finite index; and we define $\rH^1_f(F,T)\coloneqq\varprojlim_H\rH^1_f(F,T_H)$ as a $\Lambda_G$-submodule of $\rH^1(F,T)$.

\begin{lem}\label{le:bk}
Suppose that for every subgroup $H<G$ of finite index, $T_H$ belongs to $\dZ_p[\Gamma_F]^{[a,b]}_\ns$ for integers $a\leq 0\leq b$ independent of $H$, satisfying $b-a\leq p-2$. Then
\begin{enumerate}
  \item For every surjective homomorphism $\Lambda_G$ to an $O$-ring $R$ that is a finite free $O$-module, the natural map $\rH^1(F,T)\to\rH^1(F,T\otimes_{\Lambda_G}R)$ sends $\rH^1_f(F,T)$ surjectively to $\rH^1_f(F,T\otimes_{\Lambda_G}R)$ if $T\otimes_{\Lambda_G}R$ also belongs to $\dZ_p[\Gamma_F]^{[a,b]}_\ns$.

  \item The $\Lambda_G$-module $\rH^1(F,T)/\rH^1_f(F,T)$ is torsion free.
\end{enumerate}
\end{lem}

\begin{proof}
We first consider (1). For every integer $m\geq 1$, put $R_m\coloneqq R\otimes\dZ/p^m$. We first show that the image of $\rH^1_f(F,T)$ is contained in $\rH^1_f(F,T\otimes_{\Lambda_G}R)$. By \eqref{eq:bk1}, it suffices to show that the image of $\rH^1_\ns(F,T)$ under the composite map
\[
\rH^1(F,T)\to\rH^1(F,T\otimes_{\Lambda_G}R)\to\rH^1(F,T\otimes_{\Lambda_G}R_m)
\]
is contained in $\rH^1_\ns(F,T\otimes_{\Lambda_G}R_m)$ for every $m\geq 1$. Since $R_m$ is a finite ring, the homomorphism $\Lambda_G\to R_m$ factors through $\Lambda_{G/H_m}$ for some subgroup $H_m<G$ of finite index. It follows that the image of $\rH^1_\ns(F,T)$ is contained in $\rH^1_\ns(F,T\otimes_{\Lambda_G}R_m)$. Next, we show that the induced map $\rH^1_f(F,T)\to\rH^1_f(F,T\otimes_{\Lambda_G}R)$ is surjective. By \eqref{eq:bk1}, the natural map $\rH^1_f(F,T\otimes_{\Lambda_G}R)\otimes\dF_p\to\rH^1(F,T\otimes_{\Lambda_G}R)\otimes\dF_p$ is injective, which implies that the natural map $\rH^1_f(F,T\otimes_{\Lambda_G}R)\otimes\dF_p\to\rH^1_\ns(F,T\otimes_{\Lambda_G}R_1)$ is injective. By Nakayama's lemma, it suffices to show that the map $\rH^1_f(F,T)\to\rH^1_\ns(F,T\otimes_{\Lambda_G}R_1)$ is surjective. By \eqref{eq:bk}, $\rH^1_\ns(F,-)$ preserves surjectivity. It follows that the maps $\rH^1_\ns(F,T_{H_1})\to\rH^1_\ns(F,T\otimes_{\Lambda_G}R_1)$ and $\rH^1_f(F,T)=\varprojlim_H\rH^1_\ns(F,T_H)\to\rH^1_\ns(F,T_{H_1})$ are both surjective. Part (1) follows.

Now we consider (2). Put $\rH^1_?(F,T^*(1))\coloneqq\varinjlim_H\rH^1_?(F,(T_H)^*(1))$ for $?\in\{\;,\ns\}$. By \eqref{eq:bk1} and \eqref{eq:bk2}, for every subgroup $H<G$ of finite index, $\rH^1_\ns(F,T_H)$ and $\rH^1_\ns(F,(T_H)^*(1))$ are mutual annihilators under the local Tate pairing
\[
\rH^1(F,T_H)\times\rH^1(F,(T_H)^*(1))\to\dQ_p/\dZ_p.
\]
It follows that $\rH^1_\ns(F,T)$ and $\rH^1_\ns(F,T^*(1))$ are also mutual annihilators under the local Tate pairing
\[
\rH^1(F,T)\times\rH^1(F,T^*(1))\to\dQ_p/\dZ_p.
\]
Thus, we obtain an isomorphism $\rH^1(F,T)/\rH^1_f(F,T)\simeq\rH^1_\ns(F,T^*(1))^*$ of $\Lambda_G$-modules. By \eqref{eq:bk} and passing to the colimit, we know that $\rH^1_\ns(F,T^*(1))$ is a quotient $\Lambda_G$-module of
\[
\varinjlim_H\varinjlim_m D((T_H)^*[p^m](1))=\varinjlim_HD(T_H)^*
=\(\varprojlim_HD(T_H)\)^*
\]
(equality of $\Lambda_G$-modules). It follows that $\rH^1(F,T)/\rH^1_f(F,T)$ is a $\Lambda_G$-submodule of $\varprojlim_HD(T_H)$. Since $D(T_H)$ are free $O[G/H]$-modules of the same (finite) rank and the transition maps are surjective, $\varprojlim_HD(T_H)$ is a free $\Lambda_G$-module. Part (2) follows since $\Lambda_G$ is integral.
\end{proof}

\fi

\subsection{Appendix: Generalized bipartite Euler systems}
\label{ss:bipartite}

In this subsection, we generalize Howard's notion of bipartite Euler systems over Artinian local rings.

Let $R$ be a principal Artinian local ring of maximal ideal $\fm$ and residue characteristic $p>3$, with a continuous action of $\Gamma_F$.

Let $T$ be a free $R$-module of finite rank at least two equipped with a compatible continuous action of $\Gamma_F$ and a perfect symmetric pairing
\[
(\;,\;)\colon T\times T\to R(1)
\]
satisfying $(x^\sigma,y^{\tc\sigma\tc})=(x,y)^\sigma$ for every $\sigma\in\Gamma_F$. In particular, for every place $w$ of $F$, we have the local Tate pairing
\[
(\;,\;)_w\colon\rH^1(F_w,T)\times\rH^1(F_{w^\tc},T)\to R
\]
which is perfect. Throughout this subsection, we are given a self-dual, cartesian Selmer structure $(\sF,\Sigma_\sF)$ on $T$ \cite{How06}*{Definition~2.1.1~\&~Definition~2.2.2}.

Let $\fL$ be a set consisting of primes $\fl$ of $F$ satisfying
\begin{itemize}[label={\ding{118}}]
  \item $\fl\not \in\Sigma_\sF\cup\Sigma_{\{2,p\}}$,

  \item $\fl$ is inert over $F^+$,

  \item there is a unique $\Gamma_{F_\fl}$-stable decomposition $T=T_0\oplus T_1\oplus S$ of $R$-modules satisfying $T_0\simeq R$, $T_1\simeq R(1)$, and $\rH^i(F_\fl,S)=0$ for $i\in\dZ$.
\end{itemize}
For every $\fl\in\fL$, define $\rH^1_\ordi(F_\fl,T)$ to be the image of the natural map $\rH^1(F_\fl,T_1)\to\rH^1(F_\fl,T)$, so that we have
\[
\rH^1(F_\fl,T)=\rH^1_\unr(F_\fl,T)\oplus\rH^1_\ordi(F_\fl,T)
\]
in which each summand is a free $R$-module of rank one and is maximal isotropic under the local Tate pairing.

Denote by $\fN$ the set of (possibly empty) finite sets consisting of elements in $\fL$ with \emph{distinct} underlying rational primes. For $\fn\in\fN$, put $\fn\fl\coloneqq\fn\cup\{\fl\}$ for $\fl\in\fL\setminus\fn$, and $\fn/\fl\coloneqq\fn\setminus\{\fl\}$ for $\fl\in\fn$. Following \cite{How06}, we define a \emph{directed} graph $\fX$ with vertices $v(\fn)$ indexed by elements in $\fN$, and arrows $a(\fn,\fn\fl)$ from $v(\fn)$ to $v(\fn\fl)$ for $\fl\in\fL\setminus\fn$. Denote by $\fX^{\r{arrow}}$ the set of arrows of $\fX$.

For every $\fn\in\fN$, define the Selmer structure $\sF(\fn)$ by the formula
\[
\rH^1_{\sF(\fn)}(F_w,T)\coloneqq
\begin{dcases}
\rH^1_\sF(F_w,T), & w\not\in\fn,\\
\rH^1_\ordi(F_w,T), & w\in\fn.
\end{dcases}
\]
Then $(\sF(\fn),\Sigma_\sF\cup\fn)$ is again a self-dual, cartesian Selmer structure on $T$. We denote by $\Sel_{\sF(\fn)}(F,T)$ the corresponding global Selmer group. By the same argument for \cite{How06}*{Proposition~2.2.7}, there exists a unique element $e(\fn)\in\{0,1\}$ and a finitely generated $R$-module $M_\fn$, unique up to isomorphism, such that $\Sel_{\sF(\fn)}(F,T)\simeq R^{e(\fn)}\oplus M_\fn\oplus M_\fn$.

Following \cite{How06}*{Definition~2.2.8}, we denote by $\fN^\even$ and $\fN^\odd$ the subsets of $\fN$ for which $e(\fn)=0$ and $e(\fn)=1$, respectively. For $\fn\in\fN$, define the stub module to be
\[
\Stub_\fn\coloneqq
\begin{dcases}
\fm^{\length(M_\fn)}\cdot R, & \fn\in\fN^\even, \\
\fm^{\length(M_\fn)}\cdot \Sel_{\sF(\fn)}(F,T), & \fn\in\fN^\odd,
\end{dcases}
\]
which is a cyclic $R$-module.

\begin{hypothesis}\label{hy:euler}
We introduce two hypotheses for future use.
\begin{enumerate}
  \item The residue representation $T/\fm T$ is absolutely irreducible.

  \item For every nonzero element $s\in\rH^1(F,T/\fm T)$, there are infinitely many elements $\fl\in\fL$ such that $\loc_\fl(s)\neq 0$.
\end{enumerate}
\end{hypothesis}

\begin{definition}\label{de:euler}
A \emph{generalized bipartite Euler system} for $(T,\sF,\fL)$ is a collection of data
\[
\left(\{\cC_\fn\res \fn\in\fN\},\{\varrho_a\res a\in\fX^{\r{arrow}}\},\{\lambda_\fn\res\fn\in\fN^\even\},\{\kappa_n\res\fn\in\fN^\odd\}\right)
\]
in which
\begin{itemize}[label={\ding{118}}]
  \item $\cC_\fn$ is a finitely generated $R$-module;

  \item $\varrho_a\colon\cC_\fn\to\cC_{\fn\fl}$ is an $R$-linear map for $a=a(\fn,\fn\fl)$;

  \item $\lambda_\fn\colon \cC_\fn\to R$ is an $R$-linear map;

  \item $\kappa_\fn\colon \cC_\fn\to \Sel_{\sF(\fn)}(F,T)$ is an $R$-linear map;
\end{itemize}
satisfying that for every arrow $a=a(\fn,\fn\fl)\in\fX^{\r{arrow}}$, the following holds:
\begin{enumerate}
  \item when $\fn\in\fN^\even$, there exists an isomorphism $\psi_a\colon\rH^1_\ordi(F_\fl,T)\simeq R$ such that
      \[
      \lambda_\fn=\psi_a\circ\loc_\fl\circ\kappa_{\fn\fl}\circ\varrho_a\colon\cC_\fn\to R;
      \]

  \item when $\fn\in\fN^\odd$, there exists an isomorphism $\psi_a\colon\rH^1_\unr(F_\fl,T)\simeq R$ such that
      \[
      \lambda_{\fn\fl}\circ\varrho_a=\psi_a\circ\loc_\fl\circ\kappa_\fn\colon\cC_\fn\to R.
      \]
\end{enumerate}

We say that a vertex $v=v(\fn_*)$ of $\fX$ is an \emph{effective vertex} for a generalized bipartite Euler system, if $\cC_{\fn_*}\neq 0$ and $\varrho_a$ is a bijection for every arrow $a=a(\fn,\fn\fl)$ with $\fn_*\subseteq\fn$.
\end{definition}

In what follows, we fix a generalized bipartite Euler system for $(T,\sF,\fL)$ with an effective vertex $v(\fn_*)$.

\begin{lem}\label{le:euler1}
Assume Hypothesis \ref{hy:euler}. Let $\fn\in\fN$ be an element containing $\fn_*$. If $\lambda_\fn\neq 0$ (resp.\ $\kappa_\fn\neq 0$) when $\fn\in\fN^\even$ (resp.\ $\fn\in\fN^\odd$), then $M_\fn$ does not contain $R$.
\end{lem}

\begin{proof}
This follows from the same proof of \cite{How06}*{Proposition~2.3.5}.
\end{proof}

\begin{notation}\label{no:free}
For every $\fn\in\fN^\odd$, we denote by $\cC_\fn^\free$ the $R$-submodule of $\cC_\fn$ generated by elements $\phi$ such that $\kappa_\fn(\phi)$ is contained in a free $R$-submodule of $\Sel_{\sF(\fn)}(F,T)$.
\end{notation}

\begin{lem}\label{le:euler2}
Assume Hypothesis \ref{hy:euler}. For every element $\fn\in\fN$ containing $\fn_*$, we have $\lambda_\fn(\cC_\fn)\subseteq\Stub_\fn$ when $\fn\in\fN^\even$ and $\kappa_\fn(\cC_\fn^\free)\subseteq\Stub_\fn$ when $\fn\in\fN^\odd$.
\end{lem}

\begin{proof}
This follows from the same proof of \cite{How06}*{Theorem~2.3.7}.
\end{proof}

We have the Euler system sheaf $\r{ES}(\fX)$ (with $\psi^a_v\coloneqq(\psi_a)^{-1}$ for $v=v(\fn)$ with $\fn\in\fN^\even$), the stub sheaf $\Stub(\fX)$, and the core subgraph $\fX^\heartsuit\subseteq\fX$ spanned by core vertices (that is, those $v=v(\fn)$ such that $\Stub_\fn\simeq R$), similarly defined as in \cite{How06}*{\S2.4}.

\begin{definition}
For two vertices $v$ and $v'$ of $\fX$, a \emph{(directed) path} from $v$ to $v'$ in $\fX$ is a finite sequence of vertices $v=v_0,v_1,\ldots,v_k=v'$ such that there is an arrow $a_i$ from $v_i$ to $v_{i+1}$ for $0\leq i\leq k-1$. In particular, if $v=v(\fn)$ and $v'=v(\fn')$, then there exists a path from $v$ to $v'$ if and only if $\fn\subseteq\fn'$.

We say that a path is \emph{invertible} if the map $\psi^{e_i}_{v_i}\colon\Stub(v_i)\to\Stub(e_i)$ is an isomorphism for every $0\leq i\leq k-1$. Then an invertible path from $v$ to $v'$ induces a surjective map $\Stub(v')\to\Stub(v)$.
\end{definition}

\begin{lem}\label{le:euler3}
Assume Hypothesis \ref{hy:euler}.
\begin{enumerate}
  \item Every path of $\fX^\heartsuit$ is invertible.

  \item The core subgraph $\fX^\heartsuit$ is path connected.

  \item For every vertex $v$ of $\fX$, there exists a core vertex $v^\heartsuit$ that can be either even or odd and admits an invertible path from $v$.
\end{enumerate}
\end{lem}

\begin{proof}
The three parts follow from \cite{How06}*{Lemma~2.4.8,~Proposition~2.4.11, and Lemma~2.4.9}, respectively.
\end{proof}

For every $R$-module $M$ and $m\in M$, put $\ind\(m,M\)\coloneqq\sup\{i\in\dN\res m\in\fm^i M\}$. The following proposition generalizes \cite{How06}*{Theorem~2.5.1} with a similar proof.

\begin{proposition}\label{pr:euler}
Assume Hypothesis \ref{hy:euler}. Suppose that we are given a generalized bipartite Euler system for $(T,\sF,\fL)$ with an effective vertex $v(\fn_*)$, satisfying that $\cC_\fn^\free=\cC_\fn$ for every $\fn\in\fN^\odd$ containing $\fn_*$.
\begin{enumerate}
  \item We have
     \[
     \min\left\{\left.\min_{\phi\in\cC_\fn}\ind\(\lambda_\fn(\phi),R\)\right|\fn_*\subseteq\fn\in\fN^\even\right\}
     =\min\left\{\left.\min_{\phi\in\cC_\fn}\ind\(\kappa_\fn(\phi),\Sel_{\sF(\fn)}(F,T)\)\right|\fn_*\subseteq\fn\in\fN^\odd\right\}.
     \]

  \item If we denote by $\delta\leq\infty$ the quantity in (1), then $\{\lambda_\fn(\phi)\res\phi\in\cC_\fn\}$ (resp.\ $\{\kappa_\fn(\phi)\res\phi\in\cC_\fn\}$) generates $\fm^\delta\Stub_\fn$ for every $\fn\in\fN^\even$ (resp.\ $\fn\in\fN^\odd$) containing $\fn_*$.
\end{enumerate}
\end{proposition}

Note that by Lemma \ref{le:euler2}, the image of $\lambda_\fn$ or $\kappa_\fn$ is indeed contained in $\Stub_\fn$ for $\fn\in\fN$ containing $\fn_*$.

\begin{proof}
Note that for every path from $v$ to $v'$ of the core subgraph $\fX^\heartsuit$, the induced map $\Stub(v')\to\Stub(v)$ is an isomorphism. In particular, by Lemma \ref{le:euler3}(1,2), there exists a unique integer $\delta\geq 0$ at most the length of $R$ such that for every core vertex $v=v(\fn)$ with $\fn$ containing $\fn_*$, the image of $\lambda_\fn$ or $\kappa_\fn$ generates $\fm^\delta\Stub_\fn$.

Now let $v=v(\fn)$ be a general vertex of $\fX$ with $\fn$ containing $\fn_*$. By Lemma \ref{le:euler3}(3), we may find a core vertex $v^\heartsuit$ that admits a path from $v$. It follows that the image of $\lambda_\fn$ or $\kappa_\fn$ generates $\fm^\delta\Stub_\fn$. It remains to show that $\delta$ is just the quantity in (1). Indeed, for every $\fn\in\fN^\even$ containing $\fn_*$, we have
\[
\delta\leq\min_{\phi\in\cC_\fn}\ind\(\lambda_\fn(\phi),R\);
\]
and the equality holds if $v(\fn)$ is a core vertex. Similarly, for every $\fn\in\fN^\odd$ containing $\fn_*$, we have
\[
\delta\leq\min_{\phi\in\cC_\fn}\ind\(\kappa_\fn(\phi),\Sel_{\sF(\fn)}(F,T)\);
\]
and the equality holds if $v(\fn)$ is a core vertex. By Lemma \ref{le:euler3}(3), we know that there are both odd and even core vertices admitting paths from $v(\fn_*)$. The proposition follows.
\end{proof}

\end{document}